\documentclass{amsart}

\usepackage{amsmath}
\usepackage{amsfonts}
\usepackage{amssymb}
\usepackage{amsthm}
\usepackage{comment}
\usepackage{epsfig}
\usepackage{psfrag}
\usepackage{mathrsfs}
\usepackage{amscd}
\usepackage[all, import]{xy}
\usepackage{rotating}
\usepackage{lscape}
\usepackage{amsbsy}
\usepackage{verbatim}
\usepackage{moreverb}
\usepackage{pigpen}
\usepackage{eucal}
\usepackage{color}

\makeatletter
\renewcommand\subsubsection{\@startsection {subsubsection}{1}{\z@}%
                                   {-3.5ex \@plus -1ex \@minus -.2ex}%
                                   {2.3ex \@plus.2ex}%
                                   {\normalfont\bf}}
\renewcommand\subsection{\@startsection {subsection}{1}{\z@}%
                                   {-3.5ex \@plus -1ex \@minus -.2ex}%
                                   {2.3ex \@plus.2ex}%
                                   {\normalfont\bf}}
\renewcommand\part{\@startsection {part}{1}{\z@}%
                                   {3.5ex \@plus -1ex \@minus -.2ex}%
                                   {2.3ex \@plus.2ex}%
                                   {\center\large\bf}}
\makeatother
  
\newtheorem{theorem}{Theorem}[subsection]
\newtheorem{prop}[theorem]{Proposition}
\newtheorem{lemma}[theorem]{Lemma}
\newtheorem{cor}[theorem]{Corollary}
\newtheorem{corollary}[theorem]{Corollary}

\newtheorem{image}[theorem]{Figure}
\numberwithin{equation}{section}

\theoremstyle{definition}
\newtheorem{definition}[theorem]{Definition}
\newtheorem{defn}[theorem]{Definition}

\newtheorem{notation}[theorem]{Notation}

\newtheorem{example}[theorem]{Example}
\newtheorem{non-example}[theorem]{Non-example}
\newtheorem{remark}[theorem]{Remark}

\newtheorem{notation/term}[theorem]{Notation/Terminology}
\newtheorem{construction}[theorem]{Construction}

\newtheorem{observation}[theorem]{Observation}
\newtheorem{warning}[theorem]{Warning}

\newtheorem{convention}[theorem]{Convention}

\newtheorem{conjecture}[theorem]{Conjecture}
\newtheorem{question}[theorem]{Question}

\theoremstyle{remark}

\newcommand{\nc}{\newcommand}
\nc{\DMO}{\DeclareMathOperator}
\nc{\tensor}{\otimes}
\nc{\unit}{{1}}
\nc{\del}{\partial}
\DMO{\ass}{\sf{Alg}}
\DMO{\conf}{\sf{Conf}}
\DMO{\Conf}{\sf{Conf}}
\DMO{\open}{\sf{Open}}
\DMO{\Mod}{\mathsf{-Mod}}
\DMO{\Exit}{\sf Exit}
\DMO{\cech}{\mathsf{Cech}}
\DMO{\fun}{\mathsf{Fun}}
\DMO{\sbar}{\mathsf{Bar}}
\DMO{\coend}{\mathsf{Coend}}
\DMO{\modpair}{\mathsf{ModPair}}
\DMO{\maps}{\mathsf{Maps}}
\DMO{\emb}{\mathsf{Emb}}
\DMO{\Ab}{\mathsf{Ab}}
\DMO{\chain}{\mathsf{Chain}}
\DMO{\ob}{\mathsf{obj}}
\DMO{\lan}{\mathsf{Lan}}
\DMO{\orderI}{ {\pi_0\cI^{\sqcup}_{\mathsf{or}}}}
\DMO{\id}{\sf id}
\nc{\diskover}{(\disk_1^{\partial,\fr})_{/[-1,1]}}
\nc{\snglrall}{\mathsf{Snglr}^{\mathsf{csm}}}
\nc{\Sp}{\mathcal{S}p}

\DeclareMathOperator{\Ent}{\sf Entr}
\DeclareMathOperator{\Link}{\sf Link}		
\DeclareMathOperator{\Aut}{\sf Aut}
\DeclareMathOperator*{\colim}{\sf colim}

\DeclareMathOperator{\Fun}{\sf Fun}

\DeclareMathOperator{\Map}{\sf Map}

\DMO{\cmpct}{cmpt}
\DMO{\depth}{{\sf depth}}
\DMO{\snglrr}{Snglr}
\nc{\power}{\mathsf{Power}}

\DeclareMathOperator{\Cat}{\sf Cat_\infty}

\DeclareMathOperator{\shv}{\sf Shv}

\DeclareMathOperator{\op}{\sf op}

\DeclareMathOperator{\Top}{\mathsf{Top}}
\DeclareMathOperator{\Emb}{\mathsf{Emb}}
\DeclareMathOperator{\Diff}{\mathsf{Diff}}
\DeclareMathOperator{\PL}{\mathsf{PL}}

\DeclareMathOperator{\spaces}{\cS\mathsf{paces}}

\DeclareMathOperator{\disk}{\cD\mathsf{isk}}
\DeclareMathOperator{\fr}{\sf fr}

\DeclareMathOperator{\RFbn}{\sf RFib}

\DeclareMathOperator{\Fin}{\mathsf{Fin}}
\DeclareMathOperator{\Strat}{\cS{\sf trat}}
\nc{\Stratd}{{\mathsf{Strat}}}

\DeclareMathOperator{\diff}{\mathsf{Diff}}
\DeclareMathOperator{\mfld}{\cM\mathsf{fld}}

\DeclareMathOperator{\snglr}{\cS\mathsf{nglr}}
\DeclareMathOperator{\Snglr}{\cS\mathsf{nglr}}

\DeclareMathOperator{\bsc}{\cB\mathsf{sc}}
\DeclareMathOperator{\Bsc}{\cB\mathsf{sc}}

\DeclareMathOperator{\unzip}{\mathsf{Unzip}}
\DeclareMathOperator{\Psh}{\mathsf{PShv}}
\DeclareMathOperator{\psh}{\mathsf{PShv}}
\DeclareMathOperator{\Shv}{\mathsf{Shv}}

\nc{\snglrdC}{\mathsf{Snglr}^{C^0}}
\nc{\snglrd}{\mathsf{Snglr}}
\nc{\bscd}{\mathsf{Bsc}}
\nc{\mfldd}{\mathsf{Mfld}}
\nc{\Entd}{{\sf Entr^\delta}}

\DeclareMathOperator{\opens}{{\sf Op}}
\DeclareMathOperator{\cbl}{\sf cbl}

\DeclareMathOperator{\Sing}{\mathsf{Sing}}
\DeclareMathOperator{\Ran}{\mathsf{Ran}}

\DeclareMathOperator{\oo}{\infty}

\newcommand{\ra}{\rightarrow}
\newcommand{\la}{\leftarrow}
\newcommand{\xra}{\xrightarrow}
\newcommand{\xla}{\xleftarrow}
\newcommand{\ov}{\overline}
\newcommand{\un}{\underline}
\newcommand{\w}{\widetilde}
\newcommand{\into}{\hookrightarrow}

\newcommand{\Set}{\mathsf{Set}}
\newcommand{\StTop}{\mathsf{StTop}}

\newcommand{\lag}{\langle}
\newcommand{\rag}{\rangle}
\newcommand{\kink}{\mathsf{Kink}}
\newcommand{\Kan}{\sf Kan}
\newcommand{\kan}{\sf Kan}
\newcommand{\POSet}{\mathsf{Poset}}
\newcommand{\tr}{\triangleright}
\newcommand{\tl}{\triangleleft}

\newcommand{\stopen}{\mathsf{StTop}^{\mathsf{open}}}

\newcommand{\pushout}{\ar@{}[dr]|{\text{\pigpenfont R}}}
\newcommand{\pullback}{\ar@{}[dr]|{\text{\pigpenfont J}}}
\newcommand{\sm}{\mathsf{Sm}}

\newcommand{\bdelta}{\boldsymbol\Delta}

\def\cA{\mathcal A}\def\cB{\mathcal B}\def\cC{\mathcal C}\def\cD{\mathcal D}
\def\cE{\mathcal E}\def\cF{\mathcal F}
\def\cI{\mathcal I}\def\cJ{\mathcal J}\def\cK{\mathcal K}\def\cL{\mathcal L}
\def\cM{\mathcal M}\def\cP{\mathcal P}
\def\cS{\mathcal S}\def\cT{\mathcal T}
\def\cU{\mathcal U}

\def\AA{\mathbb A}\def\DD{\mathbb D}

\def\NN{\mathbb N}\def\PP{\mathbb P}
\def\RR{\mathbb R}

\def\ZZ{\mathbb Z}

\def\sB{\mathsf B}\def\sC{\mathsf C}\def\sD{\mathsf D}
\def\sG{\mathsf G}
\def\sL{\mathsf L}
\def\sN{\mathsf N}\def\sO{\mathsf O}
\def\sS{\mathsf S}\def\sT{\mathsf T}

\def\bDelta{\mathbf\Delta}

\def\fP{\mathfrak P}

\begin{document}

\title{Local structures on stratified spaces}
\author{David Ayala}
\author{John Francis}
\author{Hiro Lee Tanaka}
\date{}

\address{Department of Mathematics\\Montana State University\\Bozeman, MT 59717}
\email{david.ayala@montana.edu}
\address{Department of Mathematics\\Northwestern University\\Evanston, IL 60208}
\email{jnkf@northwestern.edu}
\address{Department of Mathematics\\Harvard University\\Cambridge, MA 02138}
\email{hirolee@math.harvard.edu}
\thanks{DA was partially supported by ERC adv.grant no.228082, by the NSF under Award 0902639, and by the NSF Award 0932078 000 while residing at the MSRI for Spring 2014. JF was supported by the NSF under Award 0902974 and Award 1207758; part of this paper was written while JF was a visitor at Paris 6 in Jussieu. HLT was supported by an NSF Graduate Research Fellowship, by the Northwestern University Office of the President, by the Centre for Quantum Geometry of Moduli Spaces, and by the NSF under Award DMS-1400761.}

\begin{abstract}
We develop a theory of conically smooth stratified spaces and their smooth moduli, including a notion of classifying maps for tangential structures. We characterize continuous space-valued sheaves on these conically smooth stratified spaces in terms of tangential data, and we similarly characterize $1$-excisive invariants of stratified spaces.  These results are based on the existence of open handlebody decompositions for conically smooth stratified spaces, an inverse function theorem, a tubular neighborhood theorem, an isotopy extension theorem, and functorial resolutions of singularities to smooth manifolds with corners.  
\end{abstract}



\maketitle

\tableofcontents

\section{Introduction}

Our present work is a foundational exposition of a theory of stratified spaces founded on the new notion of conical smoothness. 
There are already several foundations for stratified spaces, each deriving from the differing purposes of several renowned inquiries. 
Our purpose and investment in the theory of stratifications derives from the following conviction: stratifications form a basis for locality in topological quantum field theory, where the homotopy type of a moduli space of stratifications on a manifold forms the source of all local invariants of the manifold. As such, our theory deviates from the preexisting theories so as to simultaneously accommodate two priorities: smooth geometry and robust behavior in families.

\smallskip

Stratified spaces received their modern inception at the hands of Whitney, Thom, and Mather -- see \cite{whitney1}, \cite{whitney2}, \cite{thom}, and \cite{mather, mather73} -- due to questions of generic behavior of spaces of solutions of algebraic and analytic equations. Whitney was concerned with the triangulability of algebraic varieties; he showed that a singular variety admits a decomposition by smooth nonsingular manifolds~\cite{whitney0}. Thom was concerned with dynamics, such as when a smooth map $M\ra N$ was differentiably stable -- in particular surrounded in the space $C^{\oo}(M,N)$ by mappings with diffeomorphic level sets. One of his questions was when differentiable stability is generic; this was then solved in a series of papers by Mather -- see \cite{goresky} for a historical survey.

\smallskip

This geometric theory developed further in the intersection homology and stratified Morse theory of Goresky \& MacPherson in \cite{goreskymacpherson}, \cite{goreskymacpherson2}, and \cite{stratmorse}. Their work continues Whitney's, as it extends bedrock results such as Poincar\'e duality and the Lefschetz hyperplane theorem to singular varieties. Other sources of expansion have been the study of constructible sheaves, of (micro)supports, and of $D$-modules on manifolds; see \cite{kashiwaraschapira} and references therein. A different development still of the theory was given by Baas and Sullivan, see \cite{baas}, who defined cobordism of manifolds with singularities; their purpose was to define generalized homology theories whose cycles were singular manifolds.

\smallskip

In these studies, the focus of these authors is largely on stratified spaces, one at a time -- there is little development of the theory in {\it families}, either on maps between stratified spaces, or on continuous families of stratified spaces themselves. 
Many reasons exist for this omission. For Whitney, stratifications were a tool to study algebraic varieties, where there already exists an excellent theory of families. Another reason is that the naive notion of a map of Whitney stratified spaces (i.e., a stratum preserving map which is smooth on strata separately), quickly leads to pathologies in families: the total space of a naive bundle of Whitney stratified spaces over a Whitney stratified base need not itself be Whitney stratified. Examples follow as a consequence of pseudo-isotopy theory, which we will explain below.

\smallskip

There does, however, exist a homotopy-theoretic theory of stratifications, advanced by Siebenmann \cite{siebenmann} and Quinn \cite{quinn}, which possesses robust behavior in families. Siebenmann constructed spaces of stratified maps between his locally-cone stratified sets, and he showed these spaces have good local and homotopy-theoretic behavior. In particular, he proved local contractibility, which gives an isotopy extension theorem as well as the existence of classifying spaces for fiber bundles. Quinn likewise proved an isotopy extension theorem valid in higher dimensions. These results have no counterpart in the geometric theory following Whitney. However, the homotopy-theoretic stratifications are insufficient for our more geometric arguments due to basic features as nonexistence of tubular neighborhoods (see \cite{normal}) and absence of transversality.

\smallskip

Our goal in this work is thus a theory of smoothly stratified spaces which is well-behaved in families, possessing the fine geometric features of Whitney's theory on objects and the robust features of Siebenmann's theory in families, which combine so as to give strong regularity in families. In particular, we require well-behaved spaces of smooth maps, embeddings, and tangential structures (such as orientations, spin structures, or framings) on stratified spaces.

\smallskip

We satisfy these requirements by the introduction of conical smoothness of maps. The notion of conical smoothness is intrinsic and makes no reference to an ambient smooth manifold, as in Whitney's definition. It implies strong regularity along closed strata -- in particular, the Whitney conditions -- so that there exist tubular neighborhoods along singularity loci. We prove an inverse function theorem, which simultaneously implies many of the excellent features in the geometric and topological theories: on the geometric side, it should imply an openness of transversality like that due to Trotman \cite{trotman} (see Conjecture \ref{conj.transversality}); like on the topological side, it implies an isotopy extension theorem, which here follows by standard arguments from the existence of tubular neighborhoods in our conically smooth theory. 

\smallskip

Most essentially, conical smoothness allows us to prove that the natural map 
\begin{equation}
\label{eqn.cone}
\Aut(X)\xra{~\simeq~} \Aut\bigl(\sC(X)\bigr)
\end{equation}
is a homotopy equivalence for each compact conically smooth stratified space $X$. That is, the space of conically smooth automorphisms of an open cone $\sC(X):= X\times[0,1)\amalg_{X\times\{0\}}\{0\}$ is homotopy equivalent to the space of conically smooth automorphisms of the link $X$ around the cone point. (See \S\ref{sec:endomorphisms}.) As a consequence, our theory has the advantage of avoiding the aforementioned pathologies. 
In contrast, the naive theory of smooth families is marbled with pseudo-isotopies, as we now explain. 
For $Z$ a smoothly stratified space, let $\Aut^{\sf naive}(Z)$ be the subspace of those stratum-preserving homeomorphisms which restrict to diffeomorphisms on each stratum (a weaker condition than conical smoothness). Consider the basic example of $Z=\sC(X)$ with $X$ a compact smooth manifold; so $Z=\sC(X)$ is a stratified space with two strata, the cone point $\{0\}$ and its complement $X\times (0,1)$. 
In this case, restriction to the complement of the cone point defines an isomorphism 
\[
\Aut^{\sf naive}\bigl(\sC(X)\bigr) \xra{~\cong~} \Diff^+\bigl(X\times (0,1) \bigr)
\]
with the subspace $\Diff^+(X\times(0,1)) \subset \Diff(X\times (0,1))$ consisting of those components that preserve the \emph{ends} of $X\times (0,1)$: a diffeomorphism $g$ of $X\times (0,1)$ belongs to this subspace 
if, for each point $x\in X$, the limit of the composite $g: \{x\}\times (0,1)\ra X\times (0,1) \ra [0,1)$ is zero as $t\in(0,1)$ tends to zero. 

Now, choose a smooth manifold $K$ as it fits into the solid diagram
\[
\xymatrix{
&&
{\sf BDiff}(X)  \ar[d]^-{-\times \id_{(0,1)}}  \ar[drr]^-{\sC(-)}
&&
\\
K  \ar@{-->}[urr]^-{\nexists}  \ar[rr]  
&&
{\sf BDiff}^+(X\times (0,1))  
&&
{\sf BAut}^{\sf naive}(\sC(X))  \ar[ll]_-{\cong}
}
\]
yet for which there does not exist a dashed lift, as indicated.
Such a diagram exists because the
product map $\Diff(X) \ra \Diff^+(X\times (0,1))$ is not in general a homotopy equivalence, due to the relation of $\Diff^+(X\times (0,1))$ with pseudo-isotopies of $X$. This diagram classifies a fiber bundle $E\to K$ whose fibers are naively isomorphic to $\sC(X)$ and whose fiberwise link, which is a fibration $E_0 \ra K$ with fibers homotopy equivalent to $X$, is not concordant to a smooth fiber bundle.
In summary, the naive theory of smooth families of stratified spaces leads to links which lack smooth structure. (Compare with the results of Hughes--Taylor--Weinberger--Williams~\cite{htww}.) 
Consequently, there
can be no functorial resolutions of singularities for this naive theory. 
Our conically smooth theory is designed to remedy this: we prove that links in our theory are again conically smooth and that there is a functorial resolution of singularities.

\smallskip

The above discussion argues for the essential correctness of the homotopy type of spaces of conically smooth automorphisms. In our theory, we also consider spaces of conically smooth embeddings $\Emb(X,Y)$ and spaces of all conically smooth maps $\Strat(X,Y)$. The construction of spaces of conically smooth embeddings between stratified spaces allows for a definition of a tangent bundle for a stratified space: it is the value of a functor
\[
\tau\colon \snglr \longrightarrow \Psh(\bsc)
\]
from stratified spaces and open embeddings among them, to presheaves of spaces on basic singularity types. This specializes to the usual concept in the case of a smooth $n$-manifold for the following reason: the structure of the tangent bundle of a smooth $n$-manifold $M$ is equivalent to the structure of the presheaf of spaces $\Emb(-,M)$, defined by smooth embeddings into $M$, on the singularity type $\RR^n$. This equivalence is implemented by the equivalence between ${\sf GL}(\RR^n)$-bundles and $\Emb(\RR^n,\RR^n)$-bundles due to the homotopy equivalence of topological monoids ${\sf GL}(\RR^n)\simeq \Emb(\RR^n,\RR^n)$.

\smallskip

The relative version of our main result, Theorem~\ref{theorem.exit-tangent}, states that for a stratified space in our theory, there are the following equivalences:
\begin{equation}
\label{eqn.pshv}
\shv^{\cbl}(X)~ \simeq ~\psh\bigl(\Ent(X)\bigl)~ \simeq~ \shv\bigl(\snglr_{/X}\bigr)~.
\end{equation}
The first term consists of sheaves on the underlying topological space of $X$ which are constructible with respect to a given filtration; the second is presheaves on the enriched category $\Ent(X)$ of basic singularity types embedded into $X$; the third is the enriched overcategory of all stratified spaces openly embedded into $X$, where a morphism is an isotopy of embeddings over $X$. 
In the latter two cases, our (pre)sheaves are continuous with respect to the topology we have endowed on embedding sets. 
We thus advance (\ref{eqn.pshv}) as a second confirmation of our construction of spaces of conically smooth maps between stratified spaces; it shows that this topological enrichment comes along for free when one considers constructible sheaves. It also proves Corollary \ref{exitpathequivalence}, that our category $\Ent(X)$ of basic singularity types embedded in $X$ is equivalent to the opposite of the exit-path $\oo$-category of $X$ defined by Lurie \cite{HA}, after MacPherson and Treumann \cite{treumann}. For a third confirmation, for the homotopy type of all stratified maps, see (\ref{eqn.striation}).

\smallskip

To prove these results, we tailor a significant amount of differential topology for stratified spaces, so as to make parametrized local-to-global arguments over stratified spaces as one can over usual smooth manifolds. 
To give one example, we prove in Theorem \ref{open-handles} that stratified spaces have open handlebody decompositions; this is analogous to Smale's theorem, used in the proof of the h-cobordism theorem \cite{smale}, that usual smooth manifolds have handlebody decompositions.  (A similar outcome can likely be managed through the work of Goresky--MacPherson~\cite{stratmorse}) on stratified Morse theory.)
To give another example, we produce as Proposition~\ref{tot-unzip} a functorial resolution of singularities procedure, the unzipping construction; this is a useful technique for maneuvering between stratified spaces and manifolds with corners, where classical differential topology can be applied. 
While this resolution of singularities has been examined in other contexts of stratified spaces (see~\cite{almp} for an account, after unpublished work of Richard Melrose), our treatment is tailored to achieve the homotopy coherence required for Part 2, as will be overviewed shortly.  
We conceive this package of results as a d\'evissage of stratified structures in the sense of Grothendieck \cite{esquisse}. Any number of general results about stratified spaces can be proved by applying these techniques:  induction on depth and resolution of strata, via unzipping, to manifolds with corners.

\begin{remark}\label{oo.conventions}
In this work, we use Joyal's {\it quasi-category} model of $\oo$-category theory \cite{joyal}. 
Boardman \& Vogt first introduced these simplicial sets in \cite{bv} as weak Kan complexes, and their and Joyal's theory has been developed in great depth by Lurie in \cite{HTT} and \cite{HA}, our primary references; see the first chapter of \cite{HTT} for an introduction. We use this model, rather than model categories or simplicial categories, because of the great technical advantages for constructions involving categories of functors, which are ubiquitous in this work. 
More specifically, we work inside of the quasi-category associated to the model category of Joyal.  In particular, each map between quasi-categories is understood to be an iso- and inner-fibration; and (co)limits among quasi-categories are equivalent to homotopy (co)limits with respect to Joyal's model structure.

We will also make use of categories, as well as categories enriched in topological spaces or Kan complexes, such as the $\Kan$-enriched category $\snglr$ of stratified spaces and conically smooth open embeddings among them. These are comparable via the following functors:
\[
\Bigl\{{\rm topological} \medspace {\rm categories} \Bigr\}\overset{{\sf Sing}}\longrightarrow
\Bigl\{\Kan\text{-}{\rm enriched} \medspace {\rm categories} \Bigr\}\overset{\sN}\longrightarrow
\Cat
\]
The first functor assigns to a topological category $\cC$ the $\Kan$-enriched category with the same objects and with enrichment defined, for any two objects $x$ and $y$ in $\cC$, as
\[
\Sing\cC(x,y):=\Sing\bigl(\cC(x,y)\bigr),\]
the singular complex of the topological space of maps from $x$ to $y$. The second functor, the simplicial nerve $\sN$, is a generalization of the usual nerve of a category; for details see \S1.1.5 of \cite{HTT}, after \cite{cordier}. In the following, by a functor $\cS \ra \cC$ from a $\Kan$-enriched category such as $\snglr$ to an $\oo$-category $\cC$ we will always mean a functor $\sN\cS \ra \cC$ from the simplicial nerve, suppressing the $\sN$ from the notation. In particular, see Convention \ref{defn.enrichment}. In this way, we will in particular regard an ordinary category as an $\oo$-category, by taking its nerve.

Both of these two functors are, in a suitable sense, equivalences. In particular, every $\oo$-category is equivalent to one coming from a topological category. So the reader uncomfortable with $\oo$-categories can substitute the words ``topological category" for ``$\oo$-category" wherever they occur in this paper to obtain the correct sense of the results, but they should then bear in mind the proviso that technical difficulties may then abound in making the statements literally true.
\end{remark}

We now describe the linear contents of our work in detail. The paper splits in three parts. The first gives definitions, in which we make precise the notion of a stratified space and structures on it. The second part characterizes the category of tangential structures as 1-excisive functors. The third part is devoted to what one might call the differential topology of stratified spaces, in which we prove that every finitary stratified space is the interior of some compact stratified space (possibly with boundary), and that many questions about stratified spaces can be reduced to a question about smooth manifolds with corners, via resolutions of singularity. The following sections review the contents of each of these parts in more detail; we recommend the reader read this moderately thorough summary before delving into the main body.

\subsection{Overview of Part 1: stratified spaces and tangential structures}
We begin in \S\ref{section.C0-singular-manifolds} by defining the notion of a {\em topological}, or $C^0$, stratified space. A $C^0$ stratified space is always a paracompact topological space with a specified filtration (see Remark \ref{rem.filtration}), but not all filtered spaces are examples. Just as a topological manifold must look locally like $\RR^n$, a $C^0$ stratified space must look locally like a space of the form
\begin{equation}
\label{eqn.local-cone}
 \RR^n \times \sC(X).
\end{equation}
Here, $X$ is a compact, lower-dimensional, $C^0$ stratified space, and $\sC(X)$ is the open cone on $X$. Thus, the definition is inductive. For example, if $X = \emptyset$, its open cone is a single point, so this local structure is the structure of $\RR^n$ -- i.e., the local structure of a topological $n$-manifold. If $X$ is a point, its open cone is a copy of $\RR_{\geq 0}$, and the local structure is that of a topological manifold with boundary. Since spaces of the form $\RR^n \times \sC(X)$ are the basic building blocks of $C^0$ stratified spaces, we call them {\em $C^0$ basics}.

In \S\ref{section.singular-manifolds}, we define the notion of a {\em conically smooth atlas} on a $C^0$ stratified space, and define a {\em stratified space} as a $C^0$ stratified space equipped with such an atlas. As one might expect, one recovers the usual notion of a smooth atlas by considering stratified spaces whose neighborhoods are all of the form (\ref{eqn.local-cone}) for $X = \emptyset$.
 Underlying the notion of an atlas is the notion of a {\em conically smooth} embedding between stratified spaces, which we also define in \S\ref{section.singular-manifolds}. 

By the end of \S\ref{section.C0-singular-manifolds} and \S\ref{section.singular-manifolds}, we will have constructed two categories: the category of stratified spaces, $\snglrd$, and the full subcategory of basic stratified spaces, $\bscd$. Objects of $\bscd$ are stratified spaces whose underlying topological space is $\RR^n \times \sC(X)$, with an atlas induced by an atlas of $X$. We call these ``basics'' because they are the basic building blocks of stratified spaces: Every object of $\snglr$ looks locally like an object of $\bsc$. We may also refer to them as singularity types. Finally, the morphisms in these categories are given by open embeddings compatible with their atlases.

We can enrich these categories over Kan complexes, and we will denote their associated $\infty$-categories $\bsc$ and $\snglr$. These discrete and enriched categories are related by the following diagram:
\[\xymatrix{
\bscd\ar[d]\ar@{^{(}->}[r]^\iota&\snglrd\ar[d]\\
\bsc~\ar@{^{(}->}[r]^\iota&~\snglr\\
}
\]
where the functor $\snglrd \ra \snglr$ is induced by the inclusion of the underlying set on morphism spaces, equipped with the discrete topology. 

\begin{example}
The inclusion $\iota: \bsc \into \snglr$ has an analogue in the smooth setting.
Let $\mfld_n$ be the $\infty$-category of smooth $n$-manifolds, whose morphisms are open, smooth embeddings between them, and whose topology on mapping spaces is the compact-open $C^\infty$ topology (see Remark \ref{oo.conventions}). Let $\sD_n$ denote the $\infty$-category with a single object called $\RR^n$, whose endomorphisms are the space of smooth, open embeddings from $\RR^n$ to itself. There is a pullback diagram
\[
 \xymatrix{
 \sD_n \ar[r] \ar[d]
 & \mfld_n \ar[d]\\
 \bsc \ar[r]^-{\iota}
 & \snglr
 }
\]
where all arrows are fully faithful inclusions of $\infty$-categories.
\end{example}

\begin{remark}
We expect that a Whitney stratified space is an example of a conically smooth stratified space. 
See Conjecture~\ref{whitneys-count}.
\end{remark}

\subsubsection{The tangent classifier, and structures on stratified spaces}
Equipped with the categories $\bsc$ and $\snglr$, we can define what it means to put a {\em structure} on a stratified space. 
Let us first consider the smooth case. In the smooth setting, examples of structures are given by reducing the structure group of a tangent bundle to some group $G$ with a homomorphism to $\sO(n)$. Importantly, these examples can be expressed in the language of homotopy theory:
Let $B = \sB G$, and consider the fibration $B \to \sB\sO(n)$ induced by the group homomorphism to $\sO(n)$.
Given a smooth $n$-manifold $X$, one has a natural map $\tau_X : X \to \sB\sO(n)$ classifying the tangent bundle of $X$, and one can define a structure on $X$ to be a lift of $\tau_X$ to $B$:
\begin{equation}
\label{eqn.classical-structures}
 \xymatrix{
 & B \ar[d] \\
 X \ar[r]_-{\tau_X} \ar[ur]
 & \sB\sO(n).
 }
\end{equation}
The space of such lifts is the space of structures one can put on $X$. To define the notion of a structure on a stratified space, we generalize both $\tau_X$ and $B \to \sB\sO(n)$ to the stratified setting. 

We begin with $\tau_X$. Phrased more universally, $\tau_X$ defines a functor $\tau:\mfld_n \to \spaces_{/\sB\sO(n)}$ which sends a smooth manifold $X$ to the classifying map $\tau_X$. By the Grothendieck construction, a space living over $\sB\sO(n)$ is the same data as a {\em functor} from the $\infty$-groupoid $\sB\sO(n)$ to the $\infty$-category of spaces. By replacing $X \to \sB\sO(n)$ with a Kan fibration, one may consider this functor to be contravariant -- i.e., as a presheaf on $\sB\sO(n)$. Moreover, there is an equivalence of $\infty$-categories $\sD_n \simeq \sB\sO(n)$, essentially by taking an embedding to its derivative at the origin; see Theorem~\ref{theorem.basics-easy}. Hence, the data of $X \to \sB\sO(n)$ is the same data as a presheaf on the $\infty$-category $\sD_n$, and this defines an equivalence of $\infty$-categories
\begin{equation}
\label{eqn.presheaf-spaces-over}
 \psh(\sD_n) \simeq \spaces_{/\sB\sO(n)}
\end{equation}
between presheaves on $\sD_n$, and spaces living over $\sB\sO(n)$.
Since the singular version of $\sD_n$ is precisely $\bsc$, this suggests that a generalization of $\tau_X$ is given by associating to every $X$ a presheaf over $\bsc$. There is a natural candidate for such a family:

\begin{definition}[Tangent classifier]\label{def:tangents}
\label{definition.tangent-classifier}
\label{def.tangent-classifier}
The \emph{tangent classifier} is the composite functor
	\begin{equation}
	\label{eqn.tangent-classifier}
	\tau \colon\snglr \xra{\text{Yoneda}} \Psh(\snglr) \xra{\iota^\ast} \Psh(\bsc).
	\end{equation}
In particular, given a stratified space $X$, the functor $\tau(X)$ assigns to each basic $U$ the space of conically smooth open embeddings $U \into X$.
\end{definition}

\begin{example}
When $\tau$ is restricted to $\mfld_n \subset \snglr$, we recover  from the smooth setting the functor $\mfld_n \to \spaces_{/\sB\sO(n)}$. To see this, we will construct a pullback square
\[
 \xymatrix{
 \mfld_n \ar[r] \ar[d]
 & \psh(\sD_n)\ar[d] \\
 \snglr \ar[r]^-\tau
 & \psh(\bsc).
 }
\]
In Corollary~\ref{classical-tangent}, we prove that the composite $\mfld_n \to \psh(\sD_n) \simeq \spaces_{/\sB\sO(n)}$ recovers the usual functor taking a manifold $X$ to the map $\tau_X: X \to \sB\sO(n)$.
\end{example}

In the following definition of the enter-path $\oo$-category $\Ent(X)$, we use the {\it unstraightening construction}, a version of the Grothendieck construction for $\infty$-categories: for every presheaf $\cF\in \psh(\cC)$, one can construct an $\infty$-category $\cE$ with a {\em right fibration} $\cE \to \cC$. By construction, $\cE$ is the $\oo$-overcategory
\[
\cE = \cC_{/\cF}:=\cC\underset{\psh(\cC)}\times\psh(\cC)_{/\cF}
\]
consisting of pairs $c\in\cC$ with a natural transformation from the Yoneda image of $c$ to $\cF$.
This unstraightening construction defines an equivalence between presheaves on $\cC$ and right fibrations over $\cC$. (We review this in \S\ref{section.right-fibration}). 

\begin{defn}
\label{defn.E(X)}
For $X$ a conically smooth stratified space, the \emph{enter-path $\oo$-category} is
\[
\Ent(X) := \bsc_{/X}
\]
the $\oo$-overcategory of basics embedded in $X$.
The right fibration
	\begin{equation}\label{hat}
	\tau_X: \Ent(X) = \bsc_{/X} \to \bsc
	\end{equation}
is the forgetful functor to basics; or, equivalently, $\tau_X$ is the unstraightening of the tangent classifier $\tau(X): \bsc^{\op}\ra \spaces$ of Definition \ref{def:tangents}.
The space $\Ent_U(X)$ is the fiber of $\tau_X$ over $U \in \bsc$.  
The right fibration
	\begin{equation}
	\Entd(X) := \bscd_{/X} \to \bscd
	\end{equation}
is the unstraightening of the restriction of the discrete tangent classifier $\tau(X)$ to the discrete category of basics, i.e., the composite $\bscd^{\op} \ra \snglrd^{\op}\xra{\Emb(-,X)} \Set$.
\end{defn}

Now that we have generalized $\tau_X$ to the stratified setting, we likewise generalize the notion of the structure $B \to \sB\sO(n)$. By the same reasoning as in (\ref{eqn.presheaf-spaces-over}), we replace the space $\sB$ over $\sB\sO(n)$ by a presheaf on $\bsc$ -- that is, by a right fibration to $\bsc$.

\begin{definition}[Category of basics]\label{def.category-of-basics}
An \emph{$\infty$-category of basics} is a right fibration $\cB\to \bsc$.  
\end{definition}

\begin{definition}[$\cB$-manifolds]
Fix an $\infty$-category of basics $\cB\to\bsc$. The $\infty$-category of $\cB$-manifolds is the pullback
\[
\xymatrix{
\mfld(\cB)  \ar[r]  \ar[d]
&
(\RFbn_{\bsc})_{/\cB}  \ar[d]
\\
\snglr  \ar[r]^-{\tau}
&
\RFbn_{\bsc}~.
}
\]
Explicitly, a $\cB$-manifold is a pair $(X,g)$ where $X$ is a stratified space and 
	\[
	\xymatrix{
	&
	\cB  \ar[d]
	\\
	\Ent(X) \ar[r]_-{\tau_X}  \ar[ur]^-g
	&
	\bsc
	}
	\]
is a lift of the tangent classifier.
Unless the structure $g$ is notationally topical, we will denote a $\cB$-manifold $(X,g)$ simply by its underlying stratified space, $X$.  
\end{definition}

\begin{example}\label{two-examples}
In the stratified setting, the choice of $\cB$ has a new flavor which is not present in the smooth setting. Rather than simply putting a structure (like an orientation) on a stratified space, $\cB$ can also restrict the kinds of singularities that are allowed to appear.
\begin{enumerate}
\item The inclusion $\sD_n \to \bsc$ is an $\infty$-category of basics, and $\mfld(\sD_n)$ is precisely the category of smooth $n$-manifolds. 

\item Let $\cB \to \bsc$ be the composite $\sB\sS\sO(n) \to \sB\sO(n) \simeq \sD_n \to \bsc$. The resulting $\oo$-category $\mfld(\cB)$ is that of smooth, {\em oriented}, $n$-manifolds. Morphisms are open embeddings respecting the orientation. 
\end{enumerate}
\end{example}

\begin{remark}
Example~\ref{two-examples} illustrates that there are two conceptual roles that an $\infty$-category of basics plays.
The first role is as a declared list of singularity types; this list has the property that if a manifold is allowed to look locally like $U$, it is allowed to look locally like any $V$ admitting an open embedding to $U$.  
The second role is as additional structure on each singularity type in this list (such as an orientation, or a map to a background space).  This additional structure pulls back along inclusions of basics.
In Theorem~\ref{theorem.basics-easy}, we will articulate in what sense $\bsc$ is a twisting of a poset (one basic may openly embed into another, but not vice-versa unless the two are equivalent) with a collection of $\infty$-groupoids (the automorphisms of each singularity type). 
It then follows formally that any $\infty$-category of basics $\cB \to \bsc$ factors as a sequence of right fibrations $\cB \to \ov{\cB} \to \bsc$ with the first essentially surjective, and the second fully faithful.  
\end{remark}

\begin{remark}
Tangential structures play an important role in the theory of factorization homology by influencing the algebraic structures we consider. For instance, factorization homology for oriented 1-manifolds requires the input data of a unital associative algebra, while factorization homology for unoriented 1-manifolds takes as input a unital, associative algebra with an involution. Generalizing the notion of structures to the stratified setting allows us to impose such algebraic structures on module actions as well.
\end{remark}

\subsection{Overview of Part 2: structures are 1-excisive}

While the definitions of Part 1 were natural, they were not justified by universal properties. We provide such characterizations in Part 2, and we view these characterizations as some of the main results of our work.

First, consider the smooth case. While we tacitly stated that interesting structures on a smooth manifold $X$ can be understood as lifts of the map $\tau_X: X \to \sB\sO(n)$, one can actually characterize such structures using a different property.

Namely, let $\cF: \mfld_n^{\op} \to \spaces$ be a functor of $\infty$-categories -- i.e., a space-valued presheaf. By definition, $\cF$ respects the topology of embedding spaces. Moreover, if we want to consider structures that glue together, we should require that $\cF$ be a {\em sheaf} on the usual site of smooth manifolds. Hence, by a structure, one should mean an object $\cF \in \shv(\mfld_n)$. We have the following comparison to the notion of $\spaces_{/\sB\sO(n)}$, as with (\ref{eqn.presheaf-spaces-over}).

\begin{theorem}\label{theorem.smooth-basic-sheaf}
Let $\shv(\mfld_n)$ denote the $\infty$-category of sheaves on $\mfld_n$ with values in spaces. There is an equivalence of $\infty$-categories
\[
 \shv(\mfld_n)
 \simeq
 \spaces_{/\sB\sO(n)} .
\]
\end{theorem}

\begin{remark}
This result is new is presentation only. The essential idea is an old one, which underlies Smale--Hirsch immersion theory, scanning maps, and the parametrized h-principle in general.
\end{remark}

By (\ref{eqn.presheaf-spaces-over}), the equivalence in Theorem~\ref{theorem.smooth-basic-sheaf} may be written as an equivalence 
\begin{equation}
\label{eqn.smooth-presheaves}
 \shv(\mfld_n) \simeq \psh(\sD_n).
\end{equation}
In particular, a presheaf on basics induces a sheaf on the site of all smooth $n$-manifolds. Our first result of Part 2 is a natural generalization of this observation. 
Note that the full inclusion $\iota \colon \bsc\hookrightarrow\snglr$ defines an adjunction
\begin{equation}\label{extend-basics}
 \iota^\ast \colon \Psh(\snglr) \rightleftarrows \Psh(\bsc)\colon \iota_\ast
\end{equation}
with the right adjoint given by right Kan extension. 

\begin{theorem}[Continuous structures are tangential (absolute case)]\label{sheaves-basics}
\label{theorem.sheaves-basics}
The adjunction~(\ref{extend-basics}) restricts as an equivalence of $\infty$-categories
\[
 \Shv(\snglr)~\simeq~\Psh(\bsc)~.
\]
\end{theorem}

\begin{remark}
The above result says that a sheaf on $\snglrd$ that is \emph{continuous}, by which we mean it is equipped with an extension to $\snglr$, is given from a fiberwise structure on tangent `bundles', which is a structure dependent only on the basics.
\end{remark}

There is also a {\em relative} version of Theorem~\ref{theorem.sheaves-basics}, for the $\infty$-category $\snglr_{/X}$ of stratified spaces living over a fixed stratified space $X$. In the terminology of Goodwillie--Weiss manifold calculus~\cite{goodwillie-weiss}, one might think of Theorem~\ref{theorem.sheaves-basics} as a ``context-free'' version of the relative case. A beautiful outcome of the relative case is that, since we are working relative to a stratified space $X$, one can ask what homotopical data the {\em stratification} on $X$ carries. It turns out that continuous sheaves which respect the stratification on $X$ (i.e., representations of the exit-path category of $X$) are precisely the sheaves on $\snglr_{/X}$.

To state the relative version of Theorem~\ref{theorem.sheaves-basics}, fix a stratified space $X$. Recall that a sheaf $\cF$ on $X$ is called {\em constructible} if its restriction to every stratum is locally constant. (See Definition~\ref{def.cbl-sheaf} for a precise definition.) We denote by $\shv^{\cbl}(X) \subset \shv(X)$ the full subcategory of constructible sheaves on $X$. 

\begin{theorem}[Continuous structures are tangential (relative case)]
\label{exit-tangent}
\label{theorem.exit-tangent}
Fix a stratified space $X$. 
There are natural equivalences of $\infty$-categories
\[
\Shv^{\sf cbl}(X)~\underset{\rm Cbl}\simeq~\Psh\bigl(\Ent(X)\bigr)~\underset{\rm Rel}\simeq~\Shv(\snglr_{/X}).
\]
\end{theorem}

\begin{remark}
The equivalence {\rm Cbl} will identify a constructible sheaf $\cF$ with the presheaf on $\Ent(X)$ sending an object $j: U \into X$ to $\cF(j( U ))$.
Hence {\rm Cbl} identifies the locally constant sheaves inside $\Shv^{\sf cbl}(X)$ with those presheaves on $\Ent(X)$ which factor through the smallest $\infty$-groupoid containing it.
\end{remark}

\begin{cor}\label{exit-corollaries}
Let $X$ be a stratified space.  
There is a natural equivalence of spaces $\sB(\Ent(X)) \simeq X$ between the classifying space of the $\infty$-category $\Ent(X)$ to the underlying space of $X$.  
Furthermore, $X$ is an ordinary smooth manifold if and only if $\Ent(X)$ is an $\infty$-groupoid.
\end{cor}

\begin{remark}\label{better-tangent}
Corollary~\ref{exit-corollaries} gives that $\Ent(X)$ is not an $\infty$-groupoid whenever $X$ is \emph{not} an ordinary smooth manifold; and furthermore that the map $\Ent(X) \to \bsc$ retains more information than the map of spaces $\sB(\tau_X)\colon \sB(\Ent(X)) \simeq X \to \sB(\bsc)$ from the underlying space of $X$ to the classifying space of the $\infty$-category $\bsc$ of singularity types. 
For instance, the map of spaces $\sB(\tau_X)$ does \emph{not} classify a fiber bundle per se since the `fibers' are not all isomorphic.
Even so, the functor $\tau_X$ does classify a sheaf of locally free $\RR_{\geq 0}$-modules, and this sheaf is $S$-constructible.
\end{remark}

\begin{remark}\label{exit-path-category}
Let $X = (X\xra{S} P)$ be a stratified space of bounded depth.
Consider the simplicial set $\Sing_S(X)$ for which a $p$-simplex is a map of stratified topological spaces
\[
(\Delta^p\to[p])\longrightarrow (X\to P)
\]
where we use the \emph{standard} stratification $\Delta^p\to[p]$ given by $(\{0,\dots,p\}\xra{t} [0,1])\mapsto {\sf Max}\{i\mid t_i\neq 0\}$. 
This simplicial set is defined in~\S A.6 of~\cite{HA} where it is shown to be a quasi-category, and it is referred to as the \emph{exit-path $\infty$-category} of the underlying stratified topological space of $X$.  
It will be obvious from the definitions that this underlying stratified space $X\xra{S} P$ is \emph{conically stratified} in the sense of Definition~A.5.5 of~\cite{HA}, and the underlying topological space $X$ is locally compact.  
And so Theorem~A.9.3 of \cite{HA} can be applied, which states an equivalence $\Psh\bigl(\Sing_S(X)^{\op}\bigr) \simeq \Shv^{\sf cbl}(X)$.
By inspection, both $\Sing_S(X)$ and $\Ent(X)$ are idempotent complete.
Through Theorem~\ref{exit-tangent} we conclude:
\end{remark}

\begin{corollary}\label{exitpathequivalence}
Let $X=(X\xra{S} P)$ be a stratified space.
There is an equivalence of $\infty$-categories
\[
\Sing_S(X)^{\op} \simeq \Ent(X)~.
\]
In other words, the enter-path $\oo$-category is equivalent to the opposite of the exit-path $\oo$-category.
\end{corollary}

Let $\cC$ be an $\infty$-category. An object $c\in \cC$ is \emph{completely compact} if the copresheaf $\cC(c,-)$ preserves (small) colimits. The following gives an intrinsic characterization of the exit-path category, as opposed to a construction of it.

\begin{cor}
\label{cor.completely-compact}
Through this equivalence, the essential image of the Yoneda functor $\Ent(X) \hookrightarrow \Psh\bigl(\Ent(X)\bigr) \simeq \Shv^{\sf cbl}(X)$ is the full $\oo$-subcategory consisting of the completely compact objects.  
\end{cor}

In~\S\ref{prove-prop-loc} we will prove the following related result.  
\begin{prop}\label{localization}
Let $X$ be a stratified space. 
Consider the subcategory $W_X\subset \Entd(X)$ consisting of the same objects but only those morphisms $(U\hookrightarrow X) \hookrightarrow (V\hookrightarrow X)$ for which $U$ and $V$ are abstractly isomorphic as stratified spaces.  
Then functor 
$
c\colon \Entd(X) \to \Ent(X)
$
witnesses an equivalence of $\infty$-categories from the localization
\[
\Entd(X)[W_X^{-1}] \xra{~\simeq~} \Ent(X)~.
\]
In particular, $\Entd(X) \to \Ent(X)$ is final.
\end{prop}

In Definition~\ref{def.maps} we give the definition of a \emph{refinement} $\w{X}\to X$ between stratified spaces, which is an articulation of a \emph{finer} stratification than the given one on $X$.  
We will prove the following result in~\S\ref{prove-prop-loc}.  
\begin{prop}\label{refinements-localize}
Let $\w{X}\xra{r} X$ be a refinement between stratified spaces. Then there is a canonical functor
\[
\Ent(\w{X}) \longrightarrow \Ent(X)
\]
which is a localization.  
\end{prop}
\begin{remark}\label{enter-path-refinements}
Proposition~\ref{refinements-localize} can be interpreted nicely through Corollary~\ref{exitpathequivalence} as a conceptually obvious statement.  
Namely, consider a refinement $\w{X}\to X$.
Consider the collection of those paths in the exit-path category of $\w{X}$ that immediately exit a stratum of $\w{X}$ that is not present in $X$.
The statement is that inverting these paths results in the exit-path category of $X$.  
\end{remark}

While Theorem~\ref{theorem.sheaves-basics} and Theorem~\ref{theorem.exit-tangent} have been stated for the unstructured case, we also prove them for the category of $\cB$-manifolds for an arbitrary structure $\cB$. 

\begin{theorem}[Structured versions]\label{structured-versions}
\label{theorem.structured-versions}
Let $\cB$ be an $\infty$-category of basics.
The following statements are true.
\begin{enumerate}
\item There is an equivalence of $\infty$-categories
 \[
  \Shv\bigl(\mfld(\cB)\bigr) ~\simeq~\Psh(\cB)~.
 \]
\item Let $X= (X\xra{S}P, \cA,g)$ be a $\cB$-manifold. There is are canonical equivalences of $\infty$-categories
 \[
  \Shv^{\sf cbl}(X)~\underset{\rm Cbl}\simeq~\Psh\bigl(\Ent(X)\bigr)~\underset{\rm Rel}\simeq~\Shv\bigl(\mfld(\cB)_{/X}\bigr)~.  
 \]
\end{enumerate}
\end{theorem}

\begin{remark}
Note that Theorem~\ref{theorem.smooth-basic-sheaf} follows from Theorem~\ref{structured-versions} as the case $\cB= \sD_n$.
\end{remark}

\subsubsection{Excisiveness}
The final result of Part 2 further simplifies matters. 
A priori, to verify that a presheaf on $\cB$-manifolds is a sheaf amounts to checking the sheaf condition for \emph{arbitrary} open covers of $\cB$-manifolds.
Theorem~\ref{theorem.1-excision} below shows that a (continuous) presheaf on $\cB$-manifolds can be verified as a sheaf by only checking the sheaf condition for a much simpler class of open covers: collar-gluings and sequential unions.  
Roughly speaking, a {\em collar-gluing} of a stratified space $X$ is a decomposition
\[
 X = X_{\geq -\infty} \bigcup_{\RR \times X_0} X_{\leq \infty}
\]
of $X$ into two stratified spaces, $X_{\geq -\infty}$ and $X_{\leq \infty}$, whose intersection is identified with a product stratified space, $\RR \times X_0$. The situation to keep in mind is with $X_{\leq \infty}$ and $X_{\geq -\infty}$ manifolds with common boundary $X_0$, and $X$ is obtained by gluing along this common boundary after choosing collars for the boundary. We view this as the appropriate analogue of composing cobordisms in the stratified setting. See Definition~\ref{defn.collar-gluing} for details.

The following terminology is inspired by Goodwillie--Weiss calculus \cite{goodwillie-weiss}. We replace the $\infty$-category $\spaces$ by a general $\infty$-category $\cC$.

\begin{definition}[$\amalg$-excisive]\label{def.coprod-excisive}
Let $\cC$ be an $\infty$-category that admits finite pushouts and sequential colimits.
Let $\cB$ be an $\infty$-category of basics. 
The $\infty$-category of $\amalg$-excisive functors is the full subcategory 
\[
\Fun_{\amalg\text{-}\sf exc}\bigl(\mfld(\cB), \cC)~\subset ~\Fun\bigl(\mfld(\cB), \cC\bigr)
\]
consisting of those functors $F$ that satisfy the following conditions:
\begin{itemize}
\item If $X= X_{\geq -\infty}\underset{\RR\times \partial}\bigcup X_{\leq \infty}$ is a collar-gluing among $\cB$-manifolds, then the diagram in $\cC$
 \[
 \xymatrix{
 F(\RR\times \partial)  \ar[r]  \ar[d]
 &
 F(X_{\leq \infty})  \ar[d]
 \\
 F(X_{\geq -\infty})  \ar[r]
 &
 F(X)
 }
 \]
is a pushout.
\item Let $X_0\subset X_1\subset \dots \subset X$ be a sequence of open subsets of a $\cB$-manifold such that $\underset{i\geq 0} \bigcup X_i = X$.
 The diagram in $\cC$
 \[
 F(X_0)\to F(X_1)\to \dots \to F(X)
 \]
 witnesses $F(X)$ as the sequential colimit.  
\end{itemize}
We use the notation
\[
\Fun_{1\text{-}{\sf exc}}(\mfld(\cB)^{\op}, \spaces) ~:= \Bigl(\Fun_{\amalg\text{-}\sf exc}\bigl(\mfld(\cB), \spaces^{\op})\Bigr)^{\op}~\subset~\Psh(\mfld(\cB))
\]
for the $\infty$-category of \emph{$1$-excisive functors (valued in spaces)}.  
\end{definition}

\begin{theorem}[Structures are 1-excisive functors]
\label{theorem.1-excision} 
Let $\cB$ be an $\infty$-category of basics.
Let $\cC$ be an $\infty$-category that admits finite pushouts and sequential colimits. 
There is an equivalence of $\infty$-categories
 \[
  \Fun(\cB,\cC)~\simeq~\Fun_{\amalg\text{-}{\sf exc}}\bigl(\mfld(\cB) , \cC\bigr)~.
 \]
\end{theorem}

\subsection{Overview of Part 3: differential topology of stratified spaces}
The last third of the paper is devoted to proving some basic results in what one might call the differential topology of stratified spaces. The first main result is Theorem~\ref{tot-unzip}, which we state later for sake of exposition. Roughly, the theorem says that any stratified space has a functorial {\em resolution of singularities} by a manifold with corners. The upshot is that one can reduce many questions about stratified spaces to the setting of manifolds with corners. 
As usual, many arguments boil down to a partition of unity argument, which in turns relies on the paracompactness of stratified spaces.

We then define the notion of a {\em finitary} stratified space, which is roughly a stratified space which can be obtained from a basic open after a finite number of handle attachments. (The idea of being a finitary space is a well-known condition that also appears in Goodwillie calculus, where analytic functors are only determined on finitary spaces, unless one assumes that a functor commutes with filtered colimits.) See Definition~\ref{def.finitary} for details.

In the smooth case, one can prove that a finitary smooth manifold is diffeomorphic to the interior of some {\em compact} smooth manifold with boundary, and vice versa. After defining a notion of a stratified space with boundary, we prove the last result of Part 3, which we restate later as Theorem~\ref{open-handles}:

\begin{theorem}[Open handlebody decompositions]
\label{handlebody}
\label{theorem.handlebody}
Let $X$ be a stratified space.
\begin{enumerate}
 \item There is a sequence of open subsets $X_0\subset X_1\subset \dots\subset X$ with $\underset{i\geq 0}\bigcup X_i = X$ and each $X_i$ finitary.  
 \item Suppose there is a compact stratified space with boundary $\ov{X}$ and an isomorphism ${\sf int}(\ov{X}) \cong X$ from its interior.
 Then $X$ is finitary.  
\end{enumerate}
\end{theorem}

\subsection{Descendant works}

We briefly describe where these present results are and will be used in a number of successor papers.

\smallskip

 {\it Factorization homology of stratified spaces} \cite{aft2}, extends the domain of definition of factorization homology to take values on these more general objects, conically smooth stratified spaces. Motivated by mathematical physics, we incorporate stratifications so as to allow for defects and higher-dimensional operators in factorization homology, such as Wilson loops in Chern--Simons. We develop the present work further in this sequel by proving, for a conically smooth stratified space $X$, the equivalence
\[
\Ent\bigl(\Ran (X)\bigr) \simeq \disk(\bsc)^{\sf surj}_{/X}
\]
between the enter-path $\oo$-category of the Ran space of $X$ (Definition \ref{def.ran}) and the $\oo$-category of embedded finite disjoint unions of basics in $X$ (where embeddings are required to induce surjections on $\pi_0$). This allows for the geometry of the Ran space to be applied to factorization homology, since factorization homology is, by definition, a colimit indexed by this category of embedded disks. This result is used further in {\it Zero-pointed manifolds} \cite{ZP} and its successor {\it Poincar\'e/Koszul duality} \cite{PKD}, in order to establish a Poincar\'e duality theorem for factorization homology.

\smallskip

{\it A stratified homotopy hypothesis} \cite{striation} forms a second sequel to the present work. There it is proved that $\Strat$, the $\oo$-category of conically smooth stratified spaces and conically smooth maps among them, embeds fully faithfully into $\oo$-categories. More precisely, enter-paths define a functor
\begin{equation}
\label{eqn.striation}
\Ent: \Strat \longrightarrow \Cat \simeq \cS{\sf tri} \subset \Shv(\Strat) 
\end{equation}
which is fully faithful: for any two conically smooth stratified spaces $X$ and $Y$, the space of conically smooth maps from $X$ to $Y$ is homotopy equivalent to the space of functors from $\Ent(X)$ to $\Ent(Y)$. Via the restricted Yoneda functor, this embedding realizes $\Cat$ as a full $\oo$-subcategory of space-valued sheaves on conically smooth stratified spaces; and its essential image $\cS{\sf tri}$ can be characterized as {\it striation sheaves}: isotopy-invariant constructible sheaves that satisfy certain additional descent conditions, such as descent for blow-ups.
 
\smallskip

{\it A stratified homotopy hypothesis} develops the present work in a second direction: whereas here we exhibit a theory of smooth moduli of stratified spaces, \cite{striation} does likewise for a theory of {\it singular moduli} of stratified spaces. That is, we here define the notion of smooth fiber bundle of conically smooth stratified spaces: we think of a $K$-parametrized smooth moduli of stratified spaces to be a conically smooth fiber bundle whose base is the smooth manifold $K$. There is then a natural homotopy equivalence
\[
\Bigl\{{\rm conically \medspace smooth \medspace }X\text{-}{\rm bundles \medspace over \medspace} K\Bigr\}\simeq \Map\bigl(K,{\sf BAut}(X)\bigr)
\]
for each conically smooth stratified space $X$. In \cite{striation}, there is an identification of the maximal $\oo$-groupoid of $\Snglr$ with the disjoint union $\coprod_{[X]} {\sf BAut}(X)$ taken over the isomorphism classes $[X]$ of conically smooth stratified spaces. This then gives an equivalence
\[
\Bigl\{{\rm conically \medspace smooth \medspace fiber \medspace bundles \medspace over}\medspace K\Bigr\} \simeq \Map\bigl(K, \Snglr\bigr)
\]
between the space of all conically smooth fiber bundles with smooth base $K$ with and the space of functors from (the $\oo$-category associated to the space) $K$ to the $\oo$-category $\Snglr$. In further studying singular moduli, we consider a $K$-parameterized singular family of stratified spaces to be a {\it constructible bundle} whose base is a conically smooth stratified space $K$. There exists an $\oo$-category $\cB{\sf un}$ which classifies such constructible bundles, in that there is an equivalence
\[
\Bigl\{{\rm constructible \medspace bundles \medspace over \medspace} K\Bigr\} \simeq \Map\bigl(\Ent(K)^{\op}, \cB{\sf un}\bigr)
\]
which restricts to the equivalence above via a fully faithful functor $\Snglr \hookrightarrow \cB{\sf un}$. The previously alluded to moduli space of stratifications of a manifold $M$ can be now defined to be a full $\oo$-subcategory of $\cB{\sf un}_{/M}$, consisting of those morphisms from a stratified space $\widetilde{M}$ to $M$ which are a diffeomorphism on underlying smooth manifolds (but where the stratification of $\widetilde{M}$ can vary). This is the basis for {\it Factorization homology I: higher categories} \cite{Emb1A}, which develops a notion of homology indexed by such a moduli space of stratifications, and which is proposed as the essential construction for all extended topological quantum field theories -- see the introduction of \cite{Emb1A}.

\subsection{Conjectures}

The following are problems raised by this work which we suggest for the interested reader or student.

\begin{conjecture}[Conically smooth approximation]
The $C^0$ analogue of the enrichment of Definition \ref{self-enrichment} defines a $\Kan$-enriched category $ \Strat^{C^0}$ of $C^0$ stratified spaces, and the natural forgetful functor of $\Kan$-enriched categories
\[
\Strat \longrightarrow \Strat^{C^0}
\]
is fully faithful up to homotopy. That is, for any two conically smooth stratified spaces $X$ and $Y$, the map
\[
\Strat(X,Y) \longrightarrow \Strat^{C^0}(X,Y)
\]
is a homotopy equivalence of Kan complexes.

\end{conjecture}

The preceding is the conically smooth analogue of the crucial result in differential topology that the map
\[
C^{\oo}(M,N) \longrightarrow \Map(M,N)
\]
is a homotopy equivalence for smooth manifolds $M$ and $N$; this can be proven from Weierstrass approximation.

\begin{conjecture}\label{conj.transversality}
Thom transversality holds for conically smooth maps. That is, for $X$, $Y$ and $Z$ conically smooth stratified spaces and $Z\hookrightarrow Y$ a conically smooth closed embedding, then the inclusion
\[
\Strat_{\pitchfork Z}(X,Y)\longrightarrow \Strat(X,Y)
\]
induces a surjection on connected components. Here $\Strat_{\pitchfork Z}(X,Y)$ is the subspace of conically smooth maps $\Strat(X,Y)$ (see Lemma \ref{mapping-Kans}) consisting of those which are transverse to $Z$, i.e., those conically smooth maps $g:X\ra Y$ for which $X_p$ and $Z_{gp}$ intersect transversely in $Y_{gp}$ for every stratum $X_p\subset X$.
\end{conjecture}

The preceding is not immediately comparable with the transversality result of, for example, \cite{trotman} for two reasons: first, we ask that every conically smooth map is isotopic via conically smooth maps to one which is transverse to the sub-stratified space; second, we have not provided comparison results for conical smoothness versus the Whitney conditions. This leads to the following:

 \begin{conjecture}\label{whitneys-count}
 The following classes of spaces admit conically smooth stratifications: 
 \begin{enumerate}
\item real and complex algebraic varieties;
\item orbifolds;
\item Whitney stratified spaces.
\end{enumerate}
\end{conjecture}

We imagine, for instance, that this last class could be approached in the following way: the Thom--Mather Theorem~\cite{mather} ensures that a Whitney stratified manifold $M$ is locally homeomorphic to stratified spaces of the form $\RR^{n-k}\times \sC(X)$, and thus -- by induction on the dimension of $\sC(X)$ -- is an example of a $C^0$ stratified space. It would then remain to prove that these local homeomorphisms can be taken to restrict to local diffeomorphisms on each stratum; and further, that Whitney's conditions $A$ and $B$ give that these local diffeomorphisms can be further taken to be conically smooth, thereby exhibiting a conically smooth atlas.

\smallskip

Let $M$ be a smooth $n$-manifold.
For each finite cardinality $k$, consider the \emph{Ran space of $M$} of Definition \ref{def.ran}. This is the stratified space 
\[
\Ran_{\leq k}(M)~:=~ \Bigl( \{\emptyset \neq S\subset M\mid |S|\leq k\} \xra{~S\mapsto |S|~} \NN^{\op}\Bigr)
\]
consisting of nonempty subsets of $M$ whose cardinality is bounded above by $k$, stratified by cardinality.  
Consider the $\infty$-category
\[
\Ent\bigl(\Ran(M)\bigr) := \varinjlim \Ent\bigl(\Ran_{\leq k}(M)\bigr)
\]
which is the direct limit of the enter-path $\infty$-categories from Definition~\ref{defn.E(X)}.
Forgetting the ambient manifold defines a functor
\begin{equation}\label{value.tau}
\Ent\bigl(\Ran(M)\bigr) \longrightarrow  \Fin^{\sf surj}~,~ \qquad (S\subset M)\mapsto S~,
\end{equation}
to the category of nonempty finite sets and surjections among them.  
The fiber of this composite functor over $I\in \Fin^{\sf surj}$ is the space $\Conf_I(M)$ of injections from $I$ into $M$.
For example, for $I=\{1,2\}$ and $M=\RR^n$, this fiber $\Conf_2(\RR^n)\simeq S^{n-1}$ is homotopy equivalent to the $(n-1)$-sphere.

For each finite cardinality $k$, there are natural morphisms
\begin{equation}\label{interpolate}
\Top(n) 
\longrightarrow 
\Aut_{/\Fin^{\sf surj}}\Bigl(\Ent\bigl(\Ran(\RR^n)\bigr)\Bigr)
\longrightarrow
\Aut_{/\Fin^{\sf surj}_{\leq k}}\Bigl(\Ent\bigl(\Ran_{\leq k}(\RR^n\bigr)\Bigr)
\end{equation}
among group-like topological monoids.
In the preceding expression, $\Top(n)$ is the topological group of homeomorphisms of $\RR^n$, and $\Fin^{\sf surj}_{\leq k}\subset \Fin^{\sf surj}$ is the full subcategory consisting of those finite sets whose cardinality is bounded above by $k$.

\begin{remark}\label{spivak}
In the case $k=2$, the sequence~(\ref{interpolate}) interpolates between $\Top(n)$ and ${\sf hAut}(S^{n-1})$, the topological monoid of self-homotopy equivalences of the $(n-1)$-sphere.  

\end{remark}

\begin{conjecture}\label{not.Top}
The natural map
\[
\Top(n) \longrightarrow \Aut_{/\Fin^{\sf surj}}\Bigl(\Ent\bigl(\Ran(\RR^n)\bigr)\Bigr)
\]
\begin{itemize}
\item is a homotopy equivalence for $n \leq 3$;
\item is not a homotopy equivalence for $n \geq 4$.
\end{itemize}
\end{conjecture}
\noindent
We regard this last conjecture as being of special interest in the case $n=3$.

Now, consider the sections of the functor~(\ref{value.tau}):
\[
\w{\sT}M~:=~\Gamma\Bigl( \Ent\bigl(\Ran(M)\bigr) \to  \Fin^{\sf surj}\Bigr)~.
\]
Evaluation at the terminal object defines a functor
\begin{equation}\label{w.T}
{\sf ev}_\ast\colon \w{\sT}M\longrightarrow M
\end{equation}
to the underlying $\infty$-groupoid of $M$.
The fiber of this functor over $x\in M$ is canonically identified
\[
\w{\sT}_xM~\simeq ~  \Ent\bigl(\Ran(\sT_xM)\bigr)
\]
as the enter-path $\infty$-category of the Ran space of the tangent space of $M$ at $x$; the monodromy action of $\Omega_xM$ on this fiber is by automorphisms of $\Ent\bigl(\Ran(\sT_xM)\bigr)$ over $\Fin^{\sf surj}$.  

In other words, the functor~(\ref{w.T}) is classified by a map between $\infty$-groupoids:
\begin{equation}\label{w.tau}
M\longrightarrow {\sf BAut}_{/\Fin^{\sf surj}}\Bigl(\Ent\bigl(\Ran(\RR^n)\bigr)\Bigr)~.
\end{equation}
Through Remark~\ref{spivak}, the map $(\ref{w.tau})$ recovers the unstable Spivak tangent bundle of $M$; our Conjecture~\ref{not.Top}, then, is that $(\ref{w.tau})$ does \emph{not} a priori recover the micro-tangent bundle of $M$. 
This leads us to the following, somewhat informal, question.
\begin{question}
How much familiar algebraic topology about a smooth $n$-manifold $M$ can be recovered solely by the $\infty$-category $\Ent\bigl(\Ran(M)\bigr)$ as it is equipped with the functor~(\ref{value.tau}) to $\Fin^{\sf surj}$?
For instance, can the fact that $M$ is a Poincar\'e duality space be verified solely through this data?
If so, can the $\sL$-theoretic obstructions be shown to vanish solely through this data?
\end{question}

\subsection*{Acknowledgements} JF thanks everyone at Buzz: Killer Espresso.

\clearpage
\part{Stratified spaces, basics, and tangential structures}
\section{$C^0$ Stratified spaces}\label{section:c0-singular-manifolds}
\label{section.C0-singular-manifolds}

\subsection{Definition}
Given a smooth manifold $X$, one can forget its smooth atlas to obtain the underlying topological manifold. When one does the same for a smooth manifold with {\em corners}, we lose the data of the natural stratification on $X$ -- namely, the stratification by whether $x \in X$ is a point in the interior, face, or a higher codimension corner. Our viewpoint is that this stratification is a legitimate structure one can retain, even after forgetting the $C^\infty$ atlas -- more generally, that there is a notion of a $C^0$ stratified space, whose singularities are encoded by a stratification. We also demand that all singularities look locally like a thickened cone -- this is a philosophy that goes back at least to Thom~\cite{thom}.

\begin{remark}
Until \S\ref{section.kan-enrichment}, all the categories we consider will be {\em discrete} (i.e., enriched in sets) categories. We write $\snglrd$, $\mfldd$, $\bscd$ for the discrete categories of conically smooth stratified spaces, manifolds, and basics. In \S\ref{section.kan-enrichment}, we will enrich these categories over Kan complexes, and we will denote the corresponding $\infty$-categories by $\snglr$, $\mfld$, and $\bsc$.
\end{remark}

\begin{definition}[Posets as spaces]\label{poset-space}
Let $P$ be a poset. We consider it a topological space by declaring $U \subset P$ to be open if and only if it is closed upwards -- this means if $a \in U$, then any $b \geq a$ is also in $U$. Note that a map $P \to P'$ is continuous if and only if it is a map of posets. This determines a functor
\[
\POSet \hookrightarrow \Top
\]
which is fully faithful, and preserves limits.  
\end{definition}

\begin{definition}[Stratified space]
\label{def.stratified-space}
\label{def.strat-space}
For $P$ a poset, a {\em $P$-stratified space} is a topological space $X$ together with a continuous map $S \colon X \to P$, which we call a stratification of $X$. We will often refer to $X$ as the underlying topological space. 
A stratified topological space is a $P$-stratified space for some $P$. 
\end{definition}

\begin{remark}
As far as we are aware, this definition first appears in Lurie's Higher Algebra~\cite{HA}, Appendix A.5.
\end{remark}

\begin{notation}
When the context will not lead to confusion, we will write a stratified space $(S\colon X\to P)$ simply as its underlying topological space $X$.
\end{notation}

\begin{example}
Let $X_{\leq k} $ be the $k$-skeleton of a CW complex $X$. One has a stratification $X \to \ZZ_{\geq 0}$ by sending $X_{\leq k} \smallsetminus X_{\leq k-1}$ to $k \in \ZZ_{\geq 0}$. 
\end{example}

\begin{example}\label{example.product-poset}
Let $(X \xra{S} P)$ and $(X' \xra{S'} P')$ be stratified topological spaces.
Then $(X\times X'\xra{S\times S'} P\times P')$ is another stratified topological space. 
Note that the product poset has the partial order $(p, p') \leq (q,q') \iff (p\leq q) \& (p' \leq q')$.
\end{example}

\begin{example}\label{example.corner}
Let $[1]$ be the poset $\{0<1\}$ and let $\RR_{\geq 0} \to [1]$ be the map sending $0 \mapsto 0$ and $t \mapsto 1$ for $ t > 0$. By taking $n$-fold products, one obtains a stratified space
	\[
	\RR_{\geq 0}^{\times n} \to [1]^{n}
	\]
where the stratifying poset is an $n$-cube. This is a $C^0$ version of an {\em $\langle n \rangle$-manifold} as previously studied by Laures~\cite{laures} and J\"anich~\cite{janich}.
\end{example}

\begin{remark}
\label{rmk.usual-stratifications}
Thom's notion of a topological stratification~\cite{thom} is a special case of Definition \ref{def.stratified-space}. Namely, a filtration by closed subsets $\emptyset \subset X_{\leq 0} \subset \ldots \subset X_{\leq n} = X$ defines a map $X \to \ZZ_{\geq 0}$, and the $i$th stratum is given by $X_{\leq i} - X_{\leq i-1}$. For further reading on the history of stratifications, we also refer the reader to \S1.0 of Goresky and MacPherson~\cite{goreskymacpherson}.
\end{remark}

\begin{definition}[Stratified map]\label{def.strat-map}
Let $(X\to P)$ and $(X'\to P')$ be two stratified spaces.
A \emph{continuous stratified map} $f\colon (X\to P) \to (X'\to P')$ is a commutative diagram of topological spaces
	\begin{equation}
	\label{eqn.maps-of-XP}
	\xymatrix{
	X \ar[r] \ar[d]
	&	X'\ar[d] 
	\\
	P \ar[r]
	& P' .
	}
	\end{equation}
Note there is an obvious notion of composition, given by composing the horizontal arrows.
\end{definition}

\begin{definition}[Open embeddings of stratified spaces]
A continuous stratified map $f:X\ra Y$, as in Definition \ref{def.strat-map}, is a \emph{stratified open embedding} if:
\begin{itemize}
\item $f: X\ra Y$ is an open embedding of topological spaces; and
\item the restriction $f_|: X_p \ra Y_{fp}$ is an open embedding of topological manifolds for every element $p\in P$ of the poset indexing the stratification of $X$.
\end{itemize}
The category
\[
	\stopen
	\]
has objects which are stratified topological spaces whose underlying space is second countable and Hausdorff;  morphisms are stratified open embeddings.
\end{definition}

\begin{example}
Let $\Delta^2 \to [2] =  \{0 < 1 < 2\}$ be the stratification sending the interior of a $k$-cell to the number $k$. Fix some continuous map $f: [2] \to [1]$. This defines another stratified space $(\Delta^2 \to [1])$ by post-composing with $f$. The induced morphism
	\[
	\xymatrix{
	\Delta^2 \ar[r]^{\sf id} \ar[d]
	&	\Delta^2\ar[d] 
	\\
	[2] \ar[r]^f
	& [1]
	}
	\]
is a continuous map between the stratified spaces $(\Delta^2 \to [2])$ and $(\Delta^2 \to [1])$, but is not an open embedding of stratified spaces, since the bottom arrow is not an open embedding.
\end{example}

\begin{remark}
\label{remark.open-cover}
Note there is a natural notion of an open cover: a collection of morphisms $\{(U_i \to P_i) \to (X \to P)\}_{i \in I}$ is an open cover whenever both $\{U_i \to X\}_{i \in I}$ and $\{P_i \to P\}_{i \in I}$ are.
\end{remark}

The following construction plays a central role, as the singularities we consider all look like a cone over another stratified space (possibly thickened by a copy of $\RR^i$).

\begin{definition}[Cones]
Let $(X\to P)$ be a stratified topological space.  
The \emph{cone} $\sC(X\to P) = \bigl(\sC(X) \to \sC(P)\bigr)$ is as follows.  
The underlying space is the pushout topological space
	\[
	\sC(X) := \ast \coprod_{\{0\}\times X} \RR_{\geq 0} \times X.
	\]
The poset $\sC(P)$ is given by formally adjoining a minimal element $\ast$. It can likewise be written as a pushout in posets
	\[
	\sC(P) := \ast \coprod_{\{0\} \times P} [1] \times P.
	\]
The stratification $\RR_{\geq 0} \to [1]$, given by $0\mapsto 0$ and $0<t\mapsto 1$, induces a stratification $\sC(X) \to \sC(P)$ as the map between pushouts. (In particular, this map sends $\ast \mapsto \ast$.)  
\end{definition}

Now we define the category of $C^0$ stratified spaces and give examples.

\begin{definition}[$C^0$ stratified spaces]
\label{defn.stopen}
Consider the full subcategories $\cT \subset \stopen$ satisfying the following five properties:
	\begin{enumerate}
	\item \label{prop.epty} $\cT$ contains the object $(\emptyset \to \emptyset)$.
	\item \label{prop.cones} If $(X \to P)$ is in $\cT$ and both $X$ and $P$ are compact, then $\sC(X \to P)$ is in $\cT$.
	\item \label{prop.products} If $(X \to P)$ is in $\cT$, then the stratified space $(X \times \RR \to P)$, with stratification given by the composition $X \times \RR \to X \to P$, is in $\cT$. 
	\item \label{prop.subs} If $(U \to P_U) \to (X \to P)$ is an open embedding of stratified spaces, and if $(X \to P)$ is in $\cT$, then $(U \to P_U)$ is.
	\item \label{prop.covers} If $(X \to P)$ is a stratified space admitting an open cover by objects in $\cT$, then $(X \to P)$ is in $\cT$.
	\end{enumerate}	
Let $\bigcap_{\cT} \cT\subset \stopen$ be the smallest full subcategory satisfying the above properties (1)--(5). We define
	\[
	\snglrdC \subset ~{}~\bigcap_{\cT} \cT~{}~ \subset ~ \stopen
	\]
to be the full subcategory consisting of those objects whose underlying space $X$ is paracompact. A \emph{$C^0$ stratified spaces} is an object of the category $\snglrdC$.
\end{definition}

\begin{example}\label{example.topological-manifolds}
The combination of (\ref{prop.epty}) and (\ref{prop.cones}) guarantees that the point, $\ast = \sC(\emptyset\to\emptyset)$, is in $\cT$. Further, (\ref{prop.products}) ensures that $\RR^n$ with trivial stratification, $(\RR^n \to \ast)$, is in $\cT$, and any open subset of $\RR^n$ is in $\cT$ by (\ref{prop.subs}). As a result, (\ref{prop.covers}) ensures that any topological manifold is in $\cT$. Since this is true for all $\cT$, the category $\snglrdC$ contains -- as a full subcategory -- the category of topological manifolds and open embeddings among them. These are precisely those $C^0$ stratified spaces $(X \to P)$ for which $P$ has a single element.
\end{example}

\begin{example}
Regard a compact topological manifold $X$ as a $C^0$ stratified space $(X \to \ast )$. Then the cone $\sC(X \to \ast)$ is a $C^0$ stratified space admitting an open embedding from the stratified space $(\RR_{>0}\times X \to \ast)$.
\end{example}

\begin{remark}\label{remark.surjective-stratification}
The stratification map $X \to P$ is always a surjection for $(X \to P) \in \snglrdC$.  This is seen easily by verifying that the category $\cT^{\mathsf{surj}} \subset \stopen$ spanned by objects for which the stratification is a surjection satisfies properties (\ref{prop.epty}) through (\ref{prop.covers}).
\end{remark}

\subsection{Basics for $C^0$ stratified spaces}
Just as $\RR^n$ controls the local behavior of a topological manifold, there are distinguished $C^0$ stratified spaces that control local behavior; that is, which characterize the possible singularities of an object in $\snglrdC$. Lemma~\ref{lemma.topological-basis} formalizes this intuition.

\begin{definition}[$C^0$ basics]\label{def:top-bsc}
A \emph{$C^0$ basic} is a $C^0$ stratified space of the form $\RR^i\times \sC(Z)$ where $i\geq 0$, $\RR^i = (\RR^i \to \ast)$ is given the trivial stratification, and $Z$ is a compact $C^0$ stratified space.  
\end{definition}

Note that if $Z$ is an object of $\snglrdC$, then (\ref{prop.cones}) and (\ref{prop.products}) guarantee that $\RR^i \times \sC(Z)$ is as well. 
When the parameters $i$ and $Z$ are not relevant, we will typically denote $C^0$ basics with letters such as $U$, $V$, and $W$.

\begin{lemma}[Basics form a basis]\label{topological-basis}\label{lemma.topological-basis}
Let $X$ be a stratified second countable Hausdorff topological space.
Consider the collection of open embeddings
	\[
	\{U \hookrightarrow X\}
	\]
where $U$ ranges among $C^0$ basics.
Then this collection forms a basis for the topology of $X$ if and only if $X$ is a $C^0$ stratified space.
\end{lemma}

\begin{proof}
If the collection forms a basis, then it is in particular an open cover, and it follows that $X$ is a $C^0$ stratified space by Property (\ref{prop.covers}).

Now consider the collection $\cT$ of such $X$ for which $\{U \into X\}$ forms a basis for its topology.  
We must show $\cT$ has properties (\ref{prop.epty})-(\ref{prop.covers}).
\begin{enumerate}
	\item[(\ref{prop.epty})] Certainly $\emptyset \in \cT$.
	\item[(\ref{prop.products})] Let $X\in \cT$.  Then the collection of product open embeddings $\RR\times U \hookrightarrow \RR\times X$ forms a basis for the topology of $\RR\times X$.  
	It follows that $\RR\times X\in \cT$.  
	\item[(\ref{prop.cones})] Let $X \in \cT$ and suppose it is compact.  
	The collection of open embeddings $\RR^0\times \sC(X) \to \sC(X)$ which are the identity of the $X$-coordinate, form a local base for the topology of the cone point $\ast$.  Because $\sC(X)\smallsetminus \ast = (0,1)\times X$, the point just above implies $\sC(X) \in \cT$.  
	\item[(\ref{prop.subs})] Let $Y\hookrightarrow X$ be an open embedding and suppose $Y\in \cT$.  
	From the definition of a basis for a topology, we see $Y\in \cT$.
	\item[(\ref{prop.covers})] Let $\cU$ be an open cover of $X$ and suppose $\cU \subset \cT$.  
	From the definition of a basis for a topology, we see $X\in \cT$.  
\end{enumerate}
\end{proof}

\begin{remark}\label{conically-stratified}
A $C^0$ stratified space $(X\to P)$ is necessarily \emph{conically stratified} in the sense of Definition~A.5.5 of~\cite{HA}.
\end{remark}

\subsection{Strata}
We prove that the strata of $C^0$ stratified spaces are also $C^0$ stratified spaces.

\begin{defn}[Consecutive maps and $X_{|Q}$]\label{defn.consecutive}
An injection $j: Q \hookrightarrow P$ of posets is {\em consecutive} if:
\begin{itemize}
\item $j$ is full, i.e., the inequality $x \leq y$ holds in $Q$ if and only if the inequality $j(x)\leq j(y)$ holds in $P$.
\item If $p$ and $p''$ are in the image of $j$, then the set 
$
 \{p' \in P \,|\, p \leq p' \leq p'' \}
$ is also contained in the image of $j$.
\end{itemize}
We denote by $X_{|Q} = (X_{|Q} \to Q)$ the pullback of $X$ along $Q \to P$.
\end{defn}

\begin{remark}
Note that being consecutive is a property stable under pullbacks.
\end{remark}

\begin{lemma}\label{strata}
\label{lemma.strata}
If $Q\hookrightarrow P$  is consecutive, the pullback $X_{|Q}$ is a $C^0$ stratified space. Further, the map $X_{|Q} \hookrightarrow X$ is the inclusion of a sub-stratified space: $X_{|Q} \to X$ is a homeomorphism onto its image.
\end{lemma}

\begin{remark}
If $Q \into P$ is not consecutive, one can show that $X_{|Q}$ need not even be a $C^0$ stratified space. See Example~\ref{example.bsc-2}.
\end{remark}

\begin{proof}[Proof of Lemma~\ref{lemma.strata}.]
Let $\cT$ be the collection of $C^0$ stratified spaces for which the first statement is true.  
\begin{enumerate}
	\item[(\ref{prop.epty})] Certainly $\emptyset \in \cT$.
	\item[(\ref{prop.products})] Let $(X\to P)\in \cT$.  Let $Q\hookrightarrow P$ be as in the hypothesis.  Then $(\RR\times X)_{|Q} = \RR\times X_{|Q}$.  So $(\RR\times X \to P)\in \cT$.  
	\item[(\ref{prop.subs})] Fix $(X\to P)\in \cT$. Let $(U\to P')\hookrightarrow (X\to P)$ be an open embedding and $Q\hookrightarrow P'$ be as in the hypothesis.
	Then $U_{|Q} = X_{|Q}$, and so $(U_{|Q} \to Q)\in \cT$.  
	\item[(\ref{prop.covers})] Let $\cU$ be an open cover of $(X\to P)$ with $\cU\subset \cT$.  
	Let $Q\hookrightarrow P$ be as in the hypothesis.
	Then the collection $\{U_{|Q}\to P'\cap Q\mid (U\to P')\in \cU\}$ is an open cover of $(X_{|Q} \to Q)$; and so $(X\to P)\in \cT$.  
	\item[(\ref{prop.cones})] Let $(X\to P)\in \cT$ with $X$ compact.  
	Let $Q\hookrightarrow \sC(P)$ be as in the hypothesis. 
	Identify the underlying set of $Q$ with its image in $\sC(P)$.
	Denote by $Q_\ast \subset Q$ the smallest open set in $Q$ such that $\ast \in Q_\ast$. There is the open cover 
	\[
	Q = Q\smallsetminus \ast \bigcup_{Q_\ast \smallsetminus \ast} Q_\ast
	\]  
	which gives the open cover $\sC(X)_{|Q} = \sC(X)_{|Q\smallsetminus \ast} \bigcup_{\sC(X)_{|Q_\ast \smallsetminus \ast}} \sC(X)_{|Q_\ast}$.
	It is thus enough to show that each term in this cover is a $C^0$ stratified space.  
	The first term is equal to $(\RR_{>0}\times X)_{|Q\smallsetminus \ast}$, which is a stratified space by (\ref{prop.products}).
	Likewise for the middle term.
	The last term is isomorphic to $\sC(X_{|Q_\ast \smallsetminus \ast})$ -- since $X_{|Q_\ast \smallsetminus \ast}$ is compact, it remains to see that it is also a $C^0$ singular. This follows because $X\in \cT$ and the inclusion $(Q_\ast \smallsetminus \ast) \into P$ is consecutive. 
\end{enumerate}
The second statement follows from the definition of stratified topological spaces -- since $Q \into P$ is consecutive, $X_{|Q}$ inherits the subspace topology from $X$. \end{proof}

As an example, for any element $p \in P$ and $(X \to P)$ a $C^0$ stratified space; let $X_p$ and $X_{\leq p}$ denote the obvious stratified spaces, and denote by $X_{\nless p}$ the stratified subspace of points of $X$ whose image in $P$ is not less than $p$. 

\begin{corollary}
For any $p\in P$, the stratified spaces
	\[
	X_{\leq p}~\qquad ~\text{ and } ~\qquad ~ X_{p} ~\qquad ~\text{ and }~ \qquad ~X_{\nless p}
	\]
are $C^0$ stratified spaces; the middle stratified space $X_p$ is, in particular, an ordinary topological manifold.  
\end{corollary}

\begin{remark}\label{rem.filtration}
For $X\ra P$ a $C^0$ stratified space, there is a canonical map of posets $P\ra \NN$ which assigns to $p\in P$ the dimension ${\sf dim}(X_p)$ of the topological manifold $X_p$. In this way, there is a natural filtration of a $C^0$-stratified space by dimension. We elaborate on this in Lemma \ref{dim-depth}.
\end{remark}

\begin{remark}
For $X\ra P$ a $C^0$ stratified space, one can consider the set $Q$ of connected components of the strata of $X$. The set $Q$ obtains a natural poset structure from $X$, by imposing a relation $q\leq q'$ exactly if there is a containment $X_q\subset \ov{X_{q'}}$ of the connected component $X_q$ in the closure of the component $X_{q'}$. This defines a map of posets $Q \ra P$, and little would be lost from the present theory if this map were required (in the definition of a $C^0$ stratified space) to be an isomorphism.
\end{remark}

\subsection{Dimension and depth}
\label{sec:dim-depth}
Now we define fundamental local invariants of a $C^0$ stratified space. Recall the notion of {\em Lebesgue covering dimension}.
For example, the Lebesgue covering dimension of an $n$-dimensional manifold is $n$ provided $n\geq -1$. 

\begin{definition}[Dimension]
Let $X= (X\to P)$ be a nonempty $C^0$ stratified space.
The \emph{local dimension} of $X$ at $x\in X$, denoted as ${\sf dim}_x(X)$, is the covering dimension of $X$ at $x$.  
The \emph{dimension} of $X$ to be
	\[
	{\sf dim}(X) = \underset{x \in X}{\sf sup} {\sf dim}_x(X).
	\]
Finally, if the local dimension of $X$ is constantly $n$, we say $X$ has \emph{pure} dimension $n$.  
We adopt the convention that the dimension of $\emptyset$ is $-1$.	
\end{definition}

\begin{remark}
One can equally consider {\em inductive dimension} rather than covering dimension. The two dimensions do not differ on the classes of spaces we consider, except for the empty manifold (depending on conventions).
\end{remark}

\begin{remark}
Local dimension is not locally constant. For example, let $Z = {\sf pt} \coprod S^1$ be the topological manifold given by the disjoint union of a point with a circle. Then the $C^0$ stratified space $\sC(Z)$ is connected, but has local dimension 1 near ${\sf pt}$, and 2 near $S^1$.
\end{remark}

\begin{definition}[Depth]\label{def:depth}
Let $X= (X\xra{S} P)$ be a nonempty $C^0$ stratified space.
The \emph{local depth} of $X$ at $x$ is
	\[
	\depth_x(X) := {\sf dim}_x(X) - {\sf dim}_x(X_{Sx}),
	\]
i.e., the difference between the local dimension of $X$ at $x$, and the local dimension at $x$ of the stratum of $X$ in which $x$ belongs. The \emph{depth} of $X$, denoted as $\depth(X)$, is the supremum over the local depths of $X$.  
We adopt the convention that the depth of $\emptyset$ is $-1$. 
\end{definition}

\begin{example}[Cones and depth]\label{example.cone-depth}
Let $Z$ be an ordinary $n$-dimensional topological manifold which is compact. As before, we regard it as a $C^0$ stratified space by the constant stratification $Z \to [0]$. Then the cone point $\ast \in \sC(Z)$ is a point of depth $(n+1)$, and all other points are points of depth $0$.
\end{example}

\begin{example}[Products and depth]\label{example.product-depth}
If $X$ is a $C^0$ stratified space of depth $k$ and dimension $n$, then $\RR^N \times X$ is a $C^0$ stratified space of depth $k$ and dimension $N+n$. More specifically, any point $(\vec v, x) \in \RR^N \times X$ has depth $\depth_x(X)$ and dimension $N + {\sf dim}_x(X)$. 
\end{example}

\begin{example}\label{non-decreasing}
Let $U = \RR^i \times \sC(Z)$ be a $C^0$ basic.  
Combining the previous two examples, $\depth(U) = {\sf dim}(Z)+1$.
This observation will be used repeatedly later on, when we define conically smooth structures inductively.
\end{example}

\begin{remark}\label{non-increasing}\label{remark.non-increasing}
From the definitions, we see that if $X\hookrightarrow Y$ is an open embedding between $C^0$ stratified spaces, then
	\[
	{\sf dim}(X) \leq {\sf dim}(Y)~\qquad\text{ and }\qquad ~{}~{}~{}~
	\depth(X) \leq \depth(Y)~.
	\]
\end{remark}

Following this observation, consider the poset 
	\[
	\PP \subset (\ZZ )^{\op} \times (\ZZ)^{\op}
	\]
whose elements are those pairs of integers $(k,n)$ which obey the inequality $-1\leq k \leq n$. Note that the partial order means $(k,n)\leq (k',n')$ if and only if $k\geq k'$ and $n \geq n'$. 

\begin{remark}
For any $i \in \ZZ$, the translation map 
	\[
	(k,n) \mapsto 
	\begin{cases}
		(k,-1)~, & n+i \leq -1 \\
		(k,n+i)~, & \text{otherwise}
	\end{cases}
	\] 
is a continuous and open map.
\end{remark}

\begin{lemma}\label{dim-depth}
Let $(X\to P)$ be a $C^0$ stratified space.
Assigning to each element $x\in X$ its local depth $k_x$ and its local dimension $n_x$ defines a continuous map $X \to \PP$.
Furthermore, this map factors uniquely through a map of posets $P\to \PP$.  
\end{lemma}

\begin{remark}\label{remark.respect-stratification}
Given an open embedding $(X \to P) \to (X' \to P')$, the composition $X \to X' \to \PP$ is equal to the map $X \to \PP$ because local depth and local dimension are local invariants. As a consequence, an open embedding of $C^0$ stratified spaces always induces a commutative triangle
	\[
	\xymatrix{
	P \ar[rr]  \ar[dr]
	&&P' \ar[dl]
	\\
	& \PP &
	}
	\]
fitting below the diagram (\ref{eqn.maps-of-XP}). That is, any diagram in $\snglrdC$ takes place over the poset $\PP$.
\end{remark}

\begin{remark}
While the map $X \to \PP$ is a useful organizing tool, the map $P \to \PP$ is not typically an injection (for instance, when $X$ can contain different singularity types of equal depth and dimension).
\end{remark}

\begin{proof}[Proof of Lemma~\ref{dim-depth}.]
First, note that any factoring map $P \to \PP$ is uniquely determined because the stratification $X \to P$ is a surjection by Remark~\ref{remark.surjective-stratification}.

Now consider the full subcategory $\cT\subset \snglrdC$ of those $C^0$ stratified spaces $X$ for which the lemma is true.
It is routine to verify that $\cT$ possesses the five properties (1)-(5): 
	\begin{enumerate}
	\item The empty manifold $(\emptyset \to \emptyset)$ has a unique map to $\PP$.
	\item Let $(X \to P)$ be a $C^0$ stratified space. Example~\ref{example.product-depth} shows that the points $(t,x) \in \sC(X)\smallsetminus\{\ast\}$ are sent to $(\depth_x(X), {\sf dim}_x(X)+1)$. Using that $a: P \to \sC(P)$ is an injection, we obtain an induced map $\sC(P)\smallsetminus\{\ast\} \xra{a^{-1}} P \to \PP \xra{+(0,1)} \PP$. We extend this composition to a map $\sC(P) \to \PP$ via Example~\ref{example.cone-depth}, which shows that the cone point $\ast \in \sC(X)$ is sent to $(\depth(X)+1,{\sf dim}(X)+1)$.
	The map $\sC(P) \to \PP$ is obviously a map of posets, and hence continuous. Thus the composition $X \to P \to \PP$ is continuous.
	\item Example~\ref{example.product-depth} shows that the map $\RR^i \times X \to \PP$ factors as a sequence of maps
		\[
		\xymatrix{
		\RR^i \times X \ar[r]
		& X \ar[r]
		&  \PP \ar[r]^{+(0,i)}
		&  \PP
		}
		\]
	where the last map is translation by $i$ in the dimension component. Since this last map is continuous, the composition is.
	\item Given any open embedding $(U \to P_U) \to (X \to P)$, we set $U \to \PP$ to be the composition $U \to X \to P \to \PP$. This agrees with the definition $x \mapsto (\depth_x U, {\sf dim}_x U)$ because local depth and local dimension are preserved under open embeddings. Since $U \to P_U$ is a surjection and $U \to X \to P$ is the same map as $U \to P_U \to P$, the composition $P_U \to P \to \PP$ is a factorization.
	\item Let $\{j_i: (U_i \to P_i) \to (X \to P)\}$ be an open cover, and $h: X \to \PP$ the depth-dimension map. Any open set $Q \subset P$ has a pre-image given by
		\[
		h^{-1}(Q) = \bigcup_i U_i \cap j_i^{-1}(h^{-1}(Q)).
		\]
	This is a union of open sets. The map $h$ factors through $P$ because it does so for each $(U_i \to P_i)$.
	\end{enumerate} 
\end{proof}

\begin{definition}[$\snglrdC_{\leq k , \leq n}$]\label{def:snglr-n}
Let $-1\leq k\leq n$ be integers.  
The full subcategory 
	\[
	\snglrdC_{k,n} \subset \snglrdC
	\]
consists of the $C^0$ stratified spaces of depth exactly $k$ and of \emph{pure} dimension exactly $n$. Similarly, we denote by
	\[
	\snglrdC_{\leq k , n}~\subset ~\snglrdC_{\leq k , \leq n} ~\supset ~ \snglrdC_{k,\leq n}~,
	\]
the obvious full subcategories. For instance, the last consists of those $C^0$ stratified spaces that need not have pure dimension, with dimension at most $n$ and depth exactly $k$. Finally, we let
	\[
	\snglrdC_{\leq k, \leq \infty}
	\]
be the full subcategory of $C^0$ stratified spaces with depth at most $k$.
We will utilize this category repeatedly in \S\ref{section.singular-manifolds}.
\end{definition}

\begin{cor}\label{inclusion-adjoint}\label{cor.inclusion-adjoint}
Let $-1\leq k\leq k'$ and $-1\leq n\leq n'$ be integers.  
The full inclusion
\[
\snglrdC_{\leq k , \leq n}~ \hookrightarrow~ \snglrdC_{\leq k', \leq n'}
\]
has a right adjoint implementing a localization.   
\end{cor}

\begin{proof}
That the inclusion is full is obvious, as both categories are defined as full subcategories of $\snglrdC$. We define the right adjoint $R$ by the assignment $(X\to P) \mapsto (X_{\leq k , \leq n} \to P_{\leq k, \leq n})$. Specifically, we have a pullback of posets
	\begin{equation}
	\label{pull-back<=k}
	\xymatrix{
	P_{\leq k, \leq n} \ar[r] \ar[d]
	& P \ar[d]
	\\
	\{(k'',n'')\}_{k''\leq k, n''\leq n} \ar[r]
	& \PP
	}
	\end{equation}
where the right vertical map is given by Lemma~\ref{dim-depth}.
The map $P_{\leq k, \leq n} \to P$ is a consecutive injection, so let
\[
(X_{\leq k, \leq n} \to P_{\leq k, \leq n} )
:=
X_{|P_{\leq k, \leq n}}.
\]
Lemma~\ref{strata} says this is a $C^0$ stratified space.

Since any open embedding of $C^0$ stratified spaces respects the maps to $\PP$ by Remark~\ref{remark.respect-stratification}, we send any open embedding $f: (X \to P) \to (X' \to P')$ to its restriction $f|_{X_{\leq k, \leq n}}$. By the universal property of pullbacks (\ref{pull-back<=k}) (or by checking), this assignment respects compositions.

There is an open embedding $(X_{\leq k, \leq n} \to P_{\leq k, \leq n})\hookrightarrow (X\to P)$ implementing a functor $\epsilon: R \to 1$, which is easily seen to satisfy the universal property of a counit.
\end{proof}

We now collect some results about how Properties (1)-(5) interact with depth and dimension. We leave the proofs as an exercise to the reader, as most of them follow from previous examples.

\begin{prop}\label{cone-increase}\label{prop.induction}
The collection of subcategories of $\snglrdC$ introduced in Definition~\ref{def:snglr-n} has the following properties:
\begin{enumerate}
	\item[(\ref{prop.epty})] $\emptyset \in \snglrdC_{\leq k , \leq n}$ for every $k \geq -1, n \geq -1$.
	\item[(\ref{prop.cones})] Let $X\in \snglrdC_{\leq k, n}$ be compact.  Then $\sC(X) \in \snglrdC_{n+1 ,n+1}$~.
	\item[(\ref{prop.products})] Let $X\in \snglrdC_{k , n}$ be nonempty.  Then $\RR\times X\in \snglrdC_{k , n+1}$~.
	\item[(\ref{prop.subs})] Let $X \in \snglrdC_{\leq k , \leq n}$ and let $U\hookrightarrow X$ be an open embedding.  Then $U\in \snglrdC_{\leq k , \leq n}$~.
	\item[(\ref{prop.covers})] Let $\cU$ be an open cover of $X$ with $\cU \subset \ob \snglrdC_{\leq k , \leq n}$.  Then $X\in \snglrdC_{\leq k , \leq n}$~.
\end{enumerate}
Among these properties, (\ref{prop.cones}) is the only one that increases depth, and it does so strictly.  
\end{prop}

\section{Stratified spaces}
\label{section.singular-manifolds}

In this portion of the paper we transition from $C^0$ stratified spaces to {\em conically smooth stratified spaces} -- these are $C^0$ stratified spaces with a choice of an atlas with appropriate regularity among transition maps, just as smooth manifolds are topological manifolds with smooth atlases. This regularity is what we call {\em conical smoothness}. 

After defining conically smoothness by induction, we prove some basic properties of the category of conically smooth stratified spaces. Just as the category of smooth manifolds (and smooth embeddings) admits an enrichment over Kan complexes, we prove that the category of conically smooth stratified spaces (and conically smooth embeddings) admits a Kan enrichment. We prove the category admits products, and that the set of atlases forms a sheaf on the site of $C^0$ stratified spaces. We then conclude the section with examples of conically smooth stratified spaces, with a categorical framework for examining those spaces with specified singularity types, and by introducing the appropriate analogue of composing cobordisms in the stratified setting.

\subsection{Differentiability along $\RR^i$}\label{sec.conical-smoothness}

We begin by introducing the stratified version of a derivative. Let $Z$ be a compact stratified topological space, and consider the stratified space $U = \RR^i \times \sC(Z)$. 
We denote its points by $(u, [s,z])$, where $(u,s,z) \in \RR^i \times \RR_{\geq 0} \times Z$ -- if $Z=\emptyset$ it is understood that $[s,z] = \ast$.  
Consider the identification
	\begin{equation}
	\label{eqn.tangent-cones}
	\sT\RR^i \times \sC(Z) \cong \RR^i_{\vec v} \times \RR^i \times \sC(Z) =  \RR^i_{\vec v} \times U = \{(\vec v, u, [s,z])\} .
	\end{equation}
We have a homeomorphism
	\[
	\gamma: \RR_{> 0} \times \sT\RR^i \times \sC(Z)
	\to \RR_{> 0} \times \sT\RR^i \times \sC(Z)
	\]
given by the equation
	\[
	(t,
	\vec v,u,[s,z]) \mapsto (t, t \vec v + u, u, [ts,z]).
	\]
One can consider $\gamma$ as a map $(t, u) \mapsto \bigl(\gamma_{t,u}\colon \RR^i_{\vec v}\times \sC(Z) \to \RR^i_{\vec v} \times \sC(Z)\bigr)$, and one has the identities 
	\[
	\gamma_{t_2, u_2} \circ \gamma_{t_1,u_1} = \gamma_{t_2 t_1, u_2 + t_2 u_1}
	\qquad \qquad
	(\gamma_{t,u})^{-1} = \gamma_{\frac{1}{t}, -\frac{u}{t}}.
	\]

\begin{remark}
The map $\gamma$ embodies the concept of scaling and translating -- we see it as capturing what remains of a vector space structure on a basic.  
\end{remark}

\begin{defn}[$f_\Delta$]
\label{defn.f-delta}
Let  $Z = (Z\to P)$ and $Z'=(Z'\to P')$ be compact stratified topological spaces. 
Let $f\colon U = \RR^i\times \sC(Z) \to \RR^{i'}\times \sC(Z')$ be a continuous stratified map for which $\sC(P) \to \sC(P')$ sends the cone point to the cone point.  
We denote by $f_{|\RR^i}: \RR^i \times \{\ast \} \to \RR^{i'} \times \{\ast\}$ 
the restriction to the cone point stratum. 
Consider the map
	\[
	\xymatrix{
	\RR^i_{\vec v} \times U \ar[rr]^{f_{|\RR^i}\times f}
	&& \RR^{i'}_{\vec v'} \times U'.
	}
	\]
Using the identification (\ref{eqn.tangent-cones}), we define $f_\Delta$ to be the map
	\[
	f_\Delta:
	\xymatrix{
	\RR_{>0} \times \sT\RR^i \times \sC(Z)
	\ar[rrr]^{\id_{\RR_{>0}} \times f_{|\RR^i} \times f}
	&&&
	\RR_{>0} \times \sT\RR^{i'} \times \sC(Z').
	}
	\]
\end{defn}

\begin{example}\label{example.ordinary-derivatives}
If $Z = Z' = \emptyset$, then $f$ is simply a function from $\RR^i$ to $\RR^{i'}$, and 
\begin{equation}\label{eqn.difference-quotient}
 (\gamma^{-1} \circ f_\Delta \circ \gamma)  
 (t,\vec v, u) =
 \left(t, {\frac{f(t \vec v + u) - f(u)} {t}}, f(u) \right).
\end{equation}
Note that the middle term is simply the usual difference quotient.
Then $f$ is $C^1$ if the limit as $t \to 0$ exists everywhere -- i.e., if there is a continuous extension
\[
\xymatrix{
\RR_{\geq 0} \times \sT\RR^i  \ar@{-->}[rr]^{\w{\sD}f}
&&
\RR_{\geq 0} \times \sT\RR^{i'}
\\
\RR_{>0} \times \sT\RR^i \ar[u] \ar[rr]^{\gamma^{-1} \circ f_\Delta \circ \gamma}
&&
\RR_{>0} \times  \sT\RR^{i'} .\ar[u]
}
\]
 
For example, when $i=i'=1$, the existence of $\w{\sD}f$ implies the existence of a right derivative ${\sf lim}_{t \to 0^+} {\frac 1 t} ( f(t+ u) - f(u))$ at every $u \in \RR$. Moreover, this right derivative is a continuous function of $u$ since $\w{\sD}f$ is. This ensures that $f$ is differentiable, and continuously so.

In higher dimensions, the condition that ${\sf lim}_{t \to 0^+} {\frac 1 t} (f(t\vec v+ u) - f(u))$ exists for all $u$ and $\vec v$ likewise implies that all partial derivatives exist. Moreover, these partials must depend continuously on $u$ and $\vec v$ since $\w{\sD}f$ is continuous; hence $f$ must be differentiable. As a consequence, one can conclude that the restriction of $\w{\sD} f$ to $t=0$ defines fiberwise {\em linear} maps $\RR^i_{\vec v} \times \{u\} \to \RR^{i'}_{\vec v'} \times \{f(u)\}$.

\end{example}

\begin{definition}[Conically smooth along]\label{def:cone-sm}
Let $Z = (Z\to P)$ and $Z'=(Z'\to P')$ be compact stratified topological spaces.
Let $f\colon \RR^i\times \sC(Z) \to \RR^{i'}\times \sC(Z')$ be a continuous stratified map.
We say $f$ is $C^1$ \emph{along $\RR^i$} if $\sC(P) \to \sC(P')$ sends the cone point to the cone point, and if there is a continuous extension
	\[
	\xymatrix{
	\RR_{\geq 0} \times \sT\RR^i \times \sC(Z)  \ar@{-->}[rr]^{\w{\sD}f}
	&&
	\RR_{\geq 0} \times \sT\RR^{i'} \times \sC(Z')
	\\
	\RR_{>0} \times \sT\RR^i \times \sC(Z)  \ar[u] \ar[rr]^{\gamma^{-1} \circ f_\Delta \circ \gamma}
	&&
	\RR_{>0} \times  \sT\RR^{i'} \times \sC(Z'). \ar[u]
	}
	\]
	
If the extension exists, it is unique by continuity. We denote the restriction of $\w{\sD}f$ to $t=0$  by
	\begin{equation}\label{eqn.derivative}
	\sD f: \sT\RR^i \times \sC(Z)  \to  \sT\RR^{i'} \times \sC(Z').
	\end{equation}
Finally, we let $\sD_u f$ denote the composite $\RR^i_{\vec v} \times \{u\} \times \sC(Z) \to \RR^{i'}_{\vec v'} \times \{f(u,\ast)\} \times \sC(Z') \to \RR^{i'}_{\vec v'} \times \sC(Z')$ where last map is projection. For instance, on $\vec v := (\vec v, \ast) \in \RR^i_{\vec v} \times \sC(Z)$, 
	\[
	\sD_{u} f(\vec v) :=
	\underset{t \to 0}{\sf lim} \gamma_{ {\frac 1 t}, {\frac 1 t} f(u)} \circ f \circ \gamma_{t, -u}(\vec v) \in \RR^{i'}_{\vec v'} \times \sC(Z)
	\]
for any $u \in \RR^i$. 
For $r>1$, we declare that $f$ is {\em $C^r$ along $\RR^i$} if it is $C^1$ along $\RR^i$ and if $\sD f: \RR^i \times \RR^i \times \sC(Z) \to \RR^{i'} \times \RR^{i'} \times \sC(Z')$ is $C^{r-1}$ along $\RR^i \times \RR^i$. We say that $f$ is {\em conically smooth along $\RR^i$} if it is $C^r$ along $\RR^i$ for every $r \geq 1$.

\end{definition}

\begin{remark}
The notion of conical smoothness {\em along $\RR^i$} is less restrictive than the notion of {\em conical smoothness}, which we will define in Definition~\ref{defn.conical-smoothness}. In fact, being conically smooth along $\RR^i$ imposes little regularity on the behavior of $f$ away from $\RR^i \times \{\ast\}$. As an example, if $i = 0$ and $Z = \ast$, the condition that $f: \RR_{\geq 0} = \sC(Z) \to \sC(Z)$ be $C^1$ along $\RR^i$ is the condition that $f$ have a right derivative at the cone point $\ast = 0 \in \RR_{\geq 0}$. And in this case, being $C^1$ along the cone point implies that $f$ is conically smooth along the cone point, as $\sD f: \sC(Z) \to \sC(Z)$ is just a linear function with slope given by the right derivative of $f$ at 0. In Definition~\ref{defn.conical-smoothness}, we will require that $f$ also be smooth away from the cone point.
\end{remark}

\begin{remark}\label{remark.usual-tangent-bundle-map}
Notice that the map $\sD f$ fits into a commutative diagram of stratified spaces
	\[
	\xymatrix{
	\RR^i_{\vec v}\times U \ar[r]^-{\sD f}  \ar[d]
	&
	\RR^{i'}_{\vec v'} \times U' \ar[d]
	\\
	\RR^i_{\vec v}  \ar[r]^-{f_{|\RR^i}}  
	&
	\RR^{i'}_{\vec v'}~.
	}
	\]
By Example~\ref{example.ordinary-derivatives}, the restriction of $\sD f$ to $\sT\RR^i \times \{\ast\}$ factors through $\sT\RR^{i'}$ as the map of ordinary tangent bundles induced by the ordinary (smooth) map $f_{|\RR^i}\colon \RR^i \to \RR^{i'}$.  
\end{remark}

\begin{example}
The identity map is conically smooth along $\RR^i$, and $\w{\sD}(\id)$ is the identity map on $\RR_{\geq 0} \times \sT\RR^{i} \times \sC(Z)$. More generally, let $h: [0,\infty) \times Z \to Z'$ be continuous and assume $f$ lifts to a map of the form $\w{f}(u,s,z) = ( g(u), s, h_s(z)) \in \RR^{i'}\times\RR_{\geq 0} \times Z'$ for $s \geq 0$. Then $\w{\sD}f|_{t=0}$ is given by
	\[
	(1, \vec v, u, [s,z])
	\mapsto
	(1, \sD g_u(\vec v), g(u), [s,h_0(z)]).
	\]
\end{example}

\begin{example}
Let $Z$ be a point and $f: \sC(Z) \to \sC(Z)$ be given by $f(s) = s^k$. Then (i) if $k<1$, $\w{\sD }f$ does not exist, (ii) if $k=1$, then $\w{\sD } f$ is the identity map as in the previous example, and (iii) if $k>1$, $\w{\sD }f|_{t = 0}$ is the constant map to the cone point.
\end{example}

\begin{example}
While there is a polar coordinates identification of underlying topological spaces $\sC(S^{n-1})\cong \RR^n$, conical smoothness at the cone point $\ast$ is weaker than smoothness at the origin of $\RR^n$.  Indeed, for $f\colon S^{n-1} \cong S^{n-1}$ an arbitrary homeomorphism, the map of stratified spaces $\sC(f)\colon \sC(S^{n-1}) \to \sC(S^{n-1})$ is conically smooth at the cone point $\ast$.  
\end{example}

\begin{lemma}[Chain rule]\label{smooth-compose}
Let $U\xra{f} V \xra{g}W$ be maps of stratified spaces between $C^0$ basics of equal depths.  
Suppose both $f$ and $g$ are conically smooth along their respective Euclidean coordinates.
Then $gf\colon U \to W$ is conically smooth along its Euclidean coordinates, and $\sD g\circ \sD f = \sD(gf)$.  
\end{lemma}

\begin{proof}
The proof directly follows that of the chain rule of multivariable calculus.  
\end{proof}

\begin{lemma}[Inverse function theorem (weak version)]\label{ivt}
Let $f: U \to U'$ be a conically smooth map between $C^0$ basics of equal depth; assume $f$ is injective and open. Suppose $\sD f\colon \RR^i_{\vec v} \times U \to \RR^{i'}_{\vec v'} \times U'$ is also an open embedding.  Then there is an open embedding $i\colon U'\hookrightarrow U'$ which is conically smooth along $\RR^{i'}$, whose image lies in the image of $f$, such that the composite map
\[
U' \xra{i} U' \xra{f^{-1}} U
\]
is conically smooth along $\RR^{i'}$.  
\end{lemma}

\begin{proof}
The map $f^{-1}\colon U' \dashrightarrow U$ is an open embedding, where it is defined.  
Conically smooth self-embeddings $U'\hookrightarrow U'$ form a base for the topology about $\RR^{i'}$ -- i.e., any other neighborhood of $\RR^{i'} \times \{\ast\}$ contains the image of a smooth self-embedding. So we must show $f^{-1}$ is conically smooth along $\RR^{i'}$, where it is defined.  
The proof of this directly follows that of the inverse function theorem of multivariable calculus.  
\end{proof}

\begin{cor}\label{composition-open}
Let $f\colon U \to V$ be an open embedding between $C^0$ basics which is conically smooth along the Euclidean coordinates of $U$.  
Suppose $\sD f\colon \RR^i_{\vec v}\times U \to \RR^j_{\vec v}\times V$ is also an open embedding.
Then $\sD^k f \colon (\RR^i_{\vec v})^{\times k} \times U \to (\RR^j_{\vec v})^{\times k} \times V$ is an open embedding for all $k\geq 0$.  
\end{cor}

\subsection{The definition of conically smooth stratified spaces}
Based on the concept of \emph{conical smoothness along $\RR^i$} of the previous section, we give a definition of conically smooth stratified spaces as $C^0$ stratified spaces equipped with a maximal {\em conically smooth} atlas.  
The advantages of conically smooth stratified spaces over $C^0$ stratified spaces are the familiar ones from ordinary differential topology; namely, consequences of transversality. 

Before we give the definitions, we outline the construction for the benefit of the reader. The two main players are a sheaf $\sm$ of atlases, and a category $\bscd$ of basic singularity types.

\begin{enumerate}
	\item It is convenient to define an equivalence relation on the set of possible atlases. This allows us to define a presheaf 
		\[
		\sm: (\snglrdC)^{\op} \to \Set
		\] 
	where an element $\cA \in \sm(X)$ will be a choice of such an equivalence class. As will be obvious, an equivalence class $\cA$ will always have a preferred representative given by a {\em maximal} atlas. We will hence often refer to an element $\cA$ as a {\em maximal atlas} for $X$.	We will prove in Lemma~\ref{sheaf} that this presheaf $\sm$ is in fact a sheaf.
	\item We introduce the notion of {\em basic} conically smooth stratified spaces. While smooth manifolds admit covers by manifolds that are all diffeomorphic to $\RR^n$, stratified spaces do not. 
	Even so, conically smooth stratified spaces are locally standard -- every point has a small neighborhood dictating the singularity type of the point, and every singularity type is an object of the category $\bscd$. We call these building blocks {\em basics}. At the end of the day, one should think of $\bscd$ as a full subcategory of $\snglrd$, the category of all conically smooth stratified spaces. We will prove it is, after all definitions are complete, in Proposition~\ref{prop.basics-are-singular}.
	\item As we saw in Proposition~\ref{prop.induction} and Remark~\ref{remark.non-increasing}, the notions of depth and dimension give us convenient parameters for inductive definitions. 
Following this intuition, we define $\sm$ and $\bscd$ inductively by {\em depth}. When we write $\bscd_{\leq k, \leq n}$, we mean the category consisting of basics with depth $\leq k$ and dimension $\leq n$. In the induction, we will define $\bscd_{\leq k+1, \leq \infty}$ by assuming we have defined $\bscd_{\leq k, \leq \infty}$ and the presheaf $\sm$ on $\snglrdC_{\leq k, \leq \infty}$.
Then we will declare 
		\[
		\bscd = \bigcup_{k \geq -1} \bscd_{\leq k, \leq \infty}.
		\]
	\item To implement the induction, it will be convenient to introduce a functor
		\begin{equation}\label{eqn.cross-R}
		\RR \times - : \bscd_{\leq k, \leq \infty} \to \bscd_{\leq k, \leq \infty}
		\end{equation}
	that takes a basic $U$ to the basic $\RR \times U$. Note that we have also constructed an endofunctor $(\RR \times -)$ defined on $\snglrdC$, which we do not distinguish in notation from the functor defined on $\bscd$. The functor on $\bscd_{\leq k, <\infty}$ will induce a natural transformation
		\begin{equation}\label{eqn.cross-R-nat-trans}
		\sm(-) \to \sm(\RR \times -) 
		\end{equation}
	for $C^0$ stratified spaces of depth $\leq k$. The philosophy is that an atlas on $X$ determines an atlas on $\RR \times X$.
\end{enumerate}

\begin{definition}[The inductive hypothesis]
When we {\em assume the inductive hypothesis for $k \in \ZZ_{\geq -1}$}, we assume three things have been defined: (i) The category $\bscd_{\leq k, \leq \infty}$, (ii) the presheaf $\sm$ on the full subcategory $(\snglrdC_{\leq k, \leq \infty})^{\op} \subset (\snglrdC)^{\op}$, and (iii) the functor (\ref{eqn.cross-R}) along with the natural transformation (\ref{eqn.cross-R-nat-trans}).
\end{definition}

\begin{definition}[Base case: the empty manifold]
Let $\emptyset \in \snglrdC_{\leq -1, \leq \infty}$ be the empty $C^0$ stratified space. 
We declare the empty manifold to have a unique conically smooth atlas:
	\[
	\sm(\emptyset) = \{\ast\}
	\]
and that there are no basics of negative depth. That is,
	\[
	\bscd_{\leq -1, \leq \infty} = \emptyset
	\] 
is the empty category.
\end{definition}

\begin{remark}
The inductive hypothesis is satisfied for $k=-1$, as the functor (\ref{eqn.cross-R}) and the natural transformation (\ref{eqn.cross-R-nat-trans}) are uniquely determined.
\end{remark}

\begin{definition}[Basics]\label{defn.basics}
Assume the inductive hypothesis for $k$. An object of $\bscd_{\leq k+1, \leq \infty }$ is a pair
	\[
	(U,\cA_Z)
	\]
	consisting of a $C^0$ basic $U = \RR^i\times \sC(Z)$ of depth at most $k+1$, together with a maximal atlas $\cA_Z\in \sm(Z)$.

Let $U= (\RR^i\times \sC(Z),\cA_Z)$ and $V= \bigl(\RR^j \times \sC(Y),\cA_Y\bigr)$ be two objects of $\bscd_{\leq k+1, \leq \infty}$.
A morphism from $U$ to $V$ is the datum of an open embedding $f\colon U \hookrightarrow V$ of underlying stratified topological spaces which satisfies the following properties.  
\begin{enumerate}

	\item If, on the level of posets, $f$ sends the cone point to the cone point, then we require the following properties.
	\begin{itemize}
		\item $f$ is conically smooth along $\RR^i$.
		\item The map $\sD f\colon \RR^i_{\vec v} \times U \to \RR^j_{\vec v} \times V$ is injective.  
		\item There is an equality 
			\begin{equation}
			\label{eqn.atlas-preserved}
			\cA_{|f^{-1}(V\smallsetminus \RR^j)} = (f_{|f^{-1}(V\smallsetminus \RR^j)})^\ast \cA_{|V\smallsetminus \RR^j}. 
			\end{equation}
	\end{itemize}
	\item If, on the level of posets, $f$ does \emph{not} send the cone point to the cone point, then $f$ factors as open embeddings of stratified topological spaces $f\colon U \xra{~f_0~} \RR^j\times \RR_{>0}\times Y \hookrightarrow V$.  
	We require that $\{(U,f_0)\}$ is  is a member of an atlas which represents $\cA_{\RR^j\times \RR_{>0}\times Y}$
\end{enumerate}

We call such a morphism $f$ a \emph{conically smooth} open embedding.
\end{definition}

\begin{remark}
Note that to make sense of equation (\ref{eqn.atlas-preserved}), we have utilized the natural transformation (\ref{eqn.cross-R-nat-trans}): the maximal atlas on $V \smallsetminus \RR^j = \RR^j \times Y$, for instance, is induced by the maximal atlas on $Y$.   We postpone the definition of this natural transformation until Definition~\ref{defn.products-j}, after we have defined the notion of an atlas.
\end{remark}

\begin{example}
Pick an $f$ that sends the cone point to the cone point at the level of posets. By Remark~\ref{remark.usual-tangent-bundle-map}, $\sD f$ restricts to a map of tangent bundles $\sT\RR^i \times \{\ast\} \to \sT\RR^j \times \{\ast\}$. The injectivity condition in (1) ensures that the restriction to the zero section $f_{|\RR^i}: \RR^i \into \RR^j$ must be an embedding. In Lemma~\ref{only-equivalences}, we will in fact see that if $f$ is a conically smooth embedding sending the cone point to the cone point at the level of posets, $U$ and $V$ must be isomorphic stratified spaces.
\end{example}

\begin{definition}
We will refer to an object of $\bscd_{\leq k+1, \leq \infty}$ as a \emph{basic} (of depth at most $k+1$).  
We will often denote an object $(U,\cA_Z)$ of $\bscd_{\leq k+1, \leq \infty}$ by its underlying stratified topological space $U$, with the element $\cA_Z$ understood.  
\end{definition}

\begin{example}
Note that (\ref{eqn.atlas-preserved}) is satisfied for $k=-1$ automatically, as  $V \smallsetminus \RR^j = \emptyset$, and $\sm (\emptyset)$ has a unique element. Hence the category $\bscd_{\leq 0, \leq \infty}$ is equivalent to the category whose objects are $\RR^i$ for $i \in \ZZ_{\geq 0}$, and whose morphisms are smooth open embeddings with respect to the standard smooth structures. In particular, there are no morphisms between $\RR^i$ and $\RR^{i'}$ for $i \neq i'$.
\end{example}

\begin{lemma}[Basics form a basis for basics]\label{lemma.small-opens}
Let $U = \RR^i \times \sC(Z)$ be a basic of depth $k+1$. 
Then the collection of open subsets 
	\[
	\{\varphi(V) \,| \, V \xra{\varphi} U \text{ \em{ is a morphism in} } \bscd_{\leq k+1, \leq \infty}\}
	\]
forms a basis for the topology of $U$. 
\end{lemma}

\begin{proof}
We must show that for any open set $V \subset U$ and any $x \in V$, there exists a morphism $i: W \to U$ in $\bscd_{\leq k+1, \leq \infty}$ such that $x \in i(W)$ and $i(W) \subset V$. 
We proceed by induction on $k$.

Assume $x$ is in the cone stratum $\RR^i \times \{*\}$. 
Then we can choose some $s_0 \in \RR_{>0}$, and some open ball $R \subset \RR^i$, such that $V$ contains the open set 
	\[
	R \times \{[s,z] \,  | 0\leq s < s_0, \, z \in Z\}.
	\]
Fixing an smooth open embedding $r: \RR^i \to R$,
and choosing a decreasing smooth embedding $g: \RR_{\geq 0} \to [0,s_0)$, we set $i: W = U \to U$ to be the map $(u,[s,z]) \mapsto (r(u),[g(s),z])$.

If $x$ is not in the cone stratum $\RR^i \times \{\ast\}$, we can choose a basic chart $\psi: U' \to \RR^i\times \RR_{>0} \times Z \subset \RR^i\times \sC(Z)$ containing $x$, and we may assume $V$ to be contained in $\psi(U')$. Since necessarily $\depth(U')\leq k$, the proof is finished by induction.
\end{proof}

\begin{definition}[Atlases]
Assume the category $\bscd_{\leq k+1,\leq\infty}$ has been defined, and let $X\in \snglrdC_{\leq k+1, \leq \infty}$.
An {\em atlas for $X$} is a collection of pairs 
	\[
	\w{\cA} = \{\bigl( U, \varphi\colon U \hookrightarrow X\bigr)\}~,
	\]
in which $U \in \bscd_{\leq k+1, \leq \infty}$ and $\varphi$ an open embedding of stratified topological spaces, and this collection must satisfy the following properties.
\begin{itemize}
	\item[{\bf Cover:}]
	$\w{\cA}$ is an open cover of $X$.  
	\item[{\bf Atlas:}]
	For each pair $(U,\varphi),(V,\psi)\in \w{\cA}$ and $x\in \varphi(U)\cap\psi(V)$ there is a diagram 
		$
			U\xla{f} W\xra{g} V
		$
	in $\bscd_{\leq k+1, \leq \infty}$
	such that the resulting diagram of maps of stratified topological spaces
		\begin{equation}\label{eqn.atlas}
			\xymatrix{
			W  \ar[r]^g  \ar[d]^f
			&
			V  \ar[d]^\psi
			\\
			U  \ar[r]^\varphi
			&
			X
			}
		\end{equation}
	commutes and the image $\varphi f(W) = \psi g(W)$ contains $x$.  
\end{itemize}
Two atlases $\w{\cA}$ and $\w{\cA}'$ are declared to be equivalent if their union $\w{\cA}\cup \w{\cA}'$ is also an atlas.
\end{definition}

\begin{lemma}\label{lemma.atlas-transitivity}
The relation 
	\[
	\w{\cA} \sim \w{\cA'} \iff \w{\cA} \cup \w{\cA'} \text{ is an atlas}
	\]
is an equivalence relation.
\end{lemma}
\begin{proof}

The only non-obvious point is whether the relation is transitive.
We proceed by induction on depth, assuming the relation to be transitive for spaces of depth $\leq k$.

Assume that $\w{\cA_0} \cup \w{\cA_1}$ and $\w{\cA_1}\cup\w{\cA_2}$ are atlases, and fix some $x \in \varphi_0(U_0) \cap \varphi_2(U_2)$ for charts $(U_j,\varphi_j)$ in  $\w{\cA_j}$. By the covering property of $\w{\cA_1}$, the point $x$ is also in the open set $\varphi_1(U_1)$ for some chart in $\w{\cA_1}$. Since we have assumed $\w{\cA_j} \sim \w{\cA_{j+1}}$, there are basics $W_{01}$ and $W_{12}$ fitting into the commutative diagram
	\[
	\xymatrix{
	&W_{01} \ar[r] \ar[d]^{f_{01}} &U_0\ar[dd]^{\varphi_0}
	\\
	W_{12} \ar[r]^{f_{12}} \ar[d] &U_1 \ar[dr]^{\varphi_1}
	\\
	U_2 \ar[rr]^{\varphi_2} && X
	}
	\]
where the maps to the $U_j$ are morphisms in $\bscd_{\leq k+1, \leq \infty}$. We have now reduced to the problem of exhibiting morphisms $W \to W_{01}, W_{12}$ fitting into the upper-left corner of the commutative diagram. By taking $V \subset U_1$ to be the open subset given by $f_{01}(W_{01}) \cap f_{12}(W_{12})$, an application of Lemma~\ref{lemma.small-opens} and the Inverse Function Theorem (Lemma~\ref{ivt}) completes the proof.
\end{proof}

\begin{definition}[The sheaf of conically smooth structures]
Assume $\bscd_{\leq k+1, \leq \infty}$ has been defined, and let $X\in \snglrdC_{\leq k+1, \leq \infty}$.
An element $\cA\in \sm(X)$ is an equivalence class $\cA = [\w{\cA}]$ of atlases. 

Given a stratified open embedding $f: X \to X'$, the map $\sm(X') \to \sm(X)$ is the map induced by pulling back atlases. Specifically, every chart $\varphi: U \to X'$ in $\w\cA'$ defines an open set $\varphi^{-1}(\varphi(U) \cap f(X))$ inside $U$. Applying Lemma~\ref{lemma.small-opens} to these open sets for every chart, we obtain an atlas on $X$. The equivalence class of this atlas is independent of the choices made when we apply Lemma~\ref{lemma.small-opens}.
\end{definition}

\begin{remark}
Note that the pull-back of an atlas $\w{\cA'}$ is not defined in any natural way, but the pull-back of the equivalence class $\cA'$ is. 
\end{remark}

\begin{remark}\label{maximal}
Let $X$ be a $C^0$ stratified space, and let $\cA\in \sm(X)$.  
Then the set of representatives of $\cA$ is a poset, ordered by inclusion of collections. Moreover, this poset has a maximal element given by the  union $\bigcup_{\w{\cA}\in \cA} \w{\cA}$. Because of this, we will often identify an element of $\sm(X)$ with a maximal atlas for $X$.
\end{remark}

\begin{example}\label{example.smooth-manifolds}
If $X$ is a $C^0$ stratified space with a single stratum, it is a topological manifold by Example~\ref{example.topological-manifolds}. An element $\cA \in \sm(X)$ is a choice of (equivalence class of) smooth atlas on $X$, which one can think of as a maximal smooth atlas on $X$.
\end{example}

\begin{example}
Let $X = U =  \RR^4\times \sC(\emptyset)$, and consider two different one-element atlases: an atlas $\w{\cA}$ consisting of the identity $\varphi = \id: \RR^4 \to X$, and another atlas $\w{\cA'}$, where the chart $\psi: \RR^4 \to X$ is a homeomorphism that is not a diffeomorphism. Then $\w{\cA}$ and $\w{\cA'}$ are not equivalent atlases: at any point $x \in \RR^4$ for which $\psi$ is not smooth, it is impossible to find smooth maps $g, f$ as in (\ref{eqn.atlas}) such that $\psi \circ g$ and $\varphi \circ f$ are both smooth.
\end{example}

\begin{definition}[Products with $\RR$ and the natural transformation]\label{defn.products-j}
The functor
\[
 (\RR \times -): \bscd_{\leq k+1, \leq \infty} \to \bscd_{\leq k+1, \leq \infty}
\] 
(which also appears in  (\ref{eqn.cross-R})) is defined on objects and morphisms by the assignments
	\[
	(U,\cA_Z)
	\mapsto 
	(\RR \times U, \cA_Z)~,
	\qquad
	\qquad
	(f: U \to V)
	\mapsto
	(\id_{\RR} \times f: \RR \times U \to \RR \times V).
	\]  
Finally, an atlas on a $C^0$ stratified space $X$ induces an atlas on $\RR \times X$ by sending any chart $(U, \varphi)$ to the induced chart $(\RR \times U, \id_\RR \times \varphi)$. This defines a natural map of sets
	\[
	\sm(X) \to \sm(\RR \times X).
	\]
\end{definition}

\begin{definition}\label{defn.depth-k-snglr}
Assume we have defined the presheaf $\sm$ on the full subcategory $\snglrdC_{\leq k+1,\leq \infty}$. We denote by
	\[
	\snglrd_{\leq k+1, \leq \infty}
	\]
the category whose objects are pairs $(X,\cA)$ where $X \in \snglrdC_{\leq k+1,\leq \infty}$ and $\cA \in \sm(X)$. A morphism $(X,\cA) \to (X',\cA')$ is a stratified open embedding $f\colon X\hookrightarrow X'$ for which $\cA_X = f^\ast \cA_{X'}$.
\end{definition}

\begin{example}\label{example.first-smooth-manifolds}
As a special case, we let
	\[
	\mfldd := \snglrd_{\leq 0, \leq \infty}.
	\]
Parsing the definition, this category's objects are smooth manifolds, and its morphisms are smooth open embeddings. (See Example~\ref{example.smooth-manifolds}.)
\end{example}

In principle this inductive definition is finished, but we must account for stratified spaces who may have unbounded depth -- i.e., who do not lie in $\snglrd_{\leq k, \leq \infty}$ for any finite $k$.

\begin{defn}[$\sm$ for stratified spaces with unbounded depth]
\label{defn.Sm}
Let 
\[
 \snglrdC_{<\infty,\leq \infty} = \bigcup_{k \geq -1} \snglrdC_{k,\leq \infty}
\]
be the category of all $C^0$ stratified spaces with bounded depth. We extend $\sm$ by the right Kan extension
\[
 \xymatrix{
 (\snglrdC_{< \infty, \leq \infty})^{\op} \ar[r]^-{\sm}  \ar@{^{(}->}[d]
 & \Set \\
 (\snglrdC)^{\op}.
  \ar@{-->}[ur]
 }
\]
and refer to this extension as $\sm$ as well.
\end{defn}

We prove that $\sm$ is a sheaf in Lemma~\ref{lemma.sm-sheaf}.

\begin{definition}[$\snglrd$]\label{def:snglr}\label{defn.snglr}
We define the category 
\[
\snglrd
\]
of \emph{conically smooth stratified spaces} and embeddings. An object is a pair $(X,\cA)$ consisting of a $C^0$ stratified space and an element $\cA \in \sm(X)$. A morphism $(X,\cA) \to (X',\cA')$ is a stratified open embedding $f\colon X\hookrightarrow X'$ for which $\cA_X = f^\ast \cA_{X'}$.
\end{definition}

\begin{remark}
Unless the discussion requires otherwise, we will denote a conically smooth stratified space $(X,\cA)$ simply by its underlying stratified topological space $X$.  
We will refer to a section $\cA\in \sm(X)$ as a \emph{conically smooth structure on $X$}.  
\end{remark}

The following is a conically smooth analogue of Lemma~\ref{topological-basis}.

\begin{prop}[Basics form a basis]\label{basics.basis}
\label{prop.basics-form-a-basis}
For $X$ a conically smooth stratified space, the collection $\{U\hookrightarrow X\}$, of all conically smooth open embeddings from basics into $X$, forms a basis for the topology of $X$.  
\end{prop}

\begin{proof}
Because $X$ admits an atlas, then this collection of embeddings forms an open cover.
The result then follows from Lemma~\ref{lemma.small-opens}.
\end{proof}

\subsection{Auxiliary maps and $\Stratd$}
We introduce the concept of a {\em conically smooth} map between conically smooth stratified spaces. This is a more general concept than that of morphisms in $\snglrd$, just as smooth maps are more general than smooth open embeddings. 

\begin{definition}[Conically smooth]\label{def:c-infty}\label{defn.conical-smoothness}
Let $X$ and $Y$ be conically smooth stratified spaces.
Let $f\colon X \to Y$ be a continuous stratified map between their underlying stratified topological spaces.  
By induction on the depth of $Y$, we define what it means for $f$ to be \emph{conically smooth}.  
If $\depth(Y)=-1$ then $X = \emptyset = Y$ and we deem the unique such $f$ to be conically smooth.

Suppose $X = U= \RR^i \times \sC(Z)$ and $Y = V=\RR^j \times \sC(Y)$ are basics.
In this case, say $f$ is \emph{conically smooth} if it has the following properties.
\begin{enumerate}
\item If, on the level of posets, $f$ does \emph{not} send the cone point to the cone point, then $f$ factors as $f\colon U \xra{~f_0~} \RR^j\times \RR_{>0}\times Y \hookrightarrow \RR^j\times \sC(Y)$.
In this case, we require that $f_0$ is conically smooth -- this is well-defined by induction.

\item
If, on the level of posets, $f$ sends the cone point to the cone point, then we require that 
$f$ is conically smooth along $\RR^i$, and the restriction $f^{-1}(V\smallsetminus \RR^j) \to V\smallsetminus \RR^j$ is conically smooth.  
\end{enumerate}
For $X$ and $Y$ general, say $f$ is \emph{conically smooth} if for each pair of charts $\varphi\colon U \hookrightarrow X$ and $\psi\colon V \hookrightarrow Y$ with $f(\varphi(U)) \subset \psi(V)$, the composition 
\[
\psi^{-1} f \varphi \colon U \to V
\]
is conically smooth.
\end{definition}

\begin{example}
If $V$ has depth 0, then $V \smallsetminus \RR^j$ is empty. Hence a conically smooth map $U \to V$ is any map which is conically smooth along $\RR^i$. In particular, if $U$ also has depth 0, a conically smooth map $U \to V$ is simply a smooth map from $\RR^i$ to $\RR^j$. (As a consequence, if $X$ and $Y$ have depth 0, then a conically smooth map in our sense is just a smooth map in the usual sense.)
\end{example}

\begin{example}\label{ex.sub-man}
A smooth map $M\to M'$ between smooth manifolds is an example of a conically smooth map.
\end{example}

\begin{example}\label{ex.inc-is-smooth}
If $(X\to P)$ is a conically smooth stratified space with $Q\subset P$ a consecutive sub-poset, then the inclusion $X_{|Q}\hookrightarrow X$ is a conically smooth map. See Lemma \ref{lemma.strata-smooth}.
\end{example}

The following basic property, that conically smooth maps are closed under composition, follows by a routine induction argument on depth.  
\begin{prop}\label{c-infty-compose}
For $X \xra{f} Y \xra{g} Z$ conically smooth maps, the composite $gf$ is again conically smooth.
\end{prop}

\begin{example}\label{ex.diagonal}

Let $X$ be a conically smooth stratified space, and consider the product stratified space of \S\ref{sec.products}.
Then the diagonal map 
\[
X\xra{\sf diag} X\times X
\]
is an example of a conically smooth map.
To see this, it is enough to work locally and assume $X=U = \RR^i\times \sC(Z)$ is a basic.  
Inspecting the structure of $X\times X$ as a conically smooth stratified space (\S\ref{sec.products}), the problem amounts to showing the map to the join $Z \xra{\sf diag} Z\times Z \xra{\{\frac{1}{2}\}} Z\star Z$ is conically smooth.
After Proposition~\ref{c-infty-compose}, this follows by induction on depth, with the case of depth zero an instance of Example~\ref{ex.sub-man}.  
\end{example}

In the next section we will prove the following compatibility between the definition of a morphism in Definition~\ref{defn.snglr} and {\em conically smooth} stratified open embeddings:

\begin{prop}\label{open-smooth}
Let $(X,\cA_X)$ and $(X',\cA_{X'})$ be conically smooth stratified spaces. 
Let $f\colon X\hookrightarrow X'$ be a stratified open embedding.  
Then $f$ is conically smooth if and only if $f^*(\cA_{X'}) = \cA_X$.
\end{prop}

There is the following fundamental result.  
\begin{lemma}[Inverse function theorem]\label{ivt-again}
Let $f\colon X\to Y$ be a conically smooth map between conically smooth stratified spaces, and let $x\in X$ be a point.
Then there is an open neighborhood $x\in O\subset X$ for which the restriction $f_{|O}\colon O\to Y$ is an open embedding if and only if there are centered coordinates $(U,0)\xra{\varphi} (X,x)$ and $(V,0)\xra{\psi}(Y,f(x))$ for which the derivative of the composite $\sD_0 (\psi^{-1}f\varphi) \colon \sT_0 U \to \sT_0 V$ is an isomorphism.  
\end{lemma}

\begin{proof}
The ``only if'' statement is manifest.
Suppose $\sD_x f$ is an isomorphism.
The problem is a local one, so we can assume $(X,x)=(U,0)$ and $(Y,f(x))=(V,0)$ are centered basics.
After Lemma~\ref{ivt}, it is enough to show that $f\colon U \to V$ is open near $0\in U$.  
Because $\sD_0f$ is an isomorphism, then in particular it is open.
So for each open neighborhood $0\in O\subset U$ such that $0\in \sD_0f(O)\subset V$ is an open neighborhood.
Because $V$ is locally compact and Hausdorff, there is an open neighborhood $0\in P\subset \ov{P} \subset \sD_0 f(O)$ whose closure is compact and contained in this image open subset.  
Because $\sD_0f = \underset{t\to 0} {\sf lim} \gamma_{\frac{1}{t},0} \circ f\circ \gamma_{t,0}$, then for $t$ large enough there remains the inclusion $P\subset \gamma_{\frac{1}{t},0}\circ f\circ \gamma_{t,0}(O)$, and thereafter the inclusion $\gamma_{t,0}(P) \subset f\bigl(\gamma_{t,0}(O)\bigr)$.  
Provided $O$ has compact closure, then the collection of open subsets $\bigl\{\gamma_{t,0}(O)\mid t\geq 0\bigr\}$ form a base for the topology about $0\in U$.
We conclude that $f$ restricts to an open map on some neighborhood of $0\in U$, as desired.  
\end{proof}

\begin{definition}[$\Stratd$]\label{def.strat}
The category of \emph{conically smooth stratified spaces},
\[
\Stratd~,
\]
has objects which are conically smooth stratified spaces; its morphisms are conically smooth maps. 
\end{definition}

\begin{observation}\label{admits-pullbacks}
We make three observations about $\Stratd$.
\begin{itemize}
\item
Proposition~\ref{c-infty-compose} grants that morphisms in $\Stratd$ indeed compose.  
\item 
Proposition~\ref{open-smooth} grants that $\snglrd \subset\Stratd$ is the subcategory consisting of the same objects and only those morphisms which are open embeddings.  
\item 
Corollary~\ref{snglr-products}, to come, grants that $\Stratd$ admits finite products.  
\end{itemize}
\end{observation}

\subsection{Basic properties of $\snglrd$}
By Remark~\ref{remark.open-cover}, there is a natural notion of open cover, hence a Grothendieck topology, for $\snglrdC$.
\begin{lemma}\label{sheaf}
\label{lemma.sm-sheaf}
The presheaf $\sm\colon (\snglrdC)^{\op} \to \Set$ is a sheaf.  
\end{lemma}

\begin{proof}

Let $\cU\subset \snglrdC_{/X}$ be a covering sieve of $X$ (Definition \ref{def.sieve}).  
We must show that the universal map of sets
	\[
	\sm(X) \longrightarrow \sm(\cU) := \underset{O\in \cU}{\sf lim} \sm(O)
	\]
is an isomorphism.
To see that this map is injective, notice that if two atlases $\w{\cA}$ and $\w{\cA}'$ on $X$ restrict to equivalent atlases on each $O\in\cU$, then they are equivalent.
To see that the map is surjective, let $(\cA_O)\in \sm(\cU)$. Choose a representative $\w{\cA}_O$ for each $\cA_O$, and consider the collection $\w{\cA} := \underset{O\in \cU}\bigcup \{U\xra{\varphi} O \hookrightarrow X\mid (U,\varphi)\in \w{\cA}_O\}$.  
This collection is an atlas on $X$ (for instance, by applying Lemma~\ref{lemma.small-opens}).  Moreover, its restriction to each $O\in \cU$ is equivalent to $\w{\cA}_O$.  
\end{proof}

The category of basics should be thought of as a full subcategory of all conically smooth stratified spaces. We justify this with the following:

\begin{prop}\label{prop.basics-are-singular}
There is a fully faithful functor
	\[
	\bscd \hookrightarrow \snglrd.
	\]
\end{prop}

\begin{example}
Restricting to the full subcategories $\bscd_{\leq 0, \leq \infty} \to \snglrd_{\leq 0, \leq \infty}$, we recover that the category spanned by the smooth manifolds $\RR^i$, $i \geq 0$, (where morphisms are smooth open embeddings between them) is a full subcategory of the category of all smooth manifolds (where morphisms are smooth open embeddings between them).
\end{example}

To prove Proposition~\ref{prop.basics-are-singular}, we must first produce an atlas on every basic. This philosophy was of course implicit when we defined $\bscd_{\leq k+1, \leq \infty}$ in Definition~\ref{defn.basics}, and we makes the atlas explicit in the following Lemma.

\begin{lemma}\label{cone-transformation}
Let $\sm_{|\sf{cpt}}$ be the restriction of $\sm$ to the full subcategory of compact stratified spaces.
For each $i\geq 0$ there is a natural transformation
	\[
	\sm_{|\sf cpt}(-) \longrightarrow \sm\bigl(\RR^i \times \sC(-)\bigr).
	\]
\end{lemma}

\begin{proof}
Let $U =( \RR^i\times \sC(Z) , \cA_Z)$ be a basic.  
We must exhibit a section of $\cA_U\in \sm\bigl(\RR^i\times \sC(Z) \bigr)$.
We proceed by induction on $\depth(U) = {\sf dim}(Z)+1$.  
The base case is the map $\sm(\emptyset) = \ast \to \sm(\RR^i)$ which plucks out the standard smooth structure on $\RR^i$.  
By induction, conically smooth open embeddings $V\hookrightarrow U\smallsetminus \RR^i$ form a basis for the topology of $U$.
Conically smooth open embeddings $U \hookrightarrow U$ form a base for the topology along $\RR^i\subset U$.
It follows that the collection $\cA_U:= \{V \to U\}$ of morphisms in $\bscd$ forms an atlas for $U$.  
We will show that this collection is a basis for the underlying topological space of $U$.  
Let $f\colon V\hookrightarrow U$ be an open embedding between basics.  Write $V=\RR^j\times \sC(Y)$.  
Suppose there is a collection of conically smooth open embeddings $W \hookrightarrow V$ which cover $V$ and for which each composition $W\hookrightarrow V \hookrightarrow U$ is conically smooth.
It follows that $V$ is conically smooth along $\RR^j$.  
By induction, the map $f^{-1}(U\smallsetminus \RR^i) \hookrightarrow U\smallsetminus \RR^i$ is conically smooth.  
It follows that $\cA_U$ is maximal.  
The functoriality of the assignment $\cA_Z \mapsto \cA_U$ among stratified homeomorphisms in the variable $Z$ is evident, by induction.  
\end{proof}

\begin{proof}[Proof of Proposition~\ref{prop.basics-are-singular}]
An object of $\bscd$ is a pair $(\RR^i \times \sC(Z), \cA_Z)$. Since it comes equipped with a conically smooth atlas for $Z$, we can apply Lemma~\ref{cone-transformation} to define the functor on objects: an atlas for the basic is given by the image of $\cA_Z$ under the natural transformation of the Lemma.

We must now prove that any morphism between basics is a morphism in $\snglrd$ -- that is, that a morphism $f: U \to V$ of basics satisfies the property $f^*\cA_V = \cA_U$. This is proven in a similar method as in the proof of Lemma~\ref{cone-transformation}: by induction on the depth of $V$ (which, by Remark~\ref{remark.non-increasing}, is always $\geq$ the depth of $U$). 

That it is faithful follows because the composite $\bscd \to \snglrd \to \snglrdC$ is faithful.  
That it is full follows from the definition of morphisms in $\snglrd$.
\end{proof}

Via the faithful functor $\snglrd \to \snglrdC$, all definitions and results from \S\ref{sec:dim-depth} have an evident modification for conically smooth stratified spaces. Since depth and strata will serve important functions later on, we highlight two analogues here. 

Below is the \emph{conically smooth} analogue of Lemma~\ref{lemma.strata}.

\begin{lemma}\label{strata-smooth}
\label{lemma.strata-smooth}
Let $X=(X\to P,\cA)$ be a conically smooth stratified space, and let $Q\hookrightarrow P$ be a consecutive map of posets. The $C^0$ stratified space $X_{|Q}= (X_{|Q} \to Q)$ naturally inherits a conically smooth structure with respect to which $X_{|Q} \to X$ is a conically smooth map.  
In particular, for $p\in P$, the $C^0$ stratified spaces 
\[
X_{\leq p}~,\qquad X_p~,\qquad X_{\nless p}
\]
are sub-stratified spaces of $X$, and $X_p$ is moreover an ordinary smooth manifold.  
\end{lemma}

\begin{proof}
The statement is a local one, so we can assume ${\sf dim}(X)<\infty$.  
We proceed by induction on $\depth(X)$.  
If $\depth(X)\leq 0$, then no distinct elements of $P$ are related, and the statement is trivially true. 

Suppose $X = U=\RR^i\times \sC(Z)$ is a basic.  Write the stratification $Z\to P_0$ and let $Q\hookrightarrow \sC(P_0)$ be as in the hypothesis.  Identify the underlying set $Q\subset \sC(P_0)$ with its image.
If $\ast \not \in Q$, $U|_{Q}$ equals $\RR^i\times \RR_{>0}\times Z_{|Q}$, and naturally inherits a smooth atlas, by induction.
If $\ast \in Q$, note that $Z_{|Q-\ast}$ is downward closed (and $Q-\ast$ is still consecutive in $P$) so  $Z_{|Q-\ast} \subset Z$ is a closed subspace (and therefore compact).  So $X|_{Q} \cong \RR^i \times \sC(Z|_{Q - \ast})$ is a $C^0$ stratified space.  
Since $Z$ has lower depth than $X$, by induction $Z_{|Q-\ast}$ naturally inherits the structure of a smooth atlas, and therefore so does $\RR^i\times \sC(Z_{|Q-\ast})$ by Lemma~\ref{cone-transformation}.  
For $X$ general, consider an atlas $\{(\varphi\colon U\hookrightarrow X\}$.  
Let $Q\hookrightarrow P$ be as in the hypothesis.  
Then $\{U_{|Q} \hookrightarrow X_{|Q}\}$ is an open cover of $X_{|Q}$.  
It follows that $X_{|Q}$ canonically inherits the conically smooth structure, with respect to which $X_{|Q} \to X$ is conically smooth.  
\end{proof}

\begin{cor}\label{depth-smooth}
Let $X$ be a conically smooth stratified space.
For each $k\geq -1$, the locus 
	\[
	X_k := \{x\in X\mid \depth_x(X) = k\} ~\subset ~X
	\]
of depth $k$ points of $X$ is a union of strata of $X$, and is thus itself an ordinary smooth manifold.  
\end{cor}

\begin{example}\label{strata-cbl}
If $X=(X\to P)$ is conically smooth with $Q\subset P$ a consecutive subposet, then the inclusion $X_{|Q} \to X$ is conically smooth and is \emph{constructible} in the sense of Definition~\ref{def.maps}.  
\end{example}

\subsubsection{Products of conically smooth stratified spaces}\label{sec.products}
For two spaces $Z$ and $W$, let $Z \star W$ denote their join. Recall that $\emptyset \star W \cong W$, and there is a homeomorphism
	\begin{equation}\label{cone-join}
	\sC(Z)\times \sC(W) \cong \sC(Z\star W).
	\end{equation}
For example, if $Z$ and $W$ are nonempty, denote a point in $Z \star W$ by $[z,a,w]$ where $a$ is in the closed interval $[0,1]$. A homeomorphism is given by 
\[
 ([t,z],[s,w])\mapsto [s+t, [z,\frac{t}{s+t},w]].
\]
Note that if $Z$ and $W$ are stratified (empty or not), there is a natural stratification on $Z \star W$ making (\ref{cone-join}) an isomorphism of conically smooth stratified spaces.

\begin{lemma}[Products]\label{products}
Let $X$ and $Y$ be $C^0$ stratified spaces.  
Then the product stratified space 
$
 X\times Y
$
is a $C^0$ stratified space, and there is a natural transformation 
\[
 \sm(-)\times \sm(-) \to \sm(-\times -).
\]
If both $X$ and $Y$ are compact, then the join 
$	
 X\star Y
$
is a compact $C^0$ stratified space. Moreover, there is a natural transformation 
	\[
	\sm_{|\sf cpt}(-)\times \sm_{|\sf cpt} (-) \to \sm_{|\sf cpt}(-\star -).
	\]
\end{lemma}

\begin{cor}\label{basic-products}
Let $U$ and $V$ be $C^0$ basics. Then the $C^0$ stratified space $U\times V$ is a $C^0$ basic. If $U$ and $V$ be basics,  the product $U\times V$ is naturally a basic.
\end{cor}

\begin{cor}\label{snglr-products}
The category of $C^0$ stratified spaces, and continuous stratified maps among them, admits products.
Likewise, the category $\Stratd$ of conically smooth stratified spaces, and conically smooth maps among them, admits products.  
\end{cor}

\begin{proof}[Proof of Lemma~\ref{products}] 
Because $C^0$ stratified spaces admit open covers by $C^0$ stratified spaces with finite dimension (Lemma~\ref{basics.basis}), it is enough to prove the statement assuming each of the dimensions of $X$ and $Y$ are finite. We perform induction on the dimension of $X \times Y$ and $X \star Y$ -- specifically, the statement for ${\sf dim} X \times Y = k$ will imply the statement for ${\sf dim} X \star Y = k$, which then implies the statement for ${\sf dim} X \times Y = k+1$, and so forth. The base case $k=0$ is obvious since discrete spaces are $C^0$ stratified, and also admit a unique maximal atlas. (Likewise for the case $k=-1$, since $\emptyset$ admits a unique atlas.)

Through Lemma~\ref{topological-basis}, the collection of product open embeddings
\[
\{\bigl(\RR^i\times \sC(Z) \bigr) \times \bigl(\RR^j\times \sC(W) \bigr)  \hookrightarrow X\times Y\}
\]
forms a basis for the topology of $X\times Y$. Thus, to verify that $X \times Y$ is $C^0$ stratified with an induced atlas, it is sufficient to consider the case $X=\RR^i\times \sC(Z)$ and $Y = \RR^j\times \sC(W)$. If $i+j \geq 1$, $\RR^{i+j} \times \sC(Z) \times \sC(W) \cong \RR^{i+j} \times \sC(Z \star W)$, which is a $C^0$ stratified space with an atlas by Lemma~\ref{cone-transformation} and induction on dimension. So we may assume $i=j=0$. Finally, we may assume both $Z$ and $W$ are nonempty since $\sC(\emptyset) \times \sC(W) = \sC(W)$ is a basic.

Then ${\sf dim}(Z) = {\sf dim} X-1$ and ${\sf dim}(W) = {\sf dim} Y-1$.  Since ${\sf dim}(Z\star W) = {\sf dim} X + {\sf dim} Y -1 < {\sf dim} X \times Y$,  by induction $Z \star W$ is a $C^0$ stratified space with a conically smooth atlas. Since $Z\star W$ is further compact, $\sC(Z)\times \sC(W) \cong \sC(Z \star W)$ is a $C^0$ stratified space with a conically smooth atlas by property (\ref{prop.cones}) and Lemma~\ref{cone-transformation}. 

It remains to prove that if all stratified spaces of the form $X \times Y$ with dimension $k$ are $C^0$ stratified with an atlas, then so are stratified spaces of the form $Z \star W$ of dimension $k$. Observe that $Z \star W$ admits an open cover by three sets:
\[
 Z \star W 
 \cong
 \sC(Z) \times W \bigcup_{Z \times W \times \RR} Z \times \sC(W).
\]

The middle term is a $C^0$ stratified space with atlas by Lemma~\ref{cone-transformation} and induction, since ${\sf dim} Z \times W < {\sf dim} Z \star W$.  The two other terms are $C^0$ stratified spaces with an atlas because they are products, of dimension $k$. One simply notes that the inclusions $Z \times W \times \RR \into \sC(Z) \times W$, $Z \times \sC(W)$ are conically smooth open embeddings by definition. 

\end{proof}

\subsection{Examples of basic opens and conically smooth stratified spaces}\label{section:singularExamples}

\example[Stratified spaces of dimension $0$ are countable discrete spaces]{
When $n=0$, we see that $\bscd_{0,0} = \{(\RR^0\to \ast)\}$ is the terminal category consisting of the stratified space $\RR^0 = \RR^0\times \sC(\emptyset)$, equipped with the unique atlas on $\emptyset$.  So a $0$-dimensional conically smooth stratified space is a countable set, equipped with its unique atlas.
In other words, $\snglrd_{0,0}$ is the category of countable sets and injections.
}

\example[Smooth manifolds]{
As defined in Example~\ref{example.first-smooth-manifolds},
conically smooth stratified spaces of depth 0 are the same as smooth manifolds.
Specifically, $\bscd_{0,n}$ is a $\Kan$-enriched category consisting of a single object (see Lemma \ref{mapping-Kans}), which is the smooth manifold
\[
 \RR^n \times \sC(\emptyset) = \RR^n~;
\] 
and whose Kan complex of morphisms $\Emb(\RR^n,\RR^n)$ is the singular complex of the space of smooth embeddings and smooth isotopies among them. 
In this case, $\snglrd_{0,n}$ is the familiar category of smooth $n$-manifolds and smooth embeddings among them, endowed with a $\Kan$-enrichment consistent with the compact-open $C^\infty$ topology of smooth embedding spaces.
}

\example[Cones on compact manifolds]{
Manifestly, for $M$ a closed smooth manifold, then the open cone $\sC(M)$ canonically inherits a conically smooth structure.  
}

\example[Stratified spaces of dimension 1 are graphs]{
The first instance of a singularity occurs with $\bscd_{\leq 1, 1}$ -- basics of pure dimension 1, and depth at most 1.
Explicitly, 
	\[
	\ob \bscd_{\leq 1,1} = \{\RR\}\sqcup \{\RR^0 \times \sC(J) \mid \emptyset \neq J \text{ finite}\}
	\] 
where the underlying space of a basic is either $\RR$, or is the noncompact `spoke' $\sC(J)$. 
We think of it as an open neighborhood of a vertex of valence $|J|$. 

The category $\snglrd_{1}$ is summarized as follows.  Its objects are (possibly non-compact, possibly empty) graphs with countably many vertices, edges, and components. A graph with no vertices (which is a $1$-dimensional stratified space of depth 0) is simply a smooth 1-manifold. 
}

\begin{example}[Simplices]
\label{example.simplices}
The topological simplex $\Delta^n$ can be given the structure of a conically smooth stratified space: note there is a homeomorphism of stratified spaces $\Delta^n \cong \Delta^{n-1} \star \Delta^0$, then appeal to Lemma~\ref{products}.  
\end{example}

\begin{example}[Simplicial complexes]\label{example.simplicial-complex}
Let $\sS$ be a finite simplicial complex such that, for a fixed $n$, every simplex of dimension $<n$ is in the boundary of a simplex of dimension $n$. We now describe how to naturally endow the geometric realization $|\sS|$ with the structure of a compact $n$-dimensional conically smooth stratified space. (We leave it to the reader to verify that, when $\sS = \Delta^n$, the following construction yields the same maximal atlas as specified above.)

For any simplex $\sigma \subset \sS$, let $K_\sigma$ be the union of all $n$-simplices that contain $\sigma$, and $U_\sigma$ the interior of $K_\sigma$. Then $U_\sigma$ can be identified with $\RR^{n-k} \times \sC(L_\sigma)$ where $L_\sigma=\mathsf{Link}(\sigma)$ and $(n-k)$ is the dimension of $\sigma$. By induction on dimension, the link is a compact $(k-1)$-dimensional conically smooth stratified space, so the collection $\{U_\sigma\}_{\sigma \subset \sS}$ forms an open cover of $|\sS|$ by stratified basics.

For any two simplices $\sigma$ and $\sigma'$, let $F$ be the highest-dimensional simplex such that $\sigma \cup \sigma' \subset F$. Then $U_\sigma \cap U_\sigma' = U_F$ (when $F$ is empty, this intersection is empty as well). So to verify that $\{U_\sigma\}$ forms an atlas, one need only verify that the topological embedding $U_F \into U_\sigma$ is in fact a conically smooth embedding. If $\depth U_F = 0$ (i.e., ${\sf dim} F = n$) this is obvious, as $U_F \into U_\sigma$ is the inclusion of Euclidean space into a depth-zero stratum of $U_\sigma$. Noting that $\mathsf{Link}(F) \subset \mathsf{Link}( \sigma)$ is an embedding into a higher-depth basic whenever $\sigma \neq \sigma'$, we are finished by induction on depth. 
\end{example}

\begin{example}[Permutable corners]\label{example.perm.corners}
For a fixed $n$, consider the subcategory of $\bscd_{\leq n, n}$ spanned by the objects of the form $\{\RR^{n-k} \times \sC(\Delta^{k-1})\}_{0 \leq k \leq n}$. Here, $\Delta^{k-1}$ has the atlas specified in the examples above, and by $\Delta^{-1}$, we mean the empty manifold. Note that we have a conically smooth identification $\RR^{n-k}\times \sC(\Delta^{k-1}) \cong \RR^{n-k}\times [0,\infty)^k$.
If a $n$-dimensional conically smooth stratified space $M$ is given an atlas consisting only of charts from objects in this subcategory, we call $M$ a {\em manifold with permutable corners}. Manifolds with permutable corners will be examined more systematically as Example~\ref{boundary}. 

An example of a $2$-manifold with permutable corners is the teardrop in Figure~\ref{image.teardrop}.
Here are two examples of a $2$-manifold with permutable corners that are isomorphic in $\snglrd_2$, with the atlases implicit:
\begin{itemize}
\item The closure of the first quadrant in $\RR^2$, given by $M = \{(x_1,x_2) \, | \, x_1,x_2\geq 0\}$.
\item The complement of the open third quadrant in $\RR^2$, given by $M' = \{\text{either $x_1$ or $x_2$ is $\geq 0$}\}$.
\end{itemize}
Note that there is no smooth embedding $f: \RR^2 \to \RR^2$ such that $f(M) = M'$, so no identification of these subspaces can be witnessed through an ambient smooth embedding of $\RR^2$. Regardless $M \cong M'$ in $\snglrd_2$.
This is in contrast to other definitions of manifolds with corners, such as~\cite{laures}. See also \S\ref{section.corners}.
\end{example}

\begin{example}[Embedded submanifolds]\label{example.embedded}
Let $M^n$ be a smooth $n$-manifold and $L^{n-k} \subset M^n$ a properly embedded $(n-k)$-dimensional smooth submanifold. This data is an example of a $n$-dimensional conically smooth stratified space of depth $k$ as follows; we will denote it as $(L\subset M)$.  
The stratification map $M\xra{S} \{n-k<n\}$ is given by setting $S^{-1}\{n-k\} = L$ and $S^{-1}\{n\} = M\smallsetminus L$. To verify conical smoothness, notice that $M$ is covered by stratified open embeddings from the stratified spaces $\RR^n \to \{n\}$ and $\RR^{n-k}\times \sC(S^{k-1}) \to \{n-k<n\}$.
So $M$ is thus equipped with the atlas consisting of those stratified open embeddings $\{\RR^n \xra{\varphi} M\}\cup \{\RR^{n-k}\times \sC(S^{k-1}) \xra{\psi} M\}$ for which $\RR^n\xra{\varphi}M\smallsetminus L$ is smooth, $\psi_|\colon \RR^{n-k}\to L$ is smooth, and $\psi_|\colon \RR^{n-k}\times (\RR_{>0}\times S^{k-1}) \to M\smallsetminus L$ is smooth.  
It is routine to verify that this indeed defines a conically smooth atlas.  
Notice that the smooth structures of $M\smallsetminus L$ and $L$ are exploited, but the entirety of the smooth structure around $L$ is not.  
The map $(L\subset M) \to M$ is \emph{conically smooth} in the sense of Definition~\ref{defn.conical-smoothness}, and is a typical example of a \emph{refinement} in the sense of Definition~\ref{def.maps}.
\end{example}

\begin{example}[$\bscd_2$]
\label{example.bsc-2}
The underlying space of an object of $\bscd_{2}$ is $\RR^2$, $\RR\times \sC(J)$ with $J$ a nonempty finite set, or $\sC(Y)$ where $Y$ is a nonempty compact graph.  
Heuristically, an object of $\snglrd_{2}$ is a space which is locally equivalent to one of the three types of stratified spaces.
An example of such an object is the geometric realization of a (countable) simplicial complex for which each simplex is contained in some $2$-simplex. (See Example~\ref{example.simplicial-complex}.) In Figure~\ref{image.coneDelta1} we also picture a basic open obtained from the 1-skeleton of the 2-simplex. Another example of an object of $\snglrd_2$ is a nodal surface -- here the depth $2$ points are the nodes and there are no depth $1$ points. 

In Figure~\ref{image.tetra-1-skeleton} we consider $\sC(Y)$ when $Y$ is the $1$-skeleton of the tetrahedron $\Delta^3$. The stratification $\sC(Y) \to P = \{(2,2) \leq (1,2) \leq (0,2)\}\cong [2]$ is given by sending a point to its local depth and dimension. Consider a poset $Q = \{0\leq 2\}$, and the non-consecutive inclusion $Q \into P = [2]$. Then $X|_Q$ is not even a $C^0$ stratified space, as the neighborhood around $\ast$ does not admit a neighborhood of the form $\RR^i \times \sC(X)$ for a compact space $X$. This shows the necessity of consecutiveness in Lemmas~\ref{lemma.strata} and \ref{lemma.strata-smooth}.
\end{example}

\begin{figure}
		\[
		\xy
		\xyimport(8,8)(0,0){\includegraphics[height=0.8in]{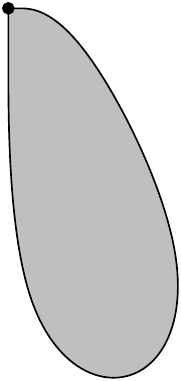}}
		\endxy
		\]
	\begin{image}\label{image.teardrop}
	
	The teardrop is a 2-manifold with corners, and in particular a 2-dimensional conically smooth stratified space of depth 2. The unique point of depth 2 is the tip of the teardrop. The interior, indicated by grey, contains only points of depth 0, and the solid border consists of points of depth 1.
	\end{image}
\end{figure}

\begin{figure}
		\[
		\xy
		\xyimport(8,8)(0,0){\includegraphics[width=3in]{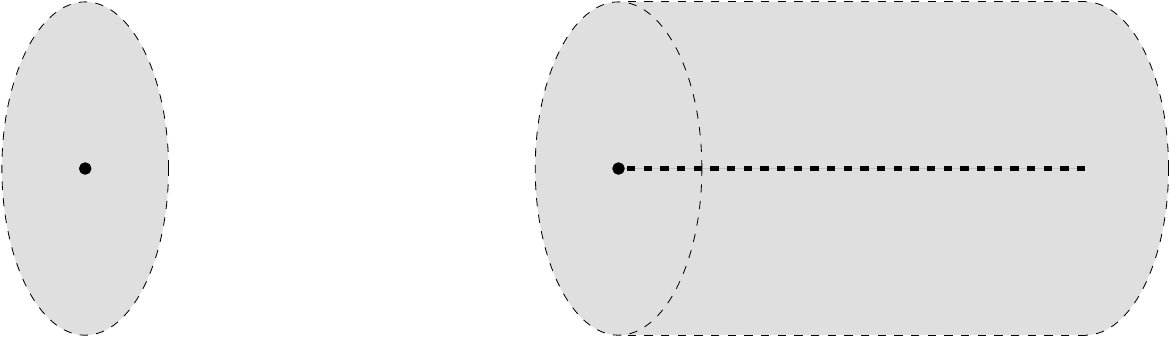}}
	,(0.6,-1)*+{(a)}
	,(6,-1)*+{(b)}
		\endxy
		\]
	\begin{image}\label{image.coneS1}
	(a) The basic $\sC(S^1)$. The marked point is the cone point $\ast \in \sC(S^1)$ (of depth 2) and the dashed boundary indicates this space is noncompact. (b) The basic open $U= \RR^1 \times \sC(S^1)$. $U$ has depth 2, and is a local model for a 1-manifold embedded in a 3-manifold.
	\end{image}
\end{figure}

\begin{figure}
		\[
		\xy
		\xyimport(8,8)(0,0){\includegraphics[height=1in]{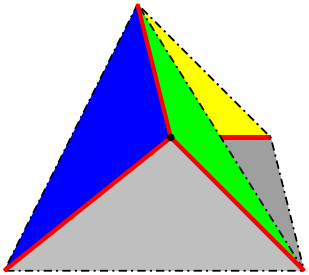}}
		\endxy
		\]
	\begin{image}\label{image.tetra-1-skeleton}
	A drawing of the basic open $\sC(Y)$, where $Y$ is the compact $1$-dimensional conically smooth stratified space given by the $1$-skeleton of a tetrahedron. This basic open has a single point of depth 2, depicted as a black dot at the center. 
	Each colored face -- in grey, blue, et cetera --  is a region whose points all have depth zero. The red solid lines indicate points of depth 1 -- the subspace of depth 1 points is a union of copies of $\RR$, one copy for each vertex of $Y$. 
		\end{image}
\end{figure}

\begin{figure}
		\[
		\xy
		\xyimport(8,8)(0,0){\includegraphics[width=2in]{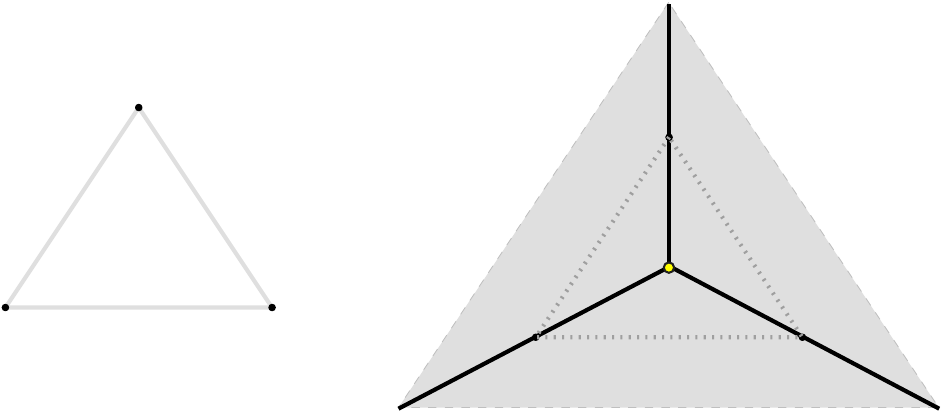}}
	,(1.2,-1)*+{(a)}
	,(5.8,-1)*+{(b)}
		\endxy
		\]
	\begin{image}\label{image.coneDelta1}
	(a) A $1$-dimensional conically smooth stratified space $X$ given by the natural stratification on the 1-skeleton of $\Delta^2$. The black vertices represent depth 1 points, while the grey edges contain only depth 0 points. (b) The basic open $\RR^2 \times \sC(X)$, with the original manifold $X$ drawn using dotted lines. The central point is a point of depth 2, and the other colors encode the same depths as before -- the black edges are depth 1, while the grey areas are depth 0 (i.e., smooth) points. Note this also encodes the singularity type found at a corner of the surface of a cube.
	\end{image}
\end{figure}

\begin{example}
For each $0\leq k\leq n$, there is a standard identification of the underlying space $U^n_{S^{k-1}} \cong \RR^n$ through, say, polar coordinates.  
However, we point out that $U^n_{S^{k-1}}\cong U^n_{S^{l-1}}$ implies $k=l$.  
This follows immediately from the definition of morphisms between basics.  (See~\S\ref{sec:endomorphisms} for a thorough discussion.)
\end{example}

\begin{example}[Global quotients]\label{ex.orbifolds}
Let $X$ be a conically smooth stratified space.
Let $G$ be a finite group acting on $X$ so that, for each subgroup $H\subset G$, the invariants $X^H \subset X$ is a constructible embedding (see Definition~\ref{def.maps}).  
We now explain how the quotient $X_G$ canonically inherits a conically smooth structure.

Because $G$ is finite, we can assume that $X$ has bounded depth, on which we will induct.  
Let $x\in X$ and denote the subgroup $G_x := \{g\mid g\cdot x=x\}$.  
Again, the finiteness of $G$ grants that there is a pair of centered charts $(U,0) \hookrightarrow (\w{U},0)\hookrightarrow (X,x)$ for which $G_x\cdot U\subset \w{U}$.  
Through the conclusions of~\S\ref{sec:endomorphisms}, there results a map $G_x \to {\sf GL}(U)$.
We conclude that there exists a centered chart $(U',0)\hookrightarrow (X,x)$ that is $G_x$-equivariant.  
So the problem of constructing the conically smooth structure on $X_G$ reduces to the case that $X=U=\RR^n\times \sC(Z)$ and the homomorphism $G \to \Aut(U)$ factors through $\sO(\RR^n)\times \Aut(Z)$ (see Definition~\ref{other-endos}).  
Splitting off the trivial representation, $\RR^n=\RR^T\times \RR^A$, we conclude an action of $G$ on $S^{A-1}\star Z$ that induces the given action on $U=\RR^n\times \sC(Z) \cong \RR^T \times \sC\bigl(S^{A-1}\star Z\bigr)$.  
By induction on depth, $(S^{A-1}\star Z)_G$ has a conically smooth structure.  
And so $U_G = \RR^T\times \sC\bigl((S^{A-1}\star Z)_G\bigr)$ has a conically smooth structure. 

\end{example}

\example[Isotropy stratification]\label{global-orbi}{
Let $G$ be a finite group.
Let $G\xra{\rho} \sG\sL(\RR^n)$ be a homomorphism.
Regard the set $\cP_{\sf quot}(G):= \{H\subset G\}$ of subgroups of $G$ as a poset, ordered by reverse inclusion. 
The map of sets $\RR^n \to \cP_{\sf quot}(G)$ given by $x\mapsto \{g\mid gx=x\}$ is continuous, and the standard smooth structure on $\RR^n$ equips this stratified topological space with a conically smooth structure.

Because $G$ is finite, it acts orthogonally on $\RR^n\cong \RR^{T}\times \RR^{A}$, where here we have made visible the trivial representation.  
In particular, $G$ acts on the unit sphere $S^{A-1}$.
The above stratification on $\RR^n$ is invariant with respect to scaling by a positive real number, and so its restriction to $S^{A-1}$ inherits the conically smooth structure. 
An inductive argument on the dimension of $S^{A-1}$ (see Example~\ref{ex.orbifolds}) verifies that the stratified topological space which is the coinvariants $(S^{A-1})_G$ canonically inherits the conically smooth structure.  
And so, $(\RR^n)_G \cong \RR^T \times \sC\bigl((S^{A-1})_G\bigr)$ is a basic stratified space. 

Likewise, let $X$ be an orbifold.
A choice of an orbifold atlas for $X$ witnesses an open cover of the underlying topological space of $X$, with each inclusion in the open cover being conically smooth.
There results a stratification of $X$ together with a conically smooth atlas.  
}

\begin{example}[Coincidences]\label{ex.coinc}
Let $I$ be a finite set.
Consider the poset $\cP_{\sf quot}(I)$ for which an element is an equivalence relation $\fP\subset I\times I$, with $\fP\leq \fP'$ meaning $\fP'\subset \fP$.  
Let $P$ be a poset. 
Consider the poset $\cP_{\sf quot}(I)\wr P$ for which an element is a pair $(\fP,\un{c})$ consisting of an equivalence relation on $I$ together with a map of underlying sets $I_{/\fP} \xra{\ov{c}} P$, with $(\fP,\un{c})\leq (\fP',\un{c}')$ meaning $\fP'\subset \fP$ and $\ov{c}q \leq \ov{c}'$ -- here $q\colon I_{/\fP'} \to I_{/\fP}$ is the quotient map.
Notice the maps of posets
\[
\cP_{\sf quot}(I) \leftarrow \cP_{\sf quot}(I)\wr P \to P^I~.  
\]

Let $X=(X\xra{S} P)$ be a conically smooth stratified space.
There is the map of underlying sets 
\[
X^{cI}~:=~ \bigl(X^I \to \cP_{\sf quot}(I)\wr P\bigr)
\]
given by $(I\xra{c}X)\mapsto \bigl(\{(i,i')\mid c(i)=c(i')\}, [i]\mapsto Sc(i) \bigr)$.  
Immediate is that this map is continuous, and therefore defines a stratified topological space.
Also clear is that a conically smooth map $X\to Y$ and a map $I\to J$ together induce a continuous stratified map $X^{cJ} \to Y^{cI}$, compatibly as these two types of maps compose.  
Furthermore, if $X\to Y$ is constructible or open then so is $X^{cJ}\to Y^{cI}$.  
\begin{prop}\label{coinc-exist}
Let $I$ be a finite set, and let $X=(X\xra{S}P)$ be a conically smooth stratified space.
The following statements are true.
\begin{enumerate}
\item The stratified topological space $X^{cI} = \bigl(X^I\to\cP_{\sf quot}(I)\wr P\bigr)$ has a standard structure of a conically smooth stratified space with respect to which the stratified continuous map $X^{cI} \to X^I$ is conically smooth.  
\item In the case that $X=U$ is a basic, then so is $U^{cI}$.  
\item Furthermore, for $X\hookrightarrow Y$ a conically smooth embedding and $I\to J$ is a map of sets, then the induced map $X^{cJ}\to Y^{cI}$ is a conically smooth embedding.  
\end{enumerate}
\end{prop}
\begin{proof}
Let $\cU$ be an open cover of $X$.
Consider the collection 
\[
\bigl\{ \underset{[i]\in I_{/\fP}}\prod (O_{[i]})^{c[i]}\mid \fP\in \cP_{\sf quot}(I)\text{ and }I_{/\fP}\xra{O_-} \cU\text{ is a map of sets} \bigr\}
\]
of subsets of $X^{cI}$. 
This collection is an open cover, by inspection.  
So the Statement~(1) and Statement~(3) follow from Statement~(2) together with showing that, for $U\hookrightarrow V$ a conically smooth embedding between basics, the map $U^{cJ} \to V^{cI}$ is conically smooth embedding.  
We focus on Statement~(2).  
Write $U=\RR^n\times \sC(Z)\cong \RR^n\times\bigl(\ast\underset{\{0\}\times Z}\coprod \RR_{\geq 0}\times Z\bigr)$.
Consider the map
\[
U^{cI} \to \RR^n\times \RR_{\geq 0}
\]
given by assigning to $I\xra{c}\RR^n\times \sC(Z)$ the sum and the scale:
\[
\Bigl(~\underset{i\in I} \sum {\sf pr}_{\RR^n}(c_i)\in \RR^n~{}~,~{}~\underset{i\in I} \sum \lVert {\sf pr}_{\RR^n}(c_i)\rVert^2+ \sum_{i\in I} | {\sf pr}_{\RR_{\geq 0}}(c_i)|^2~\Bigr)
\]
-- this map is conically smooth, by direct inspection.  
Define the sub-stratified topological space $L$ as the fiber of this map over the point $(0,1)$.  
The compactness of $Z$ grants the properness of this map, and thereafter the compactness of $L$.  
The map $\gamma\colon \RR^n\times \RR_{\geq 0} \times U \to U$ at the beginning of~\S\ref{sec.conical-smoothness}) implements translation and scaling, and acts diagonally on $U^{cI}$ to give the conically smooth map 
\[
\Gamma \colon \RR^n\times \RR_{\geq 0} \times U^{cI} \to U^{cI}~.
\]
From the construction of $L$, the restriction 
$
\Gamma\colon \RR^n\times \RR_{\geq 0} \times L\longrightarrow U^{cI}
$
descends to a stratified homeomorphism $\RR^n\times \sC(L) \cong U^{cI}$, through which $L$ can be seen to inherit the structure of a conically smooth stratified space.   
Because $L$ is compact, this witnesses $U^{cI}$ as a basic, thereby proving Statement~(2).

Now let $U\xra{f}V$ be a conically smooth embedding between basics, and let $I\to J$ be a map of sets.
The conical smoothness of the diagonally defined map $f^{c(I\to J)}\colon U^{cJ}\to V^{cI}$ can be tested by the conical smoothness of the composite map $U^{cJ} \to U^J \xra{f^{I\to J}}V^I$.  
This is the case because $f^{I\to J}$ is conically smooth and because, by construction, the map $U^{cI} \to U^I$ is conically smooth.
That $f^{c(I\to J)}$ is an embedding immediately follows because $f^{I\to J}$ is an embedding.   
\end{proof}

Notice that the standard $\Sigma_I$-action on the stratified topological space $X^{cI}$ is conically smooth.
Furthermore, for each subgroup $H\subset \Sigma_I$, the invariants $(X^{cI})^H\hookrightarrow X^{cI}$ is a constructible embedding from a conically smooth stratified space.  
\end{example}

\begin{remark}
Note that the map $X^{cI} \to X^I$ is a homeomorphism of underlying topological spaces (and so, is a \emph{refinement} in the sense of Definition~\ref{def.maps}).  
\end{remark}

\subsection{Classes of maps}
Recall the notion of a conically smooth map from Definition~\ref{defn.conical-smoothness}. There are several sub-classes of such maps that are useful to isolate. In the definition of constructible bundles, which is an inductive definition, we will make use of the links from Definition \ref{def.Y-link}.

\begin{definition}\label{def.maps}
Let $f\colon (X\xra{S} P) \to (Y\xra{T}Q)$ be a conically smooth map among conically smooth stratified spaces with underlying posets $P$ and $Q$, with $X_q :=f^{-1}Y_q$ the sub-stratified space of $X$ which is the inverse image of the stratum $Y_q$, for each $q\in Q$.
\begin{itemize}
\item[] \textbf{(Consecutive) inclusion of strata:}
The map $f$ is an \emph{inclusion of strata} if $P\ra Q$ is a subposet and the diagram of topological spaces
\[
\xymatrix{
X  \ar[r]^-f  \ar[d]_S
&
Y  \ar[d]^T
\\
P  \ar[r]^-{{f}}
&
Q
}
\]
is a pullback.
If, in addition, the map $\ov{f}$ is consecutive (in the sense of Definition~\ref{defn.consecutive}), we say $f$ is a \emph{consecutive inclusion of strata}.  
\item[] \textbf{Refinement:}
The map $f$ is a \emph{refinement} if it is a homeomorphism of underlying topological spaces, and for each $p\in P$ the restriction $f_{|}\colon X_{p} \to Y_{{f}p}$ is an isomorphism onto its image (as smooth manifolds). 

\item[] \textbf{Proper embedding:}
The map $f$ is a \emph{proper embedding} if there is a factorization of $f$ as a composite $X\xra{\w{f}} \w{Y} \xra{r} Y$ in which $\w{f}$ is a consecutive inclusion of strata and $r$ is a refinement.  
\item[] \textbf{Submersion:} 
Consider the collection $\{O_\alpha\times U \hookrightarrow X\}_{\alpha \in A}$ indexed by commutative diagrams of conically smooth maps
\[
\xymatrix{
O_\alpha \times U  \ar[r]   \ar[d]_-{\sf pr}
&
X  \ar[d]^-f 
\\
U  \ar[r]
&
Y
}
\]
in which the horizontal arrows are open embeddings, and the left vertical map is the projection.  
The map $f$ is a \emph{submersion} if this collection is a basis for the topology of $X$.

\item[] \textbf{Bundle:}
Consider the collection $\{U \hookrightarrow Y\}_{\beta \in B}$ indexed by pullback diagrams of conically smooth maps
\[
\xymatrix{
F_\beta \times U  \ar[r]   \ar[d]_-{\sf pr}
&
X  \ar[d]^-f 
\\
U  \ar[r]
&
Y
}
\]
in which the horizontal arrows are open embeddings, and the left vertical map is the projection.  
The map $f$ is a \emph{conically smooth fiber bundle} if this collection is a basis for the topology of $Y$.

\item[] \textbf{Weakly constructible bundle:}
The map $f$ is a \emph{weakly constructible bundle} if, for each $q\in Q$, the restriction $f_|\colon X_{q}\to Y_q$ is a bundle.  
\item[] \textbf{Constructible bundle:}
The map $f$ is a \emph{constructible bundle} if, by induction on depth:
\begin{itemize}
\item For $\depth(X)=0$, the map $f$ is a smooth fiber bundle of manifolds;
\item For $\depth(X)>0$, then:
\begin{itemize}
\item $f$ is a weakly constructible bundle; and
\item for each $q\in Q$ the natural map
\[
\Link_{X_{q}}(X)\longrightarrow X_{q}\underset{Y_q}\times \Link_{Y_q}(Y)
\]
is a constructible bundle.
\end{itemize}
\end{itemize}
\item[] \textbf{Pre-constructible bundle:}
A continuous map $X\xra{g} Y$ is a \emph{(weakly) pre-constructible bundle} if there is a diagram of conically smooth maps $X\xla{r} \w{X} \xra{\w{g}} Y$ with $r$ a refinement and $\w{g}$ a (weakly) constructible bundle such that the diagram in topological spaces
\[
\xymatrix{
&\w{X}\ar[dr]^{\w{g}}\ar[dl]_r\\
X\ar[rr]^{g}&&Y}
\]
is commutative.

\end{itemize}
\end{definition}

\begin{remark}
Note that for $Z\subset X$ a consecutive inclusion of strata, the depth of $\Link_Z(X)$ is strictly less than the depth of $X$. Consequently, the above inductive definition of a constructible bundle is well-defined.
\end{remark}

Example~\ref{example.embedded} and Example~\ref{coinc-exist} demonstrate some typical examples of refinements.  
In Example~\ref{strata-cbl} we saw examples of weakly constructible bundles which are not fiber bundles.  
In~\S\ref{collargluingssection} we will examine an important class of weakly pre-constructible bundles: \emph{collar-gluings} (see Remark~\ref{cg-as-cbl}).

\begin{example}\label{bundle-examples}
\begin{itemize}
\item[]
\item An ordinary smooth fiber bundle is an example of a conically smooth fiber bundle.

\item The projection from a product $X\times F \to X$ is a conically smooth fiber bundle.

\item Let $L \xra{\pi} X$ be a proper conically smooth fiber bundle.  
Then the \emph{fiberwise cone} 
\[
\sC(\pi)~:=~ X \underset{\{0\}\times L}\coprod (\RR_{\geq 0}\times L)
\]
canonically inherits the conically smooth structure, and the projection $\sC(\pi) \to X$ is a conically smooth fiber bundle whose fibers are isomorphic to basics.  
Notice the cone-locus section $X\to \sC(\pi)$.  

\end{itemize}
\end{example}

\subsection{The Ran space}
We endow the space of finite subsets $S\subset X$ of a connected conically smooth stratified space with bounded cardinality $|S|\leq i$ with the structure of a conically smooth stratified space.  
Fix a stratified space $X$. 
Recall Example~\ref{ex.coinc} introducing the stratified space $X^{cI}$.
Consider the sub-stratified space 
\[
X^{sI}\subset X^{cI}
\]
consisting of those $I\to X$ for which the map to connected components $I\to [X]$ is surjective -- this subspace is a connected component.  
Furthermore, for each surjection $I\to J$ the evident induced map $X^{sJ} \to X^{sI}$ is conically smooth -- this is direct from the definition of $X^{cI}$.  
In fact, this map is a \emph{constructible embedding} in the sense of Definition~\ref{def.maps}.  
These values assemble as a functor
\[
X^{s\bullet} \colon (\Fin^{\sf surj})^{\op} \longrightarrow \Stratd
\]
from nonempty finite sets and surjective maps to conically smooth stratified spaces and conically smooth maps.

\begin{definition}[$\Ran_{\leq i}(X)$]\label{def.ran}
Let $i$ be a cardinality.  
Let $X$ be a $C^0$ stratified space.  
The \emph{Ran space of $X$} is the $C^0$ stratified space that is the colimit
\[
\Ran_{\leq i}(X) ~:=~ \colim\Bigl((\Fin^{{\sf surj}, \leq i})^{\op}\xra{X^{s\bullet}} \Stratd \to \StTop\Bigr)~.
\]
In the case that $i$ is infinite, we omit the subscript and write $\Ran(X)$.  
\end{definition}

\begin{remark}
A wonderful result of Bott~(\cite{bott}) demonstrates a homeomorphism $\Ran_{\leq 3}(S^1) \cong S^3$.  
In~\cite{tuffley}, this result is extended to show $\Ran_{\leq 2k+1}(S^1) \cong S^{2k+1}$.  
\end{remark}

\begin{remark}
For $X$ connected, the underlying topological space $\Ran(X)$ is weakly contractible.
This is originally due to Curtis \& Nhu in \cite{curtisnhu}; Beilinson \& Drinfeld give a concise proof in \cite{bd} and they credit Jacob Mostovoy with having independently found a similar proof.
\end{remark}

\begin{remark}
Let $X = (X\xra{S} P)$ be a stratified topological space.
The stratified topological space $\Ran_{\leq i}(X)$ is stratified by the poset $\ZZ_{\geq 0}^P$ consisting of maps of sets $P\xra{c} \ZZ_{\geq 0}$ for which $\sum_{p\in P} c_p \leq i$, with $c\leq c'$ meaning the subset $\{p\mid c_p>c'_p\}$ contains no maxima of $P$.  
\end{remark}

Notice that a $\Ran_{\leq i}(X)$ is functorial among conically smooth inclusions in the argument $X$ and relations in the argument $i$.   
Notice also that, for $i\leq j$, the map
\[
\Ran_{\leq i}(X) \longrightarrow \Ran_{\leq j}(X)
\]
is a consecutive inclusion among stratified topological spaces.  
Also notice that, for $I$ a finite set, the map 
\[
{\sf Image}\colon X^{cI} \longrightarrow  \Ran_{\leq |I|}(X)
\]
is a quotient map, and that it factors through the coinvariants $X^{cI}\to (X^{cI})_{\Sigma_I}$.  
\begin{prop}
Let $i$ be a finite cardinality.
Let $X$ be a conically smooth stratified space.
Then the Ran space $\Ran_{\leq i}(X)$ canonically inherits the structure of a conically smooth stratified space.
Moreover, for each conically smooth embedding $f\colon Y\hookrightarrow X$, and each inequality $j\leq i$, the map
\[
\Ran_{\leq j}(Y) ~\hookrightarrow ~\Ran_{\leq i}(X)
\]
is a conically smooth embedding.  
Furthermore, if $X=U$ is a basic, then so is $\Ran_{\leq i}(U)$.  
\end{prop}

\begin{proof}
We explain the following pushout diagram among stratified topological spaces
\[
\xymatrix{
\underset{I\to J} \colim ~X^{sJ}  \ar[r]  \ar[d]
&
\underset{I} \colim ~X^{sI}  \ar[d]
\\
\underset{L}\colim ~X^{sL}  \ar[r]
&
\underset{K}\colim ~X^{sK}~.
}
\]
\begin{itemize}
\item The bottom right colimit is indexed by the category $\Fin^{{\sf surj},\leq i}$ whose objects are finite sets of cardinality at most $i$, and whose morphisms are surjective maps among them.
Recognize this colimit as $\Ran_{\leq i}(X)$.  
\item The bottom left colimit is indexed by the category $\Fin^{{\sf surj},<i}$ whose objects are finite sets of cardinality strictly less than $i$, and whose morphisms are surjective maps among them.  
Recognize this colimit as $\Ran_{<i}(X)$.  
\item The top right colimit is indexed by the category $\Fin^{{\sf surj}, =i}$ whose objects finite sets of cardinality $i$, and whose morphisms are isomorphisms among them.
Recognize this colimit as $(X^{sI})_{\Sigma_I}$ for some choice of finite set $I$ of cardinality $i$.  
\item The top left colimit is indexed by the category $\Fin^{{\sf surj},<i=}$ whose objects are surjective maps $I\to J$ from a set of cardinality $i$ to one of strictly less cardinality, and whose morphisms are maps of such maps which are isomorphisms on the coordinate $I$ and surjections on the coordinate $J$.  
\end{itemize}
The arrows in the diagram are evident.
To see that the diagram is pushout amounts to a comparison of indexing categories.  

By inspection, there is an equivalence of categories:
\[
\Fin^{{\sf surj},<i}  \underset{\Fin^{{\sf surj},<i=}\times\{0\}}\coprod \Fin^{{\sf surj},<i=}\times[1]  \underset{\Fin^{{\sf surj},<i=}\times\{1\}}\coprod \Fin^{{\sf surj},=i}~{}~ \xra{~\simeq~}~{}~ \Fin^{{\sf surj},\leq i}~.
\]
This identification witnesses a colimit over the target of this functor as such an iterated pushout of restricted colimits.
Because $\Fin^{{\sf surj},<i=}\to \Fin^{{\sf surj},<i}$ is coCartesian, and the functor $\Fin^{{\sf surj},<i=}\to \Fin^{{\sf surj},=}$ is Cartesian, this iterated pushout can be simplified as a pushout.
For the case at hand, this gives the above square and justifies that it is pushout.

Let us argue that the top horizontal map is a constructible embedding (see Definition~\ref{def.maps}).  
The salient point to check is that the map is injective.
So consider two maps $I\twoheadrightarrow J \xra{c} X$ and $I'\twoheadrightarrow J' \xra{c'} X$ for which there is an isomorphism $I\cong I'$ over $X$.
Because of the surjective-injective factorization system on maps among spaces, there is a $I'' \twoheadrightarrow J'' \xra{c''} X$ with $c''$ injective, and morphisms $(I\to J) \to (I'' \to J'') \la (I'\to J')$ in $\Fin^{{\sf surj},<i=}$ taking place over $X$.  
Injectivity follows.
By further inspection, it is a constructible topological embedding.  
From Example~\ref{ex.orbifolds}, $(X^{sI})_{\Sigma_I}$ has a standard structure of a conically smooth stratified space.  
Thereafter, the top left term in the diagram inherits the structure of a conically smooth sub-stratified space.
That the left vertical map is conically smooth follows by direct inspection.  
By induction on $i$, we assume $\Ran_{<i}(X)$ has the structure of a conically smooth stratified space; the base cases $i=0,1$ being trivially true.  
The result follows from Lemma~\ref{some-pushouts}. The functoriality is immediate from the construction.  
For the final statement, let $U$ be a basic. The argument verifying that $\Ran_{\leq i}(U)$ is a basic is the same argument as that verifying that $U^{cI}$ is a basic (Proposition~\ref{coinc-exist}(2)), for the map $U^{cI} \to \RR^n \times \RR_{\geq 0}$ constructed there factors through the quotient map $U^{cI} \to \Ran_{\leq |I|}(U)$.  
\end{proof}

In the following, we use the notion of a weakly constructible bundle, which we defined in Definition~\ref{def.maps}.

\begin{lemma}\label{some-pushouts}
Let $B\xla{p} E\xra{i} X$ be a diagram of weakly constructible bundles 
among conically smooth stratified spaces.
Suppose $p$ is proper and $i$ is an embedding.  
Then the pushout among stratified spaces and conically smooth maps among them
\[
B\underset{E}\coprod X
\]
exists in $\Stratd$, and its underlying stratified topological space agrees with the pushout in ${\sf StTop}$.  
\end{lemma}

\begin{proof}
An elaboration of Proposition~\ref{tubular-neighborhoods} grants the existence of a conically smooth regular neighborhood of $E\subset X$, which is to say there is a proper weakly constructible bundle $L\to E$ together with an open embedding from its fiberwise cone $\sC_E(L)\hookrightarrow X$. 
So we recognize the pushout 
\[
B\underset{E}\coprod X \cong B\underset{E}\coprod\bigl(\sC_E(L) \underset{\RR\times L}\bigcup X\smallsetminus E\bigr)~.
\]
And so we witness an open cover of the pushout $B\underset{E}\coprod X \cong \sC_B(L) \underset{\RR\times L}\bigcup  X\smallsetminus E$ involving the fiberwise cone of the composite $L\to E\to B$.  
\end{proof}

\section{Basic properties and the tangent classifier}
\label{section.basics-properties}

\subsection{$\snglr$ as an $\infty$-category}
\label{section.kan-enrichment}
 
 We describe a natural enrichment of $\snglrd$ over $\Kan$, the monoidal category of Kan complexes with Cartesian product. Afterward, see Convention \ref{defn.enrichment}, we will regard it as an $\oo$-category by applying the simplicial nerve functor to this $\Kan$-enriched category.
This $\Kan$-enrichment comes from a more primitive enrichment over the category of set-valued presheaves on $\Stratd$ with Cartesian product. This notion stems from the following definition:

\begin{definition}\label{def.parametrized}
For $X$, $Y$, and $Z$ conically smooth stratified spaces, then
\[
\Strat_Z(X,Y):= \bigl\{f\colon Z\times X \to Z\times Y\mid {\sf pr}_Z \circ f = {\sf pr}_Z\bigr\}
\]
is the set of commutative diagrams of conically smooth maps
\[
\xymatrix{
Z\times X\ar[rr]^f\ar[dr]_{{\sf pr}_Z}&&Z\times Y\ar[dl]^{{\sf pr}_Z}\\
&Z}
\]
where ${\sf pr}_Z$ are the projections onto $Z$.
\end{definition}

The Observation~\ref{admits-pullbacks} that $\Stratd$ admits finite products, and that conically smooth maps compose, validates the following definition.
\begin{definition}[$\Strat$ and $\snglr$]\label{self-enrichment}
$\Strat$ is the following category enriched over the set-valued presheaves $\Fun(\Stratd^{\op}, \Set)$. An object is a conically smooth stratified space. For $X$ and $Y$ stratified spaces, the presheaf of maps $\Strat(X,Y)\in \Fun(\Stratd^{\op},\Set)$ assigns values
\[
\Strat(X,Y) := \Bigl(Z\mapsto \Strat_Z(X,Y)\Bigr)
\]
as defined in Definition \ref{def.parametrized}.
Likewise, $\snglr$ is the category enriched over $\Fun(\Stratd^{\op}, \Set)$ for which: an object is a stratified space; and for $X$ and $Y$ stratified spaces, then the presheaf of maps is
\[
\snglr(X,Y) := \Bigl(Z\mapsto \bigl\{f \in \Strat_Z(X,Y) \mid f \text{ is an open embedding}\}\Bigr)~.
\]
We use the notation 
\[
\mfld~,~\bsc~\subset~ \snglr
\]
for the full sub-enriched categories consisting of the objects of $\mfldd$ and $\bscd$, respectively.  
\end{definition}

\begin{definition}[$\Delta^p_e$]\label{def.extended-simplex}
Define the \emph{standard cosimplicial manifold} as the functor
\[
\Delta^\bullet_e\colon \bdelta \longrightarrow \Stratd
\]
given by assigning
\[
[p]\mapsto \Delta^p_e := \Bigl\{ t\colon \{0,\dots,p\} \to \RR \ \Big| \ \sum_{i=0}^p t_i = 1\Bigr\}~.
\]
\end{definition}
Note that $\Delta^p_e$ is a smooth manifold which is noncanonically diffeomorphic to $\RR^{p}$.

Restriction along $\bdelta\xra{\Delta^\bullet_e} \Stratd$ defines a product preserving functor 
\begin{equation}\label{restrict-enrichment}
(-)_{|\bdelta}\colon \Fun\bigl(\Stratd^{\op}, \Set\bigr) \to {\sf sSet}
\end{equation}
where ${\sf sSet}:=\Fun(\bDelta^{\op}, {\sf Set})$ is the category of simplicial sets.

\begin{lemma}\label{mapping-Kans}
Let $X$ and $Y$ be conically smooth stratified spaces.
Both of the simplicial sets $\Strat(X,Y)_{|\bdelta}\colon [p]\mapsto \Strat_{\Delta^p_e}(X,Y)$ and $\snglr(X,Y)_{|\bdelta} \colon [p]\mapsto \snglr_{\Delta^p_e}(X,Y)$ are Kan complexes.  
\end{lemma}

\begin{proof}
We first show that $\Strat(X,Y)$ is a Kan complex.
For $\cU \subset \snglrd_{/X}$ a covering sieve (see Definition \ref{def.sieve}), the canonical map of simplicial sets $\Strat(X,Y) \ra \underset{O\in \cU}{\sf lim} \Strat(O,Y)$ is an isomorphism.  
So we can assume that $X$ has bounded depth.

Consider the sub-stratified space $X_{\sf d}\subset X$ consisting of the deepest strata.  
The inclusion $X_{\sf d} \hookrightarrow X$ has a conically smooth tubular neighborhood by Proposition~\ref{tubular-neighborhoods}.
It is then routine that the restriction map of simplicial sets $\Strat(X,Y) \to \Strat(X_{\sf d},Y)$ is a Kan fibration. 
We are thus reduced to the case that $X=M$ is an ordinary smooth manifold.
Write the stratifying poset $Y=(Y\to Q)$. 
Then the evident map of simplicial sets $\underset{q\in Q} \prod \Strat(M,Y_q) \ra \Strat(M,Y)$ is an isomorphism. For each factor, both $M$ and $Y_q$ are ordinary smooth manifolds, and $\Strat(M,Y_q) = C^{\infty}(M,Y_q)$ is a Kan complex.

We now show that the sub-simplicial set $\snglr(X,Y)\subset \Strat(X,Y)$ is a Kan complex.  
Choose $0\leq i \leq p$.
Choose a $p$-simplex $\sigma\colon \Delta[p]\to \Strat(X,Y)$, which we write as $X\times \Delta^p_e\xra{f}Y\times \Delta^p_e$, such that for each subset $i\notin S\subset [p]$ the composite $\Delta[S]\hookrightarrow \Delta[p]\to \Strat(X,Y)$ factors through $\snglr(X,Y)$.  
Lemma~\ref{ivt-again} grants that, for each such $S$, there is an open neighborhood $\Delta^S\subset O_S\subset \Delta^p$ for which the restriction $f_{t} \colon X\times\{t\} \to Y\times \{t\}$ is a conically smooth open embedding for each $t\in O_S$.
Choose a smooth embedding $\alpha \colon \Delta^p_e\hookrightarrow \Delta^p_e$ extending the identity map on each $\Delta^S$ whose image lies in the union $\underset{i\notin S\subset [p]}\bigcup O_S\subset \Delta^p_e$.
The composition $f\circ \alpha\colon X\times \Delta^p_e \to Y\times \Delta^p_e$ now defines a filler in $\snglr(X,Y)$ of the horn $\sigma_{|\Lambda_i[p]}$.  
\end{proof}

\begin{convention}\label{defn.enrichment}
The $\Fun(\Stratd^{\op}, \Set)$-enrichments of the categories in Definition~\ref{self-enrichment} restrict to $\Kan$-enrichments -- this is a consequence of the product preserving functor~(\ref{restrict-enrichment}) and Lemma~\ref{mapping-Kans}.
In a standard manner, for instance via the simplicial nerve functor (see \S1.1.5 of~\cite{HTT}), we then regard these $\Kan$-enriched categories as $\infty$-categories.
We double-book notation and again write
\[
\Strat,
\qquad
\snglr,
\qquad
\mfld,
\qquad
\bsc
\]
for these respective $\infty$-categories.  
Note in particular that $\bsc$ is a full $\infty$-subcategory of $\snglr$ via an iteration of Proposition~\ref{prop.basics-are-singular}. 
\end{convention}

\begin{remark}\label{remark.kan-enrichments}
Most natural functors in the non-enriched setting admit a $\sf Kan$-enrichment, simply by declaring $F(\Delta^\bullet_e\times X) = \Delta^\bullet_e\times F(X)$ for any such functor $F$.
\end{remark}

\subsection{Interlude: presheaves and right fibrations}
\label{section.right-fibration}
Before we move on, we set some notation and recall some general facts about the Grothendieck construction for $\infty$-categories.

\begin{defn}
For any $\infty$-category $\cC$, then 
\[
 \psh(\cC):= \fun(\cC^{\op},\spaces)
\] 
is the $\infty$-category of contravariant functors from $\cC$ to the $\infty$-category $\spaces$. A {\em presheaf} on $\cC$ is an object of $\psh(\cC)$.
\end{defn}

\begin{defn}
Let $f: \cE \to \cC$ be a functor of $\infty$-categories. Then we say $f$ is a {\em right fibration} if the diagram
\[
 \xymatrix{
 \Lambda^p_i \ar[r]  \ar[d]
 &
 \cE  \ar[d]
 \\
 \Delta^p  \ar[r]  \ar@{-->}[ur]
 &
 \cC
 }
\]
can be filled for any $0 < i \leq p$. Here, $\Lambda^p_i \into \Delta^p$ is the usual inclusion of a horn into the $p$-simplex.
Given two right fibrations $\cE\to \cC$ and $\cE'\to \cC$ the simplicial set $\Fun_\cC(\cE,\cE')$ of maps over $\cC$ has a maximal sub-Kan complex, naturally in $\cE$ and $\cE'$.
And so, there is the $\Kan$-enriched category of right fibrations over $\cC$
\[
\RFbn_\cC~,
\]
which we regard as an $\infty$-category.  
\end{defn}

\begin{remark}
A model-independent way to define a right fibration is to require that $f$ has the right lifting property with respect to all {\em final} functors $\cK_0 \to \cK$.
\end{remark}

In \S2.2.1 of~\cite{HTT}, Lurie defines an $\infty$-categorical version of the Grothendieck construction, which he calls the {\em unstraightening construction}. This produces a right fibration $\cE \to \cC$ from any presheaf on $\cC$. He then goes onto prove:

\begin{theorem}[Proposition 2.2.3.11, Theorem 2.2.1.2 of~\cite{HTT}]
\label{right-fibration}
The unstraightening construction induces an equivalence of $\infty$-categories
\begin{equation}\label{eqn.right-fibrations}
\Psh(\cC) \xra{~\simeq~} \RFbn_\cC.
\end{equation}
\end{theorem}

We will make use of this equivalence freely.

Since being a right fibration is a right lifting property, it is stable under pullbacks:

\begin{prop}
\label{prop.right-fibrations-pull-back}
In the pullback square
\[
 \xymatrix{
 A' \ar[r] \ar[d]^{f'}
 & A \ar[d]^f \\
 B' \ar[r]
 & B
 }
\]
if $f$ is a right fibration, so is $f'$.
\end{prop}

\begin{remark}
The above is just a statement that functors compose: a functor $B^{\op} \to \spaces$ induces a functor $(B')^{\op} \to \spaces$ by pre-composing with $B' \to B$.
\end{remark}

\begin{example}
In particular, the Yoneda embedding $\cC \to \psh(\cC)$ determines a right fibration, and for any object $X \in \cC$, the right fibration is given by the forgetful functor $\cC_{/X} \to \cC$. When $\cC = \snglr$, an edge in $\snglr_{/X}$ is given by a diagram
\[
\xymatrix{
 Y \ar[rr]^g  \ar[dr]_f
 && Y' \ar[dl]^{f'} \\
 & X &
}
\]
in the $\infty$-category $\snglr$. Explicitly, it is the data of the three open, conically smooth embeddings $g, f, f'$ as indicated in the diagram, together with an {\em isotopy} from $f' \circ g$ to $f$.
\end{example}

We will use the following in the proof of Lemma~\ref{Gamma-compares}.

\begin{lemma}\label{lemma.rfbn-over}
Let $\cE \xra{p} \cC$ be a right fibration over a quasi-category.
Composition gives an equivalence of $\infty$-categories
\begin{equation}
-\circ p \colon \RFbn_\cE  \xra{~\simeq~}(\RFbn_\cC)_{/\cE}~.
\end{equation}
\end{lemma}

\begin{proof}
The functor fits into a diagram among $\infty$-categories
\[
\xymatrix{
\RFbn_\cE  \ar[rr]^-{-\circ p}  \ar[d]
&&
(\RFbn_\cC)_{/\cE}  \ar[d]
\\
(\Cat)_{/\cC}  \ar[rr]^-{-\circ p}
&&
\bigl((\Cat)_{/\cC}\bigr)_{/\cE}~.
}
\]
The vertical functors are fully faithful, by definition.
The bottom horizontal functor is an equivalence, formally.  
It follows that the top horizontal functor is fully faithful. 
The result follows upon verifying that the top horizontal functor is essentially surjective.

Let $X\to \cC$ be a right fibration and let $X\to \cE$ be a map over $\cC$.  
Through Joyal's right model structure on ${\sf sSet}_{/\cE}$~(\cite{HTT}, \S 2), there is a factorization $X\xra{\simeq} \w{X} \to \cE$ which by a weak equivalence followed by a right fibration.  
We conclude that the functor is essentially surjective.
\end{proof}

\subsection{Computing $\bsc$}\label{sec:endomorphisms}
The main result of this section, Theorem~\ref{basic-facts}, indicates that an $\infty$-category of basics is a reasonable datum to specify in practice -- the result articulates in what sense the $\infty$-category $\bsc$ is a poset (of singularity types) up to $\infty$-groupoid ambiguity. 
Write 
\[
[\bsc] := \{[U]\mid U\in \bsc\}
\]
for the set of isomorphism classes of objects of $\bscd$, the ordinary (un-enriched) $\infty$-category of basics.
It is equipped with the relation $\{([U],[V])\mid \exists~ U\to V\}$ consisting of those pairs of isomorphism classes for which there is a morphism between representatives. We prove in Theorem~\ref{theorem.basics-easy}(\ref{enum.basic-partial-order}) that this defines a partial order relation. 
\begin{theorem}[Basics are easy]\label{basic-facts}
\label{theorem.basics-easy}
\begin{enumerate}
\item[]

\item For each object $U = \RR^i\times \sC(Z)\in \bsc$, the inclusion of Kan monoids
\[
\sO(\RR^i)\times \Aut(Z) \hookrightarrow \bsc(U,U)
\]
is a homotopy equivalence.  
In particular, $\bsc(U,U)$ is group-like.  
\item
Let $U\xra{f} V$ be a morphism in $\bsc$.
Then exactly one of the following is true:
\begin{enumerate}
\item $f$ is an equivalence in the $\infty$-category $\bsc$.  
\item The depth of $U$ is strictly less than the depth of $V$.
\end{enumerate}

\item\label{enum.basic-partial-order}
The relation on $[\bsc]$ is a partial order. 

\item The canonical functor between $\infty$-categories $[-]\colon \bsc\to [\bsc]$ is conservative.
In other words, a morphism $U\to V$ in the $\infty$-category $\bsc$ is an equivalence if and only if $U$ and $V$ are isomorphic in $\bscd$.

\item The map of sets $\depth \colon [\bsc]\to \ZZ_{\geq 0}$ is a map of posets.    In other words, the existence of a morphism $U\to V$ implies $\depth(U) \leq \depth(V)$.  
\end{enumerate}
\end{theorem}

After Lemma~\ref{lemma.small-opens}, Theorem~\ref{basic-facts}(\ref{enum.basic-partial-order}) gives the following reassuring result.
\begin{cor}[Unique local types]\label{rem.unique-locals}
Let $X$ be a stratified space and let $x\in X$ be a point.
Let $(U,0) \hookrightarrow (X,x) \hookleftarrow (V,0)$ be two coordinate charts about $x$.  
Then there is an isomorphism $U \cong V$.
In particular, writing $U=\RR^i\times \sC(Z)$ and $V=\RR^j\times \sC(Y)$, we have $i=j$ and $Z\cong Y$.  
\end{cor}

For the rest of this section we fix a basic $U = \bigl( \RR^i\times \sC(Z),\cA_Z\bigr)$.  
As usual, we will identify $\RR^i$ as the stratum $\RR^i\times \ast \subset U$.
As so, we will abbreviate the element $(0,\ast) \in \RR^i\times \sC(Z)$ by simply writing $0$ and refer to it as the \emph{origin}.
We are about to examine the Kan monoid $\bsc(U,U)$ of endomorphisms of $U$.  
We write $\sO(\RR^i)$ for the orthogonal group, and $\Aut(Z)$ for the group Kan complex of automorphisms of $Z$ in $\snglr$. 

Recall the map $\gamma \colon \Sing^{\sf{sm}}(\RR_{\geq 0}\times \RR^i) \to \bsc(U,U)$, adjoint to the map from~\S\ref{sec.conical-smoothness}.

\begin{definition}\label{other-endos}
We define the sub-Kan monoids
\begin{equation}\label{basic-equivalences}
\bsc(U,U) \supset \bsc^0(U,U) \supset \Aut^0(U)\supset {\sf GL}(U) \supset \sO(U)
\end{equation}
as follows.  Let $f: \Delta^p_e \times U  \to \Delta^p_e \times U$ be a $p$-simplex of $\bsc(U,U)$.  
\begin{itemize}
	\item
	This $p$-simplex belongs to $\bsc^0(U,U)$ provided $f(t,0) = (t,0)$ for all $t\in \Delta^p_e$.  
	\item
	This $p$-simplex belongs to $\Aut^0(U)$ provided $f(t,0) = (t,0)$ and $f_{|\{t\}\times U}$ is an isomorphism for each $t\in \Delta^p_0$.	
	\item
	This $p$-simplex belongs to ${\sf GL}(U)$ if it satisfies the equation
	\begin{equation}\label{linear-equation}
	T\circ \gamma_{t,v} = \gamma_{t,(Tv)_{\RR^i}}\circ  T
	\end{equation}
	for all $(t,v)\in \RR_{\geq 0}\times \RR^i$ -- here we have used the subscript ${}_{\RR^i}$ to indicate the projection to $\RR^i$. 
	The case $Z = \emptyset$ reduces to the consideration of those open embeddings $T: \RR^i \to \RR^i$ for which $T(t\vec u + \vec v) = t T\vec u + T\vec v$ (i.e., elements of ${\sf GL}(\RR^i)$).
	See Lemma~\ref{linear-is-iso}.
	\item
	$\sO(U) := \sO(\RR^i) \times \Aut(Z)$. This is regarded as a submonoid of $\bsc^0(U,U)$ by acting on the evident coordinates of $U = \RR^i\times \sC(Z)$.  That its image lies in ${\sf GL}(U)$ is a simple exercise.
	\end{itemize}
\end{definition}

\begin{lemma}\label{linear-is-iso}
There is an inclusion ${\sf GL}(U) \subset \Aut^0(U)$ of sub-Kan complexes of $\bsc(U,U)$.  
Moreover, if $S,T\in {\sf GL}(U)$ then $T\circ S \in {\sf GL}(U)$.  
\end{lemma}

\begin{proof}
$T(0)=0$ because $\gamma_{t,0}(0) = 0$.
We need to show that each morphism $U\xra{T} U$ in ${\sf GL}(U)$ is a surjective map of underlying topological spaces.  This follows from the defining equation~(\ref{linear-equation}) after the following two observations:
\begin{itemize}
\item the collection $\{0\in O\subset U\}$, of open neighborhoods of the origin with compact closure, covers $U$;
\item for each open neighborhood of the origin $0\in O\subset U$ with compact closure, the collection of images $\{\gamma_{t,0}(O) \mid t>0\}$ is a local base for the topology about the origin $0\in U$.
\end{itemize}
The second statement is immediate by inspecting the defining equation~(\ref{linear-equation}).   
\end{proof}

Recall from Definition~\ref{def:cone-sm} the \emph{derivative} $\sD_0f\colon U \to U$ of a conically smooth self-open embedding $f\colon U\to U$.  
\begin{lemma}\label{deriv-linear}
\label{lemma.deriv-linear}
The assignment $f\mapsto \sD_0 f$ describes a homomorphism of Kan monoids
\[
\sD_0\colon \bsc^0(U,U) \longrightarrow {\sf GL}(U)~.
\]
Moreover, this map is a deformation retraction of Kan complexes.  
\end{lemma}

\begin{proof}
The string of equalities
\begin{eqnarray}
\sD_0f \circ \gamma_{t,v}
& = &
\bigl(\underset{s\to 0}{\sf lim} \gamma_{\frac{1}{s},0} f \gamma_{s,0}\bigr)\circ \gamma_{t,v} 
\nonumber
\\
& = & 
\underset{s\to 0}{\sf lim} \gamma_{\frac{1}{s},0} f \gamma_{ts,sv}  
\nonumber
\\
& = & 
\underset{s\to 0}{\sf lim} \Bigl(\bigl( \gamma_{1,-\frac{1}{s}f(-sv)} \gamma_{1,\frac{1}{s}f(-sv} \bigr) \circ \bigl( \gamma_{\frac{1}{s},0} f \gamma_{st,sv}\bigr)\Bigr)
\nonumber
\\
& = & 
\underset{s\to 0}{\sf lim} \Bigl( \bigl(\gamma_{1,-\frac{1}{s} f(-sv)}\bigr)  \circ \bigl( \gamma_{\frac{1}{s},0} \gamma_{1,\frac{1}{s}f(-sv)} f \gamma_{st,sv} \bigr)\Bigr)
\nonumber
\\
& = & 
\Bigl(\underset{r\to 0}{\sf lim} \gamma_{1,\frac{1}{r}f(rv)} \Bigr) \circ \Bigl( \underset{s\to 0}{\sf lim} \bigl( \gamma_{\frac{1}{s},0} \circ ( \underset{p\to 0}{\sf lim} \gamma_{1,f(-pv)} f \gamma_{st,pv}) \bigl) \Bigr)
\nonumber
\\
& = & 
\gamma_{1,\sD_0f(v)} \circ \bigl( \underset{s\to 0}{\sf lim} \gamma_{\frac{1}{s},0} f \gamma_{st}\bigr)
\nonumber
\\
& = & 
\gamma_{1,\sD_0f(v)} \circ \bigl(\underset{s'\to 0}{\sf lim} \gamma_{\frac{t}{s},0} f \gamma_{s,0} \bigr)
\nonumber
\\
& = & 
\gamma_{1,\sD_0f(v)}\gamma_{t,0} \circ \bigl( \underset{s'\to 0}{\sf lim} \gamma_{\frac{1}{s},0} f \gamma_{s,0}\bigr)
\nonumber
\\
& = & 
\gamma_{t,\sD_0f(v)}\circ  \sD_0f
\nonumber
\end{eqnarray}
verifies that $\sD_0f \in {\sf GL}(U)$.
Moreover, $T\in {\sf GL}(U)$ implies $\gamma_{\frac{1}{t},0} T \gamma_{t,0} = T$ for all time $t>0$, so $\sD_0T = T$, meaning $\sD_0$ is a retract.  
The path of maps $\Sing^{\sf{sm}}([0,1])\times \bsc^0(U,U) \to \bsc^0(U,U)$ given by $(s,f)\mapsto \gamma_{\frac{1}{s},0} f \gamma_{s,0}$ is a homotopy from $\sD_0$ to the identity map.  
\end{proof}

\begin{lemma}\label{derivative}
Each of the inclusions in~(\ref{basic-equivalences}) is a homotopy equivalence of Kan complexes.  
\end{lemma}

\begin{proof}
It is enough to witness the following deformation retractions:
\begin{itemize}
\item
There is a deformation retraction of $\bsc(U,U)$ onto $\bsc^0(U,U)$ given by $(t,f)\mapsto \gamma_{1,tf(0)} \circ f$.

\item 
The deformation retract of $\bsc^0(U,U)$ onto ${\sf GL}(U)$ is Lemma~\ref{deriv-linear}.  
\item
There is a deformation retraction of ${\sf GL}(U)$ onto $\sO(U)$ given in coordinates as
	\[
	(t,T)
	\mapsto
	\Bigl( 
	(u, [s,z])
	\mapsto ( \mathsf{GrSm}_t(\sD_0 T_{|\RR^i})(u) ,
	(1-t) + t \sD_0 T[s,z]_\RR, 
	T(0,[0,z])_{\sC(Z)}\Bigr] 
	\Bigr)
	\]
-- here $\mathsf{GrSm}_t$ is the Gram--Schmidt deformation retraction at time $t$, and the subscript notation is used to denote the projection to the named coordinate. 
\end{itemize}
\end{proof}

\begin{lemma}\label{only-equivalences}
\label{lemma.only-equivalences}
Let $U=\RR^i\times \sC(Z)$ and $V=\RR^j\times \sC(Y)$ be basics.
Suppose there exists a conically smooth open embedding $U\hookrightarrow V$.
Then the following are equivalent.
\begin{enumerate}
\item There is an equality of depths $\depth(U) = \depth(V)$.
\item There is a conically smooth open embedding $f\colon U \hookrightarrow V$ for which the intersection $f(U)\cap \RR^j\subset V$ is nonempty.
\item The inclusion of the isomorphisms
\[
{\sf Iso}(U,V) \xra{~\simeq~} \bsc(U,V)
\]
is a weak equivalence of Kan complexes.
\item There is an isomorphism $U\cong V$.  
\item There is an equality $i=j$ and an isomorphism of stratified spaces $Z\cong Y$.  
\end{enumerate}
\end{lemma}
\begin{proof}
We show $(5) \implies (4) \implies (3) \implies (2) \implies (1)\implies (5)$.  
That $(5)\implies (4)$ is obvious.  
$(4)\implies (3)$ through Lemma~\ref{derivative}.  
That $(3)\implies (2)$ is obvious, from the definition of morphisms in $\bsc$.
That $(2)\implies (1)$ follows because the existence of an open embedding $\RR^i\times \sC(Z) \hookrightarrow \RR^j \times \sC(Y)$ implies $Z$ and $Y$ have the same dimension.  
Assume $(1)$.  Let $f\colon U\hookrightarrow V$ be a conically smooth open embedding.
Directly after Definition~\ref{def:cone-sm} of \emph{conically smooth}, there is a conically smooth open embedding $\w{f} \colon \RR^i\times \RR_{\geq 0}\times Z \hookrightarrow \RR^i\times \RR_{\geq 0}\times Y$ over $f$.  
Restricting $\w{f}$ gives the conically smooth open embedding $Z =  \{(0,0)\}\times Z \hookrightarrow \{(f(0),0)\}\times Y = Y$.  Because $Z$ is compact, this map is an inclusion of components.  
If this map misses a connected component, the map of cones $\{0\}\times \sC(Z) \to \{f(0)\}\times \sC(Y)$ is not open.  
We therefore conclude $(5)$.  
\end{proof}

\begin{proof}[Proof of Theorem~\ref{basic-facts}]

(1) This is Lemma~\ref{derivative}.

(2) This is an immediate consequence of Lemma~\ref{only-equivalences}.

(3) The existence of identity morphisms implies the named relation contains the diagonal relation.
Composition of morphisms witnesses transitivity for the named relation.  Lemma~\ref{only-equivalences} verifies that $[U]=[V]$ whenever $[U]\leq [V]$ and $[V]\leq [U]$.  
This proves that $[\bsc]$ is a poset. Hence $[\bsc]$ defines a category, whose nerve is an $\infty$-category. Since the non-emptiness of $\bscd(U,V)$ is equivalent to the non-emptiness of $\pi_0\bsc(U,V)$, one has a well-defined functor of $\infty$-categories $\bsc \to [\bsc]$ which factors through the homotopy category of $\bsc$.

(4) Let $U\xra{f}V$ be a morphism with $U\cong V$.  Lemma~\ref{only-equivalences} gives that $f$ is an equivalence in the $\infty$-category $\bsc$.

(5) Lemma~\ref{only-equivalences} immediately verifies that $\depth\colon [\bsc]\to \NN$ is a map of posets. 
\end{proof}

\subsection{Tangent classifiers, smooth and singular}
Recall the {\em tangent classifier} from Definition~\ref{def.tangent-classifier} , and the right fibration $\tau_X : \Ent(X) \to \bsc$ from Definition~\ref{defn.E(X)}.

\begin{definition}[$\Ent_{[U]}(X)$]\label{class-fiber}
Let $U=\RR^i\times \sC(Z)$ be a basic.  
Denote the pullback $\infty$-categories
\[
\xymatrix{
\Ent_{[U]}(X)  \ar[r]  \ar[d]_-{(\tau_{X})_|}
&
\Ent(X)  \ar[d]^-{\tau_X}
\\
\bsc_{[U]}  \ar[r]  \ar[d]
&
\bsc  \ar[d]^-{[-]}
\\
\{[U]\}  \ar[r]
&
[\bsc].
}
\]
Since $\tau_X$ is a right fibration, so is $(\tau_X)_|$, by Proposition~\ref{prop.right-fibrations-pull-back}.
\end{definition}

\begin{remark}[$E_U$ versus $E_{[U]}$]
Recall also that we defined the $\infty$-category $E_U(X)$ in Definition~\ref{defn.E(X)}. This is the fiber of $\Ent(X) \to \bsc$ over the vertex $\{U\} \into \bsc$, and is in fact an $\infty$-groupoid (i.e., Kan complex) because $\Ent(X) \to \bsc$ is a right fibration.  Meanwhile, $E_{[U]}(X)$ is the fiber over the entire sub-groupoid $\bsc_{[U]} \simeq {\sf BAut}(U)$. (Up to equivalence, this is a $\Kan$-enriched category with a single object, whose endomorphism Kan complex is $\Aut(U)$.) For example:
\begin{enumerate}
\item $E_{[U]}(U)$ has a terminal object given by the identity morphism of $U$, while the equivalence classes of objects in $E_U(U)$ are in bijection with isotopy classes of self-embeddings $U \into U$. 
\item If $X$ is a smooth $n$-manifold, the image of the right fibration $\Ent(X) \to \bsc$ lands in the subcategory $\sD \simeq \sB\sO(n) \simeq {\sf BAut}(\RR^n)$. Setting $U = \RR^n$, $\Ent_U(X)$ is homotopy equivalent to the frame bundle of $X$, while $\Ent_{[U]}(X)$ is homotopy equivalent to $X$ itself. 
\end{enumerate}
\end{remark}

Here is an immediate consequence of Theorem~\ref{basic-facts}. 
 
\begin{cor}
\label{cor.max-groupoid-bsc}
The $\infty$-category $\underset{[U]\in [\bsc]}\coprod \bsc_{[U]}$ is the maximal $\infty$-subgroupoid of $\bsc$.
Furthermore, for each conically smooth stratified space $X$ the $\infty$-category $\underset{[U]\in \bsc}\coprod \Ent_{[U]}(X)$ is the maximal $\infty$-subgroupoid of $\Ent(X)$.  
\end{cor}

Recall from Lemma~\ref{dim-depth} that a stratified space $(X\to P)$ determines a map of posets $P \to \PP$. But rather than projecting to $\PP = \{(\depth,{\sf dim})\}$, one can remember the singularity type of a specific point in $X$. The following result gives a stratification $X\to [\bsc]$ for any conically smooth stratified space $X$. 

\begin{lemma}\label{lem.unique-locals}
Let $X$ be a conically smooth stratified space with stratification $X \to P$. Then the composite map $X \to P \to \PP$ factors through a continuous map to the poset $[\bsc]$. The map to $[\bsc]$ respects open embeddings $X \to X'$ -- i.e., the composite $X \to X' \to [\bsc]$ is equal to the map $X \to [\bsc]$.
\end{lemma}

\begin{proof}
Fix $x \in X$. Consider the set of all $U \in \bsc$ such that $x$ is in the image of some open embedding $U \into X$. Since the $U$ form a basis for $X$ by Proposition~\ref{prop.basics-form-a-basis}, the set of $U$ is nonempty. Since $x$ has a local depth, the depth of all such $U$ is bounded by $\depth(x)$. Hence there is an element of maximal depth, $U_0$. By Lemma~\ref{lemma.only-equivalences}, this element is unique, and this gives a map of sets $X \to [\bsc]$ sending  $x \mapsto U_0$. 

Now we prove the map $X \to [\bsc]$ is continuous. If $X$ is empty, this is trivially true. By induction, assume that $X \to [\bsc]$ is continuous for all singular $X$ of depth $d$. Then the map $\alpha: U \to [\bsc]$ is continuous for all singular basics $U$ of depth $d+1$, as follows: any point $x \in \RR^i \times \{\ast\} \subset \RR^i \times \sC(Z)$ is sent to $[U] \in [\bsc]$. By definition of the order on $[\bsc]$, any open set containing $U$ must contain the whole image of $\alpha$, so the pre-image of any such open set (being all of $U$) is open. Any other open set in $\alpha(U)$ consists of $[V] \in [\bsc]$ with strictly smaller depth, so the inductive hypothesis guarantees that their pre-images are open in $\RR^i \times \RR_{>0} \times Z \subset U$. Now we prove $\alpha: X \to [\bsc]$ is continuous for all singulars of depth $d+1$. 

Choose an open cover of $X$ by basics, $\{\varphi_i: U_i \into X\}$. Note that the map $\alpha_{U_i} : U_i \to [\bsc]$ is the composite $\alpha \circ \varphi_i: U_i \into X \to [\bsc]$. Fix an open set $W \subset [\bsc]$. Then 
\[
 \alpha^{-1}(W) \cap \varphi_i(U_i) = \varphi_i \left( 
 (\alpha \circ \varphi_i)^{-1}(W)
 \right)
 \]
Note that $(\alpha \circ \varphi_i)^{-1}(W) = \alpha_{U_i}^{-1}(W) \subset U_i$, so it is open in $U_i$ because the map $\alpha_{U_i}$ is continuous for basics. The result follows by remembering that $\varphi_i$ is an open embedding.

Now that we have proven continuity for all conically smooth stratified spaces with bounded depth (and hence all basics), it follows that $\alpha$ is continuous for conically smooth stratified spaces. Simply repeat the argument from the previous paragraph verbatim.
\end{proof}

\begin{definition}[$X_{[U]}$]\label{class-stratum}
Let $U=\RR^i\times \sC(Z)$ be a basic.  
Denote the pullback topological space
\[
\xymatrix{
X_{[U]} \ar[r]  \ar[d]
&
X  \ar[d]
\\
\{[U]\}  \ar[r]
&
[\bsc]
}
\]
and refer to it as the $[U]$-stratum -- it is the locus of those points in $X$ about which embeddings $U\hookrightarrow X$ form a local base.
It follows from Corollary~\ref{strata-smooth} that $X_{[U]}$ is a smooth manifold of dimension $i:={\sf dim}(U)-\depth(U)$.
\end{definition}

\begin{example}
If $U$ is a basic, then then the natural map $U_{[U]} \into U$ is the inclusion $\RR^i \into \RR^i \times \sC(Z)$.
\end{example}

\begin{prop}
\label{prop.stratification-functorial}
Fix a singular basic $U$. Then the assignment
\[
 (-)_{[U]}: \snglr \to \mfld
\]
is a functor of $\kan$-enriched categories. More generally, given a consecutive embedding $Q \into [\bsc]$, the assignment $(-)_{Q}: \snglr \to \snglr$ is a functor of $\kan$-enriched categories.
\end{prop}

\begin{proof}
Lemma~\ref{lem.unique-locals} tells us that conically smooth open embeddings respect the map to $[\bsc]$. 
Hence, since $X_{[U]}$ is defined via pull-back, any conically smooth open embedding $j: X \into X'$ defines a map $X_{[U]} \into X'_{[U]}$. This map is also conically smooth, since it is given by restricting $j$ to a sub-stratified space.
In particular, any conically smooth open embedding from $\Delta^p_e \times X \into \Delta^p_e \times X'$ defines an open embedding $(\Delta^p_e \times X)_{[\RR^p \times U]} \into (\Delta^p_e \times X')_{[\RR^p \times U]}$. By functoriality of pull-backs, these assignments respect the face and degeneracy maps of the morphism Kan complexes.

The second statement follows mutatis mutandis.
\end{proof}

By construction, there is a natural equivalence of spaces
	\[
	\Ent_U(X)~ \simeq~ \snglr(U,X)
	\]
between the fiber of $\Ent(X) \xra{\tau_X} \bsc$ over $U$ and the Kan complex of embeddings of $U$ into $X$.
It is useful to think of an element of this embedding space as a very small neighborhood of a point in $X$ at whose singularity type is $[U]$, together with a parametrization of this neighborhood (which one might think of as a choice of \emph{framing} about the point).
Our first goal is to make this intuition precise -- this is Lemma~\ref{exit-class}.

After Lemma~\ref{derivative}, there is a map of $\infty$-groupoids $\bsc_{[U]} \to \sB\sO(\RR^d)$ for $d={\sf dim}(U)-\depth(U)$.  
\begin{lemma}\label{exit-class}
Let $X$ be a conically smooth stratified space.
Let $U = \RR^i\times \sC(Z)$ be a basic.
There is a natural equivalence of maps
\[
\Bigl(\Ent_{[U]} (X) \xra{(\tau_X)_|} \bsc_{[U]} \to \bsc_{[\RR^i]}\Bigr)~ \simeq ~\Bigl( \Sing\bigl(X_{[U]}\bigr)\xra{\tau_{X_{[U]}}} \sB\sO(\RR^i)\Bigr)
\]
from the restriction of the tangent classifier in the sense of Definition~\ref{def:tangents} to the map from the singular complex of the smooth manifold $X_{[U]}$ classifying its tangent bundle.  
\end{lemma}

\begin{proof}
Recall from Definition~\ref{other-endos} the $\Kan$-group $\Aut^0(U)$ of based automorphisms of $U$.
There is likewise the sub $\Kan$-enriched category $\bsc^0_{[U]}\subset \bsc_{[U]}$ consisting of the same objects and those $p$-simplices of morphisms $\Delta^p_e\times U'\xra{f} \Delta^p_e\times U''$ for which $f(t,0)=(t,0)$.
Lemma~\ref{derivative} gives that each of the functors ${\sf BAut}^0(U) \xra{\simeq} \bsc^0_{[U]}\xra{\simeq} \bsc_{[U]}$ is an equivalence of $\infty$-categories (in fact, of $\infty$-groupoids).
Consider the restricted right fibrations $\Ent^{\sf bdl}_{[U]} := \Ent_{[U]}(X)_{|{\sf BAut}(U)} \to {\sf BAut}^0(U)$ and $\Ent^0_{[U]} := \Ent_{[U]}(X)_{|\bsc^0_{[U]}} \to \bsc^0_{[U]}$ -- each of the functors $\Ent^{\sf bdl}_{[U]}(X) \xra{\simeq} \Ent^0_{[U]}(X) \xra{\simeq} \Ent_{[U]}(X)$ is an equivalence of $\infty$-groupoids.  
Evaluation at $0\in U$ gives functors $\Ent^{\sf bdl}_{[U]}(X) \to \Sing(X_{[U]})$ and $\Ent^0_{[U]}(X) \to \Sing(X_{[U]})$ to the singular complex of the $[U]$-stratum.  
The isotopy extension theorem gives that each of these functors is a Kan fibration between Kan complexes.
We now argue that each of these Kan fibration is trivial.
We do this by demonstrating that, for each $x\in X$, each map $K\to \Ent^{\sf bdl}_{[U]}(X)_x$ to the fiber from a simplicial set with finitely many non-degenerate simplices can be extended to a map $K^\tl \to \Ent^0_{[U]}(X)_x$.

Fix $x\in \Sing(X_{[U]})$, and let $K \to \Ent^{\sf bdl}_{[U]}(X)_x$ be such a map.
Explicitly, this is represented by the data of a fiber bundle $E\to |K|$, equipped with a section $z$, whose fibers are (non-canonically) based isomorphic to $(0\in U)$, together with a map $E \xra{f} |K|\times X$, over $|K|$ that restricts over each $t\in |K|$ as a conically smooth open embedding sending $z(t)$ to $x$.  
So, the image of $f$ contains a neighborhood of $|K|\times\{x\}$.
Because of the finiteness assumption on $K$, and because based conically smooth open embeddings $(0\in U) \to (x\in X)$ form a base for the topology about $x$ (Proposition~\ref{prop.basics-form-a-basis}), then there is such an embedding $\varphi\colon U\hookrightarrow X$ for which $\varphi(U)\subset f_t(E_t)$ for each $t\in |K|$.  
We are left with a map $f^{-1}\varphi\colon |K|\times U \to E$ over $|K|$ that is fiberwise a based conically smooth open embedding.  
The topological space $|K^\tl|\times U \underset{|K|\times U} \coprod E$ over $|K^\tl|$, with its given section, together with its fiberwise based conically smooth open embedding to $|K^\tl|\times X$, represents a map $K^\tl \to \Ent^0_{[U]}(X)_x$ extending the composite map $K\to \Ent^{\sf bdl}_{[U]}(X)_x \to \Ent^0_{[U]}(X)_x$.  
We have established the sequence of equivalences of $\infty$-groupoids
\[
\Ent_{[U]}(X) \xla{\simeq} \Ent^0_{[U]}(X) \xra{\simeq} \Sing(X_{[U]})~.
\]
The left equivalence lies over an equivalence $\bsc_{[U]} \xla{\simeq}\bsc^0_{[U]}$ by construction.
The right equivalence lies over the projection $(-)_{[U]}\colon \bsc^0_{[U]}  \to \bsc^0_{[\RR^i]}\simeq \sB\sO(\RR^i)$, as seen using a choice of exponential map $\sT X_{[U]} \dashrightarrow X_{[U]}$.  
\end{proof}

Finally, the following justifies why $\tau$ is a generalization of the usual tangent-classifying map from (\ref{eqn.classical-structures}).

\begin{cor}[Classical tangents]\label{classical-tangent}
\label{corollary.classical-tangent}
Let $X$ be an ordinary smooth $n$-manifold.  
There is a natural equivalence of spaces
\[
\bigl(\Ent(X) \xra{\tau_X} \bsc_{n,0}\bigr)~\simeq~\bigl(X\xra{\tau_X} \sB\sO(n)\bigr)
\]
between the tangent classifier in the sense of Definition~\ref{def:tangents} and the classical classifying map for the tangent bundle of the smooth manifold $X$.  
\end{cor}

\begin{example}\label{ordinary-tangents}
We explicate $\Ent(X)$ in some special cases.  
With the exception of the first, these expressions are validated through Remark~\ref{exit-path-category}.  
\begin{itemize}

\item Let $M$ be an ordinary smooth manifold.  Then there is a canonical equivalence of quasi-categories $\Ent(M) \simeq \Sing(M)$.  
\item There is a canonical equivalence of $\infty$-categories $\Ent(\RR_{\geq 0}^{\times p}) \simeq ([1]^{\times p})^{\op}$.

\item Consider the standard unknot $S^1\subset S^3$.  Regard this pair as a conically smooth stratified space via Example~\ref{example:Edn'}.  
Then $\Ent(S^1\subset S^3) \simeq \Sing(S^1) \star \Sing(S^1)$ is the $\infty$-categorical join;
and we see that the classifying space $\sB\bigl(\Ent(S^1\subset S^3)\bigr) \simeq S^1 \star S^1\cong S^3$ is the topological join.  
\item More generally, consider a properly embedded smooth submanifold $W\subset M$, which we regard as a stratified space whose singularity locus is $W$.  
Write $\mathsf{Link}_{W}(M)$ for the unit normal bundle of this embedding, with respect to a chosen complete Riemannian metric on $M$.  
Then there is a canonical equivalence of $\infty$-categories
\[
\Ent(W\subset M) \simeq \Sing(W) \underset{\Sing\bigl(\mathsf{Link}_{W}(M)\bigr)}\bigstar \Sing(M\smallsetminus W)
\]
to the \emph{parametrized join}: For $\cI\leftarrow \cK \to \cJ$ a diagram of $\infty$-categories,
\[
\cI\underset{\cK}\bigstar \cJ:= \cI \underset{\cK\times\{0\}}\coprod \cK\times [1] \underset{\cK\times\{1\}}\coprod \cJ
\]
is the iterated pushout of $\infty$-categories. 
\end{itemize} 
\end{example}

\subsection{Proof of propositions~\ref{localization} and~\ref{refinements-localize}}\label{prove-prop-loc}

Let $\w{X} \to X$ be a refinement between stratified spaces.
By definition of the notion of a refinement, this map is in particular a homeomorphism between underlying topological spaces.
In what follows, we therefore identify these underlying topological spaces.  
Denote by $\opens(\w{X}) = \opens(X)$ the poset of open subsets of this common underlying topological space, ordered by inclusions of subsets.  
Notice the solid functors from the discrete enter-path categories,
\[
\xymatrix{
\Entd(\w{X})   \ar@{-->}[rr]  \ar[d]_-{\sf Image}
&&
\Entd(X)  \ar[d]^-{\sf Image}
\\
\opens(\w{X})  \ar[rr]^-{=}
&&
\opens(X) .
}
\]
Both of these solid functors are fully faithful.
The image of the first functor consists of those open subsets $\w{O}\subset \w{X}$ for which the inherited stratification on $\w{O}$ is abstractly isomorphic to a basic, while the image of the second functor consists of those open subsets $O\subset X$ for which the inherited stratification on $O$ is abstractly isomorphic to a basic.  
Now, by inspection, for $\w{U} \to U$ a refinement with $\w{U}$ a basic, then $U$ too is a basic.
Therefore, there is a factorization, which is necessarily unique, as indicated by the dashed arrow in the above diagram among posets.  
We denote the values of the functor
\[
\Entd(\w{X}) \longrightarrow \Entd(X)~,\qquad  (\w{U}\hookrightarrow \w{X})\mapsto (U\hookrightarrow X)~,
\]
by omitting the tilde.

We will now prove the following result, which immediately implies Proposition~\ref{localization} as the case that the refinement $\w{X} \xra{=}X$ is the identity morphism between stratified spaces.  

\begin{lemma}\label{strength}
For each conically smooth refinement $r\colon \w{X}\ra X$ between stratified spaces, the composite functor $\Entd(\w{X}) \to \Entd(X) \ra \Ent(X)$ witnesses a localization
\[
\Entd(\w{X})[W_r^{-1}] \xra{~\simeq~} \Ent(X)
\]
where $W_r\subset \Entd(\w{X})\subset \Entd(X)$ is the subcategory consisting of the same objects $(\w{U}\hookrightarrow \w{X})$ but only those morphisms $(\w{U}\underset{\rm over~\w{X}}{\hookrightarrow}\w{V})$ for which $U$ and $V$ are abstractly isomorphic as stratified spaces.  
\end{lemma}

\begin{proof}
From Lemma~\ref{basic-facts}, it is immediate that the functor $\Entd(\w{X}) \to \Ent(X)$ factors through the named localization, and in fact it does so conservatively.
To show that this functor is an equivalence, we show it is an equivalence on underlying $\infty$-groupoids, and that it is an equivalence on spaces of morphisms.  

From Lemma~\ref{exit-class}, the underlying $\infty$-groupoid of $\Ent(X)$ is equivalent to the coproduct 
\[
\Ent(X)^{\sim}~\simeq~ \underset{[U]} \coprod X_{[U]}
\]
of the underlying spaces of the indicated smooth manifolds which are defined in Definition~\ref{class-stratum} -- here, the coproduct is indexed by isomorphism classes of basics.
The underlying $\infty$-groupoid of $\Entd(\w{X})[W_r^{-1}]$ is the classifying space $\sB W_r$.  
In light of the above coproduct, consider the full subcategory $W_r^{[U]}\subset W_r$ consisting of those $(\w{V}\hookrightarrow \w{X})$ for which there is a refinement $\w{V}\to U$.
We thus seek to show that the functor $W_r^{[U]} \to X_{[U]}$ witnesses an equivalence from the classifying space $\sB W_r^{[U]} \simeq X_{[U]}$.  
We will do this by first showing that the functor $X_{[U]}\cap - \colon W_r^{[U]} \to W_{X_{[U]}}$ is final, and therefore induces an equivalence on classifying spaces, then observing a canonical equivalence of spaces $\sB W_{X_{[U]}} \simeq X_{[U]}$.  
We use Quillen's Theorem A (Corollary~4.1.3.3 of~\cite{HTT}).
Let $(R\hookrightarrow X_{[U]})\in W_{X_{[U]}}$ be an object.  
Consider the slice category $(W_r^{[U]})^{(R\hookrightarrow X_{[U]})/}$.
An object is an object $(\w{U}'\hookrightarrow \w{X})$ of $W_r^{[U]}$ for which there is an inclusion $R\subset U'$, and a morphism is an inclusion of such.
Because such $(\w{U}'\hookrightarrow \w{X})$ form a base for the topology about $(R\hookrightarrow X_{[U]})$, this category is filtered. 
In particular, the classifying space of $(W_r^{[U]})^{(R\hookrightarrow X_{[U]})/}$ is contractible.  
It follows that the functor $X_{[U]}\cap - \colon W_r^{[U]} \to W_{X_{[U]}}$ is final, and therefore induces an equivalence on classifying spaces, as desired.  
Now, let $M$ be a smooth manifold.
The category $W_M$ is a basis for the standard Grothendieck topology on $M$, so in particular it defines a hypercover of $M$ (Definition 4.1 of \cite{dugger-isaksen}). 
It follows from Theorem 1.3 of~\cite{dugger-isaksen} that the canonical map of topological spaces
\[
\underset{(R\hookrightarrow M)\in W_M} \colim~ R~{}~\xra{~\simeq~}~{}~ M
\]
is a homotopy equivalence.  
Because each term $R$ in this colimit is contractible, 
this colimit is identified as the classifying space
\[
\sB W_M~\simeq~ \underset{(R\hookrightarrow M)\in W_M} \colim~ R~.
\]
Applying this to the case $M=X_{[U]}$, we conclude that $\sB W_{X_{[U]}} \simeq X_{[U]}$.
In summary, we have verified that the map between maximal $\infty$-subgroupoids
\[
\bigl(\Entd(\w{X})[W_r^{-1}]\bigr)^{\sim} \xra{~\simeq~} \bigl(\Ent(X)\bigr)^{\sim}
\]
is an equivalence.  
We now show that the functor induces an equivalence on spaces of morphisms.
Consider the diagram of spaces
\[
\xymatrix{
\bigl(\Entd(\w{U})[W_{r_{|\w{U}}}^{-1}]\bigr)^{\sim}  \ar[r]  \ar[d]_-{(\w{U}\hookrightarrow \w{X})}
&
\Ent(U)^{\sim}  \ar[d]^-{(U\hookrightarrow X)}
\\
\bigl(\Entd(\w{X})[W_r^{-1}]\bigr)^{(1)} \ar[r]  \ar[d]_-{{\sf ev}_1}
&
\Ent(X)^{(1)}  \ar[d]^-{{\sf ev}_1}
\\
\bigl(\Entd(\w{X})[W_r^{-1}]\bigr)^{\sim}  \ar[r]
&
\Ent(X)^{\sim}
}
\]
where a superscript ${}^{(1)}$ indicates a space of morphisms, and the upper vertical arrows are induced by the respective functors $\w{U} \hookrightarrow \w{X}$ and $U\hookrightarrow X$.
Our goal is to show that the middle horizontal arrow is an equivalence.
We will accomplish this by showing that the diagram is a map of fiber sequences, for we have already shown that the top and bottom horizontal maps are equivalences.

Now, for $c\in \cC$ an object of an $\infty$-category, and for $\cB\subset \cC$ a full $\infty$-subcategory, evaluation at the target ${\sf ev}_1\colon {\sf Ar}(\cB_{/c})\to \cB_{/c}$ is a coCartesian fibration, whose fiber over $(b\to c)\in \cB_{/c}$ is canonically identified as the $\infty$-category
\[
(\cB_{/c})_{/(b\to c)}~\simeq~ \cB_{/b}~.
\]
There results a fiber sequence on maximal $\infty$-subgroupoids
\[
(\cB_{/b})^{\sim}  \longrightarrow (\cB_{/c})^{(1)} \longrightarrow (\cB_{/c})^\sim~.
\]
Specializing to the case that $(\cB \subset \cC):=(\Bsc\subset \Snglr)$ and $(b\to c):=(U\hookrightarrow X)$ gives that the right vertical sequence is a fiber sequence.

We now show that the left vertical sequence is a fiber sequence.  
The space of morphisms of $\bigl(\Entd(\w{X})[W_r^{-1}]\bigr)^{(1)}$ is the classifying space of the subcategory of the functor category 
\[
\Fun^{W_r}\bigl([1],\Entd(\w{X})\bigr) ~\subset ~\Fun\bigl([1],\Entd(\w{X})\bigr)
\]
consisting of the same objects but only those natural transformations by $W_r$.  
Our aim now is to show that the hypothesis of Quillen's Theorem B applies to the evaluation map 
\begin{equation}\label{22}
{\sf ev}_1: \Fun^{W_r}\bigl([1],\Entd(\w{X})\bigr) \longrightarrow W_r~.
\end{equation}
Then by Quillen's Theorem B, the fiber of the lower-left vertical map over $(\w{U} \to \w{X})$  is the classifying space of the slice category $(W_r)_{/(\w{U}\hookrightarrow \w{X})}\simeq W_{r_{|\w{U}}}$.  
From above, we identify this classifying space $\sB W_{r_{|\w{U}}}\simeq \bigl(\Entd(\w{U})[W_{r_{|\w{U}}}^{-1}]\bigr)^\sim$ as the maximal $\infty$-subgroupoid of the localization.  
In this way we establish that the left vertical sequence is a fiber sequence once we verify that Quillen's Theorem B applies to the evaluation functor~(\ref{22}).  

To apply Quillen's Theorem B, we must show that each morphism $(\w{U}\underset{\rm over~\w{X}}{\hookrightarrow} \w{V})$ in $W_r$ induces an equivalence between classifying spaces $\sB\bigl((W_r)_{/(\w{U}\hookrightarrow \w{X})}\bigr) \simeq \sB\bigl((W_r)_{/(\w{V}\hookrightarrow \w{X})}\bigr)$.  
This map between spaces is canonically identified as $\sB W_{r_{|\w{U}}} \to \sB W_{r_{|\w{V}}}$.
Through the previous analysis of this proof, this map is further identified as the map between spaces $\underset{[B]} \coprod U_{[B]} \to \underset{[B]} \coprod V_{[B]}$ induced from the conically smooth embedding $U\hookrightarrow V$ between open subsets of $X$.
Because $U$ and $V$ are abstractly isomorphic as stratified spaces, then Lemma~\ref{basic-facts} gives that each such inclusion is isotopic through stratified open embeddings to an isomorphism.  
We conclude that Quillen's Theorem B applies.

\end{proof}

We now prove Proposition~\ref{refinements-localize}, whose statement we recall here:
\begin{itemize}
\item[]
Let $\w{X}\xra{r} X$ be a refinement between conically smooth stratified spaces. Then there is a canonical functor
\[
\Ent(\w{X}) \longrightarrow \Ent(X)
\]
which is a localization.  

\end{itemize}

\begin{proof}[Proof of Proposition~\ref{refinements-localize}]
Consider the functors
\[
\Entd(\w{X}) \longrightarrow \Ent(\w{X}) \longrightarrow \Ent(X)~.
\]
By Lemma \ref{strength}, both the functor $\Entd(\w{X}) \ra \Ent(\w{X})$ and the composite $\Entd(\w{X}) \ra \Ent(X)$ are localizations. By the asymmetric 2-of-3 property of localizations, this implies that the functor $\Ent(\w{X}) \ra \Ent(X)$ is a localization.

\end{proof}

\section{Categories of basics}
Recall from the introduction the following.
\begin{definition}\label{def.disc-B-mfld}
An \emph{$\infty$-category of basics} is a right fibration $\cB\to \bsc$.  
Given an $\infty$-category of basics $\cB$, the $\infty$-category of $\cB$-manifolds is the pullback
\[
\mfld(\cB)~:=~ \snglr\underset{\RFbn_{\bsc}}\times \bigl((\RFbn_{\bsc})_{/\cB}\bigr)~,
\]
and there is the further pullback $\infty$-category
We will use the notation for the pullback
\[
\mfldd(\cB)~:=~\snglrd\underset{\snglr}\times \mfld(\cB)~.
\]
\end{definition}

\begin{remark}
By definition, there is a fully faithful inclusion of $\infty$-categories, $\cB \into \mfld(\cB)$. The intuition is as follows: In examples, an object $\ov{U}\in \cB$ can be thought of as an object $U \in \bsc$ together with some structure $g$. Meanwhile, an object of $\mfld(\cB)$ is specified by a natural transformation from the representable functor $\snglr(-,U)$ to the presheaf defined by $\cB$. This natural transformation should be thought of as pulling back structures: given any morphism $f: U' \to U$, one obtains a pullback structure $f^* g \in \cB(U')$.
\end{remark}

We will now give examples of categories of basics. 
As we mentioned in the introduction, any $\infty$-category of basics $\cB \to \bsc$ factors as an essentially surjective functor followed by a fully faithful functor. The latter aspect of this factorization corresponds to a class of singularity types, and the former to tangential structures on these singularity types.  
\subsection{Specifying singularity type}

\begin{definition}\label{def.sieve}
Let $\sC$ be a $\Kan$-enriched category. 
A full subcategory $\sL\subset \sC$ is a \emph{sieve} if given any objects $c\in \sC$ and $d\in \sL$ for which $\hom_{\sC}(c,d)$ is nonempty, then the object $c$ is contained in $\sL$.
\end{definition}

Let $\sB\subset \bsc$ be a sieve.  Let $X$ be a conically smooth stratified space.  Denote by $X_\sB\subset X$ the sub-stratified space which is the open subspace 
	\[
	X_\sB = 
	\bigcup_{\sB \ni U\xra{\varphi} X }
	\varphi(U)\subset X
	\] 
where the union is taken over the set of conically smooth open embeddings from objects of $\sB$.
Because $\sB$ is a sieve, the collection $\{U\hookrightarrow X\mid U\in \sB\}$ is a basis for the topology of $X_{\sB}$.  
Conversely, given a conically smooth stratified space $X$, the full subcategory $\sB_X\subset \bsc$ consisting of those $U$ for which $\snglr(U,X)\neq \emptyset$, is a sieve.  
\begin{remark}
It is useful to think of a sieve $\sB\subset \bsc$ as a list of singularity types, this list being \emph{stable} in the sense that a singularity type of an arbitrarily small deleted neighborhood of a point in a member of this list is again a member of the list.  
\end{remark}

\begin{definition}[$\sB$-manifolds]\label{def:sieve}
Let $\sB\subset \bsc$ be a sieve.  
A $\sB$-manifold is a conically smooth stratified space for which $X_\sB\into X$ is an isomorphism.  Equivalently, a conically smooth stratified space $X$ is a $\sB$-manifold if $\snglr(U,X) = \emptyset $ whenever $U\notin \ob \sB$.
\end{definition}

\begin{example}
Let $-1\leq k \leq n$.
The inclusion $\bsc_{\leq k,n}\subset \bsc$ is a sieve.
A $\bsc_{\leq k, n}$-manifold is a conically smooth stratified space of pure dimension $n$ and depth at most $k$.  
We point out that the inclusion $\bsc_{k,n} \subset \bsc$ is \emph{not} a sieve whenever $k>0$.
\end{example}

\begin{example}
Let $\sB\subset \bsc_1$ be the full subcategory spanned by $\RR$ and $\sC(\{1,2,3\})$.  Then $\sB$ is a sieve and a $\sB$-manifold is a (possibly open) graph whose vertices (if any) are exactly trivalent.  
\end{example}

\begin{example}[Boundaries and permutable corners]\label{boundary}
Let $\sD_n^\partial\subset \bsc$ be the full subcategory spanned by the two objects $\RR^n$ and $\RR^{n-1}\times \sC(\ast)$.
Then $\sD_n^\partial$ is a sieve and a $\sD_n^\partial$-manifold is precisely a smooth $n$-manifold with boundary.

As a related example, let $\sD_n^{\un{\partial}}\subset \bsc$ be the full subcategory spanned by the objects $\{\RR^{n-k}\times \sC(\Delta^{k-1})\}_{0\leq k \leq n}$ where it is understood that $\Delta^{-1}=\emptyset^{-1}$. (See Example~\ref{example.simplices} to understand $\Delta^{k-1}$ as a stratified space.) 
By induction, conically smooth open embeddings $\RR^{n-k}\times \sC(\Delta^{k-1}) \hookrightarrow \Delta^n$ form a basis for its topology.  It follows that $\sD_n^{\un{\partial}}\subset \bsc$ is a sieve.  
A $\sD_n^{\un{\partial}}$-manifold is what we called an $n$-manifold with \emph{permutable corners} in Example~\ref{example.perm.corners}.
\end{example}

\begin{example}[Embedded submanifolds]\label{example:Edn'}
Here we follow up on Example~\ref{example.embedded} where we stated that a properly embedded submanifold $L^d \subset M^n$ of a smooth $n$-manifold can be regarded as a conically smooth stratified space of dimension $n$, whose locus of positive depth is precisely $L$.

Let $\sD_{d\subset n}^{\kink}\subset \bsc$ be the sieve with the two objects $\{\RR^n~,~\RR^{n-d}\times \sC(S^{n-d-1})\}$, the second of which we will write as $\RR^{d\subset n}$ for now.  
The morphism spaces (actually, Kan complexes) are as follows:

\begin{itemize}
\item $\sD_{d\subset n}^{\kink}(\RR^n,\RR^n) = \Emb(\RR^n,\RR^n)$ -- the space of smooth embeddings.  This space is a monoid and receives a homomorphism from $\sO(n)$ which is an equivalence of underlying spaces.
\item $\sD_{d\subset n}^{\kink}(\RR^{d\subset n},\RR^n) = \emptyset $.
\item $\sD_{d\subset n}^{\kink}(\RR^n,\RR^{d\subset n}) = \Emb(\RR^n,\RR^n \smallsetminus \RR^d)$ -- the space of smooth embeddings which miss the standard embedding $\RR^d\times\{0\}\subset \RR^n$.  This space receives a map from $\sO(n)\times S^{n-d-1}$ as a $\bigl(\sO(n),\sO(d)\times \diff(S^{n-d-1})\bigr)$-bimodule, and the map is an equivalence of underlying spaces.  
\item Cylindrical coordinates $\RR^d\times \RR_{\geq 0}\times S^{n-d-1}\xra{\sf Cyl} \RR^n$ induce a homeomorphism of underlying topological spaces $\RR^{d\subset n}\approx \RR^n$ -- this homeomorphism canonically lifts to a $\PL$ isomorphism, but not a smooth isomorphism.  
The morphism space $\sD_{d\subset n}^{\kink}(\RR^{d\subset n},\RR^{d\subset n})$ is the space of those continuous open embeddings $f\colon \RR^n \to \RR^n$ for which the diagram
\[
\xymatrix{
\RR^d\times \RR_{\geq 0} \times S^{n-d-1}  \ar[d]^{\sf Cyl}  \ar@{-->}[r]^{\w{f}}
&
\RR^d\times \RR_{\geq 0}\times S^{n-d-1}  \ar[d]^{\sf Cyl}
\\
\RR^n  \ar[r]^f  
&
\RR^n
}
\]
can be filled with $\w{f}$ a smooth map between manifolds with boundary.  
This space is a monoid and receives a homomorphism from $\sO(d)\times \Aut(S^{n-d-1})$ which is an equivalence of underlying spaces.  
\end{itemize}
Write $\sD_n^{{\PL}}$ for the $\sf Kan$-enriched category with the single object $\RR^n_{\PL}$ and whose morphism space is $\Emb^{\PL}(\RR^n,\RR^n)$ -- the space of $\sf PL$ self-embeddings of $\RR^n$.  
Notice the evident enriched functor 
\[
\sD_{d\subset n}^{\kink} \longrightarrow \sD_n^{\PL}~.
\]  
\noindent
A $\sD_{d\subset n}^{\kink}$-manifold is the data of 
\begin{itemize}
\item a smooth $n$-manifold $\w{M}$ with boundary $\partial \w{M}$,
\item a smooth $d$-manifold $L$,
\item a smooth fiber bundle $\partial \w{M} \to L$ whose fibers are diffeomorphic to a sphere.
\end{itemize}
Such data in particular determines the pushout topological space $M := L \sqcup_{\partial \w{M}} \w{M}$.
This topological space $M$ is a topological $n$-manifold, and, directly from its defining expression, it is equipped with a canonical $\sf PL$ structure.
Moreover, this $\sf PL$ manifold $M$ is equipped with a properly embedded \emph{smooth} submanifold $L\subset M$ as well as a smooth structure on $M\smallsetminus L$, in addition to further smoothness along $L$ which, among other things, provides a tubular neighborhood of $L$ and thus a smooth link about $L$.  
We refer to a $\sD_{d\subset n}^{\kink}$-manifold as a \emph{kink submanifold $L\subset M$}.  
Examples of such come from properly embedded smooth $d$-manifolds in a smooth $n$-manifold.
Not all $\sD_{d\subset n}^{\kink}$-manifolds are isomorphic to ones of this form -- this difference will be addressed as Example~\ref{example:Ekn}.  
\end{example}

\subsection{Examples of $\oo$-categories of basics}

By construction, there is a fully faithful functor $\cB \to \mfld(\cB)$.
There results restricted Yoneda functor
\begin{equation}\label{B-tangents}
\tau \colon \mfld(\cB) \longrightarrow \Psh\bigl(\mfld(\cB)\bigr) \longrightarrow \Psh(\cB)\underset{\rm Rec~\ref{right-fibration}}\simeq \RFbn_{\cB}~.
\end{equation}

\begin{example}
Let $\Theta\colon \bsc^{\op} \to \Kan$ be a fibrant functor between $\Kan$-enriched categories. The unstraightening construction applied to the simplicial nerve of $\Theta$ gives a right fibration 
	\[
	\cB_\Theta \to \bsc
	\]
whose fiber over $U$ is a Kan complex canonically equivalent to $\Theta(U)$.  
\end{example}

\begin{example}\label{over-X}
Let $X$ be a conically smooth stratified space.  Then $\Ent(X) \to \bsc$ is an $\infty$-category of basics.  The data of a stratified space $Y$ together with a conically smooth open embedding $Y\xra{f} X$ determines the $\Ent(X)$-manifold $\bigl(\Ent(Y),\Ent(f)\bigr)$.  
In general, it is not the case that every $\Ent(X)$-manifold arises in this way. 
\end{example}

\begin{definition}\label{def.coherent-sieve}
Let $\cC$ be an $\infty$-category. An $\infty$-subcategory $\cL \subset \cC$ is a \emph{sieve} if the inclusion $\cL \to \cC$ is a right fibration.  Equivalently, $\cL\subset \cC$ is a sieve if it is a full $\oo$-subcategory and for each $c\in \cL$ the space $\cC(c',c)\neq\emptyset$ being nonempty implies that $c'$ is contained in $\cL$.
\end{definition}

Note that an full $\infty$-subcategory $\cL \subset \cC$ being a sieve is a condition on its objects.

\begin{remark}
The simplicial nerve of a sub-$\Kan$-enriched category $\sL \subset \sC$ is a sieve in the sense of Definition~\ref{def.sieve} if and only if it is a sieve as $\infty$-categories. 
\end{remark}

\begin{example}[$\sD_n$]\label{example:D_n}
Define the $\infty$-category of basics 
	\[
	\sD_n\to \bsc
	\]
as the simplicial nerve of $\bsc_{0,n}$.  
A $\sD_n$-manifold is precisely an ordinary smooth manifold.  
\end{example}

\begin{example}[Framed 1-manifolds with boundary]\label{I}
Consider the sieve $\cI'\subset \bsc_{\leq 1, 1}$ whose set of objects is $\{\RR,C(\ast)\}$.  We point out that $\sC(\ast) = \RR_{\geq 0}$.
Define by $\cI\to \cI'$ the (unique) right fibration whose fiber over $\RR$ is a point $\{\RR\}$ and whose fiber over $\RR_{\geq 0}$ is two points $\{\RR_{\geq 0}\} \sqcup \{\RR_{\leq 0}\}$.
As the notation suggests, the space of morphisms between two objects of $\cI$ is the space of smooth open embedding \emph{which preserve orientation}.  
An $\cI$-manifold is an oriented smooth $1$-manifold with boundary.
\end{example}

\begin{remark}\label{strict-I}
We point out that the presented definition of $\cI$ describes it as a $\Kan$-enriched category (with three objects and with explicit morphism spaces).  
\end{remark}

\begin{example}[Framed manifolds]\label{fr}
Denote the right fibration  $\sD_n^{\fr} := (\sD_n)_{/\RR^n} \to \sD_n$.  
This is the unstraightening the topological functor $\sD_n^{\op} \to \Kan$ given by $\sD_n(-,\RR^n)$. For example, its value on $\RR^n$ (the only object) is $\Emb(\RR^n,\RR^n) \simeq \sO(n)$.  As so, it is useful to regard a vertex of $\sD_n^{\fr}$ as a framing on $\RR^n$, that is, a trivialization of the tangent bundle of $\RR^n$.  Because $\sD_n$ is an $\infty$-groupoid, the over-$\infty$-category $\sD_n^{\fr} \simeq \ast$ is equivalent to $\ast$. 
A $\sD_n^{\fr}$-manifold is a smooth manifold together with a choice of trivialization of its tangent bundle.  A morphism between two is a smooth embedding together with a path of trivializations from the given one on the domain to the pullback trivialization of the target.  
Composition is given by composing smooth embeddings and concatenating paths.  
\end{example}

\begin{example}
Let $G\xra{\rho} {\sf GL}(\RR^n)$ be a map of topological groups.  There results a Kan fibration between Kan complexes ${\sB}G \to {\sf BGL}(\RR^n)$.  
Set $D^G_n := \sB G$ and define
the right fibration $\sD_n^G \to \sD_n$ through the equivalence of Kan complexes ${\sf BGL}(\RR^n) \xra{\simeq} \sD_n$.
A $\sD_n^G$-manifold is a smooth manifold with a (homotopy coherent) $G$-structure on the fibers of its tangent bundle.  A morphism of $\sD^G_n$-manifolds is a smooth embedding together with a path from the fiberwise $G$-structure on the domain to the pullback $G$-structure under the embedding.  
Examples of such a continuous homomorphism are the standard maps from ${\sf Spin}(n)$, $\sO(n)$, and ${\sf SO}(n)$.  The case $\ast \to {\sf GL}(\RR^n)$ of the inclusion of the identity subgroup gives the category of basis $\sD_n^{\fr}$ of Example~\ref{fr}. 
\end{example}

\begin{example}[Submanifolds]\label{example:Ekn}
Recall the sieve $\sD_{d\subset n}^{\mathsf{Kink}}$ of Example~\ref{example:Edn'}.  
There was a map $\sD_{d\subset n}^{\sf Kink} \to \sD_n^{\PL}$ to the monoid of $\sf PL$ self-embeddings of $\RR^n$.
Also note the standard map $\sD_n \to \sD_n^{\PL}$.    
Define the $\infty$-category $\sD_{d\subset n}$ as the pullback
\[
\xymatrix{
\sD_{d\subset n}  \ar[r]  \ar[d]
&
\sD_n  \ar[d]
\\
\sD_{d\subset n}^{\sf Kink}  \ar[r]
&
\sD_n^{\PL}
}
\]
though we warn that the horizontal maps do not lie over $\bsc$.  
Nevertheless, $\sD_{d\subset n} \to \sD_{d\subset n}^{\sf Kink}$ is a right fibration, and therefore $\sD_{d\subset n}$ is an $\infty$-category of basics.  
We describe $\sD_{d\subset n}$ explicitly.  
The map $\sD_{d\subset n} \to \sD_{d\subset n}^{\sf Kink}$ is an equivalence on the first three of the four mapping spaces (according to the order presented in Example~\ref{example:Edn'}).  
To identify the remaining space of morphisms, we explain the homomorphisms 
\[
\sO(d) \times \sO(n-d) \xra{~\simeq~}\Emb\bigl((\RR^d\subset \RR^n),(\RR^d\subset \RR^n)\bigr) \xra{~\simeq~}\sD_{d\subset n}(\RR^{d\subset n},\RR^{d\subset n})
\]
which are equivalences of spaces.
The middle space consists of those smooth self-embeddings of $\RR^n$ which restrict to self-embeddings of $\RR^d$ and $\RR^n\smallsetminus \RR^d$.
The first map is the evident one, and it is an equivalence of spaces with homotopy inverse given by taking the derivative at the origin $0\in \RR^n$.  
The second map is the evident one, and it is an equivalence of spaces via smoothing theory.   
Using smoothing theory~(\`a la \cite{kirbysieb}), a $\sD_{d\subset n}$-manifold is a smooth manifold equipped with a properly embedded smooth $d$-manifold. 
Provided $n-d\leq 3$, the inclusion $\sO(n-d)\xra{\simeq} \Diff(S^{n-d-1})$ is a homotopy equivalence (which is Smale's conjecture for the case $n-d=3$) and thus the map $\sD_{d\subset n}\xra{\simeq} \sD_{d\subset n}^{\mathsf{Kink}}$ is an equivalence of $\infty$-categories.  For this range of $d$ then, there is no distinction between a $\sD_{d\subset n}$-manifold and a $\sD_{d\subset n}^{\mathsf{Kink}}$-manifold.  The case $(n,d) = (3,1)$ is of particular interest .  
\end{example}

\begin{example}
Let $A$ be a set whose elements we call colors.  Denote by $\sD_{d\subset n}^A\to \sD_{d\subset n}$ the right fibration whose fiber over $\RR^n$ is a point and whose fiber over $\RR^{d\subset n}$ is the set $A$.  
A $\sD_{d\subset n}^A$-manifold is a collection $\{L_\alpha\}_{\alpha\in A}$ of pairwise disjoint properly embedded smooth $d$-submanifolds of a smooth manifold $M$.

As a related example, let $S\subset A\times A$ be a subset and denote by $\sD_n^S$ the $\infty$-category of basics over the objects $\RR^n$ and $\RR^{n-1\subset n}$ whose fiber over the first object is the set $A$ of colors, and whose fiber over the second object is $S$.    The edges (and higher simplices) are given from the two projections $S \to A$.  A $\sD_n^S$-manifold is a smooth manifold together with a hypersurface whose complement is labeled by $A$, and the colors of two adjacent components of this complement are specified by $S$.
We think of such a geometric object as a \emph{defect} of dimension $n$ and indexed by $S$.  
\end{example}

\begin{example}[Framed submanifolds]\label{Ekn-framed}
Recall the $\infty$-category of basics $\sD_{d\subset n}$ of Example~\ref{example:Ekn}.
Define the $\infty$-category $\sD_{d\subset n}^{\fr}$ as the pullback in the diagram of $\infty$-categories
\[
\xymatrix{
\sD_{d\subset n}^{\fr}  \ar[r]  \ar[d]  
&
\sD_n^{\fr}  \ar[d]
\\
\sD_{d\subset n} \ar[r]
&
\sD_n
}
\]
-- though we warn that the horizontal maps do not lie over $\bsc$.  
Nevertheless, both of the projections $\sD_{d\subset n}^{\fr} \to \sD_{d\subset n} \to \bsc$ are right fibrations, and therefore $\sD_{d\subset n}^{\fr}$ is an $\infty$-category of basics.  
 
Observe that $\sD_n^{\fr}$ embeds into $\sD_{d\subset n}^{\fr}$ as the fiber over $\RR^n$.  The fiber of $\sD_{d\subset n}^{\fr}$ over the other vertex $\RR^{d\subset n}$ is what one can justifiably name the Kan complex of \emph{framings} of $\RR^{d\subset n}$.  Via~\textsection\ref{sec:endomorphisms}, this terminology is justified through the identification of the fiber over $\RR^{d\subset n}$ being canonically equivalent to the product $\sO(k) \times \Aut(S^{n-k-1})$ which is the Kan complex of choices of `trivializations of the tangent stalk of $\RR^{d\subset n}$ at it center'.    
Consider the map of $\infty$-categories $\sD_{d\subset n} \to \Delta^1$ determined by $\{\RR^n\}\mapsto 0$, $\{\RR^{d\subset n}\}\mapsto 1$.  The composite map $\sD_{d\subset n}^{\fr} \to \sD_{d\subset n} \to \Delta^1$ is an equivalence of $\infty$-categories.  This is analogous to (and restricts to) the equivalence $\sD_n^{\fr} \simeq \ast$~.  
A $\sD_{d\subset n}^{\fr}$-manifold is a framed manifold $M$ and a properly embedded smooth $d$-submanifold $L$ equipped with a null-homotopy of the Gauss map of this submanifold: $L \to {\sf Gr}_d(\RR^n)$ -- this null-homotopy is equivalent to a trivialization of both the tangent and normal bundle of $L$, compatible with the trivialization of $M$.  
\end{example}

\begin{example}[Specified intersections] 
Fix a smooth embedding $e\colon S^{n-k-1}\sqcup S^{n-l-1} \hookrightarrow S^{n-1}$.  Regard the datum of $e$ as a singular $(n-1)$-manifold, again called $e$, whose underlying space is $S^{n-1}$ with singularity locus the image of $e$.  
Consider the sieve $\sD_{k,l\subset_e n}'\to \bsc$ whose set of objects is $\{\RR^n, \RR^{k\subset n}, \RR^{l\subset n}, \sC(e)\}$.  
As with Example~\ref{example:Ekn}, there is a functor $\sD_{k,l\subset_e n}' \to \sD_n^{\PL}$.
Define the $\infty$-category of basics $\sD_{k,l\subset_e n} : = \sD_{k,l\subset_e n}' \underset{\sD_n^{\PL}}\times \sD_n$.  
A $\sD_{k,l\subset_e n}$-manifold is a smooth manifold $M$ together with a pair of properly embedded smooth submanifolds $K,L\subset M$ of dimensions $k$ and $l$, respectively, whose intersection locus is discrete and of the form specified by $e$.  
As a particular example, if $k+l=n$ and $e$ is the standard embedding into the join, then a $\sD_{k,l\subset_e n}$-manifold is a pair of submanifolds (of dimensions $k$ and $l$) of a smooth manifold which intersect transversely as a discrete subset. 
\end{example}

\begin{example}[Manifolds with corners]\label{example.mfld-corners}
This example will be foreshadow of \S\ref{section.corners}, where \emph{singular} manifolds with corners is developed.
Let $I$ be a finite set.
Consider the poset $\cP(I)$ of subsets of $I$, ordered by reverse inclusion.
For each $J\subset I$, consider the continuous map $\RR_{>0}^{I\smallsetminus J}\times \RR_{\geq 0}^I \to \cP(I)$ given by $(I\xra{x}\RR_{\geq 0})\mapsto \{i\mid x_i = 0\}$ -- it is a conically smooth stratified space, and as so it is (non-canonically) isomorphic to a basic.

Consider the $\Kan$-enriched category $\sD_{\langle n \rangle} \to \bsc$ for which an object is a finite subset $J \subset \{1,\dots,n\}=:\un{n}$, and for which the Kan complex of morphisms from $J$ to $J'$ is that of conically smooth open embedding $\RR_{>0}^{\un{n}\smallsetminus J} \times \RR_{\geq 0}^J \hookrightarrow \RR_{>0}^{\un{n}\smallsetminus J'}\times \RR_{\geq 0}^{J'}$ over the poset $\cP(\un{n})$.  
There is an evident functor $\sD_{\langle n\rangle} \to \bsc$.  
In \S\ref{section.corners} it is explained that this functor is a right fibration, and so $\sD_{\langle n \rangle}$ is an $\infty$-category of basics.  
A $\sD_{\lag n \rag}$-manifold is a smooth $n$-manifold corners, in the sense of~\cite{laures}.  
\end{example}

\begin{construction}[Stratum-wise structures]\label{stratum-wise-const}
We give a brief general account of \emph{independent} structures -- these structures are simply structures on each stratum of a stratified space separately, requiring no compatibility.  
The data required to specify independent structures is far simpler than for general structures.  
Consider the maximal $\oo$-subgroupoid $\bsc^{\sf eq}\subset \bsc$.  
From Theorem~\ref{basic-facts}, it can be described as the $\infty$-subcategory of depth-preserving morphisms.  
Let $\cE^0\to \bsc^{\sf eq}$ be any Kan fibration.  
We now construct a right fibration $\cE \to \bsc$ -- this construction is completely formal.

We define the map of simplicial sets $\cE\to \bsc$ by defining for each $p\geq 0$ and each $\Delta[p] \xra{\sigma} \bsc$ a the set $\cE_\sigma[p]$ of $p$-simplices in $\cE$ over $\sigma$; that these sets assemble to define a right fibration among simplicial sets is clear from the construction.
For $p=0$, define $\cE[0] = \cE^0_\sigma[0]$ using that $\bsc^{\sf eq}[0] = \bsc[0]$.  
Assume $\cE_{\sigma'}[p']$ has been defined whenever $p'<p$. Let $0\leq k \leq p$ be maximal such that the map $\Delta\{0,\dots,k\} \xra{\sigma_|} \bsc$ factors through $\bsc^{\sf eq}$.  
If $k=p$ define $\cE_\sigma[p] = \cE^0_\sigma[p]$.
If $k<p$ define $\cE_\sigma[p] = \cE_{\sigma_{|\Delta\{k+1,\dots,p\}}}[p-k-1]$.  
\end{construction}

\clearpage

\part{Structures are sheaves are 1-excisive functors}

Now we are equipped to give the proofs of the main theorems from the introduction:
\begin{itemize}
\item Theorem~\ref{theorem.smooth-basic-sheaf}: presheaves on $\bsc$ are equivalent to sheaves on $\snglr$.
\item Theorem~\ref{theorem.exit-tangent}: sheaves relative to $X$ are equivalent to constructible sheaves on $X$.
\item Theorem~\ref{theorem.structured-versions} generalizes the previous two theorems to the case of conically smooth stratified spaces with arbitrary structure $\cB$.
\item Theorem~\ref{theorem.1-excision}: sheaves on $\cB$-manifolds are equivalent to 1-excisive functors, and {\em covariant} functors out of $\cB$ (i.e., co-presheaves) are precisely the $\amalg$-excisive functors on $\mfld(\cB)$.
\end{itemize}

\section{Sheaves and constructible sheaves}\label{intuitions}
The following definition is from~\cite{toen-vezzosi-segal-topoi}, and subsequently used in~\cite{HTT} (cf. Definition~\ref{def.sieve}).
\begin{defn}
For $\cC$ an $\infty$-category, a \emph{sieve} is a fully faithful right fibration $\cU \to \cC$.  
For $X\in \cC$, a \emph{sieve on $X$} is a sieve $\cU_X\to \cC_{/X}$ for the slice $\infty$-category.  
\end{defn}

We now make concrete what we mean by sheaf on the site of conically smooth stratified spaces. To this end, consider the discrete category of conically smooth stratified spaces, $\snglrd$.
For every stratified space $X$, there is a functor 
\[
 {\sf Image}\colon \snglrd_{/X} \to \opens(X),
 \qquad
 (\varphi: U \into X) \mapsto \varphi(U)
\]
which maps the category of open embeddings $U \into X$ to the poset of open subsets $O \subset X$, simply by taking the image of an open embedding.
Because objects of $\snglrd$ are $C^0$ stratified spaces equipped with a section of the sheaf ${\sf Sm}$, the functor ${\sf Image}$ is a Cartesian fibration whose fibers are equivalent to the terminal category. (The functor, in particular, is an equivalence.)
And so, there is a \emph{standard} Grothendieck topology on $\snglrd$ where a sieve $\cU\subset \snglrd_{/X}$ on a conically smooth stratified space $X$ is a covering sieve if its image in $\opens(X)$ is a covering sieve in the standard sense. For the following definition, note that a covering sieve $\cU \ra \snglrd_{/X}$ is adjoint to a functor $\cU^{\triangleright} \ra \snglrd$, where $\cU^{\tr}:=\cU\times[1]\amalg_{\cU\times\{1\}}\{1\}$ is the right-cone on $\cU$ and the functor assigns $X$ to the adjoined final object.

\begin{defn}
The full $\oo$-subcategory
\[
\Shv(\snglrd)~\subset~ \Psh(\snglrd)
\]
of \emph{sheaves} on $\snglrd$ consists of those functors $\cF\colon \snglrd^{\op} \to \spaces$ such that, for each covering sieve $\cU \subset \snglrd_{/X}$, the composition $(\cU^{\op})^\tl \to \snglrd^{\op} \xra{\cF} \spaces$ is a limit diagram. In particular, the canonical map of spaces $\cF(X) \xra{\simeq} \underset{U\in \cU}{\sf lim}\cF(U)$ is an equivalence.  
More generally, a contravariant functor $\cF: \snglrd^{\op} \to \cC$ to an $\infty$-category $\cC$ is called a $\cC$-valued sheaf if the composition $(\cU^{\op})^\tl \to \snglrd^{\op} \xra{\cF} \cC$ is a limit diagram in $\cC$.
\end{defn}

\begin{remark}
Because a conically smooth stratified space is, by definition, finite dimensional, a sheaf in the above sense will satisfy the the stronger condition of descent for hypercovers (Definition 6.5.3.2 of \cite{HTT} or Definition 4.1 of \cite{dugger-isaksen}). See \S6.5.4 and Corollary 7.2.1.17 of \cite{HTT}. Descent versus hyperdescent may diverge for infinite dimensional spaces, such as the unbounded Ran space $\Ran(X)$ (Definition \ref{def.ran}).
\end{remark}

\begin{defn}
\label{definition.sheaf}
The $\infty$-category of \emph{sheaves on $\snglr$} is the pullback in the diagram among $\infty$-categories
\begin{equation}\label{sheaf-square}
\xymatrix{
\Shv(\snglr)  \ar[r]  \ar[d]
&
\Psh(\snglr)  \ar[d]
\\
\Shv(\snglrd)  \ar[r]
&
\Psh(\snglrd)~.
}
\end{equation}
So an object of $\Shv(\snglr)$ is a functor $\snglr^{\op} \xra{\cF}\spaces$ for which the composition $\snglrd^{\op} \to \snglr^{\op} \xra{\cF} \spaces$ is a sheaf.  
An analogous pullback defines the category of $\cC$-valued sheaves on $\Snglr$.
\end{defn}

\begin{defn}[Constructible sheaves]\label{def.cbl-sheaf}
Let $X \xra{S} P$ be a stratified topological space.  
The $\infty$-category of ($S$-)constructible sheaves on $X$ is the full $\infty$-subcategory
\[
\Shv^{\sf cbl}(X) \subset \Shv(X)
\]
of those $\cF\colon \opens(X)^{\op}\to \spaces$ where, for each $p\in P$, the restriction $\cF_{|S^{-1}p}\colon \opens(S^{-1}p)^{\op} \to \spaces$ is a \emph{locally constant} sheaf (see Definition A.1.12 of~\cite{HA}.)
\end{defn} 

We record a variation of Definition~\ref{definition.sheaf} that accounts for $\cB$-manifolds.
To present this definition we make a couple notational observations.  
The natural projection $\mfld(\cB)\to \Snglr$ gives an adjunction $ \Psh(\Snglr) \rightleftarrows \Psh\bigl(\mfld(\cB)\bigr)$ given by restriction and right Kan extension.  
Because $\mfld(\cB) \to \Snglr$ is a right fibration, each open cover $\cU$ of the underlying stratified space of a $\cB$-manifold $X$, which we regard as a poset by inclusion, canonically determines a functor $\cU^{\tr} \to \mfld(\cB)$.  
\begin{defn}
\label{definition.B-sheaf}
For $\cB$ an $\infty$-category of basics, the $\infty$-category of \emph{sheaves on $\mfld(\cB)$} is the pullback in the diagram among $\infty$-categories
\begin{equation}\label{sheaf-B}
\xymatrix{
\Shv\bigl(\mfld(\cB)\bigr)  \ar[r]  \ar[d]
&
\Psh\bigl(\mfld(\cB)\bigr) \ar[d]
\\
\Shv(\snglr)  \ar[r]
&
\Psh(\snglr)~.
}
\end{equation}
So an object of $\Shv\bigl(\mfld(\cB)\bigr)$ is a functor $\mfld(\cB)^{\op} \xra{\cF}\spaces$ for which, for each open cover $\cU$ of the underlying stratified space of a $\cB$-manifold $X$, the composition $(\cU^{\op})^{\tl} \to \mfld(\cB)^{\op} \xra{\cF} \spaces$ is a limit diagram.  
An analogous pullback defines the category of $\cC$-valued sheaves on $\mfld(\cB)$.
\end{defn}

\subsection{Constructible sheaves and $\Ent(X)$: 
proof of Theorem~\ref{theorem.exit-tangent}(Cbl)}\label{sec.proof-cstbl}
Fix a conically smooth stratified space $X$. 
Recall the enter-path $\infty$-category $\Ent(X)$ from Definition~\ref{defn.E(X)} -- it comes with a natural right fibration to $\bsc$, whose fiber over $U$ is the space of embeddings from $U$ to $X$.  Also recall the $\infty$-category of constructible sheaves $\shv^{\cbl}(X)$ from Definition~\ref{def.cbl-sheaf}. We prove here the equivalence (Cbl) in Theorem~\ref{theorem.exit-tangent}:
\[
 \shv^{\cbl}(X) \simeq \psh(\Ent(X)).
\]
In a nutshell, we prove Theorem~\ref{theorem.exit-tangent}~(Cbl) by showing that both sides of the equivalence respect coverings, crossing with $\RR$, and applying the cone $C$.  
\subsubsection{Characterizing the tangent classifier}
For each sieve $\cU \subset \snglrd_{/X}$, we have a composite functor 
 \[
 \cU \subset \snglrd_{/X} \to \snglr_{/X} \xra{\Ent} \Psh(\bsc)_{/\Ent(X)}
 \] 
and hence a map
\begin{equation}\label{E-sheaf-map}
\underset{O\in \cU} \colim~ \Ent(O) \longrightarrow \Ent(X)~.
\end{equation}
 of right fibrations over $\bsc$.

\begin{lemma}\label{E-covers}
For each covering sieve $\cU \subset \snglrd_{/X}$ the map~(\ref{E-sheaf-map}) is an equivalence in $\RFbn_{\bsc}$.  
\end{lemma}

\begin{proof}
Let $\cU\subset \snglrd_{/X}$ be a covering sieve.
Because colimits of presheaves are taken objectwise, the map of right fibrations $\underset{O\in \cU} \colim \Ent(O) \to \Ent(X)$ is an equivalence if and only if the restriction of this map $\underset{O\in \cU} \colim \Ent_{[U]}(O) \to \Ent_{[U]}(X)$ is an equivalence of spaces for each isomorphism class of a basic $[U]$.  
Let $[U]$ be such an isomorphism class.
Through Lemma~\ref{exit-class}, this map of spaces is an equivalence if and only if the map $\underset{O\in \cU}\colim \Sing(O_{[U]})\to \Sing(X_{[U]})$ is an equivalence.
Because the collection $\{O_{[U]}\to X_{[U]}\}$ is a covering sieve, and because $X_{[U]}$ admits partitions of unity (Lemma~\ref{part-o-1}), the result follows from Proposition~4.1 of~\cite{segal-sheaves}. If one does not want to invoke the partitions of unity result, one can also appeal to Theorem~1.1 of~\cite{dugger-isaksen}.  
\end{proof}

Consider the composite functor
\begin{equation}\label{Psh-E}
\Psh\bigl(\Ent(-)\bigr)\colon \snglrd^{\op} \longrightarrow \snglr^{\op} \xra{~\tau~} \RFbn_{\bsc}^{\op} \xra{~\Psh~} (\Cat)^{\Psh(\bsc)/} \longrightarrow \Cat~.
\end{equation}
Here is a direct consequence of Lemma~\ref{E-covers}.  
\begin{cor}\label{psh-E-sheaf}
\label{cor.psh-E-sheaf}
The functor $\Psh\bigl(\Ent(-)\bigr)$ is a $\Cat$-valued sheaf.  
\end{cor}

Because both $\RR$ and $\RR_{\geq -\infty}$ are basics, Corollary~\ref{basic-products} tells us that the inclusion $\RR \into \RR_{\geq \infty}$ induces a natural transformation between the direct product functors $(\RR \times - ) \to (\RR_{\geq -\infty} \times -)$. Hence we have a diagram 

\begin{equation}\label{E-R-maps}
\xymatrix{
\Ent(-) \ar[r]  \ar[d]
&
\Ent(\RR\times -)   \ar[d]
\\
[1]\times \Ent(-)   \ar[r]
&
\Ent(\RR_{\geq -\infty}\times -)
}
\end{equation}
which is a diagram in the $\infty$-category of functors, $\fun(\snglr,\Cat)$. The lower-left corner is the product with the nerve of the poset $\{0 < 1\}$, and the right vertical arrow is the inclusion of $\{0\} \hookrightarrow [1]$.  
\begin{lemma}\label{E-R-invariance}
In the diagram~(\ref{E-R-maps}), the horizontal arrows are equivalences of $\Cat$-valued functors.
\end{lemma}

\begin{proof}
That $\Ent(-)\to \Ent(\RR\times -)$ is an equivalence follows from Lemma~\ref{exit-class}. 
For each morphism $U \xra{f} \RR_{\geq -\infty}\times X$ from a basic, $U_{|\RR}\xra{f_|} \RR\times X$ is again morphism from a basic.  
Let $\nu: \spaces \to \{0,1\}$ be an indicator function for nonempty spaces, so $\nu(X) = 0$ if $X=\emptyset$, while $\nu(X) = 1$ otherwise.
The assignment $(U\xra{f} \RR_{\geq -\infty}\times X) \mapsto \bigl(\nu\left(U_{|\{-\infty\}}\right), U_{|\RR}\xra{f_|} \RR\times X\bigr)$ is an inverse to the bottom horizontal functor.
\end{proof}

In what follows, let $\Ent(X)^\triangleright$ denote the $\infty$-category obtained by adjoining a terminal vertex to $\Ent(X)$. Put another way, it is the join of the simplicial set $\Ent(X)$ with a 0-simplex.

\begin{lemma}\label{psh-E-cones}
Let $X$ be a compact conically smooth stratified space.
There is an equivalence of $\infty$-categories
\[
\Ent(X)^\triangleright \xra{~\simeq~}\Ent\bigl(\sC(X)\bigr)~.
\]
\end{lemma}

\begin{proof}
We construct a functor $\Ent(\RR_{\geq -\infty}\times -)\longrightarrow \Ent\bigl(\sC(X)\bigr)$ by declaring it to extend the functor $\Ent(\RR\times X) \to \Ent\bigl(\sC(X)\bigr)$ induced from the standard inclusion, and to be the constant functor $\{\sC(X)\xra{=}\sC(X)\}$ elsewhere -- this indeed defines a functor, in light of Lemma~\ref{E-R-invariance}.  
By construction, this functor factors as $\Ent(X)^\triangleright \to \Ent\bigl(\sC(X)\bigr)$ as in the statement of the lemma.  
This functor restricts as the equivalence $\Ent(X) \xra{\simeq} \Ent(\RR\times X)$ of Lemma~\ref{E-R-invariance}.  
After Corollary~\ref{cor.max-groupoid-bsc}, on maximal $\oo$-subgroupoids this functor is the map of spaces
\[
\ast \amalg \underset{[U]}\coprod \Ent_{[U]}(X) \to  \Ent_{[\sC(X)]}\bigl(\sC(X)\bigr) \amalg \underset{[U]}\coprod \Ent_{[U]}(\RR\times X)
\]
that lies over the map of indexing sets $[U]\mapsto [\RR\times U]$.  
It remains to show this functor induces an equivalence between mapping homotopy types whose source or target is not a cofactor in the indexed coproduct.  
But these mapping spaces are consistently either empty or terminal, by inspection. 
\end{proof}

\subsubsection{Characterizing constructible sheaves}
Any $\infty$-category has a Yoneda embedding $\cC \to \psh(\cC)$. By the Grothendieck construction (\ref{eqn.right-fibrations}), one obtains a functor $\cC_{/-}: \cC \to \RFbn_{\cC}$, which sends an object $X$ to the right fibration $\cC_{/X} \to \cC$. 
Let $\cC = \snglrd$ and consider the composite functor 
\[
\Psh\bigl(\snglrd_{/-}\bigr) \colon \snglrd^{\op} \xra{~\snglrd_{/-}~} (\RFbn_{\snglrd})^{\op} \xra{~\Psh~} (\Cat)_{\Psh\bigl(\snglrd\bigr)/}\longrightarrow \Cat
\]
which sends an object $X$ to the category of presheaves on $\snglrd_{/X}$. We also have the functors $\shv(-)$ and $\shv^{\cbl} (-) : \snglrd^{\op} \to \Cat$, which take an object $X$ to the $\infty$-category of sheaves on $X$, and of constructible sheaves on $X$ (with respect to the stratification $X \to P$). Note that for every $X$, $\shv^{\cbl}(X) \subset \shv(X) \subset \psh(\snglrd_{/X})$ is a sequence of full subcategories. 

In Corollary~\ref{cor.psh-E-sheaf}, we showed that presheaves on $\Ent(-)$ form a sheaf. Now we show that categories of constructible sheaves also glue together:

\begin{lemma}\label{shv-sheaf}
The functor $\Shv^{\sf cbl}(-)$ is a $\Cat$-valued sheaf. 
\end{lemma}

\begin{proof}
This is formal.
Let $s\colon \cU \into \snglrd_{/X}$ be a covering sieve.
Restrictions give the diagram consisting of left adjoints among presentable $\infty$-categories
\[
\xymatrix{
\Psh(\snglrd_{/X})  \ar[rr]^-{s^\ast}  \ar[d]
&&
\Psh(\cU)  \ar[d]^-\simeq
\\
\underset{O\in \cU}{\sf lim} \Psh(\snglrd_{/O})  \ar[rr]^-{(s^\ast_{|O})_{O\in \cU}}  
&&
\underset{O\in \cU}{\sf lim} \Psh(\cU_{/O}) ~.
}
\]
The right vertical functor is an equivalence because $\underset{O\in \cU}\colim \, \cU_{/O} \xra{\simeq} \cU$ is an equivalence of $\infty$-categories.  
Because $s$ is a sieve, the bottom horizontal functor is an equivalence.  
Because $s$ is fully faithful, the top horizontal functor is a left adjoint of a localization.
Since the right adjoint $s_*$ is the right Kan extension functor, the local objects are precisely those presheaves $\cF$ on $\snglrd_{/X}$ for which 
\[
\cF(Y) \xra{\simeq} \underset{O\in \cU_{/Y}} {\sf lim} \cF(O)
\]
is an equivalence of spaces for any object $Y \in \snglrd_{/X}$. 
Hence the top horizontal arrow restricts to an equivalence $\shv(\snglrd_{/X}) \simeq \shv(\cU)$. Moreover, the $\infty$-category of constructible sheaves is a full subcategory of both of these $\infty$-categories, and a sheaf on $\snglrd_{/X}$  is constructible if and only if its restriction to $\cU$ is. So we see that the top horizontal functor restricts to an equivalence $\Shv^{\sf cbl}(X) \xra{\simeq} \Shv^{\sf cbl}(\cU)\simeq \underset{O\in \cU}{\sf lim} \Shv^{\sf cbl}(O)$.  
\end{proof}

Analogous to (\ref{E-R-maps}), there is the diagram in $\Fun(\snglrd^{\op},\Cat)$
\begin{equation}\label{shv-R's}
\xymatrix{
\Shv^{\sf cbl}(\RR_{\geq -\infty}\times -)  \ar[r]  \ar[d]
&
\Shv^{\sf cbl}(-)^{[1]}  \ar[d]^-{{\sf ev}_1}
\\
\Shv^{\sf cbl}(\RR\times -)  \ar[r] 
&
\Shv^{\sf cbl}(-)~.
}
\end{equation}
The left vertical functor is restriction along the natural transformation $\RR\times -\hookrightarrow \RR_{\geq -\infty}\times -$, which we also utilized in constructing (\ref{E-R-maps}).
The upper-right functor assigns $X$ to the functor $\infty$-category $\fun([1], \shv^{\cbl}(X))$, and the right vertical functor is evaluation at $1\in [1]$.
The bottom horizontal functor is pulling back along the natural transformation transformation $- \into \RR\times -$.
The top horizontal functor is restriction along the pair of natural transformations $- \into (\RR \into \RR_{\geq -\infty})\times -$.
The following is the analogue of Lemma~\ref{E-R-invariance} for constructible sheaves.

\begin{lemma}\label{Shv-R-invt}
The horizontal arrows in~(\ref{shv-R's}) are equivalences.
\end{lemma}

\begin{proof}
The result follows immediately after showing that the stratified maps $\RR\to \ast$ and $\RR_{\geq -\infty}\to [1]$ are a equivalences on constructible shapes.
That the first map is an equivalence on constructible shapes is obvious, because $\RR$ is contractible.  
That the second map is an equivalence on constructible shapes follows, for instance, because the functor among $\infty$-categories $\Sing_{[1]}(\RR_{\geq -\infty}) \xra{\simeq} \Sing_{[1]}([1])$ is essentially surjective and fully faithful, by simple inspection.
\end{proof}

Restriction along the stratified transformation $-\to \ast$ to the terminal stratified space gives the transformation of functors $\snglrd^{\op} \to \Cat$
\[
\Shv^{\sf cbl}(\ast) \longrightarrow \Shv^{\sf cbl}(-)
\]
from the constant functor at $\Shv^{\sf cbl}(\ast) = \spaces$.  
Denote the pullback $\infty$-category
\[
\xymatrix{
\Shv^{\sf cbl}(X)^{[1]_\ast}   \ar[r]  \ar[d]
&
\Shv^{\sf cbl}(X)^{[1]}  \ar[d]
\\
\spaces    \ar[r]
&
\Shv^{\sf cbl}(X)~.
}
\]
\begin{lemma}\label{Shv-cones}
Let $X$ be a compact conically smooth stratified space.
There is a canonical equivalence of $\infty$-categories
\[
\Shv^{\sf cbl}\bigl(\sC(X)\bigr)\xra{~\simeq~} \Shv^{\sf cbl}(X)^{[1]_\ast}~.
\]
\end{lemma}

\begin{proof}
Witness the conically stratified space $\sC(X) \cong \ast \underset{\{-\infty\}\times X} \coprod (\RR_{\geq -\infty}\times X)$ as the pushout among stratified spaces, each of which is conically stratified.
There results an equivalence $\Shv^{\sf cbl}\bigl(\sC(X)\bigr) \xra{\simeq} \Shv^{\sf cbl}\bigl(\ast) \underset{\Shv^{\sf cbl}(X)}\times \Shv^{\sf cbl}(\RR_{\geq -\infty}\times X)$.
Apply Lemma~\ref{Shv-R-invt}.
\end{proof}

\subsubsection{Comparing characterizations}

There is a composite transformation of functors $\snglrd^{\op} \to \Cat$
\[
\Gamma\colon \Psh\bigl(\Ent(-)\bigr) \xra{\iota_\ast} \Psh\bigl(\snglr_{/-}\bigr) \xra{|\snglrd} \Psh\bigl(\snglrd_{/-}\bigr)
\]
in which the transformation $\iota_\ast$ is given by right Kan extension along the transformation $\iota \colon \bsc_{/-} \to \snglr_{/-}$ over the constant transformation $\bsc \to \snglr$.  
Through the equivalence $\Psh\bigl(\Ent(-)\bigr) \simeq \RFbn_{\Ent(-)}$ of Theorem~\ref{right-fibration}, $\Gamma$ evaluates as
\[
\Gamma_X\colon \bigl(\cE\to \Ent(X)\bigr)\mapsto \Bigl( O\mapsto \Map_{\Ent(X)}\bigl(\Ent(O), \cE\bigr)\Bigr)~.
\]

\begin{lemma}\label{Gamma-compares}
There is a canonical factorization
\[
\Gamma\colon \Psh\bigl(\Ent(-)\bigr)\longrightarrow \Shv^{\sf cbl}(-)~\subset~\Psh(\snglrd_{/-})
\]
through the constructible sheaves.  
Furthermore, this factorization sends the square~$\Psh$(\ref{E-R-maps}) to the square~(\ref{shv-R's}).  
\end{lemma}

\begin{proof}
In this proof we will make use of Lemma~\ref{lemma.rfbn-over}: there is an equivalence of $\infty$-categories $(\RFbn_\cC)_{/\cE} \xla{\simeq} \RFbn_\cE$, where $\cE\to \cC$ is an arbitrary right fibration among small $\infty$-categories.

Let $X$ be a conically smooth stratified space.
Let us explain the diagram of $\infty$-categories
\[
\xymatrix{
\snglrd_{/X}  \ar[rr]^-{\Ent_{|\snglrd}}  \ar[dd]_{(-)_{[U]}}
&&
(\RFbn_{\bsc})_{/\Ent(X)}  \ar[d]^-{|\bsc_{[U]}}
\\
&&
(\RFbn_{\bsc_{[U]}})_{/\Ent_{[U]}(X)} 
\\
\snglrd_{/X_{[U]}}  \ar[rr]^-{\Ent_{|\snglrd}}
&&
(\RFbn_{\bsc_{[\RR^i]}})_{/\Ent(X_{[U]})}  \ar[u]^-{\simeq}_{(-)^\ast_{[U]}}
}
\]
-- that this diagram commutes will be manifest from its description.
The horizontal functors the composition of $\Ent$ followed by restriction along $\snglrd \to \snglr$.  
The top right vertical functor is given by pulling back along the fully faithful embedding $\bsc_{[U]}\to \bsc$.
The left vertical functor is the $[U]$-stratum functor, after Proposition~\ref{prop.stratification-functorial}. 
The bottom right vertical functor is restriction along the $[U]$-stratum functor.  
That this bottom right vertical functor is an equivalence follows from Lemma~\ref{exit-class}.

Now, that $\Gamma$ factors through $\Shv(-)$ is immediate after Lemma~\ref{E-covers}. 
From the above paragraph, we arrive at a commutative diagram
\[
\xymatrix{
\Shv(X)  \ar[dd] 
&&
\Psh\bigl(\Ent(X)\bigr) \ar[ll]_-\Gamma    \ar[d]
\\
&&
\Psh\bigl(\Ent_{[U]}(X)\bigr)  
\\
\Shv(X_{[U]})
&&
\Psh\bigl(\Ent(X_{[U]})\bigr) \ar[ll]_-\Gamma     \ar[u]_-{\simeq}~.
}
\]
By Lemma~\ref{exit-class}, the bottom horizontal functor factors through $\Shv^{\sf cbl}(X_{[U]}) \simeq \Psh\bigl(\Sing(X_{[U]})\bigr)$.
This proves that $\Gamma$ factors through $\Shv^{\sf cbl}(-)$.  
It is routine to verify that $\Gamma$ sends $\Psh(\ref{E-R-maps})$ to (\ref{shv-R's}).  
\end{proof}

We now prove Theorem~\ref{theorem.exit-tangent}(Cbl), which we restate here for the reader's convenience.
\begin{theorem}[Theorem~\ref{theorem.exit-tangent}(Cbl)]
For each conically smooth stratified space $X$ there is natural equivalences of $\infty$-categories
\[
\Shv^{\sf cbl}(X)~\underset{\rm Cbl}\simeq~\Psh\bigl(\Ent(X)\bigr).
\]
\end{theorem}

\begin{proof}[Proof of Theorem~\ref{exit-tangent}(Cbl)]
We show that $\Gamma_X\colon \Psh\bigl(\Ent(X)\bigr) \to \Shv^{\sf cbl}(X)$ is an equivalence.
Corollary~\ref{psh-E-sheaf} and Lemma~\ref{shv-sheaf} give that the domain and codomain of $\Gamma$ are sheaves, so we can assume $X$ has bounded depth.  
After Lemma~\ref{basics-basis}, again through Lemmas~\ref{psh-E-sheaf} and~\ref{shv-sheaf} we can assume $X=U=\RR^i\times \sC(Z)$ is a basic.
After Lemmas~\ref{E-R-invariance},~\ref{psh-E-cones},~\ref{Shv-R-invt}, and~\ref{Shv-cones}, the second statement of Lemma~\ref{Gamma-compares} gives that the statement is true for $X=\RR^i\times \sC(Z)$ provided it is true for $Z$.  
The result follows by induction on $\depth(X)$, which we have assumed is finite.  
\end{proof}

\subsection{Presheaves on basics and sheaves: 
proofs of Theorem~\ref{sheaves-basics} and Theorem~\ref{exit-tangent}(Rel)}

Consider the ordinary full subcategory $\bscd\subset \snglrd$ consisting of the basic stratified spaces.  
For each sieve $\cU\subset \snglrd_{/X}$ on a conically smooth stratified space $X$ consider the pullback $\cU_{|\bscd} := \bscd  \underset{\snglrd}\times \cU \subset \bscd_{/X}$. 

\begin{defn}[$\shv(\bscd)$]
We denote by
\[
\Shv(\bscd)\subset \Psh(\bscd)
\]
the full subcategory consisting of those presheaves $\cF\colon (\bscd)^{\op} \to \spaces$ for which the canonical map $\cF(U) \xra{\simeq} \underset{(V\to U) \in  \cU_{|\bscd}}{\sf lim} \cF(V)$ is an equivalence for each covering sieve $\cU\subset \snglrd_{/U}$ of an object $U\in \bscd$.
\end{defn}

Like the adjunction~(\ref{extend-basics}), restriction and right Kan extension define an adjunction
\begin{equation}\label{iota-delta}
\iota^\ast \colon \Psh(\snglrd) \rightleftarrows \Psh(\bscd)\colon \iota_\ast~.
\end{equation}

\begin{lemma}\label{basics-basis}
The adjunction~(\ref{iota-delta}) restricts to an equivalence of $\infty$-categories $\Shv(\snglrd) \simeq \Shv(\bscd)$.  
\end{lemma}

\begin{proof}
Proposition~\ref{prop.basics-form-a-basis} gives that $\bscd$ is a basis for the standard Grothendieck topology on $\snglrd$.  
\end{proof}

Recall the adjunction
$
\iota^\ast \colon \Psh(\snglr) \rightleftarrows \Psh(\bsc)\colon \iota_\ast
$
of~(\ref{extend-basics}).  
Because $\iota$ is fully faithful, the right adjoint $\iota_\ast$ is fully faithful, and so the adjunction is a localization.  
Explicitly, the value of $\iota_\ast$ on a right fibration $\cE\to \bsc$ is the assignment
\[
X\mapsto \Map_{\bsc}\bigl(\Ent(X),\cE\bigr)~.
\]

\begin{proof}[Proof of Theorem~\ref{sheaves-basics}]
Let $\cF\in \Psh(\snglr)$ be a presheaf.  
We must argue that the following two conditions on $\cF$ are equivalent:
\begin{enumerate}
\item The restriction $\iota^\ast \cF$ is a sheaf.

\item The unit map $\cF\xra{\simeq} \iota_\ast \iota^\ast \cF$ is an equivalence.
\end{enumerate}
Lemma~\ref{E-covers} gives that $ \iota^\ast \iota^\ast \iota_\ast \cF$ is a sheaf.  
And so~(2) implies~(1).  
Suppose $\iota^\ast\cF$ is a sheaf.  
After Lemma~\ref{basics-basis}, then the map $\cF \to \iota_\ast \iota^\ast \cF$ is an equivalence if and only if $\iota^\ast \cF \to  \iota^\ast \iota_\ast \iota^\ast \cF$ is an equivalence.
Because $\iota^\ast \iota_\ast \simeq {\sf id}_{\bsc}$, (1) implies~(2) 
\end{proof}

\begin{proof}[Sketch proof of Theorem~\ref{exit-tangent}(Rel)]
This proof is identical to that of Theorem~\ref{sheaves-basics} just given, replacing $\bsc$ and $\snglr$ by $\bsc_{/X}$ and $\snglr_{/X}$, respectively, as with their discrete versions.  
We leave the details to the reader.
\end{proof}

We are now prepared to prove Theorem~\ref{theorem.structured-versions}, which we now recall for the reader's convenience.
\begin{theorem}[Theorem~\ref{theorem.structured-versions}]
The following statements are true concerning an $\infty$-category of basics $\cB$.
\begin{enumerate}
\item There is a natural equivalence of $\infty$-categories
 $
  \Shv\bigl(\mfld(\cB)\bigr) ~\simeq~\Psh(\cB)~.
 $
\item For each $\cB$-manifold $X= (X\xra{S}P, \cA,g)$ there are natural equivalences of $\infty$-categories
 $
  \Shv^{\sf cbl}(X)~\underset{\rm Cbl}\simeq~\Psh\bigl(\Ent(X)\bigr)~\underset{\rm Rel}\simeq~\Shv\bigl(\mfld(\cB)_{/X}\bigr)~.  
 $
\end{enumerate}
\end{theorem}

We first make an observation, which follows immediately from $\cB\to \Bsc$, and  $\mfld(\cB)\to \snglr$, being right fibrations for each $\infty$-category of basics $\cB$. 
\begin{observation}\label{E-no-B}
For each $\infty$-category of basics $\cB$, and each $\cB$-manifold $(X,g)$, the canonical projections
\[
\mfld(\cB)_{/(X,g)} \xra{~\simeq~} \snglr_{/X}    \qquad \text{ and } \qquad    \cB_{/(X,g)} \xra{~\simeq~} \Bsc_{/X}
\]
are equivalences of $\infty$-categories. 

\end{observation}

\begin{proof}[Proof of Theorem~\ref{structured-versions}] 
Notice that the square among $\infty$-categories
\[
\xymatrix{
\cB \ar[r]  \ar[d]
&
\mfld(\cB)  \ar[d]
\\
\Bsc  \ar[r]
&
\snglr
}
\]
is a pullback, and the vertical functors are right fibrations.  
After Theorem~\ref{theorem.sheaves-basics}, Statement~(1) follows quickly from the definition of $\Shv\bigl(\mfld(\cB)\bigr)$.

The proof of Statement~(2)(Cbl) is identical to that of Theorem~\ref{theorem.exit-tangent}(Cbl) after the identification $\cB_{/X} \simeq \Bsc_{/X}$ of Observation~\ref{E-no-B}. 

Statement~(2)(Rel) follows in the same way as Statement~(1), again using Observation~\ref{E-no-B}.  

\end{proof}

\subsection{Excisive functors: proof of Theorem~\ref{theorem.1-excision}}

\begin{lemma}\label{collar-pushout}
Let $\cB$ be an $\infty$-category of basics, $X$ be a $\cB$-manifold, and $X\xra{f}\ov{\RR}$ be a collar-gluing.
Then the diagram of right fibrations over $\cB$
\begin{equation}\label{diag.collar-pushout}
\xymatrix{
\Ent(\RR\times\partial)  \ar[r]  \ar[d]
&
\Ent(X_{\leq \infty})  \ar[d]
\\
\Ent(X_{\geq -\infty})  \ar[r]
&
\Ent(X)
}
\end{equation}
is a pushout diagram.
\end{lemma}
\begin{proof}
Any collar-gluing gives rise to an open cover, as depicted in diagram~(\ref{diag.collar-cover}), hence determines a sieve $\cU \subset \snglrd_{/X}$. The category $\cU$ receives a functor from the category 
$\bigl(X_{\geq -\infty}  \leftarrow \RR\times \partial \to X_{\leq \infty}\bigr)$, which is \emph{final}.
So
\[
\underset{O\in \cU} \colim~ \Ent(O)  ~{}~
\simeq~{}~ \Ent(X_{\geq -\infty}) \underset{\Ent(\RR\times \partial)} \coprod \Ent(X_{\leq \infty})~.
\]
We then apply Lemma~\ref{E-covers} to the lefthand side of this equivalence.
\end{proof}

We mirror Definition~\ref{def.finitary} with the following.
\begin{notation}[Seq-finite presheaves]\label{def.finite-presheaves}
Let $\cC$ a small $\infty$-category.
The $\infty$-category of \emph{seq\text{-}finite} presheaves, $\cC\subset \Psh_{\sf seq\text{-}fin}(\cC)\subset \Psh(\cC)$, is the smallest full subcategory containing the representable presheaves that is closed under finite pushouts and sequential colimits.  
\end{notation}

After Theorem~\ref{open-handles}(1), Lemma~\ref{collar-pushout} has the following immediate consequence.  
\begin{prop}\label{tangent-finite}
The tangent classifier factors 
\[
\tau \colon \snglr \longrightarrow \Psh_{\sf seq\text{-}fin}(\bsc)~\subset~\Psh(\bsc)
\]
through the seq-finite presheaves.  
\end{prop}

We now prove Theorem~\ref{theorem.1-excision}, which we restate for the reader's convenience.
\begin{theorem}[Theorem~\ref{theorem.1-excision}]
For each $\infty$-category of basics $\cB$, and for each $\infty$-category $\cC$ that admits finite pushouts and sequential colimits, there is a natural equivalence among $\infty$-categories
 \[
  \Fun(\cB,\cC)~\simeq~\Fun_{\amalg\text{-}{\sf exc}}\bigl(\mfld(\cB) , \cC\bigr)~.
 \]
\end{theorem}

\begin{proof}[Proof of Theorem~\ref{theorem.1-excision}]
The fully faithful inclusion $\iota\colon \bsc \to \snglr$ gives the (a priori, partially defined) adjunction
\[
\iota_!\colon \Fun(\bsc,\cC)\rightleftarrows \Fun(\snglr,\cC)\colon \iota^\ast
\]
with right adjoint given by restriction along $\iota$ and with left adjoint given by left Kan extension.
While $\iota^\ast$ is everywhere defined, we argue that $\iota_!$ is defined.
For this, notice that its value of $\iota_!$ on $\cB\xra{\cA} \cC$ is the composition
\[
\iota_! \cA\colon \snglr \xra{\Ent} \Psh_{\sf seq\text{-}fin}(\bsc) \xra{\Psh_{\sf seq\text{-}fin}(\cA)}\Psh_{\sf seq\text{-}fin}(\cC) \xra{\colim}\cC
\]
where the leftmost arrow is through Proposition~\ref{tangent-finite}, and the rightmost arrow exists because $\cC$ admits finite pushouts and sequential colimits, by hypothesis.

We will now show that this $(\iota_!,\iota^\ast)$-adjunction restricts to an equivalence 
\[
\Fun(\bsc,\cC) ~ \simeq~ \Fun_{\amalg\text{-}{\sf exc}}(\snglr,\cC)~.
\]  
The $(\iota_!, \iota^\ast)$-adjunction is a colocalization because $\iota$ is fully faithful.  
That $\iota_!$ factors through the $\amalg$-excisive functors follows from Lemma~\ref{E-covers} and then Lemma~\ref{collar-pushout}.  
It remains to show the counit of the $(\iota_!,\iota^\ast)$-adjunction is an equivalence.  
The full $\infty$-subcategory of $\snglr$ for which this is the case contains $\bsc$, and, manifest from the definition of $\amalg$-excisive, is closed under the formation of collar-gluings and sequential unions.  
The result follows from Theorem~\ref{open-handles}.  
\end{proof}

\clearpage

\part{Differential topology of conically smooth stratified spaces}

Here are three classical results in differential topology. In each of the following, let $M$ be a smooth manifold.
\begin{enumerate}
\item \textbf{Tubular neighborhoods:} For $W\subset M$ a properly embedded submanifold of $M$, there is a sphere bundle $\pi:\Link_{W}(M) \ra W$ associated to a tubular neighborhood of $W$.  Setting $\sC(\pi)\to W$ to be the fiberwise open cone on $\Link_{W}(M)\ra M$, the locus of cone points defines a section $W \to \sC(\pi)$, and one obtains an open embedding $\sC(\pi) \hookrightarrow M$ under $W$. 

\item \textbf{Handlebody decompositions:} If $M$ is the interior of a compact manifold with boundary, then $M$ admits a finite open handlebody decomposition.  
\item \textbf{Compact exhaustions:} There is a sequence of codimension-zero compact submanifolds with boundary $M_0 \subset M_1\subset \dots \subset M$ whose union is $M$.
\end{enumerate}
In Part 3 we establish versions of~(1),~(2), and~(3) for conically smooth stratified spaces.

To arrive at a notion of tubular neighborhoods for conically smooth stratified spaces, we first develop a notion of stratified spaces with corners, and show that every conically smooth stratified space can be resolved by a smooth manifold with corners. The rough idea is to replace a stratum $X_k \subset X$ by a spherical blow-up -- as an example, if $X_k$ is a smooth submanifold, we replace $X_k \subset X$ by the sphere bundle associated to its normal bundle. In general this results in a stratified space \emph{with boundary}. Iterating this spherical blow-up stratum by stratum, one arrives at an ordinary smooth manifold with corners. This story is the topic of \S\ref{sec.boundary-corners}, where we call this resolution \emph{unzipping}. Then the existence of collar neighborhoods of corners, which is classical, gives the existence of tubular neighborhoods of $X_k\subset X$, which is our stratified version of~(1). See Corollary~\ref{unzip-functor}.

We prove the stratified analogue of~(2) in Theorem~\ref{open-handles} by induction on depth, with base case that of ordinary smooth manifolds.
The inductive step relies on the decomposition
\[
\sC(\pi_k) \underset{\RR\times \Link_k(X)}\bigcup (X\smallsetminus X_k) \cong X
\]
which we now explain. 
Here, $X$ is a conically smooth stratified space of depth $k$, and $\Link_k(X)$ is the \emph{link} about the depth $k$ stratum of $X$. This is the generalization of the sphere bundle associated to a normal bundle, and is in fact a fiber bundle of stratified spaces as defined in Definition~\ref{def.maps}. The bundle is equipped with a projection $\Link_k(X) \xra{\pi_k} X_k$, and $\sC(\pi_k)$ is the fiberwise cone. In fact, the existence of an open embedding $\sC(\pi_k) \hookrightarrow X$ is another way to state the stratified version of~(1).  
Finally, the existence of conically smooth proper functions $X\xra{f}\RR$ gives~(3) without much trouble.

\begin{remark}\label{pseudoisotopy}
We comment on our usage of conically smooth structures on stratified spaces. In the presence of smooth structures on ordinary manifolds, (1) follows easily from the existence of open tubular neighborhoods, and from the homomorphism $\sO(d)\times\sO(n-d)\to \Emb\bigl((\RR^d\subset \RR^n), (\RR^d\subset \RR^n)\bigr)$ being a homotopy equivalence (as seen by taking derivatives, then ortho-normalizing). (2) follows easily through Morse theory, and (3) follows easily from the existence of proper smooth functions, and transversality.  
Without the presence of smooth structures, work of Kirby--Siebenmann~\cite{kirbysieb} and others accomplishes (2) and (3); however there is a genuine obstruction to (1). This is essentially because the continuous homomorphism ${\sf Homeo}(\DD^n) \xra{\nsim} {\sf Homeo}(\RR^n)$ is not a weak homotopy equivalence (see~\cite{browder}); in general the situation is marbled with pseudo-isotopy.

Accordingly, we acknowledge that much of our use of $C^\infty$ structures on stratified spaces is merely for convenience until the aspects which give us the stratified version of (1) (and we further emphasize that this can be accomplished only in the presence of \emph{$C^\infty$ structures along strata} (see Definition~\ref{def:cone-sm})).
Because of the inductive nature of stratified spaces, our methods are such that the singular version of (2), and thereafter the practice of using the singular version of (3), relies on the singular version of (1).  
Through this line of conclusion, we see the results of this article as depending fundamentally on conically smooth structures (at least along strata) of stratified spaces -- it is this local regularity that prompts the care we have taken in this article.  

\end{remark}

\section{Resolving singularities to corners}\label{sec.boundary-corners}

\subsection{Partitions of unity and generalities on bundles}\label{sec.bdl-unity}
We generalize some familiar notions from classical differential topology to the stratified setting. This is useful for several reasons: for instance, our main result on resolutions of singularity, Theorem~\ref{tot-unzip}, identifies the links of singularities as a {\em bundle} of stratified spaces (see Definition~\ref{def.maps}).

We quickly record two facts about stratified spaces which are akin to standard facts in differential topology.  
\begin{lemma}[Conically smooth partitions of unity]\label{part-o-1}
Every conically smooth stratified space admits a conically smooth partition of unity.
\end{lemma}

\begin{proof}
This follows the classical arguments founded on the existence of smooth bump functions -- the essential observations are these:
\begin{itemize}
\item Conically smooth open embeddings $U\hookrightarrow X$ from basics to a stratified space, form a basis for the underlying topology of $X$ -- this is Proposition~\ref{prop.basics-form-a-basis}.
\item For any $\epsilon>0$, there is a conically smooth map $\varphi'\colon \RR_{\geq 0} \to [0,1]$ which is $1$ on $[0,\epsilon/2]$ and $0$ outside $[0,\epsilon)$.  So the composition $\varphi\colon \RR^i\times \sC(Z) \approx \sC(S^{i-1}) \times \sC(Z) \approx \sC(S^{i-1}\star Z) \xra{\sf pr} \RR_{\geq 0} \xra{\varphi'} [0,1]$ is a conically smooth map whose support is in an $\epsilon$-neighborhood of $0 \in U$.  
\end{itemize}
The remaining points are typical given that $C^0$ stratified spaces are paracompact, by definition. 
\end{proof}

\begin{lemma}\label{propers}
For any conically smooth stratified space $X$, there exists a conically smooth proper map $f\colon X \to \RR$.
\end{lemma}

\begin{proof}
The statement is true for $X = U = \RR^i\times \sC(Z)$ a basic because it is true for $\RR^i$ and $\sC(Z)$ in a standard way.  
Choose a cover $\{U_\alpha \hookrightarrow X\}$ of $X$ by basics.  
Choose a conically smooth proper map $f_\alpha \colon U_\alpha \to \RR$ for each $\alpha$.
Choose a conically smooth partition of unity $\varphi_\alpha$ subordinate to this cover.  
The expression $f := \sum_\alpha \varphi_\alpha \cdot f_\alpha$ is a conically smooth proper map $X\to \RR$.  
\end{proof}

In what follows, recall the notion of a conically smooth fiber bundle from Definition~\ref{def.maps}.

\begin{lemma}\label{bundles-over-basics}
Let $E\xra{\pi} B$ be a conically smooth fiber bundle.
The following statements are true:
\begin{enumerate}
\item Let $B'\xra{f} B$ be a conically smooth map.
Then the pullback stratified space $f^\ast E := B'\underset{B}\times E$ exists, and the projection $f^\ast E \to B'$ is a conically smooth fiber bundle. 

\item
For each isomorphism $\alpha \colon B\cong B_0\times \RR^i\times \sC(Z)$, there is a conically smooth fiber bundle $E_0\to B_0$ and an isomorphism 
\[
\w{\alpha}\colon E~\cong~ E_0\times\RR^i\times\sC(Z)
\]
over $\alpha$ to a product conically smooth fiber bundle.  

\item 
For $E' \to E$ a conically smooth fiber bundle then the composition $E'\to E\to B$ is a conically smooth fiber bundle.  

\end{enumerate}

\end{lemma}

Taking $B_0 = \ast$, we obtain the following:

\begin{corollary}
Suppose $B\cong U$ is isomorphic to a basic.
Then there is an isomorphism
$
E~\cong~ U\times \pi^{-1}(0)
$
over $U$.  
\end{corollary}

\begin{proof}[Proof of Lemma~\ref{bundles-over-basics}.]
These arguments are conically smooth versions of the standard ones in differential topology.

Statement~(1) is local in $B'$, so we can assume $B'\xra{f} B$ factors through a basic $U\hookrightarrow B$ with respect to which there is an isomorphism $E_{|U} \cong E_0\times U$ over $U$.  
So we can assume $E\to B'$ is a projection $E_0\times B\to B$. 
This makes the problem obvious.  

We now prove statement~(2).
Choose a conically smooth map $\RR\times \bigl(\RR^i\times \sC(Z)\bigr) \xra{\varphi} \RR^i\times \sC(Z)$ having the following properties.
\begin{itemize}
\item There is an equality $\varphi_t(x) = x$ whenever $t$ belongs to $\RR_{\leq 0}$ or $x$ is near the origin $0\in \RR^i\times\sC(Z)$.  

\item The collection of images $\{{\sf Im}(\varphi_t)\mid t\in \RR\}$ is a basis for the topology about the origin $0\in \RR^i\times\sC(Z)$.  

\item There is an inclusion of the closure $\ov{{\sf Im}(\varphi_s)} \subset {\sf Im}(\varphi_t)$ whenever $0<t<s$, and this closure is compact. 

\end{itemize}
Choose a trivializing open cover $\cU_0$ of $B = B_0\times \RR^i\times \sC(Z)$.
Choose a conically smooth map $B_0\xra{\lambda} \RR$ such that, for each $b\in B_0$, the closure of the image $\ov{{\sf Im}(\varphi_{\lambda(b)})}$ lies in a member of $\cU_0$.
(Such a $\lambda$ exists, as in the proof of Lemma~\ref{part-o-1}.)
Consider the conically smooth map $\Lambda\colon \RR\times \bigl(B_0\times \RR^i\times \sC(Z) \bigr) \to B_0\times\RR^i\times \sC(Z)$ over $B_0$ given by $\Lambda_t\colon (b,x)\mapsto \bigl(b,\varphi_{t\lambda(b)}(x)\bigr)$.
Consider the pullback $\Lambda^\ast E \to \RR\times \bigl(B_0\times \RR^i\times \sC(Z)\bigr)$.  
From the choice of $\lambda$, in a standard way there is an isomorphism 
\[
E_{|B_0\times\{0\}}\times \RR^i\times \sC(Z)~\cong~\Lambda_1^\ast E
\]
over $B_0\times \RR^i\times \sC(Z)$.  
Because $\varphi_0$ is the identity map, this reduces statement~(2) to the case that $\RR^i\times \sC(Z)$ is simply $\RR$.  

Suppose $B_0$ is compact.  
Consider the largest connected subset $0\in D\subset \RR$ for which there exists an isomorphism $E_{|B_0\times D}\cong E_{|B_0\times\{0\}} \times D$.  
For $t$ in the closure of $D$, compactness of $B_0$ and the definition of a conically smooth fiber bundle grants the existence of $\epsilon>0$ for which there is an isomorphism $E_{|B_0\times(t-\epsilon,t+\epsilon)}\cong E_{|B_0\times\{t\}}\times (t-\epsilon,t+\epsilon)$.  
Without loss in generality, we may assume $t$ is the supremum of $D$, so $(t-\epsilon,t)\subset D$.  
Then we obtain isomorphisms that patch together as an isomorphism $E_{|B_0\times D\cup (t-\epsilon,t+\epsilon)}\cong E_{|B_0\times\{0\}} \times \bigl(D\cup (t-\epsilon,t+\epsilon)\bigr)$ over $D\cup (t-\epsilon,t+\epsilon)$.  
We conclude that $D$ is both open and closed, so $D=\RR$, thereby concluding (2) when $B_0$ is compact.  

For a non-compact $B_0$, choose a proper conically smooth map $B_0\xra{p}\RR$.
(Such a $p$ exists by way of Lemma~\ref{propers}.)
For each $a\in \ZZ$, the subspace $p^{-1}(a-1,a+1)\subset B_0$ lies in a compact subspace, and so there exists an isomorphism $\alpha_a\colon E_{|p^{-1}(a-1,a+1)\times \RR} \cong E_{|p^{-1}(a-1,a+1)\times \{0\}}\times \RR$ over $p^{-1}(a-1,a+1)\times \RR$.  
These $\alpha_a$ patch together as an isomorphism $\alpha_\infty\colon E\longrightarrow E_{|B_0\times\{0\}} \times \RR$ over $E\times \RR$.  

Statement~(3) is local in $B$, so we can assume $B$ is a basic.  
For this case, statement~(2) yields an isomorphism $E\cong E_0\times B$ over $B$.  
Another application of statement~(2) yields an isomorphism $E'\cong E'_0\times B$ over this isomorphism.  
In particular, $E'\to B$ is a conically smooth fiber bundle.  
\end{proof}

\subsection{Corners}
\label{section.corners}
The spherical blow-up of a properly embedded smooth submanifold $W\subset M$ results in a manifold with boundary.
A choice of collar for this boundary is a choice of a tubular neighborhood of $W$; and such a choice gives an identification of the boundary of this spherical blow-up, which we will take to be the \emph{link}, with the unit normal bundle $W\subset M$.  
This is an especially useful way to arrive at this link, for it is canonical. (See also~\cite{arone-kankaanrinta}.)

Likewise, the simultaneous spherical blow-up of a finite collection $\{W_\lambda\subset M\}$ of mutually transverse properly embedded smooth submanifolds results in a manifold with corners;
and the corners of this spherical blow-up, also referred to as links, are identified as the unit normal bundles of the various intersection loci -- these identifications are parametrized by collarings of corners, the spaces of which is weakly contractible.
Again, constructing links in this way is without choices.  
Toward proving the existence of tubular neighborhoods of singular strata, in this section we will introduce the notion of a stratified space \emph{with corners}.

\smallskip

Let $I$ be a finite set.
We will use the notation $\cP(I)$ for the poset of subsets of $I$, ordered by reverse inclusion: $T\leq T'$ means $T\supset T'$.
For each $T\subset I$ notice the inclusion $\cP(T)\subset \cP(I)$, and the natural identification $\cP(I)_{\leq T} \cong \cP(I\smallsetminus T)$.  
\begin{example}\label{corners}\label{example.corners}
Let $I$ be a finite set and let $T\subset I$ be a subset.
There is a stratification $(\RR_{>0})^{I\smallsetminus T}\times (\RR_{\geq 0})^T\to \cP(T)$ given by $(I\xra{x}\RR_{\geq 0})\mapsto x^{-1}(0)$.
The standard smooth structure on $\RR_{\geq 0}$ gives this stratified topological space a conically smooth structure.  
\end{example}

\begin{definition}[Topologically coCartesian]\label{top-cart}
Let $X=(X\to P)$ be a conically smooth stratified space.
A continuous map $X\xra{c}\cP(I)$ is \emph{topologically coCartesian} if it factors through a map of posets $P\xra{\ov{c}} \cP(I)$, and if the following condition is satisfied.  
\begin{itemize}
\item[]
For each subset $T\subset I$, the inclusion of the stratum $X_T\hookrightarrow X$ extends to a conically smooth open embedding
\[
(\RR_{\geq 0})^T\times X_T \hookrightarrow X
\]
over the inclusion of posets $\cP(T)\hookrightarrow \cP(I)$. 
\end{itemize}
\end{definition}

\begin{example}
If $P=\ast$ so $X$ is a smooth manifold, there is a unique topologically coCartesian map. It sends all of $X$ to $\emptyset \in \cP(I)$. 

Let $(X \to P)$ be a conically smooth stratified space, and $I=\ast$ be a singleton set. Then a topologically coCartesian map to $\cP(\ast)$ determines two subspaces of $X$: a boundary $X_\ast = c^{-1}(\ast)$, and an interior, $X_\emptyset = c^{-1}(\emptyset)$. Both are themselves stratified spaces.
\end{example}

\begin{remark}
We point out two easy facts about a topologically coCartesian map $X\xra{c} \cP(I)$.
\begin{itemize}
\item The factorization $P\xra{\ov{c}} \cP(I)$ is unique, because $X\to P$ is surjective, as noted in Remark~\ref{remark.surjective-stratification}. 
\item For each conically smooth open embedding $f\colon Y \hookrightarrow X$, the composition $Y \xra{f} X \xra{c} \cP(I)$ is a topologically coCartesian map -- this amounts to the fact that conically smooth open embeddings $(\RR_{\geq 0})^T\times Z \hookrightarrow (\RR_{\geq 0})^T\times Z$, under $Z$ and over the identity map of posets, form a base for the topology about $Z$ (after Proposition~\ref{basics.basis}).  
\end{itemize}
\end{remark}

\begin{definition}[$\mfld^{\langle I\rangle}$]\label{def.angle-corners}
Let $I$ be a finite set.
The $\Kan$-enriched category $\bsc^{\langle I\rangle}$ is as follows.
An object is a basic $U$ together with a topologically coCartesian map $U\xra{c} \cP(I)$. 
The Kan complex of morphisms from $\bigl(c\colon U\to P\xra{\ov{c}} \cP(I))$ to $\bigl(d\colon V\to P\xra{\ov{d}} \cP(I)\bigr)$ is the pullback
\[
\xymatrix{
\bsc^{\langle I\rangle}\bigl((U,c),(V,d)\bigr)  \ar[r]  \ar[d]
&
\bsc(U,V)  \ar[d]
\\
\Map_{\cP(I)}(P,Q)  \ar[r]
&
\Map(P,Q)
}
\]
where $\Map_{\cP(I)}(P,Q)$ consists of those maps respecting $\overline c$ and $\overline d$. Because the bottom horizontal map is an inclusion of sets, the top horizontal map of Kan complexes is an inclusion of components. 
Composition is evident.  
After the facts just above, it is straightforward to verify that the projection $\bsc^{\langle I \rangle} \to \bsc$ is an $\infty$-category of basics, thereby defining the $\infty$-category
\[
\snglr^{\langle I\rangle} := \mfld(\bsc^{\langle I \rangle})  
\]
of {\em conically smooth stratified spaces with $\langle I \rangle$-corners}. 
\end{definition}

Given a conically smooth stratified space with $\lag I \rag$-corners $\ov{X} = (X,c)$, the pre-image ${\sf int}(\ov{X}) := c^{-1}(\emptyset)$ is a stratified subspace of $X$. This defines a functor ${\sf int}: \snglr^{\lag I \rag} \to \snglr$.

\begin{definition}[$\mfld^{\langle I\rangle}(\cB)$]
For $\cB$ an $\infty$-category of basics, the $\infty$-category of {\em $\cB$-manifolds with $\langle I \rangle$-corners} is the pullback 
\[
\xymatrix{
\mfld^{\langle I\rangle}(\cB) \ar[r]\ar[d] 
& \snglr^{\langle I \rangle} \ar[d]^{\sf int} \\
\mfld(\cB) \ar[r]
& \snglr ~.
}
\]
We will typically denote $\cB$-manifolds with $\langle I \rangle$-corners as $\ov{X}$.  
When $I=\{\ast\}$, we use the notation
\[
\mfld^\partial(\cB) ~:=~\mfld^{\langle \{\ast\} \rangle}(\cB)
\]
and refer to its objects as $\cB$-manifolds \emph{with boundary}.  
\end{definition}

\begin{remark}
Let $I$ be a finite set.
Regard the poset $\cP(I)$ as a stratified topological space, by way of its identity map. 
The forgetful functor from conically smooth stratified spaces with $\langle I\rangle$-corners to stratified topological spaces factors through the overcategory of the space $\cP(I)$. This is used in Corollary \ref{unzip-functor}.
\end{remark}

\begin{definition}[Faces]\label{faces}
Let $I$ be a finite set.  
As a consequence of Proposition~\ref{prop.stratification-functorial}, for each subset $T\subset I$, the assignment $\ov{X}\mapsto (\ov{X}_{T}\into \ov{X}_{\leq T})$ determines a natural transformation of functors
\[
\partial_T\into \ov{\partial}_T\colon \mfld^{\langle I \rangle}(\cB) \longrightarrow \mfld(\cB)\subset \mfld^{\langle I\smallsetminus T\rangle}(\cB)~.
\]
We refer to the value $\partial_T \ov{X}$ as the \emph{$T$-face} (of $\ov{X}$), and the value $\ov{\partial}_T \ov{X}$ as the \emph{closed $T$-face} (of $\ov{X}$).
We use the special notation ${\sf int}(\ov{X}):= \partial_\emptyset \ov{X}$ which we refer to as the \emph{interior} of $\ov{X}$.  
\end{definition}

\begin{example}[Basic corners]\label{basic-corner}
We follow up on Example~\ref{corners}.
Let $I$ be a finite set, and let $T\subset I$ be a subset.
Let $X$ be a conically smooth stratified space.
Consider the stratified space $(\RR_{\geq 0})^T \times X$.
The composite continuous map
\[
c\colon (\RR_{\geq 0})^T \times X \to (\RR_{\geq 0})^T \to \cP(T) \to \cP(I)
\]
is manifestly topologically coCartesian, thereby witnessing this conically smooth stratified space as having $\langle I\rangle$-corners.   
\end{example}

\begin{remark}
We point out that the action of $\Sigma_I$ on the stratified space $(\RR_{\geq 0})^I$ does \emph{not} lift to an action as a stratified space \emph{with $\langle I\rangle$-corners}.  
This observation has the consequence that, there are stratified spaces with ``permutable corners,'' in the sense of Example~\ref{example.perm.corners}, that do not admit a $\langle I\rangle$-structure.
\end{remark}

\begin{prop}[Basic corners form a basis]\label{corner-basis}
Let $I$ be a finite set.
Let $\ov{X} = (X,c)$ be a stratified space with $\langle I \rangle$-corners.
Then the collection
\[
\Big\{\varphi\bigl( (\RR_{\geq 0})^T\times U\bigr)\Big\}
\]
is a basis for the topology of $X$ --
this collection is indexed by subsets $T\subset I$, basics $U$, and morphisms  $\varphi\colon (\RR_{\geq 0})^T\times U \hookrightarrow X$ of conically smooth stratified spaces with $\langle I \rangle$-corners (after Example~\ref{basic-corner}).
\end{prop}

\begin{proof}
This is immediate after Proposition~\ref{basics.basis} and Definition~\ref{top-cart} of topologically coCartesian.  
\end{proof}

\begin{example}\label{example.n-corners-manifold}
Proposition~\ref{corner-basis} immediately gives some familiar identifications.
In the case $I=\emptyset$, there is an identification $\mfld^{\langle \emptyset \rangle}(\cB) \simeq \mfld(\cB)$.  
In the case $I=\ast$, and $\cB = \sD$, there is an identification $\mfld^{\langle \ast\rangle}(\sD_n) \simeq \mfld^\partial_n$, which we use to justify the $\partial$-superscript notation. 

More generally, as the case $I=\{1,\dots,n\}$ and $\sB=\sD_n$, there is an identification $\mfld^{\langle I \rangle}_n \simeq \mfld_{\langle n\rangle}$ with $\langle n \rangle$-manifolds in the sense of~\cite{laures} (see Example~\ref{example.perm.corners}), which we use to justify the $\langle-\rangle$-notation.  
In particular, every point of a manifold with $\lag I \rag$-corners has a neighborhood covered by a chart of the form $\RR_{\geq 0}^T \times \RR^{n - |T|}$ for some $T \subset I$. 
\end{example}

\begin{example}\label{intersection-corners}
Let $\w{M}$ be an ordinary smooth $n$-manifold.
Consider an \emph{ordered} pair of transverse codimension zero submanifolds with boundary $W_1\pitchfork W_2\subset \w{M}$.
Then the intersection $\ov{X}= W_1\cap W_2\subset \w{M}$ canonically inherits the structure of an $n$-manifold with $\langle \{1,2\}\rangle$-corners.  
Here are examples of some of its faces: $\partial_{\{1,2\}} \ov{X} = \partial W_1\cap \partial W_2$, and $\partial_{\{1\}} \ov{X} = \partial W_1 \cap {\sf int}(W_2)$.  
\end{example}

\begin{observation}[Products]\label{corners-product}
Let $I$ and $J$ be finite sets.
The canonical map $\cP(I\sqcup J) \xra{\cong} \cP(I)\times \cP(J)$ is an isomorphism.  
There results a functor
\[
\snglr^{\langle I \rangle}\times \snglr^{\langle J\rangle} \xra{~\times~} \snglr^{\langle I \sqcup J \rangle}
\]
canonically lifting the product functor $\snglr\times \snglr \xra{\times} \snglr$, and that restricts to a functor $\bsc^{\langle I \rangle} \times \bsc^{\langle J\rangle} \to \bsc^{\langle I \sqcup J\rangle}$.  
\end{observation}

\begin{definition}[Interior depth]\label{def:corner-depth}
Let $\ov{X}$ be a conically smooth stratified space with corners.  
We define the \emph{interior depth} of $\ov{X}$ as 
	\[
	\depth^{\lag - \rag}(\ov{X}) := \depth\bigl({\sf int}(\ov{X})\bigr) 
	\]
	the depth of the interior.  
The \emph{depth $\langle k \rangle$} stratum $\ov{X}_{\langle k \rangle} \subset \ov{X}$ is the largest sub-stratified space with corners whose interior ${\sf int}(\ov{X}_{\langle k \rangle}) = {\sf int}(\ov{X})_{k}$ is the depth $k$ stratum of the interior.  
For $k\geq -1$, we use the notation
	\[
	\snglr^{\lag I \rag}_{\leq \lag k\rag} ~\subset ~ \snglr^{\lag I \rag}
	\]
for the full subcategory consisting of those conically smooth stratified spaces with interior depth $\leq k$.  
\end{definition}

\begin{warning}
The interior depth of a conically smooth stratified space with corners $\ov{X}=(\ov{X},c)$ does \emph{not} typically agree with the depth of the underlying stratified space $\ov{X}$.  
For instance, as a stratified space, $(\RR_{\geq 0})^I$ has depth $|I|$, the cardinality of $I$, while as a stratified space with corners via Example~\ref{basic-corner}, it has interior depth zero.  
\end{warning}

We will make use of interior depth to induct on it, which is premised on the following example.  
\begin{example}
Let $U = \RR^i\times(\RR_{\geq 0})^{I}\times \sC(Z)$ be a basic with corners.
Then $\depth^{\lag -\rag}(U) = {\sf dim}(Z) +1$.  
Consider the stratified space with corners $\w{U}:=\RR^i\times (\RR_{\geq 0})^{I}\times \RR_{\geq 0}\times Z$.
Then 
\[
\depth^{\lag-\rag}(\w{U}) = \depth^{\lag-\rag}(Z) \leq {\sf dim}(Z)< \depth^{\lag-\rag}(U)~.
\]   
\end{example}

\begin{example}
Let $\ov{Y} = \RR^I_{\geq 0} \times X$ be a stratified space with $\lag I \rag$-corners as in Example~\ref{basic-corner}. Then $\ov{Y}_{\lag  k \rag} = \RR^I_{\geq 0} \times {\sf int}(X)_{\geq k}$, where ${\sf int}(X)_{\geq k}$ is the sub-stratified-space of ${\sf int}(X)$ of depth $\geq k$.
\end{example}

\subsection{Unzipping}\label{sec:unzip}
Here we use conical smoothness to associate to each $n$-dimensional conically smooth stratified space an $n$-manifold with corners in the sense of~\cite{laures} (see Example~\ref{example.mfld-corners}).
We call this association \emph{unzipping}, and see it as a functorial resolution of singularities. 
This construction appears very close to that previously introduced in~\cite{almp}, which is based on unpublished work of Richard Melrose.
As we will see later on, unzipping serves as our conduit for efficiently transferring results in classical differential topology to the conically smooth setting. 

The basic picture is that the conically smooth stratified space $\sC(Z)$ admits a conically smooth map from $Z \times \RR_{\geq 0}$ collapsing $Z \times \{0\}$ to the cone point. We view this map as resolving the singularity at the cone point by a stratified space with boundary $Z \times \{0\}$. By iterating this process for loci of all depths, one obtains from $X$ a conically smooth stratified space $\unzip(X)$ whose local structure is no longer determined by arbitrary singularity types, but by $\lag I\rag$-corners. (See Figure~\ref{image.unzipping}.) 

We characterize this construction using a universal property in Definition~\ref{def.unzip}. In Lemma~\ref{unzip-facts}, we show that the local construction satisfies this universal property, and that the local construction can be glued together. We exhibit the functor $\unzip$ in Theorem~\ref{tot-unzip}. 

\begin{figure}
		\[
		\xy
		\xyimport(8,8)(0,0){\includegraphics[height=0.8in]{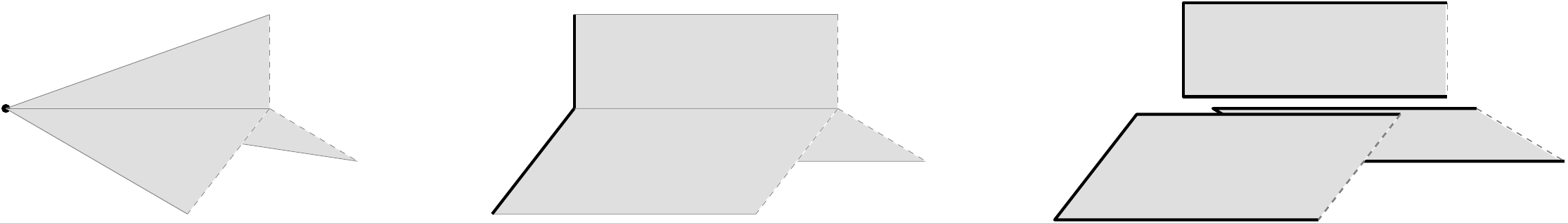}}
		\endxy
		\]
	\begin{image}[Unzipping]\label{image.unzipping}
The leftmost space is the open cone $X = \sC(Y)$ on a compact graph $Y$ with three exterior vertices and one interior vertex. By successively resolving the strata of greatest depth, can resolve $X$ by $\RR_{\geq 0} \times Y$, then by a space whose only singularities are $\lag I \rag$-corners.
	\end{image}
\end{figure}

\begin{definition}[$\unzip_k$]\label{def.unzip}
Let $k$ be a finite cardinality.
Let $\ov{X} = (X,c_X)$ be a conically smooth stratified space with $\lag I\rag$-corners, whose interior depth is at most $k$.  
Consider the category ${\sf Res}_{\ov{X}}^k$ defined as follows.
An object is a conically smooth stratified space with $\lag \ast\sqcup I\rag$-corners $(\ov{Y}\xra{c_Y} \cP(\ast \sqcup I))$, and a conically smooth map $p\colon \ov{Y} \to X$, satisfying the following conditions.
\begin{itemize}
\item The map $\ov{Y}\xra{p} X \xra{c_X} \cP(I)$ agrees with the projection $\ov{Y} \xra{c_Y} \cP(\ast \sqcup I) \cong \cP(\ast)\times \cP(I) \to \cP(I)$.  
\item The two subspaces
\[
\ov{\partial}_\ast \ov{Y} , p^{-1}\ov{X}_{\langle k\rangle}~\subset Y
\]
agree. 
\end{itemize}
We will typically denote such an object simply as $\ov{Y} \xra{p} X$.  
A morphism from $(Y,c_Y,p)$ to $(Y',c_{Y'}, p')$ is a conically smooth map $\ov{Y} \to \ov{Y}'$ over $\cP(\ast \sqcup I)\times X$. 
Composition is evident.  
A \emph{$k$-unzip} of $X$,
\[
\pi_k\colon \unzip_k(X) \longrightarrow X~,
\]
is a final object of ${\sf Res}_X^k$.  
\end{definition}

\begin{definition}[Links]\label{def.link}
For $\unzip_k(X)\xra{\pi_k} X$ a $k$-unzip, we will use the notation 
 \[
 \Link_k(X) \xra{\pi_k} X_k
 \]
for the restriction $(\pi_k)_|\colon \ov{\partial}_{\{\ast\}}\unzip_k(X) \to X_k$, and refer to the domain as a \emph{link} (of $X_k\subset X$).  
\end{definition}

\begin{example}
Let $W\subset M$ be a properly embedded smooth submanifold of an ordinary smooth manifold.  
Via Example~\ref{example:Ekn}, regard this data as a conically smooth stratified space, with singular locus $W$, and with depth equal to the codimension, $d$.  
Consider a smooth map $(0,1]\xra{\gamma} M$ that meets $W$ transversely at $1$, and only at $1$.  
This data is an example of an object of ${\sf Res}_{W\subset M}^d$.  
Another example of an object of ${\sf Res}_{W\subset M}^d$ is the spherical blow-up of $M$ along $W$, which is also a $d$-unzip. (See~\cite{arone-kankaanrinta}.) So a link of $W\subset M$ is given by a unit normal bundle.  
\end{example}

\begin{lemma}[$\unzip_k$ facts]\label{unzip-facts}
Let $I$ be a finite set, and let $X=(X,c)$ be a conically smooth stratified space with $\lag I\rag$-corners.
Let $k$ be a finite cardinality, and suppose the interior depth of $X$ is at most $k$.
\begin{enumerate}
\item If a $k$-unzip of $X$ exists, then it is unique, up to unique isomorphism in ${\sf Res}_X^k$.

\item If the interior depth of $X$ is less than $k$, then the identity map 
\[
X\xra{=} X~, 
\]
together with the corner structure $X\xra{c} \cP(I) \hookrightarrow \cP(\ast \sqcup I)$, is a $k$-unzip of $X$.  
\item Let $O\xra{f}X$ be a conically smooth map among stratified spaces with $\lag I\rag$-corners.
If $\unzip_k(X)\to X$ is a $k$-unzip of $X$,  
then the pullback stratified topological space 
\[
\unzip_k(O):= O\underset{X}\times \unzip_k(X)
\]
has a canonical structure of a conically smooth stratified space with $\cP(\ast \sqcup I)$ corners, and the projection $\unzip_k(O) \to O$ is a $k$-unzip of $O$.  
In particular, the full subcategory of $\snglrd^{\lag I\rag}_{\leq k}$ consisting of those $X$ that admit a $k$-unzip is a sieve; and on this sieve, 
\[
\unzip_k\colon \snglrd^{\lag I\rag}_{\leq \lag k\rag}~ \dashrightarrow ~ \snglrd^{\lag \ast \sqcup I\rag}_{<\lag k\rag}
\]
can be made into a covariant functor.

\item
Let $\cU$ be a hypercovering sieve of $X$.  
Suppose each $O\in \cU$ admits a $k$-unzip, $\unzip_k(O)\to O$.
Then the colimit stratified topological space 
\[
\unzip_k(X):= \underset{O\in \cU}\colim \unzip_k(O)
\]
admits a structure of a conically smooth stratified space with $\lag \ast \sqcup I\rag$-corners, and the map $\unzip_k(X) \to X$ is a $k$-unzip.  
\item 
If $\unzip_k(X) \to X$ is a $k$-unzip and $M$ is a manifold with $\lag J\rag$-corners,
then the product 
\[
M\times \unzip_k(X) \to M\times X 
\]
is a $k$-unzip.    
\item
For $Z$ a nonempty compact conically smooth stratified space of dimension $(k-1)$, the quotient map
\[
\RR_{\geq 0}\times Z \to \sC(Z)
\]
is a $k$-unzip.  
\end{enumerate}
\end{lemma}

Statements (5) and (6) say $k$-unzips exist for basic corners, as defined in Proposition~\ref{corner-basis}. Moreover, (4) tells us we can glue $k$-unzips together, so we conclude that a $k$-unzip exists for any $X$ as in the hypothesis of the Lemma. By the universal property for $\unzip_k$ as a final object, this assignment is functorial, so we have the following:

\begin{cor}
The $k$-unzip defines a functor
\[
\unzip_k\colon \snglrd^{\lag I\rag}_{\leq \lag k\rag} \to \snglrd^{\lag \ast \sqcup I\rag}_{<\lag k\rag}.
\]
\end{cor}

For the next result, we use the notation $\StTop_{/\cP(I)}$ for the category of stratified topological spaces whose stratifying poset is equipped with a map to $\cP(I)$. We also make use of the functor $\snglrd^{\lag I\rag}\to \StTop_{/\cP(I)}$ given by forgetting the atlas.

\begin{cor}\label{unzip-functor}
The diagram of categories
\[
\xymatrix{
&
\snglrd^{\lag \ast \sqcup I\rag}_{<\lag k\rag}  \ar[d]
\\
\snglrd^{\lag I\rag}_{\leq \lag k\rag}  \ar[r]  \ar[ur]^-{\unzip_k}
&
\StTop_{/\cP(I)}
}
\]
commutes up to a natural transformation $\unzip_k \xra{\pi_k} 1$ exhibiting $k$-unzips. 
Second, the restriction of this natural transformation 
\[
\Link_k(X) \xra{\pi_k} X_k
\]
as defined in Definition~\ref{def.link}, is a \emph{bundle} of \emph{compact} conically smooth stratified spaces with $\lag I \rag$-corners.   
Finally, the canonical transformation 
\[
(-)_k \underset{\Link_k(-)} \coprod \unzip_k(-) \xra{~\cong~} (-)
\]
is an isomorphism.  
\end{cor}

\begin{proof}[Proof of Corollary~\ref{unzip-functor}.]
Lemma~\ref{unzip-facts}(2)(5)(6) gives that $(\RR_{\geq 0})^T\times U$ admits a $k$-unzip for each basic $U$ of depth at most $k$ and each subset $T\subset I$.  
Through Lemma~\ref{corner-basis}, Lemma~\ref{unzip-facts}(4) gives the first statement of the corollary, after Lemma~\ref{unzip-facts}(3).  
By way of Lemma~\ref{unzip-facts}(3), the second statement is true if and only if it is true locally.
Using Lemma~\ref{corner-basis}, it is enough to verify the second statement for $(\RR_{\geq 0})^T\times U$ for $U$ a basic.
From Lemma~\ref{unzip-facts}(5) we need only verify the statement for $\sC(Z)$ for each compact conically smooth stratified space $Z$.
This case follows by inspecting the statement of Lemma~\ref{unzip-facts}(6).

Likewise, both sides of the arrow in the third statement send open covers to open covers, and the statement is true for $(\RR_{\geq 0})^T\times \RR^i\times \sC(Z)$, by inspection.  
\end{proof}

\begin{proof}[Proof of Lemma~\ref{unzip-facts}]
Statement~(1) is standard.
Statement~(2) follows by inspecting definitions.
Statement~(3) follows quickly from universal properties.
Statement~(4) follows because pullbacks of open covers are open covers, and conically smooth maps $C^\infty(-,X)$ between manifolds is a \emph{sheaf}.  
Statement~(5) follows using~(3) on the projection $M\times X\to X$.

Statement~(6) is farther away.
Let $Z$ be a compact $(k-1)$-dimensional conically smooth stratified space.
Let $\ov{Y} \xra{p} \sC(Z)$ be an object of ${\sf Res}_{\sC(Z)}^1$. 
By continuity, any two morphisms $f,g\colon \ov{Y} \to \RR_{\geq 0}\times Z$ in ${\sf Res}_{\sC(Z)}^1$ must be equal.
So we concentrate on the existence of such a morphism.
If the interior depth of $\ov{Y}$ is less than $k$, then $p$ factors through $\RR_{>0}\times Z$, and the morphism to $\RR_{\geq 0}\times Z$ is manifest.
So suppose the interior depth of $\ov{Y}$ is equal to $k$.  
It is enough to work locally, so we suppose $\ov{Y}= \RR_{\geq 0}\times U$ for $U$ a basic.

For $W$ a stratified topological space, there is the stratified homeomorphism $\gamma$ of $\RR_{>0}\times \RR_{\geq 0}\times W$ given by $\gamma(s,t,w)= (s,st,w)$, and in the case that $W\neq\emptyset$ it descends to a homeomorphism of the quotient $\RR_{>0}\times \sC(W)$.  
From Definition~\ref{def:cone-sm} of \emph{conical smoothness}, the dashed arrow in the diagram 
	\[
	\xymatrix{
	\RR_{\geq 0} \times (\RR_{\geq 0}\times U)  \ar@{-->}[rr]^{\w{\sD}p}
	&&
	\RR_{\geq 0} \times \sC(Z)
	\\
	\RR_{>0} \times (\RR_{\geq 0}\times U)  \ar[u] \ar[rr]^{\gamma^{-1} p\gamma}
	&&
	\RR_{>0} \times  \sC(Z). \ar[u]
	}
	\]
exists and is conically smooth.  
Let us explain the commutative diagram:
\[
\xymatrix{
\scriptstyle
\RR_{\geq 0} \times U  \ar@(d,l)[ddrrr]_{p}  \ar[rr]^-\iota
&&
\scriptstyle
\RR_{\geq 0}\times (\RR_{>0} \times U)  \ar[d]  \ar@{-->}[rr]
&&
\scriptstyle
\RR_{\geq 0} \times (\RR_{>0}\times Z)  \ar[d]   \ar[rr]^-{\sf pr}
&&
\scriptstyle
\RR_{\geq 0} \times Z  \ar@(d,r)[ddlll]^{\pi_k} 
\\
&&
\scriptstyle
\RR_{\geq 0} \times \RR_{\geq 0}\times U \ar[rr]^-{\w{\sD}p}  
&&
\scriptstyle
\RR_{\geq 0} \times \sC(Z)
&&
\scriptstyle
\\
&&
&
\scriptstyle
\sC(Z)
&
&.&
}
\]
The arrow labeled by $\iota$ is given by $\iota\colon (v,t,z)\mapsto \bigl(t,v,(1,z)\bigr)$.
The inner vertical arrows are the inclusions.  
All other arrows are as indicated.  
The dashed arrow exists because of the criterion $\{0\}\times U = p^{-1}(\ast)$;
that it satisfies the desired properties is manifest from its construction.
\end{proof}

Let $X=(X\to P)$ be a $n$-dimensional conically smooth stratified space.
After Lemma~\ref{dim-depth}, there is the map of posets $\depth\colon P\to [n]^{\op}$.  
Notice the map of posets ${\sf Max}\colon \cP(\un{n}) \to [n]^{\op}$.  
In the next result, for $\ov{M} = \bigl(\ov{M}\xra{c}\cP(\un{n})\bigr)$ an $\lag \un{n}\rag$-manifold, we will use the notation $\ov{\partial}_{\leq k}\ov{M} = c^{-1}\bigl({\sf Max}^{-1} k\bigr)\subset \ov{M}$.  Also recall the notation $\mfldd_{\lag n\rag}$ from Example~\ref{example.n-corners-manifold}.
\begin{theorem}[$\unzip$]\label{tot-unzip}
There is a functor 
\[
\pi\colon \unzip \colon \snglrd_n \longrightarrow \mfldd_{\lag n\rag}
\]
equipped with a natural transformation $\unzip \to 1$ by conically smooth maps that satisfies the following properties.
\begin{itemize}
\item For each $n$-dimensional $X=(X\to P)$, smooth maps among manifolds with $\lag n\rag$-corners $\ov{M} \to \unzip(X)$ are in bijection with conically smooth maps $\ov{M} \to X$ over the map of posets ${\sf Max}\colon \cP(\un{n}) \to [n]^{\op}$.

\item For each $0\leq k \leq n$, the restriction $\pi\colon \ov{\partial}_{\leq k} \unzip(X) \to X_k$ is a bundle of compact $\lag k-1\rag$-manifolds.  
\end{itemize}
\end{theorem}

\begin{proof}[Proof of Theorem~\ref{tot-unzip}.]
There are canonical identifications $\snglr_n \simeq \snglr_{\leq \lag n \rag,n}^{\lag \emptyset \rag}$, and $\mfld_n^{\lag n\rag} = \snglr_{\leq \lag 0\rag,n}^{\lag n\rag}$~.  
Define 
\[
\unzip: =  \unzip_1 \circ \cdots \circ \unzip_n~.
\] 
That $\unzip$ possesses the named property follows from Corollary~\ref{unzip-functor}.  
\end{proof}

\begin{lemma}\label{unzip-continuous}
There is a preferred lift of $\unzip$ and of $\unzip_k$ to a $\sf Kan$-enriched functor.  
\end{lemma}

\begin{proof}
The details are straightforward so we leave them to the interested reader.  
Let $X$ and $Y$ be conically smooth stratified spaces with $\lag I\rag$-corners, each with interior depth at most $k$.
Corollary~\ref{unzip-functor} gives the map
\[
\snglrd^{\lag I \rag}(\w{X},\w{Y}) \longrightarrow  \snglrd^{\lag \ast \sqcup I\rag}\bigl(\unzip_k(\w{X}) , \unzip_k(\w{Y})\bigr)
\]
functorially among conically smooth stratified spaces $\w{X}$ and $\w{Y}$ with $\lag I \rag$-corners of depth bounded above by $k$.
Apply this to the case $\w{X} = \Delta^\bullet_e\times X$ and $\w{Y}=\Delta^\bullet_e\times Y$.
Use Lemma~\ref{unzip-facts}(5), and notice that a map $\Delta^p_e\times \unzip_k(X)\xra{\unzip_k(f)}\Delta^p_e\times \unzip_k(Y)$ lies over $\Delta^p_e$ whenever $f$ does.  
\end{proof}

The construction of $\unzip_k(X)$ can be improved from one concerning $(X_k\subset X)$, where $X_k$ is the deepest stratum of a conically smooth stratified space $X$, to a construction $\unzip_Y(X)$ concerning a pair $(Y\subset X)$, where $Y$ is a closed constructible sub-stratified space. In what appears below, $\Link_Y(X)$ is the corresponding generalization of $\Link_k(X)$ from Definition~\ref{def.link}.

\begin{prop}\label{Y-unzip}
Let $Y\hookrightarrow X$ be a proper constructible embedding among conically smooth stratified spaces.
Then there is a stratified manifold with corners $\unzip_Y(X)$ fitting into a diagram among conically smooth stratified spaces (with corner structure forgotten)
\[
\xymatrix{
\Link_Y(X)  \ar[r]  \ar[d]_-{\pi_Y}
&
\unzip_Y(X)  \ar[d]
\\
Y  \ar[r]
&
X
}
\]
satisfying the following conditions.
\begin{itemize}
\item The diagram is a pullback, and is a pushout.

\item Each arrow is a conically smooth proper and weakly constructible bundle.

\item The map $\Link_{Y}(X) \to \unzip_Y(X)$ is the inclusion of the closed face of $\unzip_Y(X)$.

\item The restriction $\unzip_{Y}(X)_{|X\smallsetminus Y} \to X\smallsetminus Y$ is an isomorphism.

\end{itemize}
Furthermore, this diagram is unique, up to unique isomorphism.  

\end{prop}

\begin{proof}
This argument follows the same pattern as the other parts of this section, so we will be brief.

First, we extend the statement of the proposition for $Y\hookrightarrow X$ replaced by a proper constructible embedding among conically smooth stratified spaces \emph{with corners} for which 
\begin{equation}\label{close-interior}
\ov{Y\cap {\sf int}(X)} ~= ~Y~.
\end{equation}
So the closure of the intersection with the interior is agin $Y$.
The modification to the conclusion of the proposition is clear.  

In a routine way, the problem is local in $X$.
So we can assume $X$ has finite interior depth, and therefore $Y$ too has finite interior depth.
We proceed by induction on the interior depth of $Y$. 
Should $Y$ have depth zero, then $Y\subset X$ is identified as a (union of components of the) interior-deepest stratum of $X$.
Again, working locally in $X$, we can assume $Y$ is exactly the interior-deepest stratum of $X$.  
This case follows from Lemma~\ref{unzip-facts} and its Corollary~\ref{unzip-functor}.  

We proceed by induction on the interior depth of $Y$, and assume the interior depth of $Y$ is positive.  
Again, because the problem is local in $X$, we can assume $X =\RR^i\times \sC(Z)\times \RR_{\geq 0}^S$ is a basic with corners.
In this case the condition~(\ref{close-interior}) easily gives an identification $Y=\RR^i\times \sC(L)\times \RR^S$ for some closed constructible sub-stratified space $L\subset Z$.
Because the interior depth of $Y$ is positive, then, in particular, $L$ is not empty.  
In this case we declare $\Link_Y(X) := \RR_{\geq 0}\times \Link_{L}(Z)\times \RR^S$ and $\unzip_{Y}(X):=\RR_{\geq 0}\times \unzip_{L}(Z)\times \RR^S$.  
These declarations canonically fit into the desired square satisfying the listed properties.  
The functoriality of these declarations for open embeddings among $(Y\hookrightarrow X)$ follows the same argument proving Lemma~\ref{unzip-facts}.

\end{proof}

\begin{definition}\label{def.Y-link}
Let $Y\hookrightarrow X$ be a conically smooth embedding which is proper and constructible.  
The \emph{unzip along $Y\subset X$} is the conically smooth stratified space with corners $\unzip_Y(X)$ from Proposition~\ref{Y-unzip}, equipped with the proper and constructible bundle $\unzip_Y(X) \to X$.
The \emph{link of $Y\subset X$} is the stratified space $\Link_Y(X)$, given as the face of $\unzip_Y(X)$ as in Proposition~\ref{Y-unzip}. Note it is equipped with the proper and constructible bundle 
\[
\Link_Y(X) \xra{~\pi_Y~} Y~.
\]
\end{definition}

\section{Collar-gluing}\label{collargluingssection}
Gluing manifolds with boundary along common boundary is a useful way to construct manifolds, and it accommodates inductive arguments by dimension.  
In this subsection we give this construction for conically smooth stratified spaces, and show that a stratified space can be constructed by gluing a regular neighborhood of its deepest stratum to its remainder.  
We use this to prove Theorem~\ref{open-handles}, which states that conically smooth stratified spaces (like smooth manifolds) admit open handlebody decompositions.  

\subsection{Vector fields}
Using the functor $\unzip$, we give a notion of a vector field on a conically smooth stratified space.
This allows us to flow along vector fields for sufficiently small non-negative times.  
\emph{Parallel} vector fields are those for which flow by a given positive time is an \emph{automorphism}; so we think of these as infinitesimal automorphisms.  

\begin{notation} We fix the following notation:
\begin{itemize}
\item  For $M$ a smooth manifold,
$\Theta(M) := \Gamma(\sT M\to M)$ is the vector space of smooth vector fields on $M$.
So for each smooth map $E\xra{f} B$ there is a linear restriction map $\Theta(B)\xra{f^\ast} \Gamma(f^\ast \sT B \to E)$ as well as a linear derivative map $\Theta(E)\xra{\sD f} \Gamma(f^\ast \sT B \to E)$.  

\item Let $\ov{M}$ be a smooth manifold with corners as in Definition~\ref{example.mfld-corners}.  We write $\Theta(\ov{M})$ for the vector space of smooth vector fields on $\ov{M}$.  This sheaf is determined by declaring 
\[
\Theta\bigl(\RR^i\times \RR_{\geq 0}^j\bigr) ~:=~ \Gamma\bigl(\sT \RR^{i+j}_{|\RR^i\times \RR_{\geq 0}^j} \longrightarrow \RR^i\times \RR_{\geq 0}^j\bigr)
\]
to be the smooth sections of $\sT \RR^{i+j}$ defined near $\RR^i\times \RR_{\geq 0}^j$.  
We write 
\[
\Theta^{\parallel}(\ov{M})~\subset~ \Theta^{\sf in}(\ov{M})~ \subset ~\Theta(\ov{M})
\]
for the sub-vector spaces consisting of those vector fields that are tangent to each face $\partial_S\ov{M}$ of $\ov{M}$, and the sub-vector space consisting of those vector fields that point inward along each face.  Locally, these sub-sheaves are determined by declaring a section $\RR^i\times \RR_{\geq 0}^j\xra{V} \RR^{i+j}$ to belong to $\Theta^{\parallel}(\RR^i\times \RR_{\geq 0}^j)$, respectively to $\Theta^{\sf in}(\RR^i\times \RR_{\geq 0}^j)$, if, for each subset $S\subset\{1,\dots,j\}$, the restriction $V_{|\RR^i\times \RR_{>0}^S\times\{0\}}$ is tangent to $\RR^i\times \RR_{>0}^S\times\{0\}$, and respectively projects non-positively to each coordinate of $\RR^{S\smallsetminus j}$ the complementary coordinates.
In particular, for each face $\partial_S\ov{M}\subset \ov{M}$, there is a linear restriction map $(-)_{|\partial_S}\colon \Theta^{\parallel}(\ov{M}) \to \Theta(\partial_S \ov{M})$.  

\item For $X=(X\to P)$ a conically smooth stratified space, recall the constructible bundle $\unzip(X) \xra{\pi} X$ of~\S\ref{sec:unzip}.  For each linear subposet $S\subset P$, the restriction of $\pi$ to the face $\partial_S\unzip(X)\subset \unzip(X)$ is a smooth fiber bundle $\pi_{S} \colon \partial_S \unzip(X) \to X_{{\sf Min}(S)}$. (We set $X_{{\sf Min}(S)} = \emptyset$ for $S=\emptyset$.)

\end{itemize}
\end{notation}

\begin{definition}[Vector fields]\label{def.parallel-vect-flds}
Let $X=(X\to P)$ be a conically smooth stratified space.
The vector space $\Theta(X)$ of \emph{parallel vector fields on $X$} is the limit in the diagram among vector spaces
\[
\xymatrix{
\Theta(X)  \ar[dd]_-{(-)_{|X_-}}  \ar[rrrr]^-{\w{(-)}}
&&&&
\Theta^{\parallel}\bigl(\unzip(X)\bigr) \ar[d]^-{(-)_{|\partial_- }}
\\
&&&&  
\underset{S\underset{\rm linear}\subset P} \prod \Theta\bigl(\partial_S\unzip(X)\bigr)  \ar[d]^-{\sD\pi_{-}}
\\
\underset{p\in P}\prod \Theta(X_p) \ar[rrrr]^-{(\pi_-)^\ast}
&&&&
\underset{S\underset{\rm linear} \subset P} \prod \Gamma\bigl(\pi_S^\ast \sT X_{{\sf Min}(S)}\to \partial_S\unzip(X)\bigr)~.
}
\]
Likewise, for $\ov{X} = (\ov{X}\to P,c)$ a conically smooth stratified space with corners, 
the vector space $\Theta^{\sf in}(\ov{X})$ of \emph{inward vector fields on $\ov{X}$} is the limit in the diagram among vector spaces
\[
\xymatrix{
\Theta^{\sf in}(\ov{X})  \ar[dd]_-{(-)_{|\ov{X_-}}}  \ar[rrr]^-{\w{(-)}}\ar[dr]
&&&
\Theta^{{\sf in}}\bigl(\unzip(\ov{X})\bigr) \ar[d]^-{(-)_{|\partial_- }}
\\
&
\underset{\begin{subarray}{c} S \subset P \\ {\rm linear} \end{subarray}}\prod \Theta^{\sf in}\bigl(\partial_S\unzip(\ov{X})\bigr)  \ar[d]^-{\sD\pi_{-}}  \ar[rr]^-{\sD(\partial_-\hookrightarrow \unzip(\ov{X}))}
&&
\underset{\begin{subarray}{c} S \subset P \\ {\rm linear} \end{subarray}}\prod \Gamma\bigl(\sT \unzip(\ov{X})_{|\partial_S} \to \partial_S\bigr)
\\
\underset{p\in P}\prod \Theta^{\sf in}(\ov{X_p}) \ar[r]^-{(\pi_-)^\ast}
&
\underset{\begin{subarray}{c} S \subset P \\ {\rm linear} \end{subarray}}\prod \Gamma\bigl(\pi_S^\ast \sT \ov{X}_{{\sf Min}(S)}\to \partial_S\unzip(\ov{X})\bigr)
&&
.
}
\]

\end{definition}

Definition~\ref{def.parallel-vect-flds} is designed precisely for the following lemma:

\begin{lemma}[Flows]\label{flows}
For $X=(X\to P)$ a conically smooth stratified space, and for $V$ a parallel vector field on $X$, there is a conically smooth map
\[
\gamma^V \colon X\times \RR\dashrightarrow X~,
\]
defined on an open neighborhood of $X\times\{0\}$, satisfying the following conditions.
\begin{itemize}
\item For each conically smooth map $X\xra{\epsilon} \RR$, the composition 
\[
X\xra{({\sf id}_X,\epsilon)} X\times \RR \overset{\gamma^V}\dashrightarrow X
\]
is an isomorphism, provided it is defined.  

\item There is an equality 
\[
\gamma^V_t\bigl(\gamma^V_s(x)\bigr) = \gamma^V_{s+t}(x)
\]
whenever defined.  

\item For each point $x\in X$ belonging to a stratum $x\in X_p\subset X$,
\[
\frac{\sf d}{{\sf d}t}\gamma^V_t(x)_{|t=0}~ = ~V_{|X_p}(x)~.
\]  

\end{itemize}
Likewise, for $\ov{X}=(\ov{X}\to P,c)$ a conically smooth stratified space with corners, and for $V$ an inward vector field on $\ov{X}$, there is a conically smooth map 
\[
\gamma^V\colon \ov{X}\times \RR_{\geq 0} \dashrightarrow \ov{X}
\]
defined on an open neighborhood of $\ov{X}\times\{0\}$, satisfying the following conditions.  
\begin{itemize}
\item For each conically smooth map $\ov{X}\xra{\epsilon} \RR_{\geq 0}$, and for each linear subposet $S\subset P$, the composition 
\[
\partial_S\unzip(\ov{X})\xra{({\sf id},\epsilon_{|\partial_S})} \partial_S\unzip(\ov{X})\times \RR_{\geq 0} \overset{\gamma^V}\dashrightarrow \unzip(\ov{X})
\]
is a refinement onto its image, provided it is defined.

\item There is an equality 
$
\gamma^V_t\bigl(\gamma^V_s(x)\bigr) = \gamma^V_{s+t}(x)
$
whenever defined.  

\item For each point $x\in \ov{X}$ belonging to a stratum $x\in \ov{X_p}\subset \ov{X}$,
 \[
 \frac{{\sf d}}{{\sf d}t}\gamma^V_t(x)_{|t=0} = V_{|\ov{X}_p}(x)~.
 \]  

\end{itemize}

\end{lemma}

\begin{definition}
A map $\gamma^V$ as in Lemma \ref{flows} is a \emph{flow} of the vector field $V$. 
\end{definition}

\begin{proof}[Proof of Lemma~\ref{flows}]
This result is classical for smooth manifolds (compare Lemma 2.4 from \cite{morsetheory}), from the existence and uniqueness of smooth solutions to smooth first order ordinary differential equations with smooth dependence on initial conditions.
We thus have such a flow $\gamma^{V_{|X_p}}$ for each stratum $X_p\subset X$.   
We have likewise for smooth manifolds with corners, since we demand tangency along faces.
We thus have such a flow $\gamma^{\w{V}}$ for $\unzip(X)$, and, for each linear subposet $S\subset P$ it restricts to a flow $\gamma^{\w{V}_{|\partial_S}}$ on the face $\partial_S\unzip(X)$.  
By design, for each conically smooth map $X\xra{\epsilon} \RR$, and for each linear subposet $S\subset P$, the diagram
\[
\xymatrix{
X_{{\sf Min}(S)}  \ar[d]_-{\gamma^{V_{|X_p}}}  
&&
\partial_S\unzip(X)  \ar[rr]  \ar[ll]_-{\pi_S}     \ar[d]^-{\gamma^{\w{V}_{|\partial_S}}}
&&
\unzip(X)   \ar[d]^-{\gamma^{\w{V}}}  
\\
X_{{\sf Min}(S)} 
&&
\partial_S \unzip(X)   \ar[rr]  \ar[ll]_-{\pi_S}
&&
\unzip(X)
}
\]
commutes.
Iterating the pushout expression of Corollary~\ref{unzip-functor}, it follows that $\gamma^{\w{V}}\colon \unzip(X) \times \RR\dashrightarrow \unzip(X)$ descends to a conically smooth map $\gamma^V\colon X\times \RR\dashrightarrow X$.  This map satisfies the named conditions because $\gamma^{\w{V}}$ does. This proves the first statement.

The second statement follows the same logic, premised on the case that $\ov{X}$ is a smooth manifold with corners, which follows because it is the \emph{inward} vector fields that are considered.
\end{proof}

\begin{remark}[Tangent bundles for $\lag I \rag$-manifolds]
One advantage of unzipping a stratified space is that one can resolve the strata of $X$ to all be of the form $\RR_{\geq 0}^T \times \RR_{>0}^k$. In particular, one can define a tangent bundle for $\unzip(X)$.

Classically, a manifold $X$ with boundary has a tangent bundle -- in the interior, it agrees with the usual notion, and along the boundary, the tangent bundle is a direct sum of $\sT \del X$ (the tangent bundle of $\del X$ as a smooth manifold) with a trivial line bundle (its normal bundle). Moreover, the trivial line bundle comes with a natural orientation pointing inward.

The situation remains analogous for manifolds with $\lag I \rag$-corners. One can define a vector bundle on all of $\ov{X}$ such that, along each stratum $X_T \subset \ov{X}$, we have
 \[
 \sT \ov{X} _{| X_T} \cong \sT  X_T \oplus \un{\RR}^T.
 \]
Note that the trivial bundle $\RR^T$ is no longer oriented for $|T| \geq 2$, as the preferred inward pointing sections do not have a preferred order. The construction of the tangent bundle is as follows: First, note that we define the tangent bundle on $\RR_{\geq 0}^T$ to be the trivial rank $|T|$ vector bundle as usual. To understand the setting of manifolds with $\lag I \rag$-corners, one simply has to unravel the atlas structure on products of basics as in \S\ref{sec.products}. Let $f$ be a transition between charts of the form $\RR^{T}_{\geq 0}$. Then the transition maps from one trivializing bundle $\RR_{\geq 0}^T \times \RR^{|T|}$ to another is given fiberwise on the $\RR^{|T|}$ by values of the derivative $\sD f$. By the continuity of $\w{\sD}f$ and conical smoothness of the transition functions, one sees that these fiber-wise functions are conically smooth and linear on the fibers.

Note also that one has a notion of a Riemannian metric on an $\lag I \rag$-manifold once we have defined a tangent bundle.
\end{remark}

\begin{example}
Let $\ov{X} \in \mfld_{\lag 1 \rag}$. Then the interior of $\ov{X}$ is a smooth 1-manifold, and hence is equipped with the usual tangent bundle. We know that an open chart along the boundary $\del_*\ov{X}$, is given by conically smooth maps $\sC(*) \cong \RR_{\geq 0} \into \ov{X}$ for whom the derivative $\sD f$ is injective. (See Definition~\ref{defn.basics}). Recall that if a map $\RR_{\geq 0} \to \RR_{\geq 0}$ is conically smooth, it has a right derivative at the cone point $0$; the injectivity of $\sD f$ forces this right derivative to be a non-zero number, and continuity of $\w{\sD}f$ forces the number to be positive. Given two charts covering a boundary point of $\ov{X}$, we can write a transition function for the tangent bundle in the usual way, where all the transition functions along the boundary stratum are given by positive numbers. Hence we have a non-vanishing section of a trivial line bundle defined on $\del_* \ov{X}$. This trivial line bundle is the tangent bundle of $\del_\ast \ov{X}$.
\end{example}

\subsection{Tubular neighborhoods}
By way of the functor $\unzip$, we prove the existence of tubular neighborhoods of singular strata.  
\begin{lemma}\label{face-collars}
Let $\ov{X}$ be a $n$-dimensional conically smooth stratified space with $\langle I\rangle$-corners.
Let $\ast \in I$ be an element.
Then there is a morphism of $n$-dimensional conically smooth stratified spaces with $\langle I \rangle$-corners
\[
\RR_{\geq 0} \times \ov{\partial}_\ast \ov{X} \hookrightarrow \ov{X}
\]
under the closed $\ast$-face $\ov{\partial}_\ast \ov{X}$.  
\end{lemma}

\begin{cor}\label{corner-collars}
Let $\ov{X}$ be a $n$-dimensional conically smooth stratified space with $\langle I\rangle$-corners.
Let $T \subset I$.
Then there is a morphism of $n$-dimensional conically smooth stratified spaces with $\langle I \rangle$-corners
\[
(\RR_{\geq 0})^T \times \partial_T \ov{X} \hookrightarrow \ov{X}
\]
under the $T$-face $\partial_T\ov{X}$.  
\end{cor}

\begin{proof}[Proof of Lemma~\ref{face-collars}]
Suppose the interior depth of $\ov{X}$ is zero.
Then the result follows from classical methods.
For instance, consider the short exact sequence of vector bundles over the manifold with corners $\ov{\partial}_\ast \ov{X}$
\[
0 \longrightarrow \sT \ov{\partial}_\ast \ov{X} \longrightarrow \sT \ov{X}_{|\ov{\partial}_\ast \ov{X}} \longrightarrow N \ov{\partial}_\ast \ov{X} \longrightarrow 0
\]
-- this quotient is 1-dimensional and oriented.  
Choose a vector field $V\colon \ov{X} \dashrightarrow \sT \ov{X}$, defined in a neighborhood of the closed $\ast$-face $\ov{\partial}_\ast \ov{X}$, satisfying the following conditions.
\begin{enumerate}
\item For each $\ast\notin T \subset I$, $V$ is tangent to the $T$-face, which is to say the restriction $V_{|}\colon \partial_T\ov{X} \to \sT \ov{X}_{|\partial_T\ov{X}}$ factors through $\sT \partial_T\ov{X}$.  
\item The composition
\[
\ov{\partial}_\ast \ov{X} \xra{V_{|}} \sT \ov{X}_{|\ov{\partial}_\ast \ov{X}} \to N \ov{\partial}_\ast\ov{X}
\]
is positive. 
\end{enumerate} 
Such a choice is possible locally (i.e., for the case $\ov{X} \cong \RR^i\times (\RR_{\geq 0})^{n-i}$), and a global choice is facilitated from local choices through smooth partitions of unity.  
Flowing along $V$ gives a smooth map 
\[
\gamma^V\colon \RR_{\geq 0}\times \ov{\partial}_\ast \ov{X} \dashrightarrow \ov{X}
\]
defined in a neighborhood of $\ov{\partial}_\ast \ov{X}$.
Condition~(1) on $V$ grants that $\gamma^V$ is \emph{stratified}.
Condition~(2) on $V$ grants that $\gamma^V$ is a stratified \emph{open embedding} near $\ov{\partial}_\ast \ov{X}$.
Stratified smooth open self-embeddings $e\colon \RR_{\geq 0}\times \ov{\partial}_\ast \ov{X} \hookrightarrow \RR_{\geq 0}\times \ov{\partial}_\ast \ov{X}$ under $\ov{\partial}_\ast \ov{X}$ form a basis for the topology about $\ov{\partial}_\ast \ov{X}$.
So there exists such a self-embedding $e$ that composes with $\gamma^V$ as a desired smooth open embedding
\[
\gamma^V\circ e\colon \RR_{\geq 0}\times \ov{\partial}_\ast \ov{X} \hookrightarrow \ov{X}~.
\]

Now suppose $\ov{X} = (c\colon \ov{X}\to P \xra{\ov{c}} \cP(I))$ is a general $n$-dimensional conically smooth stratified space with $\langle I\rangle$-corners.
Consider  $\unzip(\ov{X})$ as an $n$-manifold with $\langle \un{n}\sqcup I\rangle$-corners. 
Then one can choose a vector field $V\colon \unzip(\ov{X})\dashrightarrow \sT \unzip(\ov{X})$, defined in a neighborhood of the closed $\ast$-face $\ov{\partial}_\ast \unzip(\ov{X})$, satisfying conditions (1) and (2) as in the above case in addition to the following condition.
\begin{enumerate}
\item[(3)] Let $p\in P$ and consider the stratum $X_p$ -- it is a smooth manifold. 
Consider the restriction $\pi_{|X_p}\colon \unzip(\ov{X})_{|X_p} \to X_p$ -- it is a conically smooth map which is a bundle of compact manifolds with corners.
There is diagram involving vector bundles
\[
\xymatrix{
\sT \unzip(\ov{X})_{|X_p}  \ar[rr]^-{\sD\pi_{|X_p}}    \ar[d]
&&
\sT X_p  \ar[d]
\\
\unzip(\ov{X})_{|X_p}  \ar[rr]^-\pi  \ar@(l,l)[u]^-{V_{|X_p}}
&&
X_p  \ar@(r,r)@{-->}[u]_-{\ov{V}_p}
}
\]
whose straight square commutes, and whose left upward arrow is the restriction of $V$.  
We require that the composition $\sD\pi_{|X_p}\circ V_{|X_p}$ factors through $X_p$, as indicated.
\end{enumerate}
Such a choice is possible locally over $\ov{X}$, and a global choice is facilitated through paracompactness and conically smooth partitions of unity (Lemma~\ref{part-o-1}).  
As in the above case, flowing along $V$ gives a map $\gamma^V\colon \RR_{\geq 0}\times \ov{\partial}_\ast \unzip(\ov{X})\dashrightarrow \unzip(\ov{X})$, defined in a neighborhood of $\ov{\partial}_\ast \unzip(\ov{X})$.  Condition~(1) gives that $\gamma^V$ is \emph{stratified}, where it is defined; condition~(2) gives that this map is a stratified \emph{open embedding} near $\ov{\partial}_\ast \unzip(\ov{X})$; condition~(3) gives that $\gamma^V$ descends to a stratified map
\[
\ov{\gamma}_V\colon \RR_{\geq 0}\times \ov{\partial}_\ast \ov{X} \dashrightarrow \ov{X}~,
\]
defined in a neighborhood of $\ov{\partial}_\ast \ov{X}$, where it is an open embedding.  
Stratified conically smooth open self-embeddings of $\RR_{\geq 0}\times \ov{X}$ under $\{0\}\times \ov{X}$ being a basis for the topology about $\ov{\partial}_\ast \ov{X}$, there is such a self-embedding $\ov{e}$ for which the composite
\[
\ov{\gamma}_V\circ \ov{e}\colon \RR_{\geq 0}\times \ov{\partial}_\ast \ov{X} \hookrightarrow \ov{X}
\]
is everywhere defined and is a conically smooth open embedding.  
This proves the result.  
\end{proof}

Let $X$ be a conically smooth stratified space of depth at most $k$.  
Recall from Definition~\ref{def.link} the conically smooth map $\Link_k(X)\xra{\pi_k} X_k$ -- Corollary~\ref{unzip-functor} states that it is a bundle of compact $(k-1)$-dimensional conically smooth stratified spaces.  
Example~\ref{bundle-examples} exhibits the fiberwise cone $\sC(\pi_k) \to X_k$, which is equipped with a conically smooth section.  
\begin{prop}[Tubular neighborhoods of strata]\label{tubular-neighborhoods}
Let $X$ be a conically smooth stratified space of dimension less or equal $n$ and of depth at most $k$.
There is a conically smooth open embedding
\[
\sC(\pi_k) ~\hookrightarrow~ X
\]
under $X_k$.  
\end{prop}

\begin{proof}
From Lemma~\ref{face-collars} there is the conically smooth open embedding 
\[
\RR_{\geq 0} \times \Link_k(X) \hookrightarrow \unzip_k(X)
\]
under $\Link_k(X)$.  
There results a conically smooth open embedding
\[
\sC(\pi_k) \cong X_k \underset{\Link_k(X)}\coprod \bigl(\RR_{\geq 0}\times \Link_k(X)\bigr)\hookrightarrow X_k\underset{\Link_k(X)}\coprod \bigl(\unzip_k(X)\smallsetminus X_k\bigr) \underset{\rm Cor~\ref{unzip-functor}}\cong X
\]
under $X_k$.
\end{proof}

\begin{remark}\label{C-0-failure}
We see Proposition~\ref{tubular-neighborhoods} as a focal statement that summarizes the attributes of \emph{conically smooth} stratified spaces.  
Indeed, while the statement of Proposition~\ref{tubular-neighborhoods} has an evident $C^0$ version, that statement is false.
For, consider a fiber bundle $E\to B$ among topological manifolds, equipped with a section, whose fibers are based-homeomorphic to $(0\in \RR^n)$.
Regard $(B\subset E)$ as a $C^0$ stratified space whose singularity locus is $B$.  
The fiber bundle $E\to B$ is classified by a map $B\to {\sf BHomeo}_0(\RR^n)$.
There is the subgroup ${\sf Homeo}_0(\DD^n)\subset {\sf Homeo}_0(\RR^n)$ consisting of those origin preserving homeomorphisms that extend (necessarily, uniquely) to a homeomorphism of the closed $n$-disk.  
Because 
\[
{\sf Homeo}_0(\RR^n)/{\sf Homeo}_0(\DD^n)~\simeq~ \Top(n)/{\sf Homeo}(S^{n-1})
\]
is not contractible, it is possible to choose $E\to B$ so that the spherical fibration $E\smallsetminus B \to B$ is not concordant to a sphere bundle.  
\end{remark}

We record an improvement of Proposition~\ref{tubular-neighborhoods} from a statement concerning a deepest stratum $X_d\subset X$ to a statement concerning a closed constructible subspace $Y\subset X$.
Recall from Definition~\ref{def.Y-link} the link $\Link_Y(X) \xra{\pi_Y} Y$.

\begin{prop}[Regular neighborhoods]\label{regular-neighborhoods}
Let $Y\hookrightarrow X$ be a proper constructible embedding among stratified spaces.
Then there is a conically smooth map
\[
\sC(\pi_Y) \longrightarrow X
\]
under $Y\hookrightarrow X$ such that
\begin{enumerate}
\item the map is a refinement onto its image,
\item the image is open, and
\item $\sC(\pi_Y)$ is the fiberwise open cone of the constructible bundle $\Link_Y(X) \xra{\pi_Y} Y$.  
\end{enumerate}
\end{prop}

\begin{proof}
Like in the argument proving Proposition~\ref{tubular-neighborhoods}, the statement is implied upon proving the existence of a conically smooth map 
\[
\Link_Y(X) \times \RR_{\geq 0} \longrightarrow \unzip_Y(X)
\]
under $\Link_Y(X)$ satisfying (1) and (2).   
Also like in that argument, by applying $\unzip$ the statement reduces to showing that, for $\ov{M}$ a smooth manifold with boundary, and for $\ov{\partial}\ov{M}\subset \ov{M}$ the union of its positive codimension faces, there is a conically smooth map
\[
\ov{\partial}\ov{M} \times \RR_{\geq 0} \longrightarrow \ov{M}
\]
under $\ov{\partial}\ov{M}$ satisfying (1) and (2).  
This again follows by the existence of positive-time flows along smooth inward vector fields for smooth manifolds with corners. 
To finish, we must verify that there exists such an inward vector field that is non-vanishing along each positive codimensional face.  
Conically smooth partitions of unity grant the existence of such a vector field, provided the existence of such for a basic corner $\RR^{S}\times \RR_{\geq 0}^T$.  
For this, take the vector field $\sum_{t\in T} \partial_t$.  

\end{proof}

The isotopy extension theorem is a standard consequence of the existence of regular neighborhoods of submanifolds from ordinary differential topology.
Here we state an analogous result for the stratified setting.
We use the notation $\Aut(X)\subset \Snglr(X,X)$ for the sub-Kan complex consisting of the automorphisms; and $\Emb(Z,X)\subset \Strat(Z,X)$ for the sub-Kan complex consisting of those conically smooth maps $Z\times\Delta^p_e \xra{f} X\times \Delta^p_e$ over $\Delta^p_e$ for which, for each $t\in \Delta^p_e$, the map $f_t\colon Z\to X$ is an embedding.  
\begin{theorem}[Isotopy extension theorem]\label{isotopy-extension}
Let $Z\hookrightarrow \w{X}$ be a proper and constructible embedding among stratified spaces with $Z$ compact, and let $\w{X}\to X$ be a refinement.  
Then both of the restriction maps
\begin{eqnarray}
\nonumber
\Aut(X) \longrightarrow \Emb(Z,X)
&
\qquad \text{ and }\qquad
&
\Emb(X,X)\longrightarrow \Emb(Z,X)
\end{eqnarray}
are Kan fibrations.  
\end{theorem}

\begin{proof}
We concern ourselves with the second map.

In a standard way, because ${\sf Strat}$ admits countable coproducts and finite products, the extended cosimplicial stratified space $\Delta^\bullet_e\colon \bdelta\to {\sf Strat}$ induces a functor from countable simplicial sets to simplicial conically smooth stratified spaces.  
We will employ this maneuver implicitly in the coming argument.
In doing so, for $K$ a finite simplicial set, by $k\in K$ we mean an element of a member of the diagram associated to $K$.  

Let $K$ be a finite simplicial set.  
It is sufficient to prove that every diagram among simplicial sets
\begin{equation}\label{kan-problem}
\xymatrix{
K \ar[rr]  \ar[d]_-{\{0\}} 
&&
\Emb(X,X)  \ar[d]
\\
K\times \Delta[1]  \ar[rr]  \ar@{-->}[urr]  
&&
\Emb(Z,X)
}
\end{equation}
can be filled.
Let us unwind this, through standard adjunctions.
Consider a solid diagram among diagrams of stratified spaces
\begin{equation}\label{of-K}
\xymatrix{
Z\times K  \ar[r]  \ar[d]_-{\{0\}}
&
X\times K  \ar[d]^-{\{0\}}  \ar@(u,u)[ddr]^-{e^X}
&
\\
Z\times K\times [0,1]   \ar[r]  \ar@(d,d)[drr]_-{e^Z}
&
X\times K\times[0,1]  \ar@{-->}[dr]^-{\w{e}}
&
\\
&&
X\times K\times[0,1]
}
\end{equation}
in which the horizontal arrows are the standard ones, and the diagonal arrows lie over $K\times[0,1]$ and, for each $(k,t)\in K\times[0,1]$, the restriction $e^Z_{k,t}$ is an embedding and the restriction $e^X_{k,0}$ is an isomorphism. 
The desired filler in~(\ref{kan-problem}) is supplied by finding a filler in~(\ref{of-K}) by an embedding over $K\times[0,1]$.  

For each $k\in K$, the expression
\[
\frac{\sf d}{{\sf d}t} \bigl((e^X_{k,0})^{-1}\circ e^Z_{k,t}\bigr)_{|t=s}
\]
defines a time-dependent parallel vector field $V^Z_{k,s}$ of $X$ which, at time $s$ is defined on the image $e^Z_{k,s}(Z)\subset X$.  
The flow of this time-dependent parallel vector field is a conically smooth map $\gamma^{V^Z}\colon X\times K\times[0,1]\dashrightarrow X\times K\times[0,1]$ over $K\times[0,1]$, which is defined on the image of $e^Z$.  
By construction, the composition with this flow recovers the conically smooth map
\[
Z\times K\times[0,1] \xra{e^Z_{|Z\times K\times\{0\}}\times {\sf id}_{[0,1]}} X\times K\times[0,1]\overset{\gamma^{V^Z}}\dashrightarrow X\times K\times[0,1]~,\qquad (z,k,t)\mapsto e^Z_{k,t}(z)
\]
over $K\times[0,1]$.

Using Proposition~\ref{regular-neighborhoods}, the assumption on $Z\subset X$ implies there is a regular neighborhood $e^Z(Z)\subset O\subset X\times K\times[0,1]$ over $K\times[0,1]$.  
In particular, there is an extension of this $K$-family of time-dependent vector fields $V^Z_{k,s}$ to a $K$-family of time-dependent parallel vector fields of $X$ defined on this neighborhood $O$.  
With a $K\times[0,1]$-family of partitions of unity subordinate to the $K\times[0,1]$-family of open covers $\{O_{|(k,s)},X\smallsetminus e^Z_{k,s}(Z)\}$, one can thereafter easily construct a $K$-family of time-dependent parallel vector fields $V_{k,s}$ on all of $X$ extending the named one along the image of $e^Z$ and which is zero outside of a compact neighborhood of this image.
The flow for this $K$-family of time-dependent parallel vector fields $V_{k,s}$ is a conically smooth map
\[
\gamma^V\colon X\times K\times[0,1] \dashrightarrow X\times K \times[0,1]
\]
over $K\times [0,1]$, which is a priori defined on an open neighborhood of $X\times K\times\{0\}$.  
The construction of $V_{k,s}$ grants that, in fact, this flow is everywhere-defined.  
The desired filler can be taken to be the composition
\[
\w{e} \colon X\times K\times [0,1] \xra{e^X\times {\sf id}_{[0,1]}} X\times K\times[0,1]\xra{~\gamma^V~}X\times K\times[0,1]~.
\]
This demonstrates a filler for~(\ref{of-K}), as desired, and therefore proves that the second map in the proposition is a Kan fibration.  

Because, necessarily, the flow $\gamma^V$ is an isomorphism, then $\w{e}$ too is an isomorphism provided $e^X$ is an isomorphism.  This above argument thus specializes to prove that the first map of the proposition too is a Kan fibration.  

\end{proof}

\subsection{Finitary stratified spaces}\label{sec.finitary}
We introduce a convenient class of \emph{finitary} structured stratified spaces -- these are those conically smooth stratified spaces that admit finite open handlebody decompositions.

\begin{notation}
We will use the notation $\ov{\RR}=[-\infty,\infty]$ for the closed interval.  It contains subspaces $\RR_{\geq -\infty}:= \{-\oo\}\cup\RR$, $\RR$, and $\RR_{\leq \infty} :=\RR\cup\{\oo\}$.  
\end{notation}

\begin{definition}\label{defn.collar-gluing}
A \emph{collar-gluing} is a continuous map $X\xra{f} \ov{\RR}$ from a conically smooth stratified space for which $f^{-1}(0)\subset X$ is a sub-stratified space, together with an isomorphism $\alpha\colon f^{-1}(\RR) \cong \RR\times f^{-1}(0)$ over $\RR$.   
Given such a collar-gluing, we will sometimes denote the sub-stratified spaces of $X$:
\[
X_{\geq -\infty}:=f^{-1}(\RR_{\geq -\infty})~,~\qquad \partial := f^{-1}(0)~,~\qquad X_{\leq \infty}:= f^{-1}(\RR_{\leq \infty})~.
\]  
We will often denote a collar-gluing $(X\xra{f}\ov{\RR}, \alpha)$ simply as $(X\xra{f} \ov{\RR})$, or even as $X=X_{\geq-\infty}\underset{\RR\times\partial} \bigcup X_{\leq \infty}$, depending on emphasis.  
\end{definition}

\begin{remark}\label{cg-as-cbl}
In a collar-gluing $(X\xra{f}\ov{\RR}, \alpha)$ the map $f$ is a \emph{pre-constructible bundle} in the sense of Definition~\ref{def.maps}.
Conversely, given a pre-constructible bundle $f\colon X\ra\ov{\RR}$, the collection of trivializations $\alpha$ forms a contractible groupoid.  In this sense, we regard a collar-gluing simply as a pre-constructible bundle to a closed interval.
\end{remark}

\begin{example}
Let $M$ be an ordinary manifold, and let $\partial \subset M$ be a hypersurface that separates $M$.
The choice of a collaring $\RR\times \partial \hookrightarrow M$ determines a collar-gluing $f\colon M\to \ov{\RR}$ by declaring $f$ to be the projection on the collaring.
\end{example}

\begin{example}
Let $X$ be a conically smooth stratified space of depth at most $k$.
After Proposition~\ref{tubular-neighborhoods}, there is a collar-gluing
\[
\sC(\pi_k) \underset{\RR\times \Link_k(X)} \bigcup X\smallsetminus X_k~ = ~ X~.
\]
\end{example}

A collar-gluing $X\xra{f}\ov{\RR}$ determines the a pullback diagram of stratified spaces
\begin{equation}\label{diag.collar-cover}
\xymatrix{
\RR\times \partial  \ar[r]  \ar[d]
&
X_{\leq \infty}  \ar[d]
\\
X_{\geq -\infty}  \ar[r]
&
X
}
\end{equation}
witnessing an open cover of $X$.

\begin{definition}\label{def.finitary}
The category of \emph{finitary} conically smooth stratified spaces, $\bscd \subset \snglrd^{\sf fin}\subset \snglrd$ is the smallest full subcategory containing the basics that is closed under the formation of collar-gluings, in the following sense:
\begin{itemize}
\item[]
Let $X\xra{f}\ov{\RR}$ be a collar-gluing.  
Suppose $f^{-1}(\RR_{\geq -\infty})$ and $f^{-1}(0)$ and $f^{-1}(\RR_{\leq \infty})$ are \emph{finitary}.
Then $X$ is \emph{finitary}.  
\end{itemize}
The intermediate $\oo$-category $\bsc\subset \snglr^{\sf fin}\subset \snglr$ is the image of $\snglrd^{\sf fin}$.
\end{definition}

We have the likewise definition given specified basics $\cB$.

\begin{definition}\label{defn.B-collar-finitary}
Let $\cB$ be an $\infty$-category of basics.  
A \emph{collar-gluing among $\cB$-manifolds} is a $\cB$-manifold $(X,g)$ together with a collar-gluing $X\xra{f}\ov{\RR}$ of its underlying stratified space.
The $\infty$-category of \emph{finitary} $\cB$-manifolds, $\cB\subset \mfld^{\sf fin}(\cB) \subset \mfld(\cB)$ is the smallest full $\infty$-subcategory containing $\cB$ that is closed under the formation of collar-gluings.  
\end{definition}

\begin{example}
Here some examples of stratified spaces that are finitary.
\begin{itemize}
\item Each basic $U = \RR^i \times \sC(Z)$ is finitary, by definition.
\item Because the product of basics is a basic, it follows that $X\times X'$ is finitary whenever $X$ and $X'$ are finitary. 
\item A compact manifold with boundary $M$ is finitary.
\end{itemize}
Here are some examples of stratified spaces that are \emph{not} finitary:
\begin{itemize}
\item An infinite disjoint union of Euclidean spaces is not finitary.  
\item An infinite genus surface is not finitary.
\item Consider $\RR^3$ with a trefoil knot in a small ball around each vector with integer coordinates.
By way of Example~\ref{example:Edn'} this data determines a conically smooth stratified space.
This stratified space is not finitary.
\end{itemize}
\end{example}

\begin{lemma}\label{finitary-closures}
Let $E\xra{\pi} X$ be a conically smooth fiber bundle. 
If $X$ is finitary and each fiber $\pi^{-1}(x)$ is finitary, then $E$ is finitary.
\end{lemma}

\begin{proof}
By inspection, products distribute over collar-gluings.  This verifies the assertion for the case of a trivial bundle.

Consider the collection of those stratified spaces $X$ for which the assertion is true.
Through Lemma~\ref{bundles-over-basics}(3), this collection contains the basics, because it is true for trivial bundles.
Let $X\xra{f}\ov{\RR}$ be a collar-gluing among stratified spaces that belong to this collection.
Let $E\xra{\pi} X$ be a conically smooth fiber bundle with finitary fibers.  
Lemma~\ref{bundles-over-basics}(2) gives an isomorphism $E_{|\RR} \cong \RR\times E_{|\partial}$ over the isomorphism $X_{|\RR} \cong \RR\times \partial$.
We thereby exhibit $E$ as a collar-gluing among finitary stratified spaces.
\end{proof}

We will prove the following result in the subsections that follow.  
\begin{theorem}[Open handlebody decompositions]\label{open-handles}
Let $X$ be a conically smooth stratified space.
\begin{enumerate}
\item Suppose there is a compact stratified space with corners $\ov{X}$ and an isomorphism ${\sf int}(\ov{X}) \cong X$ from its interior.
Then $X$ is finitary.  
\item There is a sequence of open subsets $X_0\subset X_1\subset \dots\subset X$ with $\underset{i\geq 0}\bigcup X_i = X$ and each $X_i$ finitary.  
\end{enumerate}
\end{theorem}

We draw an easy corollary of Theorem~\ref{open-handles}.
\begin{cor}
Let $X = (X\to P)$ be a conically smooth stratified space, and let $Q\subset P$ be a consecutive sub-poset.
If $X$ is finitary, then $X_{|Q}$ is finitary. 
\end{cor}

\begin{proof}
Necessarily, $X$ has finite depth.
We will proceed by induction on depth.  
If $X$ has depth zero, the result is tautological. 

Suppose $X = U = (\RR^i\times \sC(Z) \to \sC(P_0))$ is a basic. 
If $Q\subset \sC(P_0)$ contains the cone point, then $U_Q$ is again a basic, and is in particular finitary.
Suppose $Q\subset \sC(P_0)$ does not contain the cone point.  Then $Q\subset P$ and $U_Q \cong \RR^i\times \RR \times Z_Q$.  
After Lemma~\ref{finitary-closures}, $U_Q$ is finitary provided $Z_Q$ is finitary.  
By induction on depth, $Z_Q$ is finitary provided $Z$ is finitary.
Because of Theorem~\ref{open-handles}(1), compactness of $Z$ implies it is finitary.

Now suppose $X$ is general, but finitary.
Observe that a collar-gluing $\bigl(X\xra{f}\ov{\RR}, \RR\times f^{-1}(0) \cong f^{-1}(\RR)\bigr)$ restricts to a collar-gluing $\bigl(X_{|Q}\xra{f_|} \ov{\RR}, \RR\times f^{-1}(0)_{|Q} \cong f^{-1}(\RR)_{|Q}\bigr)$.  
So the assertion is true if and only if it is true for $X$ a basic, which was verified in the previous paragraph.
\end{proof}

\subsection{Proof of Theorem~\ref{open-handles}(1): handlebody decompositions}
Here we prove that compact stratified spaces with boundary admit finite open handlebody decompositions.  
More precisely, we prove Theorem~\ref{open-handles}(1) which asserts that the interior of a compact stratified space with boundary is \emph{finitary} in the sense of Definition~\ref{def.finitary}.  
\begin{proof}[Proof of Theorem~\ref{open-handles}(1)]
Let $\ov{X}$ be a compact stratified space with corners whose interior is identified as $X$.
By compactness, $n:={\sf dim}(X)$ and $k:=\depth(\ov{X})=\depth(X)$ are finite, and so we proceed by induction on $k$.

Suppose $k=0$, so that $X$ is an ordinary smooth manifold, and $\ov{X}$ is a compact manifold with $\langle \un{n} \rangle$-corners.
For this case fix a Riemannian metric on $\ov{X}$.
Let us use the terminology ``open disk''/``closed disk''/``sphere'' for a manifold with corners that is isomorphic to the standard open disk/closed disk/sphere in $\RR^n$ intersected with $\RR^i\times \RR_{\geq 0}^{n-i}$.  
We proceed by induction on the dimension, $n$.
Should $d=0$ then necessarily $X \cong (\RR^0)^{\sqcup I}$ is a finite set, and the assertion is easily verified.  
Consider the collection of smooth maps $f\colon \ov{X} \to [-1,1]$ for which 
\begin{itemize}
\item for each $n\notin T\subset \{1,\dots,n\}$, the $f_|\colon \partial_T\ov{X} \to [-1,1]$ is a Morse function, and the restriction of the gradient $(\nabla f)_{|\partial_T \ov{X}}$ is tangent to $\partial_T\ov{X}$,

\item for each $n\in T\subset\{1,\dots,n\}$, the restriction $f_{|\partial_T\ov{X}}$ is constant.  
\end{itemize}
Each member of this collection has a finite number of critical points;
choose a member $f$ that minimizes this number, $r$.

If $r=0$, then $\ov{X}\cong [-1,1]\times f^{-1}(0)$ as witnessed by gradient flow.
By induction on dimension $n$, ${\sf int}(f^{-1}(0))$ is finitary, and thereafter $X = (-1,1)\times{\sf int}(f^{-1}(0))$ too is finitary.  
Suppose $r=1$ and let $p\in \ov{X}$ be the unique critical point of $f$, with critical value $a\in (-1,1)$.
Consider the locus of those $x\in \ov{X}$ for which the gradient flow from $x$ limits to $p$ in either positive or negative time.
This locus intersects $f^{-1}{-1} \subset \partial_{\{n\}}\ov{X}$ as a $S\subset \partial_{\{n\}}\ov{X}$, and the normal bundle of this sphere is equipped with a trivialization.
Choose a closed tubular neighborhood $S\subset \nu \subset \partial_{\{n\}}\ov{X}$, and denote its sub-sphere bundle as $\nu_0\subset \partial_{\{n\}}\ov{X}$.
Write $C:= \partial_{\{n\}}\ov{X} \smallsetminus (\nu\smallsetminus \nu_0)$ for the complement of the interior -- it is a sub-manifold with corners. 
Gradient flow gives determines the product sub-manifolds with corners $\nu_0\times [-1,1]\subset \ov{X}$ and $C\times [-1,1]\subset\ov{X}$, as well as the sub-manifold with corners $D\subset \ov{X}$ that is the flow of $\nu$.
By construction, $D$ contains the critical point $p$ in its interior which is isomorphic as to an open disk, which is a basic manifold with corners.  
Now, choose a map $\ov{X} \xra{g}[-1,1]$ that restricts to a fiber bundle over $(-1,1)$, and such that $\nu_0\times [-1,1] = g^{-1}0$, thereby witnessing a collar-gluing 
$\ov{X} = \w{C}\times[-1,1]\underset{\nu_0\times [-1,1]}\bigsqcup \w{D}$
-- here we are denoting $D\subset \w{D}\subset \ov{X}$ and $C\subset \w{C}\subset \ov{X}$ for the open sub-manifolds with corners obtained by adjoining collars onto interior faces.
Because $\w{D}$ is a basic, we conclude this case of $r=1$ through the inductive hypothesis on $n$.

Suppose $r>1$.
We can assume there are two distinct critical values $a<b$ of $f$, otherwise locally perturb $f$ through Morse functions in a standard fashion.
Choose $a<e<b$.
By induction on $r$, each of $f^{-1}([-1,b))$, $f^{-1}(a,b) \cong (a,b)\times f^{-1}(e)$, and $f^{-1}((a,1])$ is finitary.
Because $X= f^{-1}([-1,b)) \underset{(a,b)\times f^{-1}(e)}\bigcup f^{-1}((a,1])$ is a collar-gluing, we conclude that $X$ is finitary.  
Now suppose $k>0$.  
Proposition~\ref{tubular-neighborhoods} gives the collar-gluing
\[
\sC(\pi_k) \underset{\RR\times \Link_k(X)} \bigcup (X\smallsetminus X_k)~\cong~ X~.
\]
We will explain that the collar-gluands are each finitely, from which it will follow that $X$ is finitary.  
\begin{itemize}
\item[$X\smallsetminus X_k$:]
Because the map $\unzip_k(\ov{X}) \to \ov{X}$ is proper, and the codomain is compact, then the domain, too, is compact.  
Because this codomain has strictly less interior depth than that of $\ov{X}$, then by induction on $k$, ${\sf int}\bigl(\unzip_k(\ov{X})\bigr)$ is finitary.  
Through the isomorphisms of stratified spaces ${\sf int}\bigl(\unzip_k(\ov{X})\bigr) = {\sf int}\bigl(\unzip_k(\ov{X}) \smallsetminus \partial_k\unzip_k(\ov{X})\bigr) \xra{\cong} X\smallsetminus X_k$ we conclude that $X\smallsetminus X_k$ is finitary.

\item[$\Link_k(X)$:] The deepest locus $\ov{X}_k\subset \ov{X}$ is a closed sub-stratified space with corners.  Because $\ov{X}$ is compact, $\ov{X}_k$ too is compact.  Because the depth of $\ov{X}_k$ is zero, then $\ov{X}_k$ is finitary, by the base case of the induction.  
There is the conically smooth fiber bundles $\Link_k(X) \xra{\pi_k} X_k$.
Because the fibers of $\pi_k$ are compact and have strictly less depth than $X$, then Lemma~\ref{finitary-closures} gives that $\Link_k(X)$ is finitary.  
\item[$\sC(\pi_k)$:] There is the fiberwise cone $\sC(\pi_k) \to X_k$.
Because the fibers of $\pi_k$ are compact, the fibers of $\sC(\pi_k) \to X_k$ are basics, and Lemma~\ref{finitary-closures} gives that $\sC(\pi_k)$ is finitary.  
\end{itemize}
\end{proof}

\subsection{Proof of Theorem~\ref{open-handles}(2): compact exhaustions}
We prove Theorem~\ref{open-handles}(2) which asserts that every stratified space is a sequential open union of finitary stratified spaces.
The arguments follow those in ordinary differential topology.

\begin{definition}\label{def:regular-value}
Let $X=(X\xra{S} P)$ be a conically smooth stratified space and let $f\colon X \to \RR$ be a conically smooth map.  Say $a\in \RR$ is a \emph{regular value for $f$} if, for each $p\in P$, $a$ is a regular value for the restriction $f_|\colon S^{-1}(p)\to\RR$. 
\end{definition}

\begin{lemma}\label{cut-manifold}
Let $f\colon X \to \RR$ be a conically smooth map with dimension ${\sf dim}(X)\leq n$.  
Let $a<b \in \RR$ be regular values for $f$.
Then the subspace $f^{-1}([a,b]) \subset X$ canonically inherits the structure of a conically smooth stratified space with boundary $\partial f^{-1}([a,b]) = f^{-1}(\{a,b\})$.  
\end{lemma}
\begin{proof}
Consider the composite map $\unzip(X) \xra{\pi} X \xra{f} \RR$ -- it is a smooth map from a manifold with $\langle \un{n}\rangle$-corners.  
Both $a$ and $b$ are regular values for $f\pi$; 
and it is classical that $(f\pi)^{-1}([a,b])\subset \unzip(X)$ has the structure of a manifold with $\langle \un{n}\rangle$-corners, with boundary.  
The quotient stratified topological space $f^{-1}([a,b]) = \pi\bigl((f\pi)^{-1}([a,b])\bigr)\subset X$ thus canonically inherits the structure of a stratified space with boundary.  
\end{proof}

\begin{lemma}\label{open-dense}
The set of regular values of a conically smooth proper map $f\colon X\to \RR$ is a dense subspace of $\RR$.
\end{lemma}
\begin{proof}
Write $X = (X \to P)$.  
From Lemma~\ref{dim-depth} there is the coarser stratification $\depth \colon X\to P \to \NN$.  
The statement is true if it is true for $f_{|(a,b)}\colon f^{-1}(a,b)\to (a,b)$ for each $a<b\in \RR$.  
Let $a<b\in \RR$.  
The set of regular values for $f_{|(a,b)}$ is the intersection of the set, indexed by $k\in \NN$, of regular values for the (ordinary) smooth map $f_|\colon f^{-1}(a,b) \cap \depth^{-1}(k) \to \RR$.  
Because $f$ is proper, there are only finitely many $k\in \NN$ for which $f^{-1}(a,b)\cap \depth^{-1}(k)$ is nonempty.  
The statement follows from Sard's theorem.  
\end{proof}

\begin{proof}[Proof of Theorem~\ref{open-handles}(2)]
Use Lemma~\ref{propers} and choose a conically smooth proper map $f\colon X \to \RR$.  
Choose a sequence of regular values $(a_i)$ and $(b_i)$ such that $a_{i+1} <a_i<0<b_i<b_{i+1}$ and $a_i \to -\infty$ and $b_i \to \infty$ -- that such sequences exist follows from Lemma~\ref{open-dense}.  
Through Lemma~\ref{cut-manifold}, the subspace $f^{-1}([a_i,b_i])\subset X$ has the canonical structure of a conically smooth stratified space with boundary, and it is compact because $f$ is proper.  
We conclude the argument by taking $X_i:=f^{-1}(a_i,b_i)\subset X$.  
\end{proof}

\end{document}